\documentclass[a4paper,reqno]{amsart}
\usepackage{amssymb}
\usepackage{amsmath}
\usepackage{amsthm}
\usepackage[dvipsnames]{xcolor}
\usepackage[shortlabels]{enumitem}
\usepackage{braket,comment,soul,cancel}
\usepackage{hyperref}
\usepackage{pdfpages,faktor}
\usepackage{graphicx,mathabx,bbm}
\usepackage{caption}
\usepackage{subcaption}
\usepackage{mathrsfs} 
\usepackage{tikz-cd} 
\usepackage{bm}
\usepackage{enumitem}
\usepackage{mathtools}
\usepackage{pgfplots}
\usepackage{import}
\usepackage{xifthen}
\usepackage{transparent}
\usepackage{xargs}

\usepackage[colorinlistoftodos,prependcaption,textsize=tiny]{todonotes}
\newcommandx{\unsure}[2][1=]{\todo[linecolor=red,backgroundcolor=red!25,bordercolor=red,#1]{#2}}
\newcommandx{\change}[2][1=]{\todo[linecolor=blue,backgroundcolor=blue!25,bordercolor=blue,#1]{#2}}
\newcommandx{\info}[2][1=]{\todo[linecolor=OliveGreen,backgroundcolor=OliveGreen!25,bordercolor=OliveGreen,#1]{#2}}
\newcommandx{\improvement}[2][1=]{\todo[linecolor=Plum,backgroundcolor=Plum!25,bordercolor=Plum,#1]{#2}}

\numberwithin{equation}{section}
\theoremstyle{plain}
\newtheorem{theorem}{Theorem}[section]

\newtheorem{prop}[theorem]{Proposition}
\newtheorem{lem}[theorem]{Lemma}
\newtheorem{cor}[theorem]{Corollary}

\newtheorem{example}{Example}
\newtheorem{question}[theorem]{Question}
\newtheorem*{question*}{Question}
\newtheorem{rmk}[theorem]{Remark}
\newtheorem{conj}[theorem]{Conjecture}

\theoremstyle{definition}
\newtheorem{defn}[theorem]{Definition}
\newtheorem{remark}[theorem]{Remark}

\theoremstyle{definition}
\newtheorem{thmx}{Theorem}

\newcommand{\R}{\mathbb{R}}

\newcommand{\C}{\mathbb{C}}

\newcommand{\Proj}{\mathbb{P}}

\newcommand{\Z}{\mathbb{Z}}
\newcommand{\Hyp}{\mathbb{H}}

\newcommand{\D}{\mathbb{D}}
\newcommand{\E}{\mathscr{E}}

\newcommand{\cG}{\mathcal{G}}
\newcommand{\cH}{\mathcal{H}}

\newcommand{\cK}{\mathcal{K}}
\newcommand{\cL}{\mathcal{L}}
\newcommand{\cD}{\mathcal{D}}
\newcommand{\cT}{\mathcal{T}}
\newcommand{\cU}{\mathcal{U}}
\newcommand{\RT}{\mathscr{T}}
\newcommand{\RV}{\mathscr{V}}

\DeclareMathOperator{\Rat}{Rat}
\DeclareMathOperator{\Cor}{Cor}

\DeclareMathOperator{\PSL}{PSL}

\DeclareMathOperator{\Res}{Res}

\DeclareMathOperator{\Aut}{Aut}
\DeclareMathOperator{\Int}{Int}

\pgfplotsset{compat=1.18} 

\numberwithin{figure}{section}

\makeatletter
\DeclareFontFamily{U}{tipa}{}
\DeclareFontShape{U}{tipa}{m}{n}{<->tipa10}{}
\newcommand{\arc@char}{{\usefont{U}{tipa}{m}{n}\symbol{62}}}

\newcommand{\arc}[1]{\mathpalette\arc@arc{#1}}

\newcommand{\arc@arc}[2]{
  \sbox0{$\m@th#1#2$}
  \vbox{
    \hbox{\resizebox{\wd0}{\height}{\arc@char}}
    \nointerlineskip
    \box0
  }
}
\makeatother

\date{\today}

\begin{document}

\title[Degeneration in spaces of algebraic correspondences]{Teichm{\"u}ller spaces, polynomial loci,\\ and degeneration in spaces of\\ algebraic correspondences}

\begin{abstract}
    We develop an analog of the notion of a character variety in the context of algebraic correspondences. It turns out that matings of certain Fuchsian groups and polynomials are contained in this ambient character variety. This gives rise to two different analogs of the Bers slice by fixing either the polynomial or the Fuchsian group. The Bers-like slices are homeomorphic copies of Teichm{\"u}ller spaces or combinatorial copies of polynomial connectedness loci. We show that these slices are bounded in the character variety, thus proving the analog of a theorem of Bers. To produce compactifications of the Bers-like slices, we initiate a study of degeneration of algebraic correspondences on trees of Riemann spheres, revealing a new degeneration phenomenon in conformal dynamics. There is no available analog of Sullivan's `no invariant line field' theorem in our context. Nevertheless, for the four times punctured sphere, we show that the compactifications of Teichm{\"u}ller spaces are naturally homeomorphic.
\end{abstract}

\begin{author}[Y.~Luo]{Yusheng Luo}
\address{Department of Mathematics, Cornell University, 212 Garden Ave, Ithaca, NY 14853, USA}
\email{yl3769@cornell.edu, yusheng.s.luo@gmail.com}
\thanks{Y.L. was partially supported by NSF Grant DMS-2349929.}
\end{author}

\begin{author}[M.~Mj]{Mahan Mj}
\address{School of Mathematics, Tata Institute of Fundamental Research, 1 Homi Bhabha Road, Mumbai 400005, India}
\email{mahan@math.tifr.res.in, mahan.mj@gmail.com}
\thanks{M.M. was partially supported by  the Department of Atomic Energy, Government of India, under Project Identification No. RTI 4014, an endowment of the Infosys Foundation, and a DST JC Bose Fellowship.}
\end{author}

\begin{author}[S.~Mukherjee]{Sabyasachi Mukherjee}
\address{School of Mathematics, Tata Institute of Fundamental Research, 1 Homi Bhabha Road, Mumbai 400005, India}
\email{sabya@math.tifr.res.in, mukherjee.sabya86@gmail.com}
\thanks{S.M. was partially supported by the Department of Atomic Energy, Government of India, under Project Identification No. RTI 4014, an endowment of the Infosys Foundation, and SERB research project grant MTR/2022/000248.}
\end{author}

\date{\today}

\maketitle

\setcounter{tocdepth}{1}
\tableofcontents

\section{Introduction}

About a hundred years ago, Fatou and Julia laid the foundations of the theory of iterated rational maps on the Riemann sphere \cite{Fat19,Fat26,Jul18,Jul22}. Prior to this, Poincar{\'e} and Klein studied the geometry and dynamics of Fuchsian and Kleinian groups (discrete subgroups of $\PSL_2(\R)$ and $\PSL_2(\C)$, respectively) acting on the sphere \cite{Poi82,Poi83,Kle83}. Based on various empirical similarities between these two branches of conformal dynamics, Fatou put forward a suggestion that the dynamics of Kleinian groups and rational maps can be studied in the common framework of iterated algebraic correspondences (finite-to-finite multi-valued maps with holomorphic local branches) \cite{Fat29}. Several decades later, Sullivan proposed a more systematic dictionary (a collection of analogies between results and proof techniques) between rational dynamics and Kleinian groups \cite{Sul85a}. This is now known as the \emph{Sullivan dictionary}.

The construction of such algebraic correspondences was initiated in \cite{BP94} for quadratic polynomials and the modular group, and has been further developed in the recent years \cite{BL20,BL24,LLMM3,LMM24, MM2, LLM24, BLLM24}. We refer the reader to §\ref{subsec:background} for a detailed discussion of the historical development of these constructions.

Among the recent works mentioned above, the \emph{single-valued mating framework} developed in \cite{MM2, LLM24} is of particular importance to the theme of the current paper. This framework gives rise to matings between Fuchsian genus zero orbifold groups and polynomials with connected Julia sets. The resulting space of correspondences has a product structure. The horizontal and vertical fibers in this product space are `copies' of Teichm{\"u}ller spaces and polynomial connectedness loci respectively.

\subsection*{Bers compactifications in spaces of correspondences}
As an application of the single-valued mating framework, {a biholomorphism from the Teichm{\"u}ller space of $S_{0,d+1}$ (sphere with $d+1\geq 3$ punctures) onto a domain in $\C^{d-2}$ was constructed in \cite{MM2}. This domain is the parameter space of a family of bi-degree $(2d-1,2d-1)$ correspondences on $\widehat{\C}$, such that each correspondence is the mating of a Fuchsian group $\Gamma\in\mathrm{Teich}(S_{0,d+1})$ and the fixed polynomial $z^{2d-1}$.}
This embedding can be regarded as an analog of the \emph{Bers embedding} of the Teichm{\"u}ller space in the world of algebraic correspondences. Motivated by well-known results in Teichm{\"u}ller theory and Kleinian groups, the following questions were raised in \cite{MM2}.

\begin{question}\cite[Question~7.3]{MM2}\label{qn1}\\ 
Let $\mathscr{E}:\mathrm{Teich}(S_{0,d+1})\longrightarrow\C^{d-2}$ be the biholomorphic embedding mentioned above.
\noindent\begin{enumerate}
\item Is the image $\E(\mathrm{Teich}(S_{0,d+1}))$ pre-compact in $\C^{d-2}$?

\item Describe the dynamics of the correspondences lying on the boundary of $\E(\mathrm{Teich}(\Sigma))$. In particular, do such boundary correspondences furnish new examples of matings between polynomials and Kleinian groups?
\end{enumerate}
\end{question}

 We answer the first question affirmatively in a more general setting. Let $\cH_{n}$ be the principal hyperbolic component in the space of {monic, centered}, degree $n\geq 2$ polynomials. For a fixed $P\in\cH_{2d-1}$, we denote by $\mathfrak{B}(P)$ the moduli space of correspondences arising as matings between all Fuchsian groups $\Gamma\in\mathrm{Teich}(S_{0,d+1})$ and $P$. The space $\mathfrak{B}(P)$ is called the \emph{Bers slice} of $(d+1)$-punctured spheres passing through $P$ (see \S~\ref{qf_bers_slice_subsec} for a formal definition).
In this generality, the role of the ambient space $\C^{d-2}$ of the Bers slice $\mathfrak{B}(z^{2d-1})$ is played by what we call the \emph{character variety} in the space of bi-degree $(2d-1,2d-1)$ correspondences. The character variety serves as a natural habitat for all the Bers slices. The precise notion of the character variety for correspondences and the description of the topology on it needs some doing. For now, we refer the reader to \S~\ref{char_var_subsec}.

\begin{thmx}\label{bers_pre_comp_thm_intro}
Let $P\in\cH_{2d-1}$. Then $\mathfrak{B}(P)$ is pre-compact in the character variety of bi-degree $(2d-1,2d-1)$ algebraic correspondences on $\widehat{\C}$. 
\end{thmx}
\noindent (See Theorem~\ref{pre_comp_thm_1} for a more precise statement.)

\subsection*{Uniqueness of Bers compactifications}
A celebrated theorem of Kerckhoff and Thurston states that in general, the natural homeomorphism between two (classical) Bers slices does not extend continuously to the boundaries \cite{KT90}. However, in the special case of $4$-times punctured spheres (or once punctured tori), all the Bers compactifications turn out to be naturally homeomorphic \cite{Ber81}. The last statement is a consequence of rigidity of triply punctured spheres and Sullivan's \emph{no invariant line field} theorem for Kleinian groups. It is natural to seek analogs of these results for the compactifications of the Bers slices $\mathfrak{B}(P)$, $P\in\cH_{2d-1}$, constructed in Theorem~\ref{bers_pre_comp_thm_intro}.
The case of $4$-times punctured spheres turns out to be special in our case as well. 

\begin{thmx}\label{bers_homeo_thm_intro}
Let $P,Q\in\cH_{2d-1}$. Then, the Bers slices of $(d+1)$-punctured spheres $\mathfrak{B}(P)$ and $\mathfrak{B}(Q)$ (passing through $P,Q$, respectively) are naturally homeomorphic.
\end{thmx}

Though an analog of Sullivan's no invariant line field theorem is unavailable to us at this stage, we nevertheless prove the following.
\begin{thmx}\label{bers_closure_homeo_thm_intro}
Let $P,Q\in\cH_{5}$, and $\mathfrak{B}(P), \mathfrak{B}(Q)$ be the Bers slices of $4$-times punctured spheres passing through $P,Q$, respectively. Then, the natural homeomorphism between $\mathfrak{B}(P)$ and $\mathfrak{B}(Q)$ extends to a homeomorphism between their closures.
\end{thmx}
\noindent (See Theorems~\ref{bers_slice_hoemo_thm} and ~\ref{bers_slice_closure_hoemo_thm} for more precise statements.)

\subsection*{Compactness of external fibers and a new degeneration phenomenon for correspondences}
So far, we have focused on Bers slices obtained by fixing a polynomial in the principal hyperbolic component and varying the group in a Teichm{\"u}ller space. We now turn to the complementary perspective: fixing a Fuchsian group $\Gamma$ and studying the associated parameter slice, known as the \emph{external fiber} $\mathscr{B}_\Gamma$, which is transverse to the Bers slices. 
In the single-valued mating framework of \cite{MM2, LLM24}, a conformal mating between a polynomial $P$ and $\Gamma$ yields a \emph{$B$-involution}; i.e.\ a meromorphic map defined on a subset {$\Omega$ of $\widehat{\C}$ (with piecewise-analytic boundary) inducing an involution on the boundary of $\Omega$}.  
For the $B$-involution,
\begin{itemize}[leftmargin=*]
    \item the non-escaping dynamics mirrors $P|_{\mathcal{K}(P)}$, and
    \item the escaping dynamics is conjugate to the Bowen–Series map of $\Gamma$.
\end{itemize}
(See \S~\ref{subsec:background} for a brief summary and \S~\ref{corr_mating_subsec} for a detailed account of this framework). The space $\mathscr{B}_\Gamma$, defined by fixing the Bowen-Series map of $\Gamma$ as the {\em external map}, contains such $B$-involutions and is naturally interpreted as a space of algebraic correspondences on trees of spheres (see \S~\ref{degen_poly_like_map_ext_fiber_subsec}, cf. \cite[\S 15, Appendix~A]{LLM24}). As such, it embeds into the character variety of bi-degree~$(2d-1, 2d-1)$ correspondences (see \S~\ref{char_var_subsec}), and we show that the external fiber $\mathscr{B}_\Gamma$ is compact under this embedding.

\begin{thmx}\label{vert_slice_comp_thm_intro}
Let $\Gamma\in\mathrm{Teich}(S_{0,d+1})$. Then the external fiber $\mathscr{B}_\Gamma$ is compact.
\end{thmx}
\noindent (See Theorem~\ref{vert_slice_comp_thm} for a precise statement.)
\smallskip

We refer the reader to \cite[\S 6.2, \S 9]{LLM24} for a parallel compactness result in the antiholomorphic setting and its role in the proof of the fact that antiholomorphic analogs of $\mathscr{B}_\Gamma$ (spaces of Schwarz reflections having the \emph{Nielsen map} of an ideal $(n+1)$-gon reflection group as their external map) bear strong resemblances with the connectedness locus of degree $n$ antiholomorphic polynomials.

\subsection*{The motivating example}
    We remark that the compactness of the external fiber $\mathscr{B}_\Gamma$ stated in Theorem~\ref{vert_slice_comp_thm_intro} relies crucially on working in the space of algebraic correspondences on trees of spheres (see Example~\ref{degree_three_signature_example} and Example~\ref{degree_four_signature_example} in \S~\ref{corr_mating_subsec} and the corresponding Figure~\ref{mating_1_fig} and Figure~\ref{mating_2_fig}).

    We now present an explicit example of a family of correspondences $\mathfrak{C}_t \in \mathscr{B}_\Gamma, t\in (0,1)$, defined on a single Riemann sphere, which converges, under suitable rescaling, to a correspondence on a tree of Riemann spheres as $t \to 0$.

    We note that some of the technical notation used in the example is introduced later in \S~\ref{degenerations_sec} and \S~\ref{sec:cv}. The reader may find it helpful to return to this example after reading those sections. We include it here because it serves as the motivating model for many of the technical developments in \S~\ref{degenerations_sec} and \S~\ref{sec:cv}, and it illustrates several of the subtleties underlying the compactness result.

    \begin{example}\label{DegenerationFamily}
    Consider the degree $3$ rational map 
    $$
    R_t(z) = \frac{\frac12 z^3 - \frac{3-3t}{2(3-2t)}z^2-\frac{1}{1-t} z}{z+1-t}.
    $$
    We remark that for each $t \in (0,1)$, $R_t$ is univalent on $\D$ and it has three simple critical points on $\mathbb{S}^1 = \partial \D$ (see Figure~\ref{fig:SRD}).
    
    Consider the associated family of correspondences $\mathfrak{C}_t, t\in (0,1)$, defined by
    $$
    \mathfrak{C}_t:=\{(x, y) \in \widehat{\C} \times \widehat{\C}: \frac{R_t(x) - R_t\circ \eta(y)}{x - \eta(y)} = 0\},
    $$
    where $\eta(y) = \frac1y$ (see \S~\ref{corr_mating_subsec} for motivation behind this definition).
    
    One can show that {for each $t \in (0,1)$,}
    $\mathfrak{C}_t$ is the mating between the unique Fuchsian group $\Gamma$ associated to the sphere with 2 punctures and an order 2 orbifold point and the real quadratic polynomial $z^2+ c(t)$, where $c(t)$ is a continuous strictly increasing function satisfying $\lim_{t\to 0^+} c(t) = -\frac34$ and $\lim_{t\to 1^-} c(t) = 0$.

    As $t \to 0^+$, we have
    $$
    R_t(z) = \frac{\frac12 z^3 - \frac{3-3t}{2(3-2t)}z^2-\frac{1}{1-t} z}{z+1-t} \to \frac{\frac12 z^3 - \frac{1}{2}z^2-z}{z+1} = \frac12z^2-z.
    $$ 
    Thus, the rational map $R_t$ converges compactly to the degree $2$ rational map $\frac12z^2-z$ away from the hole $-1\in \widehat{\C}$ (see the discussion in \S~\ref{rat_map_subsubsec}). 
    
    On the other hand, let $M_t(z) = tz-1$. Then as $t \to 0^+$,
    $$
    R_t \circ M_t (z) = \frac{\frac12 (tz-1)^3 - \frac{3-3t}{2(3-2t)}(tz-1)^2-\frac{1}{1-t} (tz-1)}{(tz-1)+1-t} \to \frac{9z+7}{6(z-1)}.
    $$
    Thus, the rational map $R_t \circ M_t$ converges compactly to the degree $1$ rational map $\frac{9z+7}{6(z-1)}$ away from the hole $\infty \in \widehat{\C}$.

    Therefore, as $t\to 0^+$, $R_t$ converges to a rational map $\pmb{R}_0: (\RT, \widehat\C^\RV) \longrightarrow \widehat{\C}$ from a tree of Riemann spheres $(\RT, \widehat\C^\RV)$ consisting of two spheres to $\widehat{\C}$ in the sense of Definition~\ref{defn:cvrationalmap}.
    The correspondence $\mathfrak{C}_t$ converges to a correspondence $\pmb{\mathfrak{C}}_0$ on $(\RT, \widehat\C^\RV)$. One can show that this correspondence $\pmb{\mathfrak{C}}_0$ is the mating between $\Gamma$ and the fat Basilica quadratic polynomial $z^2-\frac34$ (c.f. Example~\ref{degree_three_signature_example} and Figure~\ref{mating_1_fig} for the mating between $\Gamma$ and a post-critically finite Basilica quadratic polynomial $z^2-1$).
    \end{example}

\begin{figure}[ht]
\captionsetup{width=0.96\linewidth}
  \centering
  \includegraphics[height=5cm]{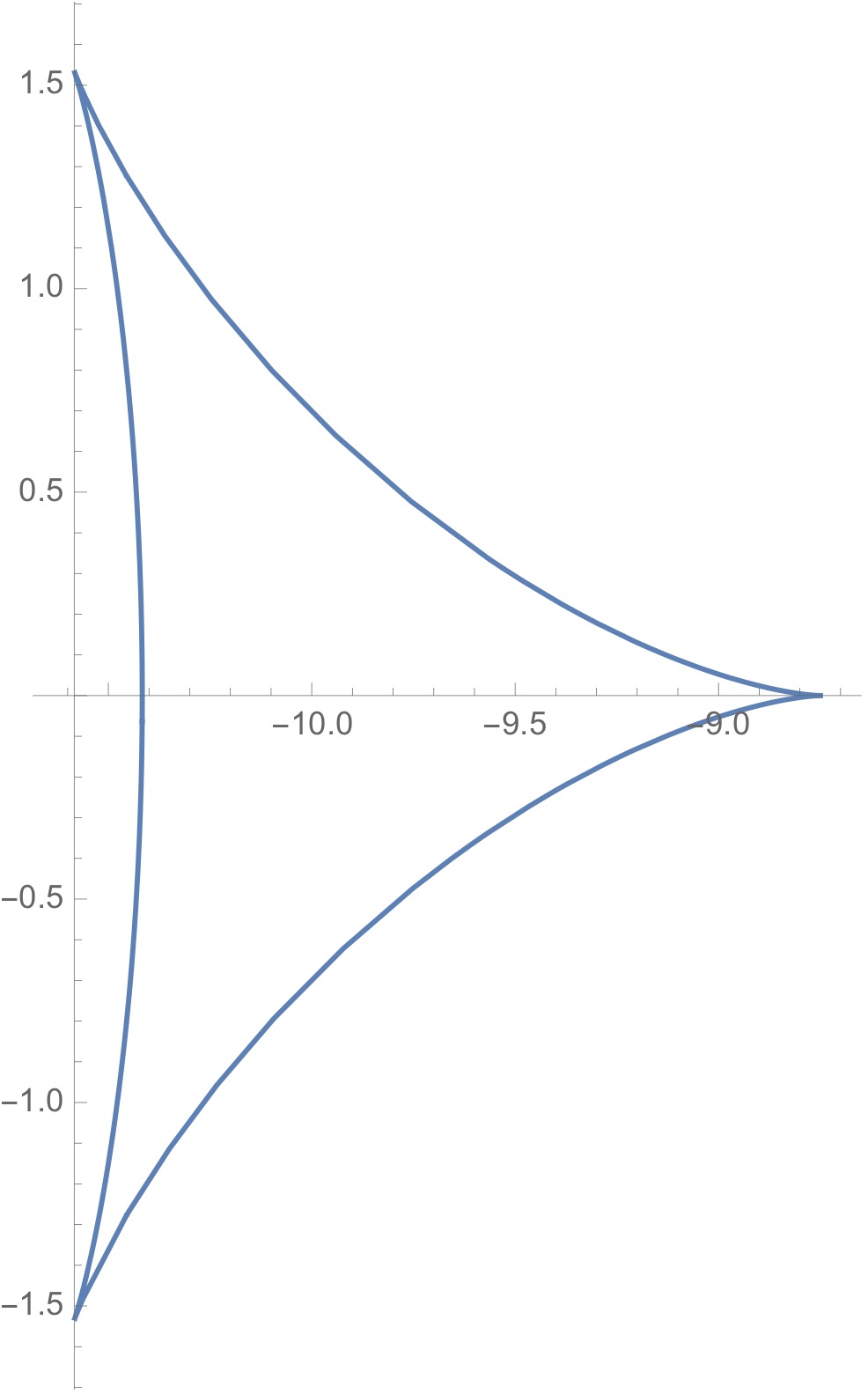}
  \includegraphics[height=5cm]{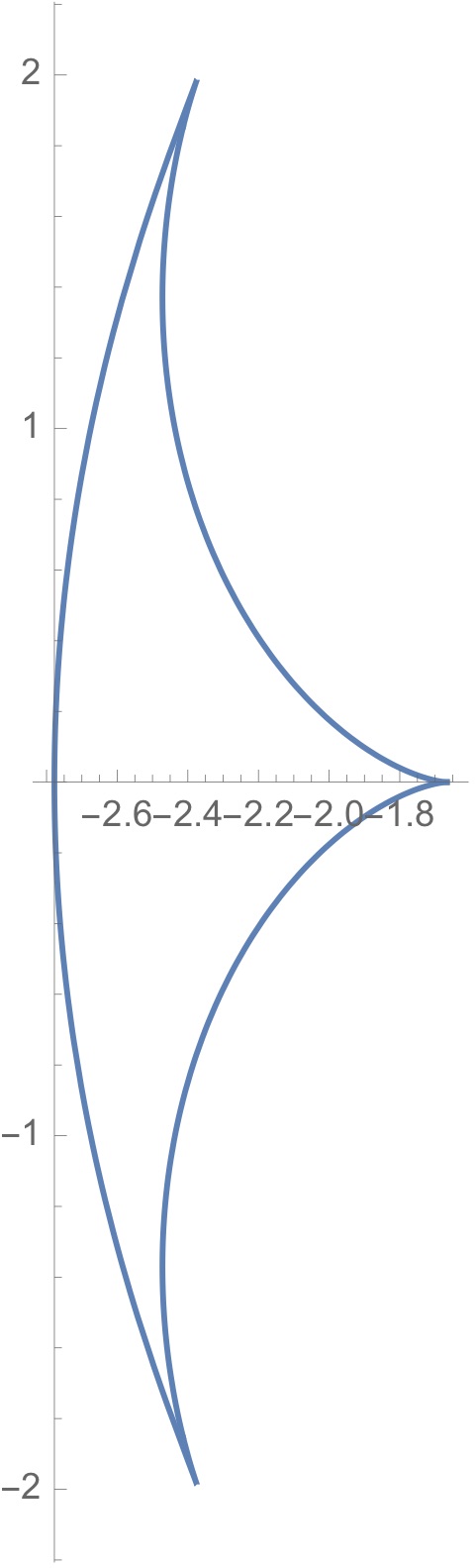}
  \includegraphics[height=5cm]{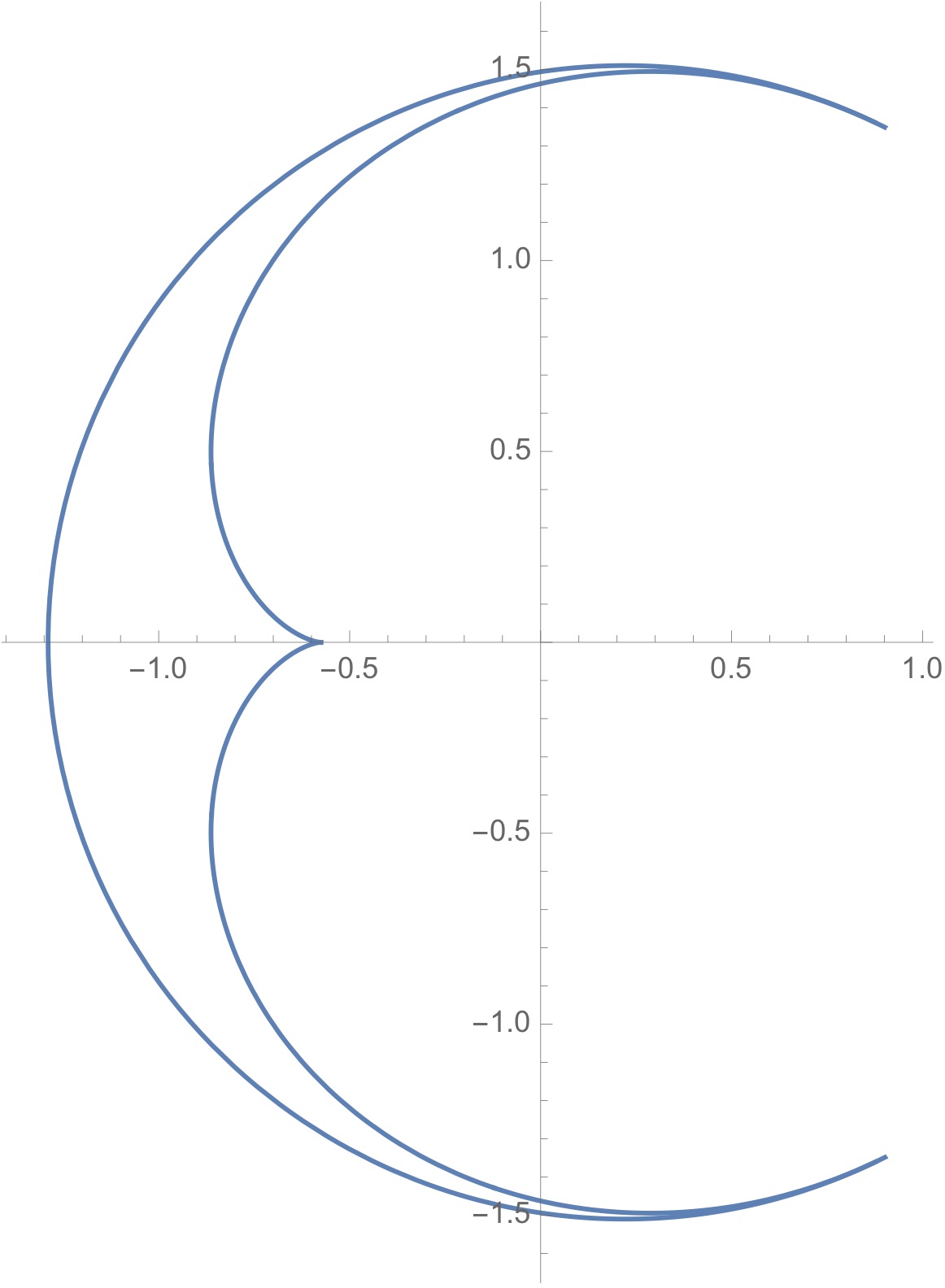}
  \includegraphics[height=5cm]{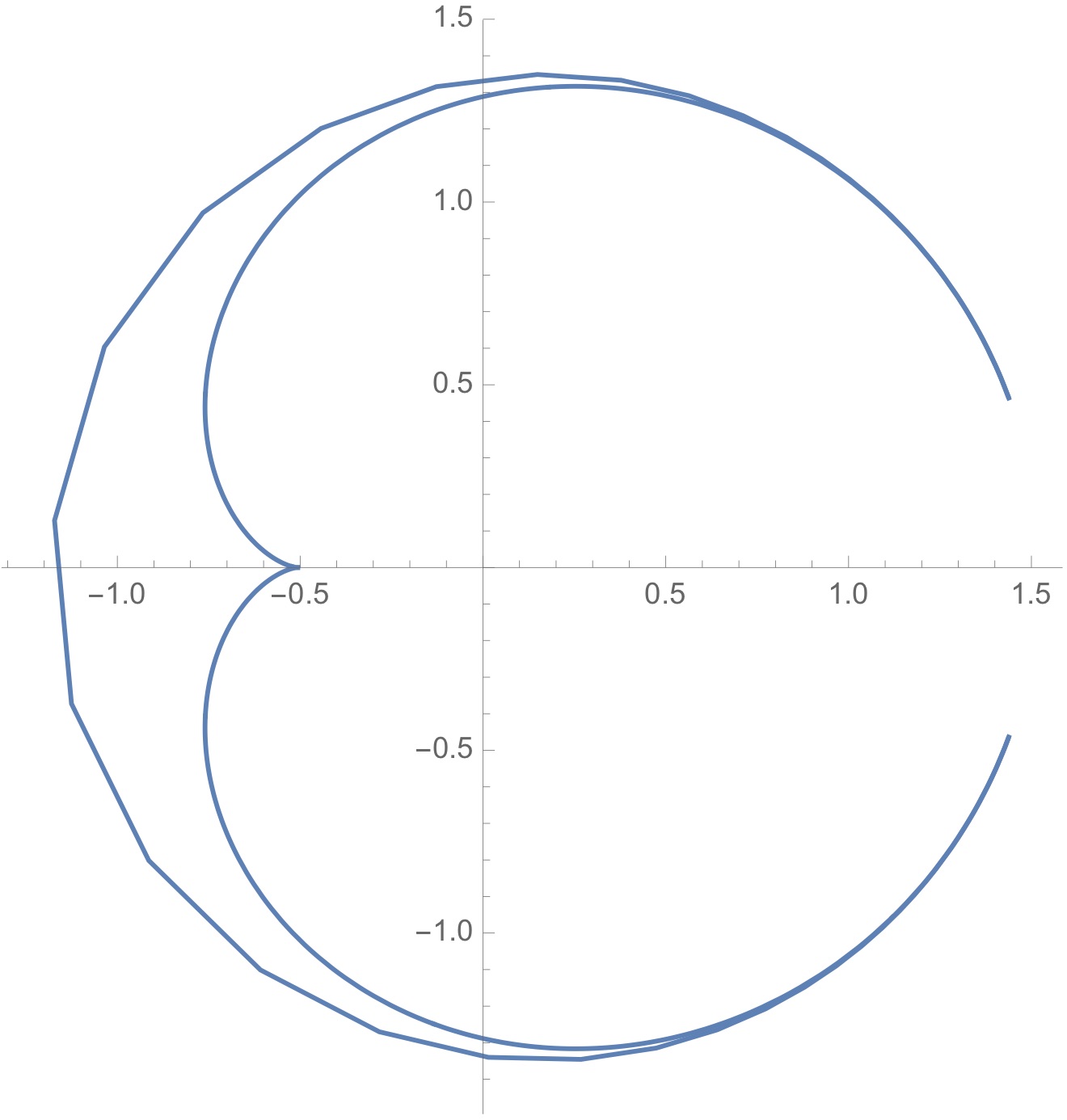}
  \includegraphics[height=5cm]{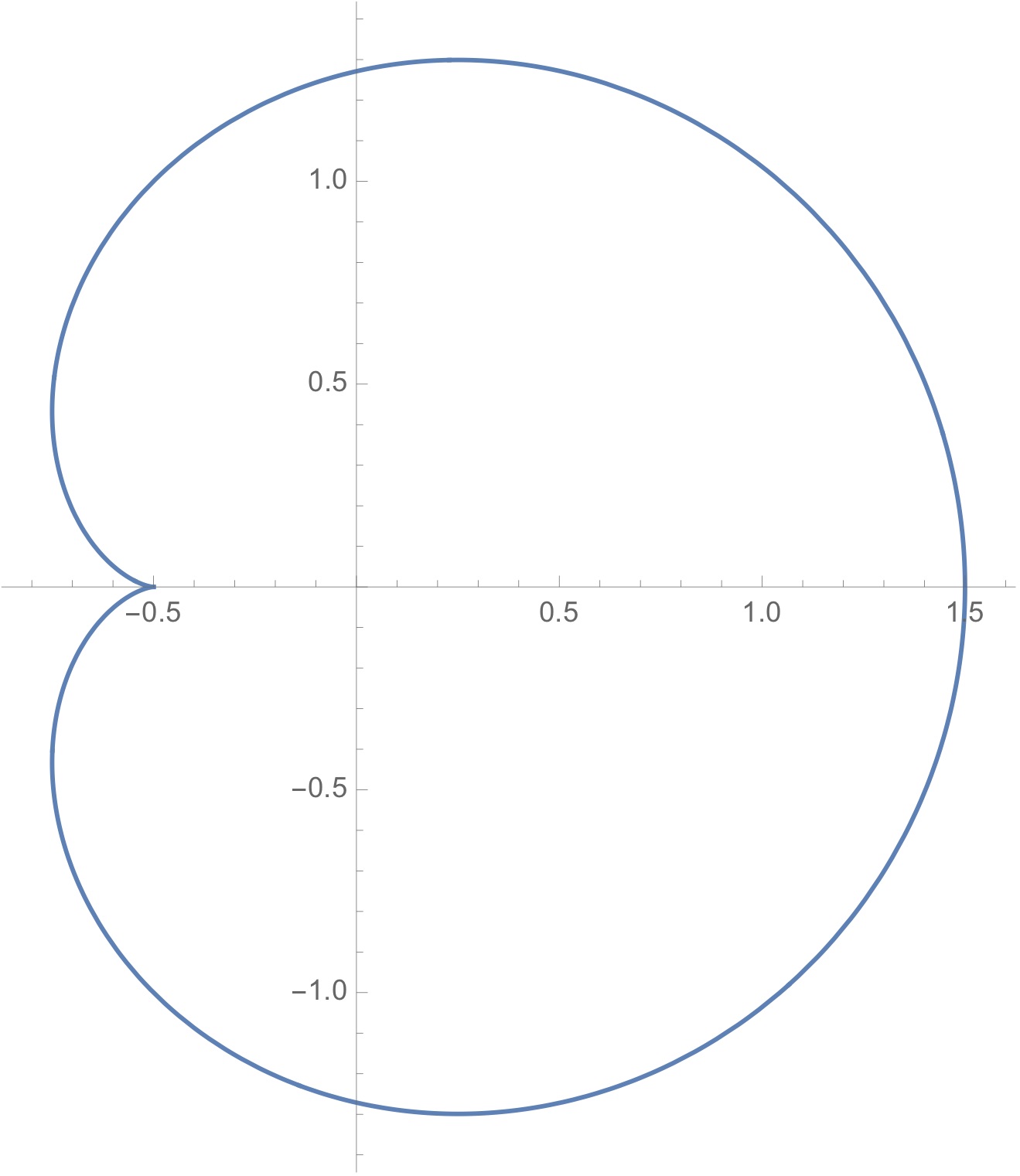}
  \caption{The deltoid family and circle-and-cardioid family of matings have been studied extensively in \cite{LLMM1,LLMM2}. This figure illustrates how the deltoid family converges to the circle-and-cardioid family (c.f. Figure~\ref{mating_1_fig}).
  These are images of the unit circle $\mathbb{S}^1$ under the rational map $R_t$ in Example~\ref{DegenerationFamily} as $t \to 0^+$. The outward cusps are the images of the three critical points on $\mathbb{S}^1$. The last figure depicts only the cardioid, which is the image of $\mathbb{S}^1$ under the limiting quadratic rational map $\frac12z^2-z$. The image of $i\R \cup \{\infty\} = \lim_{t\to 0^+} M_t^{-1}(\mathbb{S}^1)$ under the other (degree one) limiting map $\frac{9z+7}{6(z-1)}$ gives the {smallest} circle tangent to the cardioid at $\frac32$ 
  {and containing it. This circle is realized as a limit of the \emph{outer} non-singular arc in the picture to its left.}}
  \label{fig:SRD}
\end{figure}

\subsection*{Analytic compactifications of Teichm{\"u}ller and Blaschke product spaces}
Embeddings of Teichm{\"u}ller spaces in $\PSL_2(\C)$-character varieties and that of expanding Blaschke products in spaces of rational maps yield various compactifications of Teichm{\"u}ller and Blaschke product spaces. Analyzing the interrelationships among these compactifications is an important problem in conformal dynamics (cf. \cite{McM91}). Our main results allow one to ask and study such questions in a unified framework.

\subsection*{Remark on genus 0 orbifold matings}
The construction of matings between groups and polynomials has been extended to various genus 0 orbifold groups (see \cite{MM2, LLM24}). The construction of the character variety in \S~\ref{sec:cv} for these correspondences proceeds identically to the manifold case. Our main results — Theorem \ref{bers_pre_comp_thm_intro}, Theorem \ref{bers_homeo_thm_intro}, and Theorem \ref{vert_slice_comp_thm_intro} — also generalize to the orbifold setting, with proofs requiring only minor modifications (see Remark \ref{rmk:CVOrbifold}, Remark \ref{rmk:orbifoldModification} and Remark \ref{rmk:orbifoldModification_1}).

\subsection{Comments on the proofs}
Theorem~\ref{bers_pre_comp_thm_intro}, one of the main results of the paper, concerns the pre-compactness of Bers slices. While classical Bers slices are pre-compact in the $\PSL_2(\C)$-character variety due to compactness of normalized conformal maps \cite{Ber70,Mas70}, our setting requires a more delicate approach. We first define an appropriate character variety for bi-degree $(n,n)$ correspondences in \S~\ref{sec:cv}, and show that its \emph{regular part} is Hausdorff. Then, to prove pre-compactness of $\mathfrak{B}(P)$ in this character variety, we must prevent `degree drop' in the limit. To this end, we adapt rescaling techniques from rational dynamics to study degenerations of correspondences (\S~\ref{degenerations_sec}) and use this to rule out such degenerations in $\mathfrak{B}(P)$ (\S~\ref{pre_comp_bers_sec}). Preparations from \S~\ref{char_var_subsec} also help us show that the closure $\overline{\mathfrak{B}(P)}$ lies in the Hausdorff part of the character variety for correspondences.

It is worth mentioning that this pre-compactness result sets the stage for the dynamical investigation of the `Bers boundary' correspondences and its topology in the parameter space.

The proof of Theorem~\ref{bers_homeo_thm_intro} is an application of quasiconformal surgery adapted to the setting of correspondences. After preparing the necessary tools in \S~\ref{qc_def_sec}, we construct a homeomorphism between Bers slices by replacing the polynomial 
$P$-dynamics in a correspondence $\mathfrak{C}_{P,\Gamma}\in\mathfrak{B}(P)$ (that is a mating between $P$ and a Fuchsian group $\Gamma\in\mathrm{Teich}(S_{0,d+1})$) with that of another polynomial $Q$, producing a correspondence $\mathfrak{C}_{Q,\Gamma}\in\mathfrak{B}(Q)$. This surgery defines the desired homeomorphism, as shown in \S~\ref{bers_slice_homeo_sec}.

Theorem~\ref{bers_closure_homeo_thm_intro} relies crucially on the one-dimensionality of the Teichm{\"u}ller space. In the absence of a general `no invariant line field' theorem for correspondences, we instead employ techniques from the theory of {\em straightening maps}, in the spirit of continuity results for baby Mandelbrot sets (cf. \cite{DH85}), where the parameter space's one-dimensionality plays a pivotal role.

Finally, Theorem~\ref{vert_slice_comp_thm_intro} is proved using the degeneration framework developed in \S~\ref{degenerations_sec}. A key distinction emerges between Bers slices and external fibers: while the closure $\overline{\mathfrak{B}(P)}$, $P\in\cH_{2d-1}$, consists of correspondences defined on a single sphere (with no degenerations), the external fiber $\mathscr{B}_\Gamma$ admits controlled degenerations onto finite trees of spheres with varying combinatorics. This richer structure warrants additional care in managing the rescaling limits, defining the limiting correspondences, and showing that the limiting correspondences live in the Hausdorff part of the character variety (see \S~\ref{comparison_proof_subsec} for a detailed comparison). We emphasize, at the cost of being repetitive, that correspondences in $\mathscr{B}_\Gamma$ exhibit a degeneration phenomenon that is not observed in rational dynamics; namely, the convergence of a sequence of correspondences defined on $\widehat{\C}$ to a limiting correspondence of the \emph{same bi-degree} on a \emph{finite} tree of spheres.

\subsection{Connections with related work}\label{subsec:background}
\subsection*{Algebraic correspondences as matings: early examples} The world of algebraic correspondences is much larger and wilder than those of rational maps and Kleinian groups, and a general dynamical theory of correspondences seems out of reach as of now. To get a handle on the dynamics and parameter spaces of correspondences from the perspective of the Sullivan dictionary, it is important to produce large classes of algebraic correspondences that display features of Kleinian groups as well as of rational dynamics in their dynamical planes.
The first concrete step in this direction  was taken by Bullett and Penrose, who proved the existence of bi-degree $(2,2)$ algebraic correspondences on the sphere that combine the actions of certain quadratic polynomials and the modular group \cite{BP94}. Many more explicit examples of correspondences that realized such combinations were produced in the subsequent years \cite{BH00,BF03,BF05,BH07}, but a general theory remained lacking.

\subsection*{Algebraic correspondences as matings: a systematic theory}
A new chapter in this program started to emerge in the last decade with the introduction of new techniques and insights. This development occurred in two separate veins. 

The Bullett-Penrose family of bi-degree $(2,2)$ correspondences was studied systematically by Bullett and Lomonaco, and they proved that this family contains combinations/matings of every quadratic parabolic rational map (with a connected Julia set) with the modular group \cite{BL20}. They proceeded to demonstrate that the relevant parameter space of correspondences is in fact homeomorphic to the connectedness locus of parabolic quadratic rational maps \cite{BL24}. This produced an embedding of a rational parameter space in a family of correspondences (which are matings of varying rational maps with a fixed group) along the lines of Fatou's original suggestion. The related parameter space of algebraic
correspondences realizing matings between quadratic polynomials and discrete, faithful representations of the modular group in $\PSL_2(\C)$ (with connected domain of discontinuity) was recently studied in \cite{Lau24}.

The other thread in this story came from a rather unrelated and unexpected direction. It was observed that a class of antiholomorphic maps called \emph{Schwarz reflections} in \emph{quadrature domains}, which are important objects in various problems concerning statistical physics and complex analysis, often arise as amalgams of antiholomorphic rational maps and Kleinian reflection groups \cite{LM16,LLMM1,LLMM2,LMM21,LMM22}. Using the algebraicity of Schwarz reflection maps, antiholomorphic analogs of Bullett-Penrose correspondences (in arbitrary bi-degree) were constructed and their parameter spaces were investigated in \cite{LLMM3,LMM24}.

Motivated by these results in the antiholomorphic setting, a strategy to manufacture algebraic correspondences combining rational maps and suitable Fuchsian groups was devised in \cite{MM1,MM2} and further strengthened in \cite{LLM24,BLLM24}. The construction comprises four key steps. The first step is to extract a single-valued map from a Fuchsian group such that the map bears resemblance with polynomial maps, and yet remembers some of the key features of the group (e.g., a preferred generating set of the group). Examples of such maps are given by the so-called \emph{Bowen-Series maps} of punctured spheres \cite{BS79,MM1} and their generalizations termed as \emph{factor Bowen-Series maps} \cite{MM2}. One then mates such a map conformally with a polynomial or a rational map to produce a single-valued meromorphic map (which is called the \emph{conformal mating}). This requires a uniformization argument involving the \emph{David integrability theorem} that is a generalization of the classical \emph{measurable Riemann mapping theorem} (cf. \cite{LMMN25}). This is followed by the identification of the conformal mating as an algebraic function. 
To this end, a new class of algebraic functions called \emph{B-involutions} (holomorphic counterpart of Schwarz reflections) was introduced in \cite{LLM24,BLLM24}, and it was shown that the conformal matings produced in the previous step are B-involutions. Finally, one uses this algebraic description to lift the conformal mating to an algebraic correspondence that behaves like the polynomial in one part of its dynamical plane and captures the full Fuchsian group structure on the complementary part. We will refer to the above recipe as the \emph{single-valued mating framework}.

\subsection*{Capturing Bers boundary groups in the mating framework}
While the above approach allows one to combine large classes of Fuchsian genus zero orbifold groups with polynomials, the recipe has a limitation in that Kleinian groups lying on boundaries of Teichm{\"u}ller spaces of genus zero orbifolds (i.e., Bers boundary groups) cannot be thus mated with polynomials. This is related to the mismatch between group invariant geodesic laminations and polynomial laminations which are $z^d$-invariant (cf. \cite[\S 7.2]{MM1}).

One of the goals of the current paper is to surmount this obstacle by constructing compactifications of spaces of correspondences arising from the single-valued mating framework. On the one hand, this sheds new light on parameter spaces of algebraic correspondences, and on the other hand, furnishes new compactifications of Teichm{\"u}ller spaces of genus zero orbifolds thereby producing candidate correspondence which can potentially combine polynomial maps and Bers boundary groups. To this end, we will develop a general language to study degenerations/boundedness phenomena in parameter spaces of correspondences (modeled on a related theory for rational maps). Reaping benefits from this degeneration framework, we will study boundedness questions for various dynamically defined subsets/slices in spaces of correspondences, and also study interrelations between various compactifications.

\subsection*{Kerckhoff-Thurston discontinuity and straightening discontinuity in the same parameter space}
The total space
$$
\mathfrak{T}:=\bigsqcup_{\Gamma\in\mathrm{Teich}(S_{0,d+1})} \mathscr{B}_\Gamma
$$
houses all conformal matings between degree $2d-1$ polynomials with connected Julia sets and Fuchsian $(d+1)$-punctured sphere groups. As mentioned above, the space $\mathfrak{T}$ embeds into the character variety of of bi-degree $(2d-1,2d-1)$ correspondences (see \S~\ref{degen_poly_like_map_ext_fiber_subsec}). 

The space $\mathfrak{T}$ admits a natural Cartesian product structure 
$$
\mathfrak{T}\cong \mathrm{Teich}(S_{0,d+1})\times \mathscr{B}_{\Gamma_0},
$$
where $\Gamma_0$ is any base-point in $\mathrm{Teich}(S_{0,d+1})$ (by \cite[Theorem~A.3]{LLM24}). Fixing the first coordinate as $\Gamma\in\mathrm{Teich}(S_{0,d+1})$ in this product space yields the external fiber $\mathscr{B}_\Gamma$. On the other hand, when the second coordinate is fixed as the mating between $\Gamma_0$ and some $P\in\cH_{2d-1}$, one gets back the Bers slice $\mathfrak{B}(P)$. (See Figure~\ref{total_space_fig}.) {We note that the above dynamically arising product structure of $\mathfrak{T}$ is not necessarily topological, see Question~\ref{str_discont_qn} and Remark~\ref{cart_prod_rem}. However, one does obtain a topological product structure on the subset of $\mathfrak{T}$ containing matings between groups in $\mathrm{Teich}(S_{0,d+1})$ and polynomials in the principal hyperbolic component $\mathcal{H}_{2d-1}$, see \S~\ref{bers_slice_homeo_sec} (cf. \cite{Ber60} and \cite{HLP25} for similar results for quasi-Fuchsian groups and quasi-Blaschke products).} 

The total space $\mathfrak{T}$ captures the folklore line in the Sullivan dictionary between Kerckhoff-Thurston discontinuity phenomena and discontinuity of straightening maps in holomorphic dynamics. 
While these two are now established phenomena in both the Kleinian groups world \cite{KT90} and the holomorphic dynamics world \cite{Ino09,IM21}, analogous statements are conjectural in the hybrid world that is the theme of this paper.
Indeed, the (conjectural) lack of continuous boundary extension of the natural homeomorphism between two `horizontal slices' $\mathfrak{B}(P)$, $\mathfrak{B}(Q)$ (where $P,Q\in\cH_{2d-1}$) of $\mathfrak{T}$ is akin to the Kerckhoff-Thurston result, while `conjectural' discontinuity of the natural bijection between two `vertical slices' $\mathscr{B}_{\Gamma_1}$, $\mathscr{B}_{\Gamma_2}$ (where $\Gamma_1,\Gamma_2\in\mathrm{Teich}(S_{0,d+1})$) is the analog of straightening discontinuity in the current~setup.

\subsection{Open questions}
The analogous  theory in the context of Kleinian groups is essentially complete today. We give a quick survey of some of the highlights here. Bers \cite{Ber70} proved the pre-compactness of the classical Bers slice and the existence of Bers boundary groups. Thurston \cite{thurston-hypstr2,otal-book,kapovich-book}, via his Double Limit Theorem, identified the hyperbolic geometry of  these groups. The Ending Lamination Theorem of Brock-Canary-Minsky \cite{Min10,BCM12} completely classified Kleinian groups appearing on the Bers boundary. They also proved the Bers density conjecture proving that all Bers boundary groups may be approximated in a strong sense by groups lying in the Bers slice.
Mj \cite{mahan-split,mahan-kl} established the local connectivity of the limit sets of these groups. 
These lead to the following fundamental questions regarding the Bers compactifications constructed in Theorem~\ref{bers_pre_comp_thm_intro}.

\begin{question}\label{qn2} Assume the setup of Theorem~\ref{bers_pre_comp_thm_intro}.
\begin{enumerate}[leftmargin=*]\upshape
\item Does the dynamically natural homeomorphism between two Bers slices $\mathfrak{B}(P)$, $\mathfrak{B}(Q)$ extend continuously to the boundaries? A negative answer to this question would be  the analog of Kerckhoff-Thurston's result \cite{KT90}.

\item Classify the correspondences on the boundary $\partial\mathfrak{B}(P)$. This is an analog of the
Ending Lamination Theorem \cite{Min10,BCM12}.

\item Are the limit sets of the correspondences on $\partial\mathfrak{B}(P)$ locally connected? In particular, are the `geometrically infinite' correspondences on Bers boundaries matings of polynomials and singly degenerate groups?  This is the analog of \cite{mahan-split}.

\item Do the limit sets of correspondences on $\partial\mathfrak{B}(P)$ support non-trivial quasiconformal deformations? This is the analog of Sullivan's no invariant line field theorem \cite{Sul81,Sul85b}.

\item Do the limit sets of correspondences on $\partial\mathfrak{B}(P)$ have zero area? The analogous result for Kleinian groups was known as the `Ahlfors measure conjecture', and was answered positively by Brock-Canary-Minsky \cite{Min10,BCM12}, making essential use of work of Agol and Calegari-Gabai \cite{Ago04,CG06}.
Further, do the limit sets of `geometrically infinite' correspondences on Bers boundaries have Hausdorff dimension $2$? This question is motivated by an analogous result of Bishop and Jones \cite{BJ97}.

\item Are the Bers boundaries $\partial\mathfrak{B}(P)$ (for spheres with at least five punctures) non-locally connected? This question is motivated by a result of Bromberg \cite{Br11}.

\item Do there exist self-bumpings on the Bers boundaries $\partial\mathfrak{B}(P)$? This question is motivated by a result of McMullen \cite{McM97}.

\item Prove an analog of Thurston's double limit theorem for correspondences realizing matings of $P\in\cH_{2d-1}$ and $\Gamma\in\mathrm{Teich}(S_{0,d+1})$ (as the polynomial $P$ tends to $\partial\cH_{2d-1}$ and the group $\Gamma$ tends to $\partial\mathrm{Teich}(S_{0,d+1})$).
\end{enumerate}
\end{question}

We would like to point out that the proofs of many of the landmark results mentioned above use the hyperbolic geometry of the $3$-manifolds associated with Kleinian groups. In the correspondence setup, such techniques are not readily available.

\subsection{Future work}
According to \cite[\S 7]{MM1}, there are (only) finitely many quasiconformal deformation classes of Kleinian groups on the (usual) Bers boundary of $\mathrm{Teich}(S_{0,d+1})$ that admit Bowen-Series maps, and these groups can be conformally mated with polynomials lying in principal hyperbolic components. The resulting conformal matings admit algebraic descriptions, and hence give rise to algebraic correspondences realizing matings of the said groups and polynomial maps.

In a follow-up work, we will study the geometrically finite correspondences on Bers boundaries $\partial\mathfrak{B}(P)$, $P\in\cH_{2d-1}$, and show that they realize matings between the polynomials $P$ and arbitrary geometrically finite groups on classical Bers boundaries of $\mathrm{Teich}(S_{0,d+1})$.
 
In a sequel, we will also construct sequences of correspondences in $\mathfrak{B}(P)$ whose geometric limits are strictly larger than their algebraic limit. This phenomenon should give rise to analogs of the Kerckhoff-Thurston discontinuity result, non-local connectivity and self-bumpings on boundaries of Bers slices $\mathfrak{B}(P)$, etc.

\subsection*{Notation}
$\D:=\{z\in\C:\vert z\vert<1\}$,\quad $\D^*:=\{z\in\widehat{\C}:\vert z\vert>1\}$.

\subsection*{Acknowledgment} The authors would like to thank the referees for their diligent and careful reading. Comments and feedback from them on an earlier draft have been particularly helpful in improving the quality of the paper.

\section{Degeneration of rational maps and correspondences}\label{degenerations_sec}
In this section, we will first set up notation and discuss some general theory of degeneration for rational maps and correspondences, highlighting a fundamental difference between their limiting dynamics.
We will discuss a new phenomenon specifically for correspondences: the convergence of a sequence of holomorphic correspondences of bi-degree $(d_1, d_2)$ on the Riemann sphere $\widehat{\C}$ to a holomorphic correspondence of the same bi-degree on a nodal sphere, which is a finite tree of Riemann spheres.
We want to emphasize here that 
\begin{enumerate}[leftmargin=*]
    \item the nodal sphere has only finitely many singularities;
    \item the nodal sphere is totally invariant under the limiting correspondence;
    \item the limiting correspondence has the same bi-degree.
\end{enumerate}
This phenomenon can only happen if $\min\{d_1, d_2\} \geq 2$.

In contrast, consider a sequence $[f_n] \in \Rat_d(\C)/\PSL_2(\C)$. Here, we encounter a dichotomy: either $[f_n]$ is bounded in $\Rat_d(\C)/\PSL_2(\C)$, or it converges to a limiting branched covering of degree $d$ on an $\R$-tree with no totally invariant point (see \cite{Luo21, Luo22a}). In the latter case, any point in the $\R$-tree has infinite backward orbits. Thus, any totally invariant tree of Riemann spheres must consist of infinitely many spheres.

The new phenomenon for  correspondences mentioned above needs to be taken into consideration while setting up the appropriate ambient space for the study of the deformation space of matings of polynomials and groups in \S \ref{sec:cv}.

\subsection{The space of rational maps and correspondences}\label{rat_map_corr_subsec}

\subsubsection{Rational maps on $\widehat{\C}$}\label{rat_map_subsubsec}
Following the notation in \cite{DeM05}, the space $\Rat_d(\C)$ of rational maps of degree $d\geq 1$ is an open variety in $\Proj^{2d+1}_\C$.
More concretely, fixing a coordinate system on $\widehat\C \cong \Proj^1_\C$, a rational map can be expressed as a ratio of homogeneous polynomials 
$$
f(z:w) = (P(z,w): Q(z,w)),
$$ 
where $P$ and $Q$ have degree $d$ with no common divisors. Using the coefficients of $P$ and $Q$ as parameters, 
$$
\Rat_d(\C) = \Proj^{2d+1}_\C - V(\Res),
$$
where $V(\Res)$ is the vanishing locus for the resultant of $P$ and $Q$.

A natural compactification of $\Rat_d(\C)$ is $\overline{\Rat_d(\C)} = \Proj^{2d+1}_\C$. 
Every map $f\in \overline{\Rat_d(\C)}$ determines the coefficients of a pair of degree $d$ homogeneous polynomials. We write
$$
f= (P: Q) = (Hp:Hq)
$$
where $H = \gcd (P, Q)$. The map $\varphi_f = (p:q)$ is a rational map of degree at most $d$. {Note that if $\deg{H}\geq 1$, then $\phi_f$ is a rational map of degree less than $d$. 
A zero of $H$ is called a {\em hole} of $f$ and the set of zeros of $H$ is denoted by
$\mathcal{H}(f)$. While this lower degree map $\phi_f$ is defined everywhere on $\widehat{\C}$, the set $\cH(f)$ should be regarded as the locus of indeterminacy of $f\in\overline{\Rat_d(\C)}$.}
We define the degree of $f\in \overline{\Rat_d(\C)}$ as the degree of $\varphi_f$.
If $f_n \in \Rat_d(\C)$ converges to $f \in \overline{\Rat_d(\C)}$, we write this as
$$
f = \lim_{n\to\infty} f_n.
$$

The following lemma is well known (see \cite[Lemma 4.2]{DeM05}):
\begin{lem}\label{comcon}
Let $f_n\in\Rat_d(\C)$ converge to $f \in \overline{\Rat_d(\C)}$. Then $f_n$ converges compactly to $\varphi_f$ on $\widehat\C - \mathcal{H}(f)$.
\end{lem}

The next lemma follows from the proof of \cite[Lemma 4.5 and 4.6]{DeM05}:
\begin{lem}\label{lem:ns}
{ Let $f_n\in\Rat_d(\C)$ converge to $f \in \overline{\Rat_d(\C)}$.
Let $a\in \mathcal{H}(f)$ and let
 $U$ be a neighborhood of $a$.}
\begin{itemize}
\item If $\deg(\varphi_f) \geq 1$, then there exists $K>1$ so that $f_n(U) = \widehat\C$ for all $n\geq K$. 
Moreover, there exists a sequence of critical points $c_n$ of $f_n$  converging to~$a$.
\item If $\varphi_f\equiv C$ is a constant map, then for any compact set $X\subseteq \widehat\C - \{C\}$, there exists $K>1$ so that $X \subseteq f_n(U)$ for all $n \geq K$.
\end{itemize}
\end{lem}

\subsubsection{Correspondences on $\widehat{\C}$}\label{corr_subsubsec}
Let $\pi_1$ and $\pi_2$ be the two canonical projections from $\widehat{\C}\times \widehat{\C}$ to its factors. A \textit{holomorphic correspondence} on $\widehat{\C}$ is a subvariety $\mathfrak{C}$ in $\widehat{\C}\times \widehat{\C}$ of complex dimension $1$ containing no fibers of $\pi_1$ and $\pi_2$. Denote by $d_1$ and $d_2$ the degree of $\pi_1|_\Gamma$ and $\pi_2|_\Gamma$. We call $(d_1, d_2)$ the bi-degree of $\mathfrak{C}$.
Let $\Cor_{d_1, d_2}$ be the space of all bi-degree $(d_1, d_2)$ holomorphic correspondences on $\widehat{\C}$. We give $\Cor_{d_1, d_2}$ the Hausdorff metric induced by the product of spherical metrics on $\widehat{\C} \times \widehat{\C}$
$$
d(\mathfrak{C}, \mathfrak{C}') = \inf\{r>0 \colon \mathfrak{C}' \subseteq B(\mathfrak{C}, r), \mathfrak{C} \subseteq B(\mathfrak{C}', r)\}.
$$
Note that the correspondence can be viewed as a multi-valued map from $\widehat{\C}$ to itself. {As a convention, we will call the branches of the multi-valued map $\pi_2\circ\left(\pi_1\vert_{\Gamma}\right)^{-1}$ from the first coordinate to the second one (respectively, the branches of the multi-valued map $\pi_1\circ\left(\pi_2\vert_{\Gamma}\right)^{-1}$ from the second coordinate to the first one) the \emph{forward branches} (respectively, the \emph{backward branches}) of the correspondence.} The convergence in Hausdorff metric is equivalent to uniform convergence of multi-valued maps.
We also remark that the space $\Cor_{d_1, d_2}$ is not compact.
To deal with degenerations of correspondence, we need the following definition of convergence.
\begin{defn}\label{corr_conv_def}
    Let $\mathfrak{C}_n \in \Cor_{d_1, d_2}$ and $\mathfrak{C} \in \Cor_{e_1, e_2}$. Let $S_1$ and $S_2$ be two compact sets {in $\widehat{\C}$}. We say that $\mathfrak{C}_n$ converges compactly to $\mathfrak{C}$ away from $(S_1, S_2)$ if for any compact set $K \subseteq \widehat{\C}\times \widehat{\C} - (S_1 \times \widehat{\C}) - (\widehat{\C} \times S_2)$, we have
    $$
    d_K(\mathfrak{C}_n, \mathfrak{C}) \to 0,
    $$
    where $d_K$ is the Hausdorff metric on $K$.
\end{defn}

In this paper, we shall focus on the following construction of holomorphic correspondences, and some of its variations.
Let $f, g$ be two rational maps, with degrees $d_1$ and $d_2$ respectively.
Then the \textit{coincidence locus/{equalizer}} 
$$
\mathfrak{C}\coloneqq\{(x,y) \in \widehat{\C} \times \widehat{\C} \colon f(x) = g(y)\}
$$ 
is an algebraic correspondence of bi-degree $(d_1, d_2)$. For simplicity, we shall also denote $\mathfrak{C} = (f, g)$.
Therefore, we have a natural embedding of $\Rat_{d_1} \times \Rat_{d_2}$ into $\Cor_{d_1, d_2}$.
We denote its image by $\widehat{\Cor}_{d_1, d_2}$. In this way, we obtain a natural compactification
$$
\overline{\widehat{\Cor}_{d_1, d_2}} = \overline{\Rat_{d_1}(\C)} \times \overline{\Rat_{d_2}(\C)} \cong \Proj^{2d_1+1}_\C \times \Proj^{2d_2+1}_\C.
$$
The following lemma follows immediately from Lemma \ref{comcon}
\begin{lem}\label{comconCorrepondence}
Let $\mathfrak{C}_n =(f_n, g_n) \in \widehat{\Cor}(d_1, d_2)$ converge to $\mathfrak{C}\coloneqq (f, g) \in \overline{\Rat_{d_1}(\C)} \times \overline{\Rat_{d_2}(\C)}$. Let $\varphi_\mathfrak{C}\coloneqq (\varphi_f, \varphi_g)$. Then $\mathfrak{C}_n = (f_n, g_n)$ converges compactly to $\varphi_\mathfrak{C}$ away from $(\mathcal{H}(f), \mathcal{H}(g))$.
\end{lem}

{
\begin{remark}
We note that for two rational maps $f,g$, and a M{\"o}bius map $M$, the correspondences $(f,g)$ and $(M\circ f,M\circ g)$ are precisely the same subsets of $\widehat{\C} \times \widehat{\C}$.
\end{remark}}

\subsubsection{Rescaling limits}\label{rescaling_subsubsec}
In many applications, the compactification $\overline{\Rat_d(\C)} = \Proj^{2d+1}_\C$ or $\overline{\widehat{\Cor}_{d_1, d_2}} = \overline{\Rat_{d_1}(\C)} \times \overline{\Rat_{d_2}(\C)}$ is not sufficient to fully capture the global dynamics of how a rational map or a correspondence degenerates.
It is essential to consider all rescaling limits defined in the following.
{ We remark that a related notion of rescaling limit was introduced in \cite{Kiw15} that is more appropriate to study the degeneration of rational dynamics. In particular the degree of the limit is required to be at least two. Our rescaling limits are non-dynamical, and we only require the limit to be non-constant.}

\begin{defn}\label{def-equivalentscaling}
    Let $M_{a,n}, M_{b,n} \in \PSL_2(\C)$ be two sequences of M\"obius maps, which we also call rescalings. We say that they are \textit{equivalent} if the sequence $M_{b,n}^{-1} \circ M_{a,n} \to M \in \PSL_2(\C)$. 
    They are \textit{bounded equivalent} if $M_{b,n}^{-1} \circ M_{a,n}$ is bounded in $\PSL_2(\C)$. 
    They are \textit{independent} if $M_{b,n}^{-1} \circ M_{a,n} \to \infty$ in $\PSL_2(\C)$. 

    Let $f_n \in \Rat_d(\C)$, and $M_{a,n}, \widetilde{M}_{c,n} \in \PSL_2(\C)$ be two rescalings. We call $f \in \overline{\Rat_d(\C)}$ the rescaling limit  from $a$ to $c$ if
    $$
    f = \lim_{a\to c} f_n \coloneqq \lim \widetilde{M}_{c,n}^{-1} \circ f_n \circ M_{a,n}
    $$
    exists and has degree $\geq 1$.
\end{defn}

We will use the notation $M_{\bullet,n}$ to denote rescalings of the domain and $\widetilde{M}_{\bullet,n}$ for rescalings of the codomain.
In the following, we will associate rescalings as coordinates on a tree of Riemann spheres. The notation $\bullet$ in the subscript represents a vertex in the underlying tree (see \S~\ref{nodal_sphere_subsubsec}).

The following lemma follows from \cite[Lemma 2.1]{DF14} and \cite[Lemma 8.5, Lemma 8.8]{Luo21}.
\begin{lem}[Existence of rescaling limits]\label{lem:rl}
    Let $f_n \in \Rat_{d}(\C)$, and $M_{a,n} \in \PSL_2(\C)$. Then after passing to a subsequence, there exists {$\widetilde{M}_{c,n} \in \PSL_2(\C)$ so that
    $$
    f_{a\to c} = \lim_{a\to c} f_n
    $$
    exists and has degree $\geq 1$. Moreover, the sequence $\widetilde{M}_{c,n}$ is unique up to bounded equivalence.}

    On the other hand, let {$\widetilde{M}_{c,n} \in \PSL_2(\C)$. Then after passing to a subsequence, there exist pairwise independent $M_{a_1,n}, \dots, M_{a_k,n} \in \PSL_2(\C), k \leq d$, so that for each $i\in\{1,\cdots, k\}$, the limit 
    $$
    f_{a_i\to c} = \lim_{a_i\to c} f_n
    $$ 
    exists and has degree $d_i \geq 1$.} Moreover, $\sum_i^k d_i = d$, and the sequence $M_{a_i,n}$ is unique up to bounded equivalence and permutation of the {$a_i$'s}.
\end{lem}

See Example~\ref{DegenerationFamily} for a concrete example of such rescaling limits.
Let us briefly discuss the intuition behind the above lemma.  One can identify the Riemann sphere $\widehat{\C} \cong S^2$ as the boundary of the Poincar\'e ball model for hyperbolic $3$-space $\Hyp^3 \cong B({\bf 0}, 1)$. Note that a M\"obius map extends to an isometry of $\Hyp^3$ under this identification.
Via the barycentric extension, any rational map $f$ can be extended to a map $\E f: \Hyp^3 \longrightarrow \Hyp^3$. Let $\bf 0 \in \Hyp^3$ be the center of the ball model. Given the rescalings $M_{a,n}$,  the rescaling $\widetilde{M}_{c,n} \in \PSL_2(\C)$ in Lemma~\ref{lem:rl} can be chosen to satisfy
$$
\widetilde{M}_{c,n} ({\bf 0}) = \E f_n (M_{a,n}({\bf 0})).
$$
Note that such a sequence is well-defined up to pre-composition with the stabilizer of $\bf 0$, which is the compact group $SO(3)$.
Given the rescalings $\widetilde{M}_{c,n}$, the sequences $M_{a_i,n}$ in the second part of Lemma~\ref{lem:rl} can be chosen similarly so that
$$
\widetilde{M}_{c,n} ({\bf 0}) = \E f_n (M_{a_i,n}({\bf 0})).
$$
We refer the readers to \cite{Luo21} for a more detailed discussion.

Lemma~\ref{lem:rl} will be used below to zoom in at different points of the Riemann sphere to extract different conformal limits. One can think of this process as the conformal analog of the notion of geometric limits occurring in the theory of Kleinian groups \cite[Ch. 9]{thurstonnotes} in the following sense. One considers a rational map between Riemann spheres at a sequence of scales tending to zero, and by zooming in or out at appropriate locations, using the M\"obius maps $M_{a,n} \in \PSL_2(\C)$ and $\widetilde{M}_{c,n}\in \PSL_2(\C)$, one extracts a limiting map. These limiting maps are rational maps, but may have strictly smaller degree. Lemma~\ref{lem:rl}  allows one to extract these non-trivial 
(i.e.\ non-constant) limits.

This lemma will also play a key role in showing that the analog of the character variety in our setting is Hausdorff (Proposition~\ref{prop-charvarhausdorff}).

The following lemma follows immediately from Lemma \ref{comconCorrepondence}
\begin{lem}\label{lem:rlc}
    Let $\mathfrak{C}_n = (f_n, g_n) \in \widehat{\Cor}_{d_1,d_2}$. {Let $\widetilde{M}_{c, n} \in \PSL_2(\C)$.}
    Let $M_{a,n}, M_{b,n} \in \PSL_2(\C)$ be rescalings so that 
    $$
    f_{a\to c} = \lim_{a\to c} f_n \text{ and } g_{b\to c} = \lim_{b\to c} g_n
    $$
    exist, and have degree $\geq 1$. Then 
    \begin{align*}
    \mathfrak{C}_{a,b,n}&\coloneqq\{(x,y) \colon (M_{a,n}(x), M_{b,n}(y)) \in \mathfrak{C}_n\}\\
    &=\{(x,y) \colon f_n \circ M_{a,n}(x) = g_n\circ M_{b,n}(y)\}
    \end{align*}
    converges compactly to $\varphi_{\mathfrak{C}_{a,b}} = (\varphi_{f_{a\to c}}, \varphi_{g_{b\to c}})$ away from $(\mathcal{H}(f_{a\to c}), \mathcal{H}(g_{b\to c}))$.
\end{lem}

\subsubsection{Convergence on trees of spheres}\label{nodal_sphere_subsubsec}
We shall need 
to organize the 
set of rescalings of rational maps. Degenerations 
of rational maps are uniquely hierarchically organized, giving rise to a natural tree structure,
see \cite{Kiw15,arfeux,Luo22a} for instance.
This motivates the introduction of the notion of trees of spheres, a structure that will underpin this organization and allow us to make sense of 
a non-dynamical  rational map between trees of Riemann spheres (c.f. \cite{Luo22b} for a dynamical version).

\begin{defn}[Graph/tree of Riemann spheres]\label{defn:trs} A \emph{graph of Riemann spheres} is a graph of spaces
(cf.~\cite[p. 153-155]{scott-wall}), where 
\begin{enumerate}[leftmargin=*]
    \item vertex spaces are Riemann spheres with a spherical metric of diameter one, and 
    \item edge spaces are points.
\end{enumerate}
The underlying graph is denoted as $\mathscr G$, and it sets of vertices and edges are denoted as $\RV, \mathscr{E}$, respectively.

If $\mathscr G$ is a tree $\RT$, then we call it a \emph{tree of Riemann spheres} and denote it as $(\RT, \widehat\C^\RV)$.
Note that we have equipped $\RT$ with the simplicial metric, where each edge has length one.
We refer to the resulting metric $d_\RT$ on $\RT$ as the \emph{edge metric}.
 \end{defn}

A graph of Riemann spheres may thus be thought of as follows. It consists of
\begin{enumerate}[leftmargin=*]
    \item an abstract  simple graph $\mathscr{G}$ with vertex set $\RV$ and edge set $\mathscr{E}$,
    \item a collection $\{\widehat{\mathbb{C}}_a: a\in \RV\}$ of Riemann spheres equipped with spherical metrics (these are the the vertex spaces),
    \item a collection $\{I_e: e\in \mathscr{E}\}$, where each $I_e$ is isometric to $[0,1]$ and connects two
    \emph{unspecified} points of $\widehat{\C}_a$ and $\widehat{\C}_b$, where $a, b \in \RV$ are the two endpoints of $e$.  The intervals $I_e$ may thus be identified with  $\{e\} \times [0,1]$, $e\in \mathscr{E}$.
\end{enumerate} 

The metric $d_\RT$ is not important; we make a choice to streamline our discussion below. We now proceed to define a \emph{marking} on a tree of spheres, which specifies the points at which the edges are glued to the spheres. 

We elucidate Definition~\ref{defn:trs} and the above notation with the following example. Let $\RT$ be an interval with two vertices $\RV =\{a,b\}$ and one edge $\mathscr{E} = \{e\}$. Now connect $\widehat{\C}_a$ and $\widehat{\C}_b$ with an edge $I_e$ at $x \in \widehat{\C}_a$ and $y \in \widehat{\C}_b$. For any choice of such attaching points $x \in \widehat{\C}_a, y \in \widehat{\C}_b$, we  use the notation $(\RT, \widehat\C^\RV)$ to represent any such tree of spheres. However, different choices of $x, y$ give different markings as defined below.

\begin{defn}[Tangent space and marking]\label{def-tangentdirn}
   Let $(\RT, \widehat\C^\RV)$ be a tree of Riemann spheres.
   For a vertex $a$ of $\RT$, we define the \emph{tangent space} $T_a\RT$ at $a$ to be the set of edge segments in $B_{d_\RT}(a,\epsilon)-\{a\}$, for $\epsilon>0$ small enough.
   Further, if a component of $B_{d_\RT}(a,\epsilon)-\{a\}$ is contained, necessarily uniquely, in the edge $[a,b]$ of $\RT$, then it is said to be a \emph{tangent vector at $a$ 
   in the direction of $b$.}

   {A} {\em marking} of the tangent space $T_a \RT$ is an  injective map  $\xi_a: T_a\RT \xhookrightarrow{} \widehat\C_a$ so that the edge $[a,b]$ is attached to $\xi_a(v) \in \widehat\C_a$ where $v$ is the tangent vector at $a$ in the direction of $b$.
   The image $\Xi_a :=\xi_a(T_a\RT)$ is called the {\em singular set} of $\widehat{\C}_a$, $a\in\RV$. We define $\Xi = \bigcup_{a\in \RV}\Xi_a$, and call it the \emph{singular set} of $(\RT,\widehat{\C}^{\RV})$. A point in the complement $\widehat \C^\RV - \Xi$ is called a {\em smooth point}. A tree of spheres $(\RT, \widehat\C^\RV)$ equipped with markings $\{\xi_a:T_a\RT \xhookrightarrow{} \widehat\C_a: a\in\RV\}$ will be referred to as a tree of spheres marked by $\xi$ ($= \{\xi_a\}$), or simply a
   marked tree of spheres when the marking is clear from context.
\end{defn}
 Heuristically, $T_a\RT$ is the set of `tangent directions' of edges of $\RT$ emanating from $a$. Further, for a simplicial  map $\pmb{F}:\RT\to\widetilde{\RT}$
between trees, we define its \emph{derivative} at a vertex $a$ of $\RT$ to be the induced map $D\pmb{F}_a:T_a\RT\to T_{\pmb{F}(a)}\widetilde{\RT}$ on the tangent directions.

For the purposes of this paper, we focus on the case where the tree $\RT$ is finite: what this means is that we shall consider only finitely many independent rescalings.

\begin{remark}
    The above two-step definition 
    (Definitions~\ref{defn:trs} and ~\ref{def-tangentdirn}) will be important in \S~\ref{char_var_subsec}, where we will introduce suitable parameter spaces of rational maps defined on a given tree of spheres $(\RT,\widehat{\C}^\RV)$ whose markings will be allowed to vary.
\end{remark}

\begin{defn}[Rational map between trees of Riemann spheres]\label{defn:trs-map} 
Let $(\RT, \widehat\C^\RV)$ and $(\widetilde{\RT}, \widehat\C^{\widetilde{\RV}})$ be a pair of trees of Riemann spheres.
A {\em rational map} $(\pmb{F}, \pmb{R}): (\RT, \widehat\C^\RV) \longrightarrow (\widetilde{\RT}, \widehat\C^{\widetilde{\RV}})$ {consists of 
\begin{enumerate}
    \item a simplicial map 
$$
\pmb{F}: (\RT, \RV) \longrightarrow (\widetilde{\RT}, \widetilde{\RV}) \text{ that sends an edge to an edge or to a vertex},
$$ 
\item markings $\xi_a$, $\widetilde{\xi}_{\pmb{F}(a)}$
for vertices $a \in \RV, \pmb{F}(a) \in \widetilde{\RV}$, and
\item a union of maps $\pmb{R}:= \bigcup_{a\in \RV} R_a$
\end{enumerate}
 so that the following  hold.}
\begin{enumerate}[leftmargin=*]
\item $R_a: \widehat\C_a \longrightarrow \widehat \C_{\pmb{F}(a)}$ is a rational map of degree at least $1$. 
\item Let $v \in T_a\RT$, and $E=[a,b]$ be the edge of $\RT$ associated with the tangent direction $v$.
\begin{itemize}[leftmargin=*]
    \item If $\pmb{F}(E)$ is an edge, then
    $$
    R_a \circ \xi_a(v) = \widetilde{\xi}_{\pmb{F}(a)} \circ D\pmb{F}_a(v),
    $$
    where $D\pmb{F}_a(v) \in T_{\pmb{F}(a)}\widetilde{\RT}$ is the tangent direction at $\pmb{F}(a)$ associated with $\pmb{F}(E)$. 
    \item If $\pmb{F}(E)$ is a vertex, then
    $$R_a \circ \xi_a(v) = R_b \circ \xi_b(w),$$
    where $w \in T_b\RT$ is the tangent direction at $b$ associated with $E$.
\end{itemize}
\end{enumerate}
$(\pmb{F}, \pmb{R})$ is said to have degree $d$ if every smooth point in $\widehat\C^{\widetilde{\RV}}$ has exactly $d$ preimages counted with multiplicities.
We say that $(\pmb{F}, \pmb{R})$  is an {\em isomorphism} if $\pmb{F}$ is an isomorphism between trees and each rational map has degree $1$; i.e., is a M\"obius map. It is an {\em automorphism} if the domain and codomain are the same and the tree map is a {simplicial automorphism.} 
\end{defn}

\begin{defn}\label{defn:cvrationalmap}
    Let $(\RT, \widehat\C^\RV)$ be a  tree of Riemann spheres marked by $\{\xi_a : T_a\RT \xhookrightarrow{} \widehat{\C}_a\}$. A collection of M\"obius maps $M_{a,n} \in \PSL_2(\C), a\in \RV$, is said to be a set of rescalings for $(\RT, \widehat\C^\RV)$ if for any {distinct} pair $a, b \in \RV$, the map $M_{a,n}^{-1} \circ M_{b,n}(z)$ converges to the constant map $\xi_a(v)$, where $v\in T_a\RT$ is the tangent vector in the direction of $b$.
    
    A sequence $f_n:\widehat{\C}\to\widehat{\C}$ of degree $d$ rational maps is said to {\em converge} to a rational map $(\pmb{F}, \pmb{R}):(\RT, \widehat\C^\RV) \longrightarrow (\widetilde{\RT}, \widehat\C^{\widetilde{\RV}})$ with respect to rescalings $M_{a,n} \in \PSL_2(\C)$ and $\widetilde{M}_{c,n}\in \PSL_2(\C)$ if
    \begin{itemize}
        \item $M_{a,n}, \widetilde{M}_{c,n}$ are rescalings for $(\RT, \widehat\C^\RV)$ and $(\widetilde{\RT}, \widehat\C^{\widetilde{\RV}})$ respectively;
        \item $\widetilde{M}_{\pmb{F}(a),n}^{-1} \circ f_n \circ M_{a,n}(z) \to R_a(z)$ compactly on $\widehat \C_a- \Xi_a$.
    \end{itemize}
    We say $f_n$ {\em converge} to $(\pmb{F}, \pmb{R})$ if $f_n$ converges to $(\pmb{F}, \pmb{R})$ with respect to some rescalings.
\end{defn}

{
\begin{rmk}
We remark that it is convenient to regard the rescalings $M_{a,n}$ as maps between two different Riemann spheres $M_{a,n}: \widehat{\C}_a \longrightarrow \widehat{\C}$. With this in mind, $\widetilde{M}_{\pmb{F}(a),n}^{-1} \circ f_n \circ M_{a,n}$ is a map between $\widehat{\C}_a$ to $\widehat{\C}_{\pmb{F}(a)}$, which has the same domain and codomain as $R_a: \widehat{\C}_a \longrightarrow \widehat{\C}_{\pmb{F}(a)}$.
Similarly, $M_{a,n}^{-1} \circ M_{b,n} : \widehat{\C}_b \longrightarrow \widehat{\C}_a$ converges (compactly away from the hole) to the point $\zeta_a(v) \in \widehat{\C}_a$, where $v\in T_a\RT$ is the tangent vector in the direction of $b$ (see Figure~\ref{fig:TS}).

We also remark that the rational maps on tree of spheres does not provide a compactification of $\Rat_d$, but it provides a certain bordification.
\end{rmk}
}

\begin{figure}[ht]
\captionsetup{width=0.96\linewidth}
  \centering
  \includegraphics[width=0.7\linewidth]{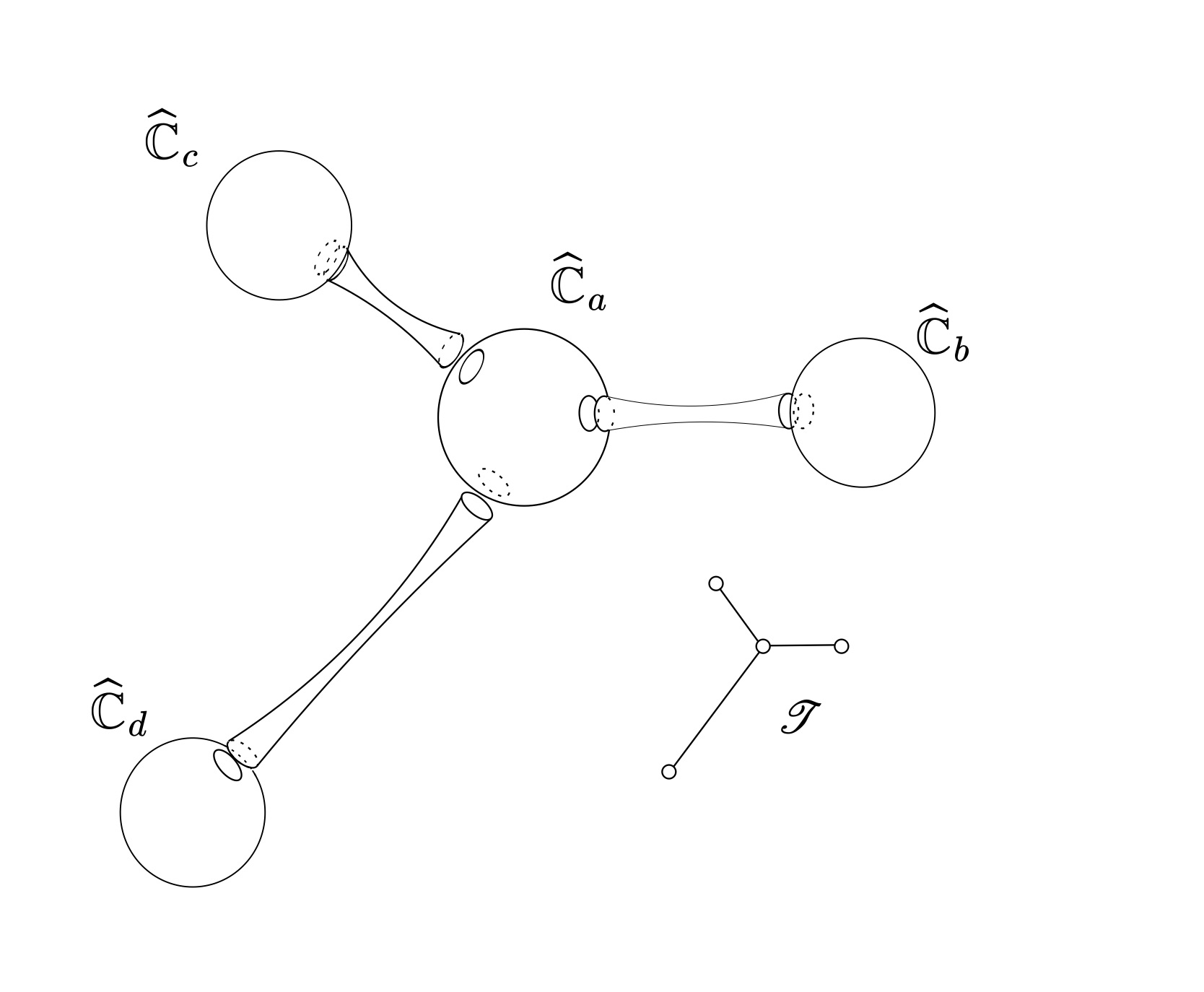}
  \caption{Illustrated is the degeneration of a Riemann sphere to a tree of spheres. The underlying tree is also depicted.}
  \label{fig:TS}
\end{figure}

The above definition naturally generalizes to correspondences.
\begin{defn}
    Let $(\RT, \widehat\C^\RV)$ be a  tree of Riemann spheres marked by $\{\xi_a : T_a\RT \xhookrightarrow{} \widehat{\C}_a\}$. A {\em holomorphic correspondence} on $(\RT, \widehat\C^\RV)$ is a subvariety $\mathfrak{C}$ in $\widehat\C^\RV\times \widehat\C^\RV$ of complex dimension $1$ with the property that there exist $d_1, d_2\geq 1$ so that the projection maps $\pi_1$ and $\pi_2$ have degree $d_1$ and $d_2$ over the set of smooth points $\widehat \C^\RV - \Xi$. We call $(d_1, d_2)$ the bi-degree of the correspondence $\mathfrak{C}$.

    A sequence of correspondences $\mathfrak{C}_n$ on $\widehat{\C}$ of bi-degree $(d_1, d_2)$ is said to converge to a bi-degree $(d_1, d_2)$ correspondence $\mathfrak{C}$ on $(\RT, \widehat\C^\RV)$ with respect to rescalings $M_{a,n} \in \PSL_2(\C)$ if $M_{a,n}$ are rescalings for $(\RT, \widehat\C^\RV)$ and for any pair $a, b \in \RV$, we have that
    $$
        \mathfrak{C}_{a,b,n}\coloneqq\{(x,y) \colon (M_{a,n}(x), M_{b,n}(y)) \in \mathfrak{C}_n\}
    $$
    converges to $\mathfrak{C} \cap (\widehat\C_a \times \widehat\C_b)$ compactly away from $(\Xi_a, \Xi_b)$. 
    Similarly, we say $\mathfrak{C}_n$ {\em converges} to $\mathfrak{C}$ if $\mathfrak{C}_n$ converges to $\mathfrak{C}$ with respect to some rescalings.
\end{defn}

Let $(\RT, \widehat\C^\RV)$ and $(\widetilde{\RT}, \widehat\C^{\widetilde{\RV}})$ be a pair of (marked) trees of Riemann spheres, and $(\pmb{F}, \pmb{R})$ and $(\pmb{G}, \pmb{S})$ be rational maps from $(\RT, \widehat\C^\RV)$ to $(\widetilde{\RT}, \widehat\C^{\widetilde{\RV}})$ of degree $d_1$ and $d_2$ respectively.
Then it follows from the definition that the \textit{coincidence locus} 
$$
\mathfrak{C}\coloneqq\{(x,y) \in \widehat\C^\RV \times \widehat\C^\RV\colon \pmb{R}(x) = \pmb{S}(y)\}
$$
gives a correspondence on $(\RT, \widehat\C^\RV)$.
\begin{prop}\label{prop:convergenceCorrespondence}
    Let $f_n, g_n:\widehat{\C}\to\widehat{\C}$ be sequences of rational maps of degree $d_1$ and $d_2$ respectively. Suppose there exist (marked) trees of spheres $(\RT, \widehat\C^\RV)$ and $(\widetilde{\RT}, \widehat\C^{\widetilde{\RV}})$ 
    equipped with rescalings $M_{a,n}$ and 
    $\widetilde{M}_{c,n}$ respectively
    so that $f_n$ and $g_n$ converge to rational maps $(\pmb{F},\pmb{f}), (\pmb{G},\pmb{g})$ of degree $d_1$ and $d_2$ from $(\RT, \widehat\C^\RV)$ to $(\widetilde{\RT}, \widehat\C^{\widetilde{\RV}})$ with respect to these rescalings.
    Then $\mathfrak{C}_n = (f_n, g_n)$ converges to a bi-degree $(d_1, d_2)$ correspondence $\mathfrak{C}$ on $(\RT, \widehat\C^\RV)$ with respect to rescalings $M_{a,n}$.

    Conversely, suppose $\mathfrak{C}_n = (f_n, g_n)$ converges to a bi-degree $(d_1, d_2)$ correspondence $\mathfrak{C} = (\pmb{f},\pmb{g})$ on $(\RT, \widehat\C^\RV)$ with respect to rescalings $M_{a,n}$. Then there exist 
    a (marked) tree of spheres $(\widetilde{\RT}, \widehat\C^{\widetilde{\RV}})$ and rescalings $\widetilde{M}_{c,n}$ for  $(\widetilde{\RT}, \widehat\C^{\widetilde{\RV}})$ such that $f_n$ and $g_n$ converge to $(\pmb{F},\pmb{f}), (\pmb{G},\pmb{g}):(\RT, \widehat\C^\RV)\longrightarrow (\widetilde{\RT}, \widehat\C^{\widetilde{\RV}})$ with respect to rescalings $M_{a,n}$ and~$\widetilde{M}_{c,n}$.
\end{prop}

{ Before proceeding with the proof, we point out that the rescalings  $M_{a,n}$ and $\widetilde{M}_{c,n}$ work \emph{simultaneously} for the sequences $f_n$ as well as $g_n$.

Note also that the convergence of the correspondences $\mathfrak{C}_n$ only needs rescalings $M_{a,n}$ in the domain. Hence the first statement is the easier one as one forgets information. For the converse direction, we need to recover the rescalings $\widetilde{M}_{c,n}$ of the codomain. This necessitates the use of Lemma \ref{lem:rl}.
}
\begin{proof}
    {Let $a, b \in \RV$ be such that $\pmb{F}(a)=\pmb{G}(b)=:c\in\widetilde{\RV}$.
    By assumption,
    $$
    \widetilde{M}_{c,n}^{-1} \circ f_n \circ M_{a,n}(z) \to f_a(z) \text{ and } \widetilde{M}_{c,n}^{-1} \circ g_n \circ M_{b,n}(z) \to g_b(z) 
    $$ compactly on $\widehat \C_a- \Xi_a$ and $\widehat \C_b- \Xi_b$.
    
    We define a correspondence $\mathfrak{C}$ on $(\RT, \widehat\C^\RV)$ so that $\mathfrak{C}_{a,b} = \{(x,y) \in \widehat\C_a \times \widehat\C_b\colon f_a(x) = g_b(y)\}$.
    Therefore, by Lemma~\ref{lem:rlc}, 
    $$
    \mathfrak{C}_{a,b,n}\coloneqq\{(x,y) \colon (M_{a,n}(x), M_{b,n}(y)) \in \mathfrak{C}_n\} \to \mathfrak{C}_{a,b}
    $$
    compactly away from $(\Xi_a, \Xi_b)$.

    Conversely, let $a, b \in \RV$ such that $\mathfrak{C}_{a,b}$ is non-empty; in particular, $\pmb{F}(a)=\pmb{G}(b)=:c\in\widetilde{\RV}$.
    By Lemma \ref{lem:rl}, there exists $\widetilde{M}_{c,n} \in \PSL_2(\C)$ so that after passing to a subsequence, $f_{a\to c}\coloneqq\lim_{a\to c} f_n$ exists and has degree $\geq 1$. Since $\mathfrak{C}_{a,b,n} \to \mathfrak{C}_{a,b}$, we conclude that $g_{b\to c}\coloneqq\lim_{b\to c} g_n$ exists and has degree $\geq 1$ with respect to the rescalings $M_{b,n}$ and $\widetilde{M}_{c,n}$. Moreover, after modifying $M_{b,n}$ in its bounded equivalence class, we can assume that $(f_{a\to c}, g_{b\to c}) = (\pmb{f}|_{\widehat{\C}_a}, \pmb{g}|_{\widehat{\C}_{b}})$. Since this is true for any subsequential limit and any $a, b\in\RV$ satisfying $\pmb{F}(a)=\pmb{G}(b)$, it follows that $f_n, g_n$ converges to~$\pmb{f}, \pmb{g}$.}
\end{proof}

\subsection{Pointed disks and Carath\'eodory convergence}\label{pointed_disk_subsec}
In this subsection, we summarize some well-known results for Carath\'eodory limits of pointed disks. These will be used to construct rescaling limits of rational maps and correspondences.

A {\em disk} is an open simply connected  proper subset in $\widehat\C$.
It is said to be {\em hyperbolic} if it admits a hyperbolic metric; i.e., if it is not of the form $\widehat{\C} - \{a\}$ for some $a \in \widehat\C$.
We have the following definition of Carath\'eodory convergence of pointed disks, and proper holomorphic maps (see \S 5 in \cite{McM94}).
\begin{defn}
Let $(U_n, u_n)$ be a sequence of pointed disks. We say $(U_n, u_n)$ converges to $(U, u)$ in Carath\'eodory topology if 
\begin{itemize}
\item $u_n \to u$;
\item For any compact set $K \subseteq U$, $K \subseteq U_n$ for all sufficiently large $n$;
\item For any open connected set $N$ containing $u$, if $N \subseteq U_n$ for all sufficiently large $n$, then $N \subseteq U$.
\end{itemize}

We say that a sequence of proper holomorphic maps between pointed disks $f_n: (U_n, u_n) \longrightarrow (V_n, v_n)$ converges to $f:(U, u) \longrightarrow (V, v)$ if 
\begin{itemize}
\item $(U_n, u_n), (V_n,v_n)$ converge respectively to $(U, u), (V, v)$ in Carath\'eodory topology;
\item For all sufficiently large $n$, $f_n$ converges to $f$ uniformly on compact subsets of $U$.
\end{itemize}
\end{defn}

We have the following important compactness result.
\begin{theorem}\cite[Theorem 5.2]{McM94}\label{thm:cmpc}
The set of disks $(U_n, 0)$ containing $B(0,r)$ for some $r>0$ is compact in Carath\'eodory topology.
\end{theorem}

For a sequence of pointed hyperbolic disks, we define the corresponding rescalings (c.f. \cite[Definition 4.2]{Luo22b}).
\begin{defn}\label{defn:rescalingPointedDisk}
    Let $(U_n, u_n)$ be a sequence of pointed hyperbolic disks. A sequence $M_{u,n} \in \PSL_2 (\C)$ is defined to be a {\em rescaling} for $(U_n, u_n)$ if $M_{u,n}^{-1}(U_n,u_n)$ converges to a pointed hyperbolic disk $(U,u)$ in the Carath\'eodory topology.
\end{defn}

\begin{prop}\label{prop:ConvergencePointedDisk}
Let $(U_n, u_n)$ be a sequence of pointed hyperbolic disks. Then after passing to a subsequence if necessary, there exist
{ rescalings} $M_{u,n}\in \PSL_2(\C)$ so that $M_{u,n}^{-1}(U_n, u_n)$ converges to a pointed hyperbolic disk.

Suppose $M_{u,n}, L_{u,n} \in \PSL_2(\C)$ are two rescalings for $(U_n, u_n)$. Then after passing to a subsequence if necessary, they are equivalent
{(in the sense of Definition~\ref{def-equivalentscaling})}.
\end{prop}
\begin{proof}
The first statement follows immediately from Theorem \ref{thm:cmpc} {(for instance, by requiring that $M_{u,n}^{-1}(u_n)=0$ and $M_{u,n}^{-1}(U_n)$ contains a definite round disk $B(0,r)$ for all $n$)}.

{Let $(U^1,u^1) = \lim M_{u,n}^{-1}(U_n, u_n)$ and $(U^2, u^2) = \lim L_{u,n}^{-1}(U_n, u_n)$. Then after passing to a subsequence if necessary, $L_{u,n}^{-1}\circ M_{u,n}$ converges compactly to a conformal map between $(U^1, u^1)$ and $(U^2, u^2)$, so they are equivalent.}
\end{proof}

The following proposition allows us to construct rescaling limits of rational maps from proper maps between pointed hyperbolic disks.
\begin{prop}\label{prop:ConvergenceProperMap}
    Let $f_n\colon(U_n, u_n)\longrightarrow (V_n, v_n)$ be a sequence of proper maps between pointed hyperbolic disks. 
    Let {$M_{u,n}$ and $\widetilde{M}_{v,n}$} be rescalings for $(U_n, u_n)$ and $(V_n, v_n)$ respectively.
    Then after passing to a subsequence if necessary, {the map $\widetilde{M}_{v,n}^{-1} \circ f_n \circ M_{u,n}$} converges compactly to a proper holomorphic map $f_{u\to v}: (U, u) \longrightarrow (V,v)$ between two pointed hyperbolic disks.

    Moreover, suppose that there are exactly $e-1$ critical points (counted with multiplicity) that are bounded hyperbolic distance away from $u_n$. Then $f_{u\to v}:(U, u) \longrightarrow (V,v)$ has degree $e$.
\end{prop}
\begin{proof}
    By Proposition \ref{prop:ConvergencePointedDisk}, after passing to a subsequence if necessary, {$M_{u,n}^{-1}((U_n, u_n))$ and $\widetilde{M}_{v,n}^{-1}((V_n, v_n))$} converge in the Carath\'eodory topology to pointed hyperbolic disks $(U, u)$ and $(V, v)$ respectively. {Note that $\widetilde{M}_{v,n}^{-1} \circ f_n \circ M_{u,n}$ are proper maps from $M_{u,n}^{-1}((U_n, u_n))$ to $\widetilde{M}_{v,n}^{-1}((V_n, v_n))$}. Thus, the proposition follows from case (3) of \cite[Theorem 5.6]{McM94}.
\end{proof}

\begin{cor}\label{cor:caratheodorylimitandrescaling}
    Let $f_n \in \Rat_d(\C)$. Suppose $f_n$ restricts to a proper holomorphic map between pointed hyperbolic disks $(U_n, u_n)$ and $(V_n, v_n)$. Let {$M_{u,n}$ and $\widetilde{M}_{v,n}$} be the corresponding rescalings.
    Then after passing to a subsequence if necessary, the limit
   { $$
    f_{u\to v} = \lim_{u\to v} f_n = \lim \widetilde{M}_{v,n}^{-1} \circ f_n \circ M_{u,n}
    $$}
    exists, and has degree~$\geq 1$.

    Moreover, suppose that there are exactly $e-1$ critical points (counted with multiplicity) that are bounded hyperbolic distance away from $u_n$. Then $f_{u\to v}$ has degree~$\geq e$.
\end{cor}
\begin{proof}
    After passing to a subsequence, we may assume $f_{u\to v} = \lim_{u\to v} f_n$ exists. {(Note that the limit is a priori possibly constant.)} By Proposition \ref{prop:ConvergenceProperMap}, after passing to a further subsequence if necessary, {$\widetilde{M}_{v,n}^{-1} \circ f_n \circ M_{u,n}$} converges to some proper holomorphic map between two pointed hyperbolic disks $(U, u)$ and $(V, v)$. This is impossible if the limit $f_{u\to v}$ has degree $0$. Therefore, $f_{u\to v}$ has degree $\geq 1$. 
    For the second part, note that by Proposition \ref{prop:ConvergenceProperMap}, $f_{u\to v}|_{U}$ has degree $e$. Therefore, $f_{u\to v}$ has degree $\geq e$ since the global degree of the rational map $f_{u\to v}$ is at least the {degree of the restriction $f_{u\to v}|_{U}$}.
\end{proof}

\section{A character variety of correspondences}\label{sec:cv}
In this section, we will briefly summarize the single mating construction of polynomials and Fuchsian groups. Such matings are realized as correspondences that satisfy some additional desirable structures. We will also introduce an ambient space to study these `mating correspondences'.

\subsection{Algebraic correspondences as matings of polynomials with punctured spheres}\label{corr_mating_subsec}
While the matings framework and the main result of this paper remain valid for a general class of genus zero orbifolds, we restrict ourselves to punctured spheres for simplicity of exposition.

In this subsection, we will briefly summarize the construction of such matings as algebraic correspondences. We refer the readers to \cite{MM2, LLM24} for more details.

\subsection*{Step I: Bowen-Series maps} The first step in the construction is to replace the action of a Fuchsian group uniformizing a punctured sphere with a piecewise Fuchsian, Markov, expansive, degree $k\geq 2$ circle covering map, called the \emph{Bowen-Series map}.

\begin{figure}[ht]
\captionsetup{width=0.98\linewidth}
	\begin{tikzpicture}
		\node[anchor=south west,inner sep=0] at (0.5,0) {\includegraphics[width=0.32\linewidth]{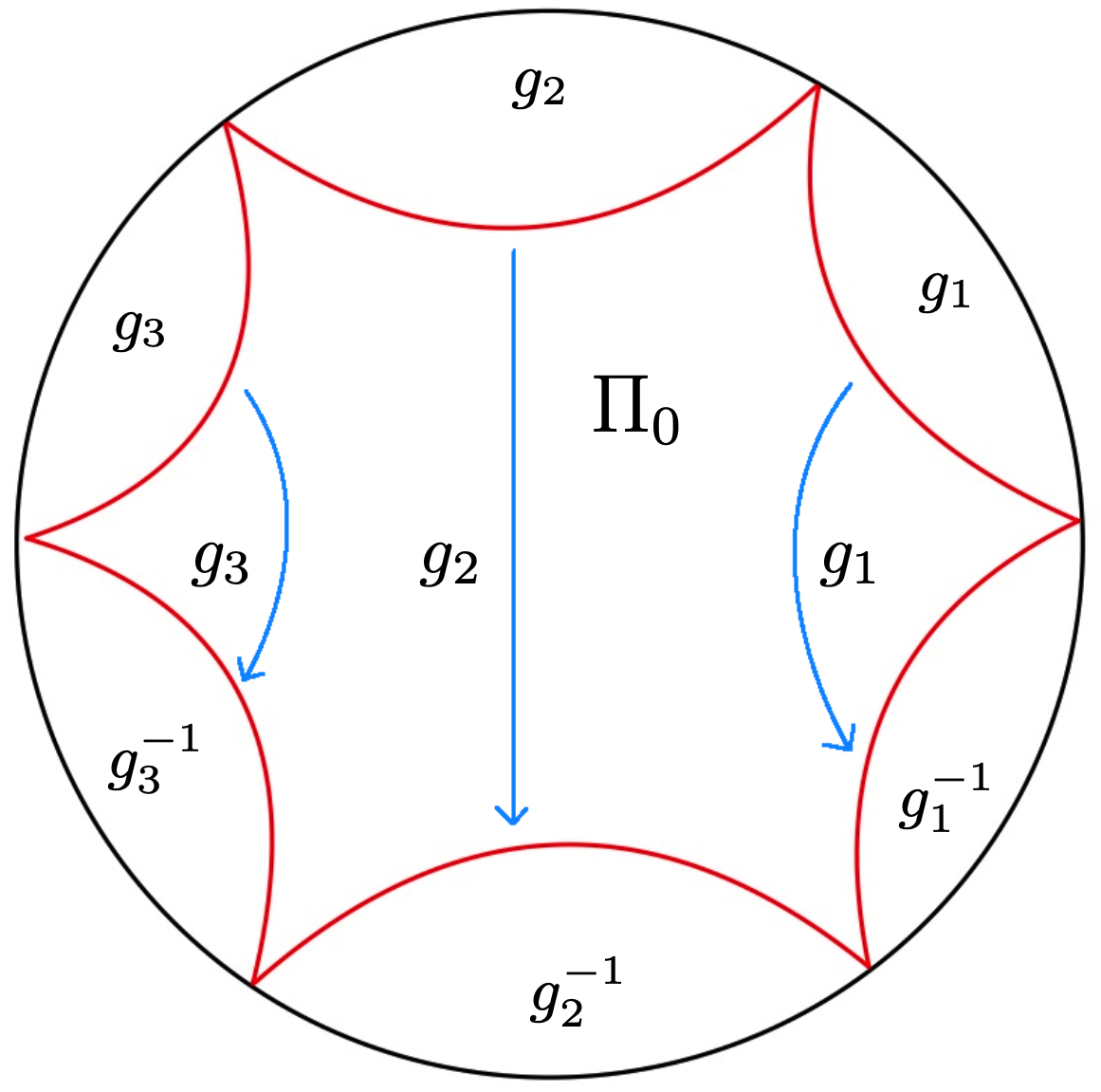}};
	\end{tikzpicture}
	\caption{Pictured is the preferred fundamental domain $\Pi_0$ and the action of the associated Bowen-Series map for a sphere with four punctures.}
	\label{punc_sphere_bs_fig}
\end{figure}

For $d\geq 2$, let $\Pi_0$ be a regular $2d-$sided closed ideal polygon in the unit disk with its vertices at the $2d$-th roots of unity. {As noted in \cite{MM1}, the degree $k$ of the associated Bowen-Series map will be 
equal to $2d-1$}. One can define side-pairing transformations for $\Pi_0$ that take each edge of $\Pi_0$ to its real-symmetric edge, and such that these side-pairing transformations generate a Fuchsian group $\Gamma_0$ which uniformizes a sphere with $d+1$ punctures. The action of the associated Bowen-Series map $A_{\Gamma_0}:\mathcal{D}_0:=\overline{\D}-\Int{\Pi_0}\to\overline{\D}$ is shown in Figure~\ref{punc_sphere_bs_fig}. We refer the reader to \cite{BS79,MM1} for general background on Bowen-Series maps and its special properties for punctured spheres.

Note that each element of the Teichm\"uller space $\mathrm{Teich}(S_{0,d+1})$ of spheres with $d+1$ punctures is realized by a discrete, faithful, strongly type-preserving representation $\rho:\Gamma_0\to\Gamma\leq\mathrm{PSL}_2(\R)$ such that
$$
\rho(g)=\phi_\rho\circ g \circ\phi_\rho^{-1},\  \forall g\in\Gamma_0,
$$
for some quasiconformal homeomorphism $\phi_\rho:\overline{\D}\to\overline{\D}$. We normalize $\phi_\rho$ such that $\phi_\rho(1)=1$. 
{ As noted in \cite[p. 1454]{MM1},
 the quasiconformal homeomorphism $\phi_\rho$ is not unique. However, any two normalized quasiconformal homeomorphisms that induce the same representation extend to the same quasisymmetric homeomorphism $\phi_\rho$ on the circle, and conjugate the marked generators of $\Gamma_0$ to those of $\Gamma$. In particular, the images of the ideal boundary points of $\Pi_0$ under $\phi_\rho$ are independent of the choice of $\phi_\rho$. We denote the hyperbolic convex hull (in $\D$) of the $\phi_\rho-$images of the ideal boundary points of $\Pi_0$ by $\Pi_\Gamma$, and note that $\Pi_\Gamma$ is an ideal polygon in $\D$ (i.e., it has geodesic edges) and is a preferred fundamental domain for the marked group $\Gamma$. One can now define the Bowen-Series map 
$$
A_{\Gamma}:\mathcal{D}_\Gamma:=\overline{\D}-  \Int{\Pi_\Gamma}\to\overline{\D}
$$ 
as a piecewise M{\"o}bius map whose pieces are given by the marked generators of $\Gamma$ on appropriate hyperbolic half-planes.

In fact, one can choose a representative quasiconformal map $\phi_\rho$ with the property $\phi_\rho(\Pi_0)=\Pi_\Gamma$. With such a choice, one has  
$$
\cD_\Gamma=\overline{\D}-\Int{\phi_\rho(\Pi_0)}=\phi_\rho(\cD_0),\qquad \textrm{and}\qquad A_\Gamma=\phi_\rho\circ A_{\Gamma_0}\circ \phi_\rho^{-1}.
$$}

\subsection*{Step II: Conformal matings} The Bowen-Series map $A_\Gamma$ can be conformally mated with complex polynomials. The resulting conformal mating is a meromorphic map in one complex variable that behaves like a polynomial on one part of its dynamical plane and resembles the Bowen-Series map on the remaining part.

More precisely, let $P$ be a monic, centered complex polynomial of degree $k$ with a connected and locally connected Julia set. We denote its basin of attraction of infinity by $\mathcal{B}_\infty(P)$ and its filled Julia set by $\mathcal{K}(P)$. The \emph{B{\"o}ttcher coordinate} of $P$ is the unique conformal map
$$
\psi_P:\widehat{\C}-\overline{\D}\to\mathcal{B}_\infty(P),
$$ 
that conjugates $z^{k}$ to $P$, and is normalized to be tangent to the identity map near infinity.
Due to local connectedness of the Julia set $\mathcal{J}(P)$, the map $\psi_P$ extends continuously to $\mathbb{S}^1$ to yield a semi-conjugacy between $z^k\vert_{\mathbb{S}^1}$ and $P\vert_{\mathcal{J}(P)}$. 
There exists a homeomorphism $\mathfrak{h}_\Gamma:\mathbb{S}^1\to\mathbb{S}^1$ that conjugates $z^k$ to $A_{\Gamma}$. We normalize $\mathfrak{h}_\Gamma$ so that it sends the fixed point $1$ of $z^k$ to the fixed point $1$ of $A_{\Gamma}$.

We define an equivalence relation $\sim$ on $\mathcal{K}(P)\bigsqcup \overline{\D}$ generated by 
\begin{equation}
\psi_{P}(\zeta)\sim \mathfrak{h}_\Gamma(\overline{\zeta}),\ \textrm{for all}\ {\zeta\in\mathbb{S}^1.}
\label{conf_mat_equiv_rel}
\end{equation}

\begin{defn}\label{conf_mat_def}
The maps $P$ and $A_\Gamma$ are said to be \emph{conformally mateable} if there exist a continuous map $F: \overline{\Omega}\to\widehat{\C}$ (called a \emph{conformal mating} of $A_\Gamma$ and $P$) that is complex-analytic in the interior $\Omega$ and continuous maps 
	$$
	\mathfrak{X}_P:\mathcal{K}(P)\to\widehat{\C}\ \textrm{and}\ \mathfrak{X}_\Gamma: \overline{\D}\to\widehat{\C},
	$$
	conformal on $\Int{\mathcal{K}(P)}$ and $\D$ (respectively), satisfying
	\begin{enumerate}[leftmargin=*]
		\item\label{topo_cond} $\mathfrak{X}_P\left(\mathcal{K}(P)\right)\cup \mathfrak{X}_\Gamma\left(\overline{\D}\right) = \widehat{\C}$,
		
		\item\label{dom_cond} $\Omega= \mathfrak{X}_P(\mathcal{K}(P))\cup\mathfrak{X}_\Gamma(\mathcal{D}_\Gamma)$,
		
		\item $\mathfrak{X}_P\circ P(z) = F\circ \mathfrak{X}_P(z),\quad \mathrm{for}\ z\in\mathcal{K}(P)$,
		
		\item $\mathfrak{X}_\Gamma\circ A_\Gamma(w) = F\circ \mathfrak{X}_\Gamma(w),\quad \mathrm{for}\ w\in
		\mathcal{D}_\Gamma$,\quad and

		\item\label{identifications} $\mathfrak{X}_P(z)=\mathfrak{X}_\Gamma(w)$ if and only if $z\sim w$ where $\sim$ is the equivalence relation on $\mathcal{K}(P)\sqcup \overline{\D}$ defined by Relation~\eqref{conf_mat_equiv_rel}.
	\end{enumerate}
\end{defn}

For a conformal mating $F$, we set
$$
\mathcal{K}(F):=\mathfrak{X}_P(\mathcal{K}(P)),\quad \mathcal{T}(F):=\mathfrak{X}_\Gamma(\D),\quad \mathrm{and}\quad \Lambda(F):=\mathfrak{X}_P(\mathbb{S}^1)=\mathfrak{X}_\Gamma(\mathbb{S}^1).
$$
We call the above sets the \emph{non-escaping set}, the \emph{tiling set}, and the \emph{limit set} of $F$, respectively.

\begin{theorem}\cite[Theorem~15.8, \S 15.1]{LLM24}\cite[Theorem~B, Proposition~4.1]{MM2}\label{conf_mating_bs_poly_thm}
Let $\Gamma$ be a (marked) Fuchsian group uniformizing a sphere with $d+1$ punctures.
Let $P$ be a degree $k=2d-1$ polynomial with connected Julia set which is either
\begin{itemize}
\item geometrically finite; or 
\item periodically repelling, finitely renormalizable.
\end{itemize}
Then there exists a conformal mating $F:\overline{\Omega}\to\widehat{\C}$, unique up to M{\"o}bius conjugation, between $P$ and $A_\Gamma$. Further, $\overline{\Omega}$ is homeomorphic to the quotient of $\overline{\D}$ by a finite geodesic lamination.
\end{theorem}

\begin{remark}\label{fixed_ray_lami_rem}
The topology of the domain of definition $\overline{\Omega}$ of the conformal mating $F$ between a polynomial $P$ and a Fuchsian group $\Gamma$ only depends on the combinatorics of $P$, and not on $\Gamma$. Specifically, consider the set of angles $\mathcal{A}:=\{\pm\frac{j}{2d}:j\in\{0,\cdots,d\}\}\subset\R/\Z$ and the equivalence relation on $\mathcal{A}$ that identifies two angles if the associated external dynamical rays of $P$ land at the same point of $\mathcal{J}(P)$. {We can turn this  equivalence relation into a finite geodesic lamination on $\overline{\D}$ by connecting the points of an equivalence class by hyperbolic geodesics. We call this lamination a \emph{coarse lamination}.} The quotient of $\overline{\D}$ by this coarse lamination produces a topological model for $\overline{\Omega}$ (cf. \cite[\S 15.1, \S 15.5]{LLM24}).
\end{remark}

We reproduce \cite[Figures~8, 9]{MM2} here (see Figures~\ref{mating_1_fig} and~\ref{mating_2_fig}) to illustrate the topology of the domains of definition of matings between polynomials and Bowen-Series maps. Here,  $\overline{\Omega}$ is homeomorphic to the quotient of $\overline{\D}$ by a finite geodesic lamination as stated in Theorem~\ref{conf_mating_bs_poly_thm} above.
In the top figures, we depict the action of the Bowen-Series map $A_\Gamma$ (left) and the filled Julia set of the polynomial $P$ along with the corresponding dynamical rays of $P$ (right). In the bottom left figure, the domain of $A_\Gamma$ is shown outside $\D$ {(under conjugation of the upper figure by $z\mapsto 1/\overline{z}$)}. The boundary of this domain is marked in red. Inside the disk, we draw a blue hyperbolic geodesic connecting the ideal boundary points of the red polygon if the corresponding dynamical rays of $P$ land at the same point. In the bottom right figure, we show the domain of the conformal mating (boundary in red). Topologically, this domain is obtained by pinching the blue geodesics to points.

\begin{figure}[h!]
\captionsetup{width=0.96\linewidth}
\begin{tikzpicture}
\node[anchor=south west,inner sep=0] at (0,3.2) {\includegraphics[width=0.75\textwidth]{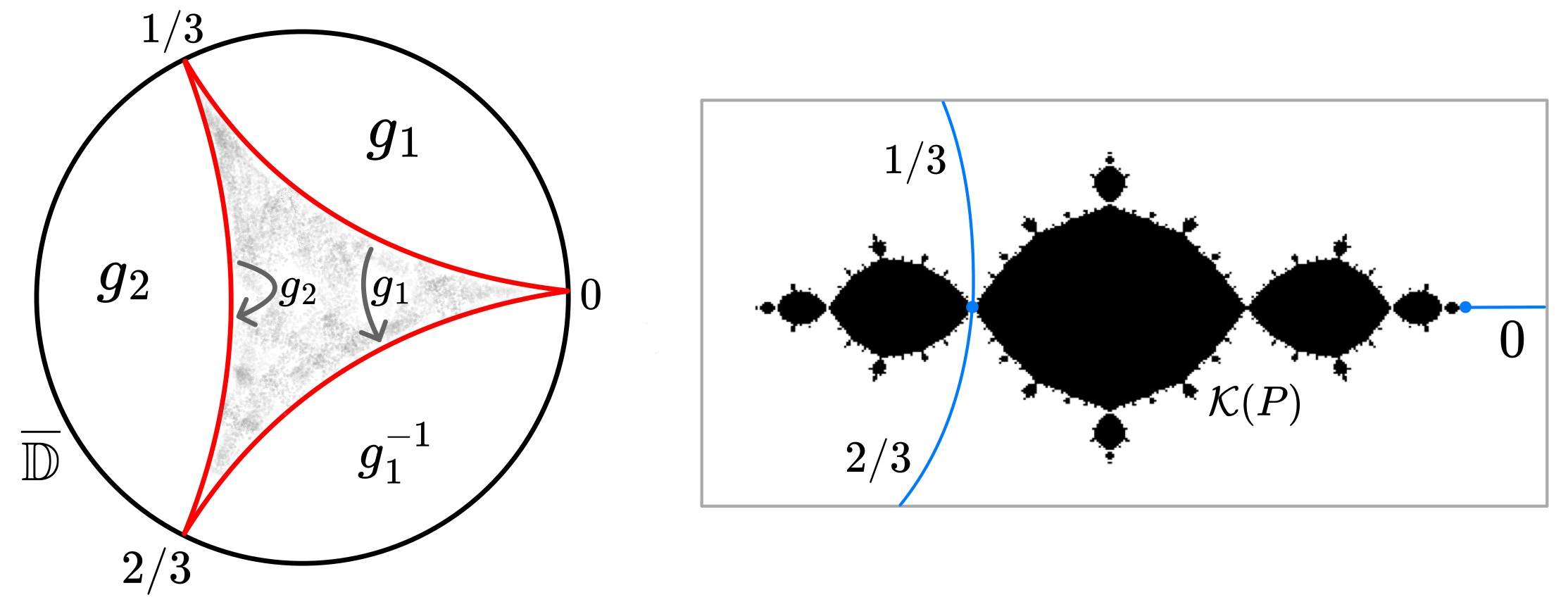}}; 
\node[anchor=south west,inner sep=0] at (0.6,0) {\includegraphics[width=0.66\textwidth]{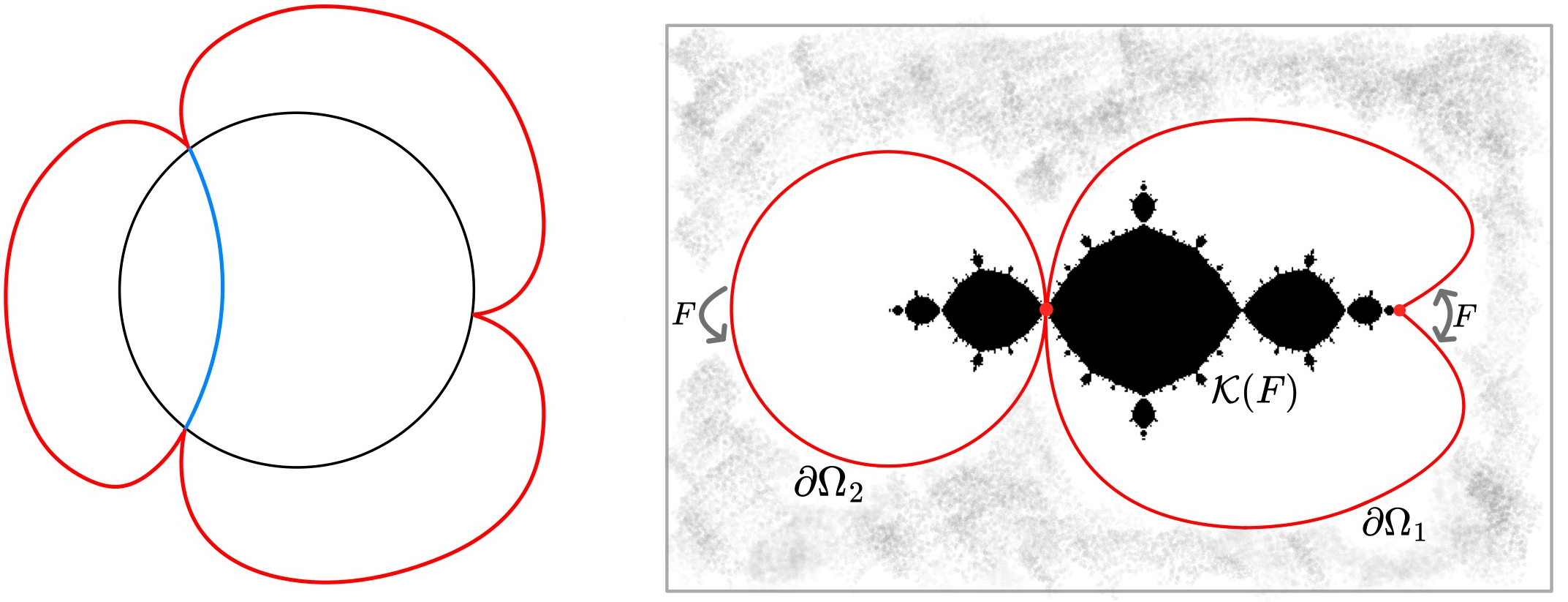}}; 
\end{tikzpicture}
\caption{The conformal mating of the Basilica polynomial $z^2-1$ with a Bowen-Series map of the sphere with $2$ punctures and an order $2$ orbifold point.}
\label{mating_1_fig}
\end{figure}

\begin{figure}[h!]
\captionsetup{width=0.96\linewidth}
\begin{tikzpicture}
\node[anchor=south west,inner sep=0] at (0,5.4) {\includegraphics[width=0.75\textwidth]{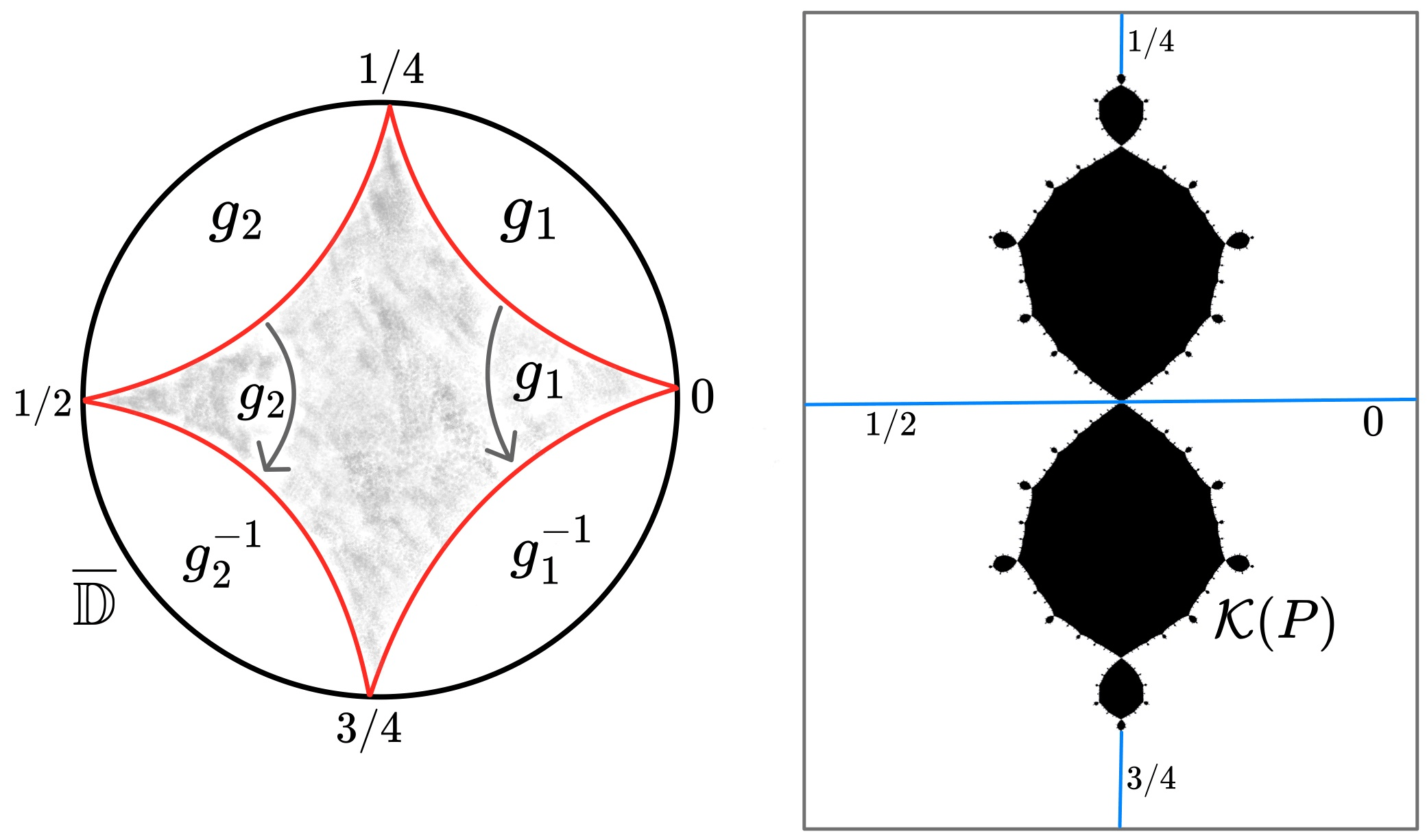}}; 
\node[anchor=south west,inner sep=0] at (0,0) {\includegraphics[width=0.75\textwidth]{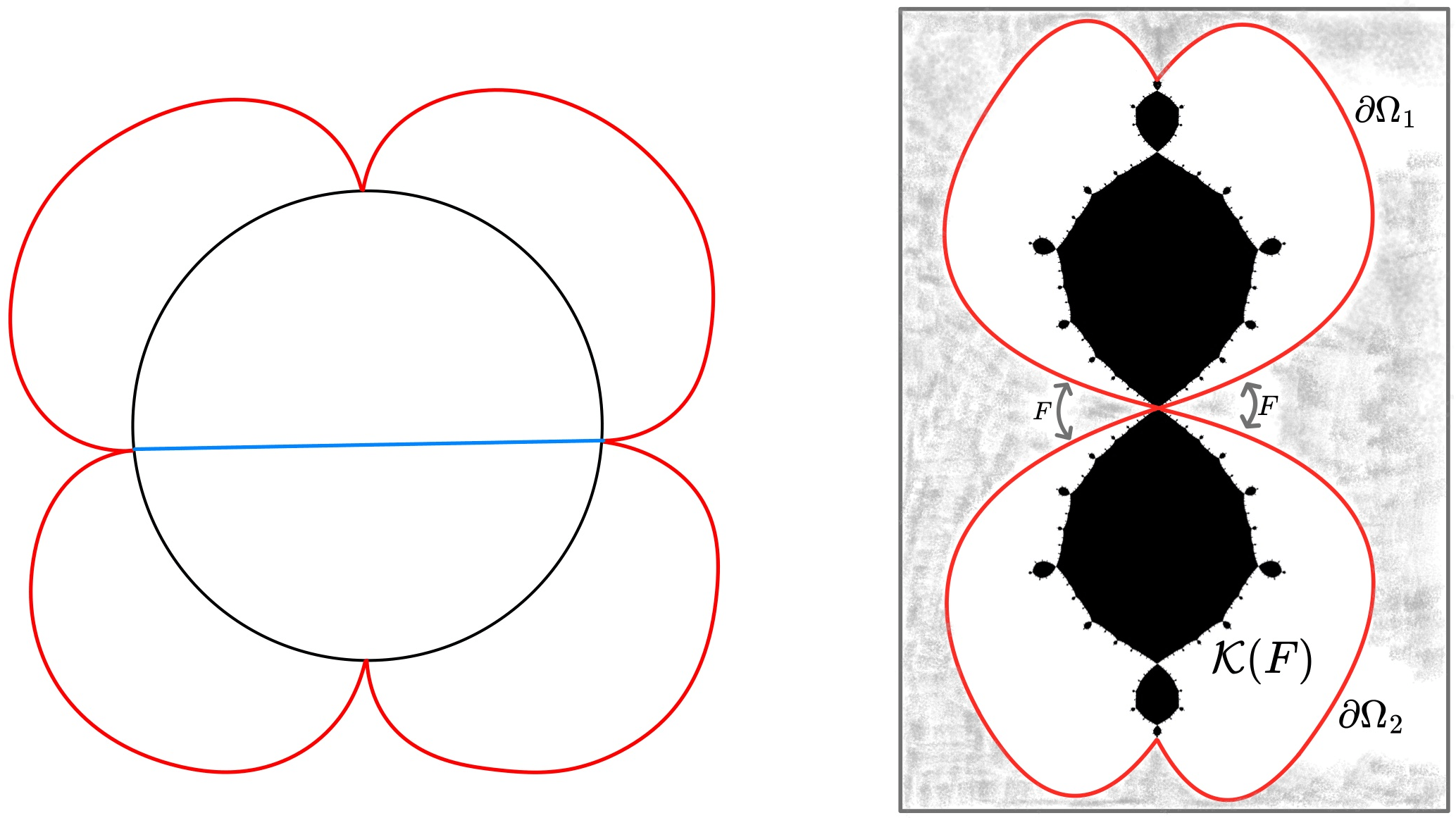}}; 
\end{tikzpicture}
\caption{The conformal mating of a cubic critically fixed polynomial with a Bowen-Series map of the thrice punctured sphere.}
\label{mating_2_fig}
\end{figure}

\subsection*{Step III: Algebraic description of conformal matings}
The promotion of the conformal matings described above to algebraic correspondences can be regarded as a process of globalization. To this end, one needs an algebraic description of the conformal matings. Heuristically speaking, the algebraicity of the matings is a consequence of algebraicity of rational maps and Kleinian groups.

Let us first remark that the domain of definition of a conformal mating $F$ between $P$ and $A_\Gamma$ may have more than one (but finitely many) interior components.
The following class of algebraic functions were introduced in \cite{LLM24}.

\begin{defn}\label{b_inv_def}
Let $\RV$ be a finite index set and $\{\Omega_a: a\in\RV\}$ be a disjoint collection of proper simply connected sub-domains of $\widehat{\C}$ such that $\Int{\overline{\Omega_a}}=\Omega_a$, $a\in\RV$, and let $\Omega:=\displaystyle\bigsqcup_{a\in\RV}\Omega_a$. Further, let $\mathfrak{S}\subset\partial\Omega$ be a (possibly empty) finite set such that $\partial^0\Omega:=\partial\Omega-\mathfrak{S}$ is a finite union of disjoint non-singular real-analytic curves.

\noindent The set $\Omega$ is called an \emph{inversive multi-domain} if it admits a continuous map $F:\overline{\Omega}\to\widehat{\C}$ satisfying the following properties.
\begin{enumerate}[leftmargin=*]
\item\label{mero_cond} $F$ is meromorphic on $\Omega$.
\item\label{permute_cond} For each $a\in\RV$, there exists $b\in\RV$ such that $F(\partial\Omega_a)=\partial\Omega_{b}$.
\item\label{inv_cond} $F:\partial\Omega\to\partial\Omega$ is an orientation-reversing involution preserving $\mathfrak{S}$; i.e., $F(\mathfrak{S})~=~\mathfrak{S}$.
\end{enumerate}
The map $F$ is called a \emph{B-involution} of the inversive multi-domain $\Omega$.

\noindent When $|\RV|=1$, the domain $\Omega$ is called an \emph{inversive domain}.
\end{defn}

We set $\eta(z):=1/z$.

\begin{theorem}\cite[Theorem~14.5, Proposition 15.4]{LLM24}\cite[Proposition~4.9]{MM2}\label{b_inv_thm}
\noindent\begin{enumerate}[leftmargin=*]
\item Let $F:\overline{\Omega}\to\widehat{\C}$ be a conformal mating of $P$ and $A_\Gamma$, where $\Gamma$ is a (marked) Fuchsian group uniformizing a sphere with $d+1$ punctures, and $P$ is a degree $2d-1$ polynomial with connected Julia set satisfying one of the conditions of Theorem~ \ref{conf_mating_bs_poly_thm}.

Then, {$\Omega$ is an inversive multi-domain, with set of connected components $\{\Omega_a:\ a\in\RV\}$ (where $\RV$ is a finite index set)}, and $F$ is a B-involution. 

\item For a B-involution $F:\overline{\Omega}=\displaystyle\overline{\bigsqcup_{a\in\RV}\Omega_a}\to\widehat{\C}$, there exist Jordan domains $\mathfrak{D}_a$ and rational maps $R_a$, $a\in\RV$, and an order two permutation $\tau$ of the set $\RV$,
such that the following hold.
\begin{enumerate}[leftmargin=*]
\item $\eta:\mathfrak{D}_a\to\widehat{\C}-\overline{\mathfrak{D}_{\tau(a)}}$ is a homeomorphism.
\item $\partial\mathfrak{D}_a$ is a piecewise non-singular real-analytic curve.
\item  $R_a:\mathfrak{D}_a\to\Omega_a$ is a conformal isomorphism.
\item $F\vert_{\Omega_a}\equiv R_{\tau(a)}\circ\eta\circ(R_a\vert_{\mathfrak{D}_a})^{-1}$.
\item $\displaystyle\sum_{a\in\RV} \deg(R_a)=2d$.
\end{enumerate}
\end{enumerate}
In particular, if $P$ lies in the principal hyperbolic component of {monic, centered,} degree $2d-1$ polynomials, then $F$ is a B-involution in an inversive domain and $1\in\partial\mathfrak{D}$.
\end{theorem}

\subsection*{Step IV: Lifting conformal matings to algebraic correspondences}\label{corr_from_conf_mating_subsec}
Finally, one uses the algebraic description of conformal matings between polynomials and Fuchsian groups to define algebraic correspondences on trees of Riemann spheres. Thanks to the hybrid structure of the conformal matings, the correspondences
thus constructed turn out to be matings of polynomials and Fuchsian groups.

Let the polynomial $P$, the Fuchsian punctured sphere group $\Gamma$, the conformal mating $F:\overline{\Omega}=\overline{\bigsqcup_{a\in \RV}\Omega_a}\to\widehat{\C}$, the involution $\tau:\RV\to\RV$, the rational maps $R_a$, and the Jordan domains $\mathfrak{D}_a$, $a\in\RV$, be as in Theorem~\ref{b_inv_thm}.  In particular, $P$ has degree $(2d-1)$.
We set $\delta(a):=\deg(R_a)$, so $\sum_{a\in \RV} \delta(a)=2d$. The degrees of the rational maps $R_a$ can be read from the coarse lamination described in Remark~\ref{fixed_ray_lami_rem} (see \cite[\S 15.2]{LLM24}).

Let $\RT$ be the graph with vertex set $\RV$ such that $[a, b]$ is an edge in $\RT$ if and only if $\partial \Omega_a \cap \partial \Omega_b \neq \emptyset$.
Since $\overline{\Omega}$ is homeomorphic to the quotient of $\overline{\D}$ by a finite geodesic lamination by Theorem \ref{conf_mating_bs_poly_thm}, $\RT$ is a finite tree.
Let $(\RT, \widehat\C^\RV)$ be a marked tree of Riemann spheres with the markings defined as follows (cf. \cite[\S~15.5]{LLM24}). For two adjacent vertices $a,b\in\RV$, there exist unique points $x\in\partial\mathfrak{D}_a\subset\widehat{\C}_a$ and $y\in\partial\mathfrak{D}_b\subset\widehat{\C}_b$ such that $\{R_a(x)\}=\{R_b(y)\}=\partial\Omega_a\cap\partial\Omega_b$. We attach an edge between $\widehat{\C}_a$ and $\widehat{\C}_b$ (in $\widehat{\C}^\RV$) at the points $x, y$ (see Definition~\ref{defn:trs}). In the language of Definition~\ref{def-tangentdirn}, this is equivalent to saying that $\xi_a(v)=x$ and $\xi_b(w)=y$, where $v\in T_a\RT$ is the tangent vector at $a$ in the direction
of $b$, and $w\in T_b\RT$ is the tangent vector at $b$ in the direction
of $a$.
The involution $\tau$ on $\RV$ induces an involution $\pmb{\tau}$ on the tree $\RT$.
{We define an involution on the tree of Riemann spheres}
$$
(\pmb{\tau},\pmb{\eta}): (\RT, \widehat\C^\RV) \longrightarrow (\RT, \widehat\C^\RV)
$$
where
\begin{itemize}
    \item $\pmb{\tau}: (\RT, \RV)\longrightarrow (\RT, \RV)$ is the induced involution; and
    \item $\eta_a: \widehat{\C}_a \longrightarrow \widehat{\C}_{\tau(a)}$ is given by $\eta_a(z) = \frac{1}{z}$.
\end{itemize}
Let $\pmb{R}: (\RT, \widehat\C^\RV) \longrightarrow \widehat{\C}$ be the rational map from $(\RT, \widehat\C^\RV)$ to the trivial tree of Riemann spheres, defined as $R_a: \widehat{\C}_a \longrightarrow\widehat{\C}$. The markings on $(\RT,\widehat{\C}^\RV)$ specified above ensure that the map $\pmb{R}$ satisfies the conditions of Definition~\ref{defn:trs-map}. We refer to the rational map $\pmb{R}$ as the \emph{uniformizing rational map} for the B-involution $F$.
We define an associated correspondence on $(\RT, \widehat\C^\RV)$ by the equation
$$
\mathfrak{C} = \mathfrak{C}_{\pmb{R}}:= \{(x,y) \in \widehat\C^\RV \times \widehat\C^\RV: \frac{\pmb{R}(x) - \pmb{R}\circ \pmb{\eta}(y)}{x-\pmb{\eta}(y)} = 0\}.
$$
By construction, $\mathfrak{C}$ defines a correspondence of bi-degree $(2d-1, 2d-1)$ on the tree of spheres $(\RT, \widehat\C^\RV)$. We can partition $(\RT, \widehat\C^\RV)$ into
$$
\widetilde{\mathcal{K}}:=\pmb{R}^{-1}(\mathcal{K}(F))\quad \mathrm{and}\quad  \widetilde{\mathcal{T}}:= \pmb{R}^{-1}(\mathcal{T}(F)).
$$
These sets, called the \emph{non-escaping set} and the \emph{tiling set} of the correspondence, are completely invariant under $\mathfrak{C}$. 
Their common boundary is the \emph{limit set}
$\widetilde{\Lambda}:=\pmb{R}^{-1}(\Lambda(F))$.
\begin{theorem}\label{corr_mating_thm_1}
Let $F:\overline{\Omega}\to\widehat{\C}$ be a conformal mating of $P$ and $\Gamma$, where $\Gamma$ is a (marked) Fuchsian group uniformizing a sphere with $d+1$ punctures, and $P$ is a degree $2d-1$ polynomial with connected Julia set satisfying one of the conditions of Theorem~ \ref{conf_mating_bs_poly_thm}.
Then the correspondence $\mathfrak{C}$ on the tree of spheres $(\RT, \widehat\C^\RV)$ is a mating of $P$ and $\Gamma$; i.e.,
{\begin{enumerate}[leftmargin=*]
	\item On $\widetilde{\cT}$, the dynamics of $\mathfrak{C}$ is equivalent to the action of a group of conformal automorphisms acting properly discontinuously with $\cT/\mathfrak{C}$ biholomorphic to $\D/\Gamma$.
	
	\item $\widetilde{\cK}$ is the union of two copies $\widetilde{\cK}^+, \widetilde{\cK}^-$ of $\cK(P)$ (where $\cK(P)$ is the filled Julia set of $P$), such that 
    \begin{enumerate}
        \item $\widetilde{\cK}^+$ and $\widetilde{\cK}^-$ intersect in finitely many points, and 
        \item $\mathfrak{C}$ has a forward (respectively, backward) branch carrying $\widetilde{\cK}^+$ (respectively, $\widetilde{\cK}^-$) onto itself with degree $2d-1$, and this branch is topologically conjugate (conformally on the interior) to $P\vert_{\mathcal{K}(P)}$. 
    \end{enumerate} \end{enumerate} }
\end{theorem}
\noindent  We refer the reader to Figure~\ref{corr_fig} for an illustration and to \cite[Definition~1.8]{LLM24} for the precise definition of a correspondence realizing the mating of a polynomial and a Fuchsian group.
\begin{figure}[ht!]
\captionsetup{width=0.96\linewidth}
\begin{tikzpicture}
\node[anchor=south west,inner sep=0] at (0,0) {\includegraphics[width=0.98\textwidth]{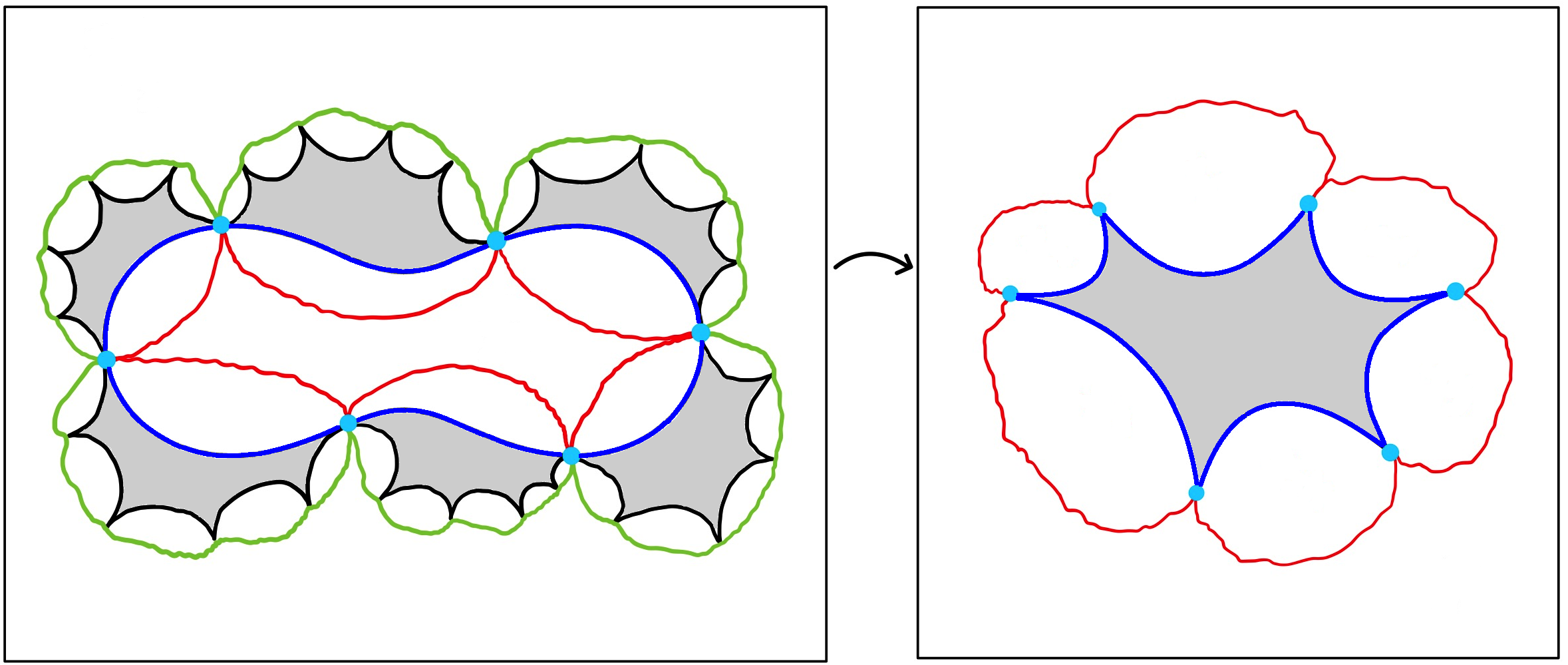}}; 
\node at (1.2,4.75) {\begin{large}$\widetilde{\mathcal{K}}^-$\end{large}};
\node at (4,2.64) {\begin{large}$\widetilde{\mathcal{K}}^+$\end{large}};
\node at (11.4,0.6) {\begin{large}$\mathcal{K}(F)$\end{large}};
\node at (9.8,2.64) {\begin{large}$\Omega^c$\end{large}};
\node at (10.8,4.44) {\begin{large}$\Lambda(F)$\end{large}};
\node at (5.84,2.675) {\begin{tiny}$1$\end{tiny}};
\node at (0.45,2.36) {\begin{tiny}$-1$\end{tiny}};
\node at (6.9,3.5) {\begin{large}$\pmb{R}$\end{large}};
\node at (1,0.4) {\begin{large}$\partial\mathfrak{D}$\end{large}};
\draw [thin,->] (1.05,0.6) -- (1.5,1.66);
\end{tikzpicture}
\caption{{The dynamical plane of the conformal mating $F$ {(right)} and the correspondence $\mathfrak{C}$ {(left)}, for some $P\in\cH_5$ and $\Gamma\in\mathrm{Teich}(S_{0,4})$, are displayed. The rational map $R$, that mediates between the $\mathfrak{C}$-plane and the $F$-plane, is injective on the interior of the blue curve (marked as $\partial\mathfrak{D}$).}}
\label{corr_fig}
\end{figure}
\begin{remark}\label{fixed_ray_lami_tree_rem}
As explained in Remark~\ref{fixed_ray_lami_rem} and \cite[Lemma~15.5]{LLM24},
the topology of $\overline{\Omega}$, and hence the topology of the tree of spheres $(\RT, \widehat\C^\RV)$, as well as the degree of the rational map $\pmb{R}$ on each sphere of $(\RT, \widehat\C^\RV)$ can be read off purely from the landing patterns of finitely many dynamical rays of $P$ (i.e., the coarse lamination).
\end{remark}

\begin{example}\label{degree_three_signature_example}
Consider the correspondence arising from the mating illustrated in Figure~\ref{mating_1_fig}. In this case, $(\RT, \widehat\C^\RV)$ consists of two spheres, the involution $(\pmb{\tau},\pmb{\eta})$ fixes each of these two spheres, and the rational map $\pmb{R}$ has degree $2$ on one sphere and degree $1$ on the other.   
\end{example}

\begin{example}\label{degree_four_signature_example}
Consider the correspondence arising from the mating illustrated in Figure~\ref{mating_2_fig}. Here also, $(\RT, \widehat\C^\RV)$ consists of two spheres, but the involution $(\pmb{\tau},\pmb{\eta})$ interchanges these two spheres. The rational map $\pmb{R}$ has degree $2$ on each of the two spheres. 
\end{example}

\subsection{Character variety}\label{char_var_subsec}
Motivated by the mating construction, in this subsection, we define the character variety which serves as the ambient space where all mating correspondences live. As explained above, the structure of the tree of spheres on which the mating correspondences are defined as well as the local/global degrees of the rational maps that define these correspondences depend on the combinatorics of the polynomials being mated with Fuchsian groups. We package this topological/combinatorial information in the so-called \emph{characteristic data} defined below. Finally, we assemble the parameter spaces of correspondences associated with various characteristic data to form a naturally stratified character variety that contains all the mating correspondences. This assembly  will primarily be used in \S~\ref{vertical_compact_subsec}.

Let $(\RT, \widehat\C^\RV)$ be a finite tree of Riemann spheres.
We {introduce a class of involutions} $(\pmb{\tau},\pmb{\eta}): (\RT, \widehat\C^\RV) \longrightarrow (\RT, \widehat\C^\RV)$ where
\begin{itemize}[leftmargin=*]
    \item $\pmb{\tau}: (\RT, \RV)\longrightarrow (\RT, \RV)$ is an involution; and
    \item $\eta_a: \widehat{\C}_a \longrightarrow \widehat{\C}_{\tau(a)}$ is given by $\eta_a(z) = 1/z$ (where $\tau:\RV\to\RV$ is the restriction of $\pmb{\tau}$ on the vertex set).
\end{itemize}
Let $\delta: \RV \longrightarrow \Z_{\geq 1}$ be a function, which we call the {\em degree function}.
Let $\overline{\delta}_a: T_a\RT \longrightarrow \{1,\cdots, \delta(a)\}, a\in \RV$, be a collection of functions, which we call the {\em local degree function}.
We define the {\em characteristic data} as the collection 
$$
\kappa = \left((\pmb{\tau}, \pmb{\eta}): (\RT, \widehat\C^\RV) \longrightarrow (\RT, \widehat\C^\RV), \delta, (\overline{\delta}_a)_{a\in \RV}\right).
$$
We remark that the characteristic data $\kappa$ encodes the following information.
\begin{enumerate}[leftmargin=*]
    \item the structure of the tree of spheres $(\RT, \widehat\C^\RV)$ (on which the relevant correspondences will be defined),
    \item the involution $(\pmb{\tau},\pmb{\kappa})$ on this tree of spheres (which goes into the definition of the correspondences),
    \item the positive integers $\delta(a)$, $a\in\RV$ (which will determine the degrees of the relevant rational maps on each sphere), and
    \item the positive integers $\overline{\delta}_a(v)$, $a\in\RV, v\in T_a\RT$ (which record the local degrees of the rational maps at the marked singular points on each sphere).
\end{enumerate}
We define the degree of the characteristic data $\kappa$ by
$$
\deg(\kappa) := \sum_{a \in \RV} \delta(a).
$$

\begin{defn}\label{rat_kappa_def}
Let us fix a characteristic data $\kappa$. We define $\Rat_\kappa(\C)$ to be the space of rational maps $\pmb{R}: (\RT, \widehat\C^\RV) \longrightarrow \widehat{\C}$ with $\deg(R_a) =\delta(a)$ and $\deg_{\xi_a(v)}(R_a) = \overline{\delta}_a(v)$, for $a\in\RV, v\in T_a\RT$, where $\xi_a: T_a \RT \longrightarrow \widehat{\C}_a$ is the marking of the singular points.
\end{defn}
We remark that here the codomain is the trivial tree of Riemann sphere. The tree map sends $\RT$ to the unique vertex. 

Let $\pmb{R} \in \Rat_\kappa(\C)$. We define the associated correspondence on $(\RT, \widehat\C^\RV)$ by the equation
\begin{equation}
\mathfrak{C} = \mathfrak{C}_{\pmb{R}}:= \{(x,y) \in \widehat\C^\RV \times \widehat\C^\RV: \frac{\pmb{R}(x) - \pmb{R}\circ\pmb{\eta}(y)}{x-\pmb{\eta}(y)} = 0\}.
\label{corr_eqn}
\end{equation}
We also call $\pmb{R}$ the {\em uniformizing rational map} for $\mathfrak{C}$.
Two such correspondences $\mathfrak{C}, \widetilde{\mathfrak{C}}$ are said to be \emph{conjugate} if there exists an automorphism $M$ on $(\RT, \widehat\C^\RV)$ 
{(necessarily acting on the underlying tree $\RT$ by a simplicial automorphism)} so that $\widetilde{\mathfrak{C}} = M^{-1} \circ \mathfrak{C} \circ M$; i.e.,
$$
\widetilde{\mathfrak{C}} = \{(x,y) \in \widehat\C^\RV \times \widehat\C^\RV: (M(x), M(y)) \in \mathfrak{C}\}.
$$

Denote by $\Aut_{(\pmb{\tau},\pmb{\eta})}(\RT, \widehat\C^\RV)$ the group of automorphisms of $(\RT, \widehat\C^\RV)$ that commutes with $(\pmb{\tau},\pmb{\eta})$.
We introduce an equivalence relation on $\Rat_\kappa(\C)$ by
$$
\pmb{R} \sim \widetilde{\pmb{R}} \quad \mathrm{if}\quad \widetilde{\pmb{R}}=M_2\circ \pmb{R}\circ M_1, 
$$
where $M_1\in \Aut_{(\pmb{\tau},\pmb{\eta})}(\RT, \widehat\C^\RV)$ and $M_2\in \Aut(\widehat{\C})$.
It follows from the definition that $\pmb{R} \sim \widetilde{\pmb{R}}$ if and only if $\mathfrak{C}_{\pmb{R}}$ and $\mathfrak{C}_{\widetilde{\pmb{R}}}$ are conjugate.
We refer to this moduli space of correspondences, or equivalently, its parameter space $\Rat_\kappa(\C)/\!\!\sim$, as the {\em character variety} with characteristic data $\kappa$.

{
\subsubsection{Simple, regular and exceptional locus}
We now define the regular and exceptional loci of the above character variety. These subsets will play an important role in establishing the Hausdorffness of the appropriate quotient topology. Characteristic data arising from mating correspondences satisfy the regularity and simplicity condition, stated below (see Proposition~\ref{total_space_regular_lem}).
\begin{defn}\label{simple_sign_def}
    Let $\kappa = \left((\pmb{\tau}, \pmb{\eta}): (\RT, \widehat\C^\RV) \longrightarrow (\RT, \widehat\C^\RV), \delta, (\overline{\delta}_a)_{a\in \RV}\right)$ be some given characteristic data.
    Then $\kappa$ is called \emph{regular} if 
    \begin{enumerate}[leftmargin=*]
        \item\label{reg:item:1} for any two adjacent vertices $a, b\in \RV$, then $\max \{\delta_a, \delta_b\} \geq 2$; and
        \item\label{reg:item:2} for any $a \in \RV$ with $\delta_a > 1$ and any $v \in T_a \RT$, we have $\overline{\delta}_a(v) < \delta_a$.
    \end{enumerate}
    Also, $\kappa$ is called {\em exceptional} if it is not regular.
    It is called {\em simple} if for any $a \in \RV$, and $v \in T_a \RT$, we have $\overline{\delta}_a(v) = 1$.
    
    Let $\pmb{R} \in \Rat_\kappa(\C)$.
    We say that $\pmb{R}$ is {\em exceptional} if 
    \begin{itemize}[leftmargin=*]
        \item either $\kappa$ is exceptional; or
        \item $\RT$ is trivial, and the local degree of $\pmb{R}$ at a fixed point of $\pmb{\eta}(z)=\frac1z$ 
        is equal to the degree of $\pmb{R}$.
    \end{itemize}
    We say that $\pmb{R}$ is {\em regular} if it is not exceptional. We denote the set of regular maps by $\Rat^{reg}_\kappa(\C) \subseteq \Rat_\kappa(\C)$.
\end{defn}
\begin{rmk}\label{rem:regular}
Since $\Rat_\kappa(\C)/\!\!\sim$ is a singleton when $\deg(\kappa)=1$, we will always assume that $\deg(\kappa)\geq 2$. By definition, if $\RT$ is trivial, then $\kappa$ is regular. Further, if $\kappa$ is simple and satisfies condition \eqref{reg:item:1}, then $\kappa$ is regular. 
Note that by definition
$$
\Rat^{reg}_\kappa(\C)  = \begin{cases}
    \Rat_\kappa(\C) & \text{ if $\kappa$ is regular and $\RT$ is non-trivial,}\\
    \text{a proper subset of $\Rat_\kappa(\C)$} & \text{ if $\RT$ is trivial,}\\
    \emptyset & \text{ if $\kappa$ is exceptional.}
\end{cases}
$$
\end{rmk}
}

\begin{lem}\label{lem:HausdorffRat}
    Let $\pmb{R}\in \Rat_\kappa(\C)$ be regular. Then the orbit of $\pmb{R}$ under the group $\Aut(\widehat{\C}) \times \Aut_{(\pmb{\tau},\pmb{\eta})}(\RT, \widehat\C^\RV)$ is closed in $\Rat_\kappa(\C)$.
\end{lem}
\begin{proof}
    Suppose for contradiction that the orbit is not closed. Then there exists $(L_n, \pmb M_n) \in \Aut(\widehat{\C}) \times \Aut_{(\pmb{\tau},\pmb{\eta})}(\RT, \widehat\C^\RV)$ so that 
    \begin{itemize}
        \item $L_n \circ \pmb R \circ \pmb M_n \to \pmb S \in \Rat_\kappa(\C)$; and
        \item $\pmb S$ is not in the orbit of $\pmb R$.
    \end{itemize}

    \noindent \textbf{Case (I):} Suppose that $\RT$ is trivial. In this case, $\pmb R = R \in \Rat_d(\C)$ is a rational map on $\widehat{\C}$ and $\pmb M_n = M_n$ is a M\"obius map.
    After a change of coordinates, we assume that $\eta(z) = -z$.
    Since $M_n$ commutes with $\eta$, after passing to a subsequence if necessary, we conclude that $M_n(z) = \lambda_n z$ or $M_n(z) = \lambda_n/z$ for all $n$. 
    Let us assume that we are in the first case, the second case can be handled by applying the same argument to the regular rational map $R(\frac1z)$.
    Suppose $\lambda_n$ is contained in a compact set of $\C -\{0\}$. Then $M_n$ is bounded. Since $S \in \Rat_d(\C)$, $L_n$ is also bounded by Lemma~\ref{lem:rl}. Therefore, after passing to a subsequence, we may assume $M_n \to M$ and $L_n \to L$. Then $S = L \circ R \circ M$, which is a contradiction.
    Therefore, after passing to a subsequence, $\lambda_n \to 0$ or $\lambda_n \to \infty$.
    Let us assume that $\lambda_n \to 0$. (The case $\lambda_n \to \infty$ can be treated similarly.) Since $\lambda_n \to 0$, the rescaling $M_n$ amounts to zooming in at $0$. 
    Since $R$ is regular, by definition, the local degree $e$ of $R$ at $0$ is strictly less than $\deg(R) = d$. It is not hard to see, by an explicit computation using the Taylor expansion of $R$ at $0$, that the rescaling limit $S$ has degree equal to the local degree of $R$ at $0$. In fact, $S = A(z^e)$ for some M\"obius map $A$.
    This is a contradiction to the assumption $S \in \Rat_d(\C)$. Therefore, the orbit is closed.
\smallskip

\noindent \textbf{Case (II):} Suppose that $\RT$ is non-trivial.
    After passing to a subsequence, we assume that the maps induced by $\pmb M_n$ on $\RT$ are the same for all $n$.
    Since $\kappa$ is regular, let $a, b \in \RV$ be so that $\delta_b > 1$ and $M_{a,n}: \widehat{\C}_a \longrightarrow \widehat{\C}_b$.
\smallskip

\noindent \textbf{Case (IIa):} Suppose that the singular set $\Xi_a$ contains three or more points. Since $L_n \circ \pmb R \circ \pmb M_n \to \pmb S \in \Rat_\kappa(\C)$, we conclude that the limit of $M_{a,n}(\Xi_a)$ contains the same number of points. Thus, $M_{a,n}$ is bounded. This implies that $L_n$ is bounded too. Therefore, $M_{v,n}$ is bounded for all $v \in \RV$ by Lemma~\ref{lem:rl}. After passing to a subsequence, we assume that $L = \lim L_n$ and $\pmb M = \lim \pmb M_n$. Then $\pmb S = L \circ \pmb R \circ \pmb M$.
    Hence $\pmb S$ is in the orbit of $\pmb R$, which is a contradiction.  
\smallskip

\noindent \textbf{Case (IIb):} Suppose that the singular set $\Xi_a$ consists of two points. After a change of coordinate, we may assume that $\Xi_a = \{0, \infty\} \subseteq \widehat{\C}_a$ and $\Xi_b = \{0, \infty\} \subseteq \widehat{\C}_b$. Then $M_{a,n}(\{0, \infty\}) = \{0, \infty\}$. Thus, after passing to a subsequence, $M_{a,n} = \lambda_n z$ or $M_{a,n} = \lambda_n/z$ for all $n$. Since $\kappa$ is regular, a similar argument as in Case(I) gives the contradiction.
\smallskip

\noindent \textbf{Case (IIc):} Suppose that the singular set $\Xi_a$ consists of a single point. After a change of coordinates, we assume $\Xi_a = \{\infty\} \subseteq \widehat{\C}_a$ and $\Xi_b = \{\infty\}\subseteq \widehat{\C}_b$. Therefore, $M_{a,n}(z) = \lambda_n z + c_n$. 
    By the same argument as before, we can assume that $M_{a,n}$ is unbounded.
    Then a simple computation gives that $L_n \circ R_b \circ M_{a,n}(z)$ converges to $A(z^e)$ for some M\"obius map $A$ and $e \leq \delta_b$. Similarly, $L_n \circ R_a \circ M_{b,n}(z)$ converges to $B(z^f)$ for some M\"obius map $B$ and $f \leq \delta_a$.
    Since $S_a(z) = A(z^e)$, we have that $\delta_a = e \leq \delta_b$.
    Similarly, $\delta_b = f \leq \delta_a$.
    Therefore, $e = f = \delta_a = \delta_b$,
    and the local degree of $S_b$ at the singular point $\infty$ is $f = \delta_b > 1$. This is a contradiction {to the fact that $\pmb{S}\in\Rat_\kappa(\C)$ and that $\kappa$ is regular}.
\end{proof}

\begin{cor}\label{cor:HausdorffReg}
    The regular character variety $\Rat^{reg}_\kappa(\C)/\!\!\sim$ is Hausdorff.
\end{cor}

\subsubsection{A partial order on characteristic data}\label{singnature_order_subsec}
Let $\kappa_1$ and $\kappa_2$ be two characteristic data of the same degree. We say $\kappa_1$ is dominated by $\kappa_2$ under $\pi$, denoted by $\kappa_1 \prec_\pi \kappa_2$, if $\pi: (\RT_2, \RV_2) \longrightarrow (\RT_1, \RV_1)$ is a surjective map so that 
\begin{itemize}
    \item $\pmb{\tau}_1 \circ \pi = \pi \circ \pmb{\tau}_2$;
    \item for any vertex $a \in \RV_1$, $\pi^{-1}(a)$ is a subtree of $\RT_2$, and 
    $$
    \delta_1(a) = \sum_{b\in \RV_2\cap\pi^{-1}(a)} \delta_2(b).
    $$
\end{itemize}
We denote the fibers by
$$
\RT_a = \pi^{-1}(a) \quad \text{ and }\quad \RV_a = \RV_2 \cap \RT_a.
$$
We remark that the projection map between $\kappa_2$ and $\kappa_1$ may not be unique.

\subsubsection{A stratification of character varieties}\label{sec-stratfn}
The character varieties for various characteristic data of the same degree are naturally stratified.
Suppose that $\kappa_1\prec_\pi~\kappa_2$. Then the character variety $\Rat_{\kappa_2}(\C)/\!\!\sim$ can be attached to the character variety $\Rat_{\kappa_1}(\C)/\!\!\sim$ at infinity.
We define the topology in the following way (c.f. Definition \ref{defn:cvrationalmap}). 

\begin{defn}\label{char_var_top_def}
Let $\kappa_1, \kappa_2$ be two characteristic data of the same degree. Suppose $\kappa_1 \prec_\pi \kappa_2$.
Let $\pmb{R}_n \in \Rat_{\kappa_1}(\C)$ and $\pmb{R} \in \Rat_{\kappa_2}(\C)$. We say $\pmb{R}_n$ converges to $\pmb{R}$ if for any vertex $a \in \RV_1$, $R_{a,n}: \widehat{\C}\longrightarrow \widehat{\C}$ converges to $\pmb{R}:(\RT_a, \widehat{\C}^{\RV_a}) \longrightarrow \widehat{\C}$. 
More precisely, this means that there exist sequences $(M_{b,n})_n, b\in\RV_2$, such that for each $n\geq 1$, the tuple $\{M_{b,n}: b\in\RV_2\}\in \Aut_{(\pmb{\tau_2},\pmb{\eta_2})}(\RT_2, \widehat\C^{\RV_2})$, and for any $a\in \RV_1$, 
\begin{itemize}
    \item the M{\"o}bius maps $M_{b,n}$, where $b \in \RV_a \subseteq \RV_2$, are rescalings for $(\RT_a, \widehat{\C}^{\RV_a})$ (see Definition~\ref{defn:cvrationalmap}); and
    \item $R_{a,n} \circ M_{b,n} \to R_b$ compactly away from the singular set .
\end{itemize}

We say that $[\pmb{R}_n] \in \Rat_{\kappa_1}(\C)/\!\!\sim$ converges to $[\pmb{R}] \in \Rat_{\kappa_2}(\C)/\!\!\sim$ if $\pmb{R}_n \to \pmb{R}$ for some representatives.

Similarly, let $\mathfrak{C}_n$ and $\mathfrak{C}$ be the associated correspondences for $\pmb{R}_n$ and $\pmb{R}$. We say that $\mathfrak{C}_n$ converges to $\mathfrak{C}$ if for any pair of vertices $a, a' \in \RV_1$, 
$$
\mathfrak{C}_{a, a',n} = \mathfrak{C}_{n} \cap \left( \widehat{\C}_a \times \widehat{\C}_{a'}\right)
$$
converges to the restriction of the correspondence $\mathfrak{C}$ on $(\RT_a, \widehat{\C}^{\RV_a}) \times (\RT_{a'}, \widehat{\C}^{\RV_{a'}})$.
Finally, we say that $[\mathfrak{C}_n]$ converges to $[\mathfrak{C}]$ if $\mathfrak{C}_n \to \mathfrak{C}$ for some representatives.
\end{defn}

\begin{defn}\label{char_var_def}
We define the \emph{character variety of degree $d$} as the union of all character varieties $\Rat_{\kappa}(\C)/\!\!\sim$, as $\kappa$ runs over all characteristic data of degree $d$. We denote this space (equipped with the topology defined above) by~$\mathcal{CV}_{d}$.

We define the \emph{regular character variety of degree $d$} as the union of all regular character varieties $\Rat^{reg}_{\kappa}(\C)/\!\!\sim$, as $\kappa$ runs over all characteristic data of degree $d$. We denote this space by~$\mathcal{CV}^{reg}_{d} \subseteq \mathcal{CV}_{d}$.

We define the \emph{simple character variety of degree $d$} as the union of all character varieties $\Rat_{\kappa}(\C)/\!\!\sim$, as $\kappa$ runs over all simple characteristic data of degree $d$. We denote this space  by~$\mathcal{CV}^{simp}_{d}\subseteq \mathcal{CV}_{d}$.
\end{defn}

\begin{prop}\label{prop-charvarhausdorff}
    The regular character variety $\mathcal{CV}^{reg}_{d}$ is Hausdorff.
\end{prop}
\begin{proof}
    Let $\pmb{R} \in \Rat^{reg}_{\kappa_1}(\C)$ and $\pmb{S} \in \Rat^{reg}_{\kappa_2}(\C)$.
    We assume that $[\pmb{R}]\neq [\pmb{S}]$.
    We will show that $[\pmb{R}]$ and $[\pmb{S}]$ can be separated by open sets.
    For simplicity of exposition, we assume that $\kappa_1$ is the trivial tree of spheres. The general case can be proved similarly by applying the same argument to each vertex of the tree for $\kappa_1$.

    We will argue by contradiction and borrow heavily from \cite{Luo21}. Suppose that there exists $R_n \in \Rat_{d}(\C)$, $M_{b,n}, L_n \in \Aut(\C)$ with $b \in \RV_2$ so that
    $$
    R_n \to R \text{ and } L_n \circ R_n \circ M_{b,n} \to S_b.
    $$
     Since $[\pmb{R}]\neq [\pmb{S}]$ and $\Rat^{reg}_{\kappa_1}(\C)/\!\!\sim$ is Hausdorff by Corollary~\ref{cor:HausdorffReg}, we conclude that $\kappa_2$ is non-trivial.

    Suppose $L_n$ is bounded. Possibly after passing to a subsequence, we assume that $L_n\to L_\infty$ in $\PSL_2(\C)$. As $R_n \to R $, it follows that $L_n\circ R_n\circ \widehat{M}_{b,n}$ converges to $L_\infty\circ R$ in $\PSL_2(\C)$, where $\widehat{M}_{b,n}=\mathrm{Id}$ for all $n$. Moreover, by assumption, $L_n \circ R_n \circ M_{b,n} \to S_b$ in $\PSL_2(\C)$. By the uniqueness part of the second part of Lemma~\ref{lem:rl}, the sequence $M_{b,n}$ is bounded for each $b\in \RV_2$. This contradicts the fact that $M_{b,n}$ is a rescaling for $\RT_2$ and that $\RT_2$ is non-trivial.
    Therefore, $L_n \to \infty$ in~$\PSL_2(\C)$.

    Note that $L_n$ extends as an isometry of $\Hyp^3$, where $\Hyp^3$ is the ball model for the hyperbolic $3$-space.
    Let $r_n = d_{\Hyp^3}(L_n(\mathbf{0}), \mathbf{0})$, where $\mathbf{0}$ is the origin of $\Hyp^3$. Then $r_n \to \infty$. By \cite{Luo21, Luo22a}, the barycentric extensions $\mathcal{E}(R_n): \Hyp^3 \longrightarrow\Hyp^3$ of $R_n:\widehat{\C}\to\widehat{\C}$ converges to a degree $d$ branched covering $F$ on the asymptotic cone $(^r\Hyp, p, d) = \lim (\Hyp^3, \mathbf{0}, d_{\Hyp^3}/r_n)$, which is an $\R$-tree. 
    Since $R_n \to R$, we conclude that $p$ is totally invariant under $F$. 
    Let $q = \lim L_n^{-1}(\mathbf{0})$, where the limit is taken in $^r\Hyp$ with respect to the metric $d$. We denote the geodesic in $\Hyp^3$ connecting $\mathbf{0}$ and $L_n^{-1}(\mathbf{0})$ by $\gamma_n$. Since $L_n$ is an isometry of $\Hyp^3$, we have $\ell_{\Hyp^3}(\gamma_n)= d_{\Hyp^3}(\mathbf{0},L_n^{-1}(\mathbf{0}))=d_{\Hyp^3}(L_n(\mathbf{0}), \mathbf{0})=r_n$.    
    Thus by construction, $\gamma_n$ converges to $[p,q]$ under rescaling, and $[p,q]$ is a geodesic of length $1$ in $^r\Hyp$. Since $p$ is totally invariant under $F$, by \cite[Proposition~7.2]{Luo21} $F^{-1}([p,q]) = \bigcup_{j=1}^k [p, s_j]$, where each geodesic $[p, s_j]$ is mapped homeomorphically to $[p,q]$. In particular, each $s_j$ is an end point of $F^{-1}([p,q])$. By \cite[Proposition~5.5]{Luo21}, each end point $s_j \in F^{-1}([p,q])$ corresponds to some $M_{b,n}$. More precisely, for each $j$, $s_j = \lim M_{b,n}(\mathbf{0})$, where the limit is taken in $^r\Hyp$. Conversely, for each $b \in \RV_2$, there exists $s_j$ so that $s_j = \lim M_{b,n}(\mathbf{0})$. Note that this also induces a surjective continuous map $\RT_2 \longrightarrow F^{-1}([p,q])$.
    
    Suppose $k \geq 3$. Then there exists a branch point in $F^{-1}([p,q])$ {(i.e., removal of the point from $F^{-1}([p,q])$ yields at least three components)}. Thus $\RT_2$ also has a branch point which does not correspond to any of the rescalings $M_{b,n}$, which is a contradiction. Therefore $k = 2$. Suppose that there exist multiple $b_1, b_2,..., b_l \in \RV_2,\ l \geq 2$, with $s_1 = \lim M_{b_i,n}(\mathbf{0}),\ i = 1, 2,\cdots, l$.
    Let 
    $$
    t_n := \max\{d_{\Hyp^3}(M_{b_i,n}(\mathbf{0}), M_{b_j, n}(\mathbf{0})),\ i \neq j,\ i,j\in\{1,2,\cdots,l\} \}.
    $$
    Then $t_n \to \infty$, but $t_n/r_n \to 0$. Then by a similar argument applied to the limit $L_n \circ R_n \circ M_{b_1,n}$ with rescaling factor $t_n$, one can show that $\RT_2$ contains a branch point which does not correspond to any of the rescalings $M_{b,n}$. Once again, this is a contradiction. 
    
    Therefore, $\RT_2$ has exactly two vertices, corresponding to $s_1$ and $s_2$. Note that the multiplicities of $F$ on $[p, s_j],\ j = 1, 2$; i.e., the scaling factors of $F:[p, s_j] \longrightarrow [p,q]$, add up to $d$ by (\cite[Proposition~7.2]{Luo21}, see \cite[Chapter 9]{BR10} for more details). This means that the local degrees of $R$ at the two corresponding points add up to $d$.  Thus, the two rescaling limits on the Riemann spheres $\widehat{\C}_{s_1}$ and $\widehat{\C}_{s_2}$ have local degree $d_1, d_2$ at the singular point respectively, where $d_i$ is the degree of the rescaling limit on $\widehat{\C}_{s_i}$. Therefore, we conclude that  $\kappa_2$ is exceptional. This is a contradiction.
\end{proof}

\begin{prop}
    Suppose $[\pmb{R}_n] \in \Rat_{\kappa_1}(\C)/\!\!\sim$ converges to $[\pmb{R}]\in \Rat_{\kappa_2}(\C)/\!\!\sim$. Then the associated correspondences $[\mathfrak{C}_n]$ converges to $[\mathfrak{C}]$.

    Conversely, suppose $[\mathfrak{C}_n]$ converges to $[\mathfrak{C}]$. Then the uniformizing rational maps $[\pmb{R}_n]$ converges to $[\pmb{R}]$.
\end{prop}

\begin{proof}
    Suppose $[\pmb{R}_n] \in \Rat_{\kappa_1}(\C)/\!\!\sim$ converges to $[\pmb{R}]\in \Rat_{\kappa_2}(\C)/\!\!\sim$.
    Let $a, a' \in \RV_1$. Suppose that $R_{a,n}$ and $R_{a',n}$ converge to rational maps $\pmb{R}|_{(\RT_a, \widehat{\C}^{\RV_a})}: (\RT_a, \widehat{\C}^{\RV_a}) \longrightarrow \widehat{\C}$ and $\pmb{R}|_{(\RT_{a'}, \widehat{\C}^{\RV_{a'}})}: (\RT_{a'}, \widehat{\C}^{\RV_{a'}}) \longrightarrow \widehat{\C}$ with respect to rescalings $M_{s,n}$, $s \in \RV_a$ and $M_{t,n}$, $t \in \RV_{a'}$.
    Therefore, by Proposition \ref{prop:convergenceCorrespondence}, $\mathfrak{C}_{a, a',n} = \mathfrak{C}_{n} \cap \left( \widehat{\C}_a \times \widehat{\C}_{a'}\right)$ converges to the restriction of the correspondence $\mathfrak{C}$ on $(\RT_a, \widehat{\C}^{\RV_a}) \times (\RT_{a'}, \widehat{\C}^{\RV_{a'}})$. Therefore, $[\mathfrak{C}_n]$ converges to $[\mathfrak{C}]$ {in the sense of Definition~\ref{char_var_top_def}}. The proof of the converse is similar.
\end{proof}

\begin{remark}\label{rmk:CVOrbifold}
   1) Some comments regarding the nomenclature `character variety' are in order. One can embed the regular character variety $\Rat_\kappa^{reg}(\C)$ with characteristic data $\kappa$ into the GIT quotient of suitable spaces of rational maps under the action of the group $\Aut(\widehat{\C}) \times \Aut_{(\pmb{\tau},\pmb{\eta})}(\RT, \widehat\C^\RV)$, and this GIT quotient is an algebraic variety (cf. \cite{Mum65}). Further, in many concrete cases, parts of $\Rat_\kappa^{reg}(\C)$ containing matings between Fuchsian groups and polynomials carry explicitly describable algebraic structure (cf. \cite[\S 7]{MM2}). Finally, our choice of the word character variety is motivated in part by the closely related and more standard notion of $\PSL_2(\C)$-character variety (cf. \cite{Mar16}).

\noindent 2) As mentioned earlier, matings between more general genus $0$ orbifolds and polynomials can be realized as algebraic correspondences on trees of Riemann spheres (cf. \cite[Theorem~B]{MM2}, \cite[Theorem~1.9]{LMM24}). The character variety defined above contains such correspondences as well, so long as degrees are consistent.
\end{remark}

\section{Mating locus, Bers slices, and external fibers}\label{hor_vert_slices_sec}

\subsection{Hybrid simultaneous uniformization locus and Bers slices}\label{qf_bers_slice_subsec}
Let $\cH_{k}$ denote the principal hyperbolic component in the connectedness locus of degree $k$ monic, centered complex polynomials.
By Theorems~\ref{conf_mating_bs_poly_thm},~\ref{b_inv_thm}, and~\ref{corr_mating_thm_1}, for each $\Gamma\in\mathrm{Teich}(S_{0,d+1})$ and $P\in\cH_{2d-1}$, there exists a degree $2d$ rational map $R_{\Gamma,P}$ of $\widehat{\C}$ such that the associated correspondence $\mathfrak{C}_{\Gamma,P}$ (defined by Equation~\eqref{corr_eqn}) is a mating of $\Gamma$ and $P$. Note that in this case, the correspondence $\mathfrak{C}_{\Gamma,P}$ is defined on the trivial tree of spheres. There is  unique characteristic data of degree $2d$ on the trivial tree of spheres, and we denote the corresponding character variety by $\Rat_{2d}(\C)/\!\!\sim$ {and the set of regular maps in $\Rat_{2d}(\C)/\!\!\sim$ by $\Rat_{2d}^{reg}(\C)/\!\!\sim$} (see \S~\ref{char_var_subsec}). 

According to \cite[\S 6]{MM2}, the map
\begin{align*}
& \mathfrak{G}: \mathrm{Teich}(S_{0,d+1})\times \cH_{2d-1} \longrightarrow \mathrm{Rat}_{2d}(\C)/\!\!\sim \\
& \hspace{3.6cm} (\Gamma,P) \mapsto [R_{\Gamma,P}]
\end{align*}
is an injection. {Since each rational map $R\in\mathfrak{G}(\mathrm{Teich}(S_{0,d+1})\times \cH_{2d-1})$ has degree $2d\geq 4$ and has simple critical points at the fixed points $\pm 1$ of $\eta(z)=\frac1z$ (cf. \cite[Corollary~4.18]{MM2}), it follows that $\mathfrak{G}(\mathrm{Teich}(S_{0,d+1})\times \cH_{2d-1})$ is contained in $\Rat_{2d}^{reg}(\C)/\!\!\sim$.}

The image $\mathfrak{G}(\mathrm{Teich}(S_{0,d+1})\times \cH_{2d-1})$ is the analog of the classical quasi-Fuchsian locus (in the Kleinian world) and the space of quasi-Blaschke products (in the complex dynamics world). While quasi-Fuchsian groups (respectively, quasi-Blaschke products) simultaneously uniformize {a combination of} two surfaces/Fuchsian groups (respectively, combine two Blaschke products), a correspondence in $\mathfrak{G}(\mathrm{Teich}(S_{0,d+1})\times \cH_{2d-1})$ combines a Fuchsian group and a Blaschke product. We call the space $\mathfrak{G}(\mathrm{Teich}(S_{0,d+1})\times \cH_{2d-1})$ the \emph{Hybrid Simultaneous Uniformization} or the \emph{HSU~locus}.

The proof of \cite[Theorem~7.2]{MM2} shows that for each $P\in\cH_{2d-1}$, the map 
$$
\mathfrak{G}:\mathrm{Teich}(S_{0,d+1})\times \{P\}\to\mathrm{Rat}_{2d}^{reg}(\C)/\!\!\sim
$$ 
is holomorphic.

\begin{defn}\label{bers_def}
For $P\in\cH_{2d-1}$, we define its \emph{horizontal Bers slice} in the HSU locus $\mathfrak{G}(\mathrm{Teich}(S_{0,d+1})\times \cH_{2d-1})$ as
$$
\mathfrak{B}(P):=\mathfrak{G}\left(\mathrm{Teich}(S_{0,d+1})\times \{P\}\right).
$$ 
Similarly, for $\Gamma\in\mathrm{Teich}(S_{0,d+1})$, the \emph{vertical Bers slice} through $\Gamma$ is defined as
$$
\mathfrak{B}(\Gamma):=\mathfrak{G}\left(\{\Gamma\}\times \cH_{2d-1} \right).
$$
\end{defn}

\subsubsection{Dimensions of the spaces}\label{dim_count_subsec}
Note that $\mathrm{Teich}(S_{0,d+1})\times \cH_{2d-1}$ has complex dimension $(d-2)+(2d-2)=3d-4$. 

{We will now see that $\mathfrak{G}(\mathrm{Teich}(S_{0,d+1})\times \cH_{2d-1})$ is also contained in a complex $(3d-4)-$dimensional subspace of $\mathrm{Rat}_{2d}^{reg}(\C)/\!\!\sim$. Since this fact will not be used explicitly in the paper, we refrain from furnishing a complete proof.}
To this end, first observe that the complex dimension of $\mathrm{Rat}_{2d}(\C)$ is $4d+1$. According to \cite[Corollary~4.15]{MM2}, each rational map in the image of $\mathfrak{G}$ has $2d$ of its critical points at $\pm 1, \alpha_{1}^{\pm 1},\cdots,\alpha_{d-1}^{\pm 1}$ (for some $\alpha_1,\cdots,\alpha_{d-1}\in\C^*$). Such rational maps lie in the algebraic variety
$$
V_d :=\bigg\{R\in\mathrm{Rat}_{2d}(\C): R'(\pm 1)=0,\ \mathrm{sRes}_{2j}(R', R'\circ\eta) =0,\ j\in\{1,\cdots,d-1\}\bigg\},
$$
where $\mathrm{sRes}_j(R', R'\circ\eta)$ denotes the $j$-th principal subresultant coefficient of the numerators of $R'$ and $R'\circ\eta$ (note that if $x$ is a common zero of $R'$ and $R'\circ\eta$, then so {is} $\eta(x)$). {The $(d+1)$ algebraic equations defining $V_d$ are evidently independent; i.e., a rational map $R\in\mathrm{Rat}_{2d}(\C)$ satisfying a subset of the defining conditions does not necessarily satisfy the other conditions. Thus, standard transversality arguments can be used to deduce that $V_d$ has $\C$-dimension $(4d+1)-(d+1)=3d$.} Finally, since the equivalence relation $\sim$ is given by pre-composition with elements of the centralizer $C(\eta)$ 
of $\eta$ {in $\PSL_2(\C)$} and post-composition with elements of $\mathrm{PSL}_2(\C)$, the image of $\mathfrak{G}$ lies in the
$(3d-4)$ complex-dimensional space $V_d/\!\!\sim$.

{
\begin{example}\label{four_punc_example}
The product space $\mathrm{Teich}(S_{0,4})\times \cH_5$ of the Teichm{\"u}ller space of $4$-times punctured spheres and the principal hyperbolic component of {monic, centered,} degree $5$ polynomials embeds into the space of equivalence classes of degree $6$ rational maps (equivalently, in the space of bi-degree $(5,5)$ holomorphic correspondences on $\widehat{\C}$). The image of this embedding is the HSU locus in this case. The rational maps lying in the horizontal Bers slice $\mathfrak{B}(z^5)$ (in this HSU locus) have the following explicit form (see \cite[\S 7.1]{MM2}):
$$
R_c(z)=z+\frac{c}{z}-\frac{c}{3z^3}+\frac{1}{5z^5}.
$$
\end{example}
}

\subsection{Connection with degenerate polynomial-like maps, and external fibers}\label{degen_poly_like_map_ext_fiber_subsec}

For certain parameter space considerations, it is convenient to work with an appropriate class of maps that contains the conformal matings of polynomials and Bowen-Series maps. Recall from Theorem~\ref{conf_mating_bs_poly_thm} that the domain of definition of the conformal mating $F:\overline{\Omega}\to\widehat{\C}$ of $P$ and $A_\Gamma$ is a pinched disk. The restriction $F:\overline{F^{-1}(\Omega)}\to\overline{\Omega}$ can be viewed as a pinched/degenerate analog of classical polynomial-like maps. This point of view allows one to align the mating structure of $F$ with the classical description of polynomial-like maps as matings of internal classes and external maps. 

{Roughly speaking, the dynamics of a polynomial-like map splits into an \emph{internal class} (that records its dynamics on the non-escaping set) and an \emph{external map} (that models its action on the escaping set). A similar description is available for degenerate polynomial-like maps as well. As it so happens, for a conformal mating $F$ of $P$ and $A_\Gamma$, the internal class of the associated degenerate polynomial-like map is given by the polynomial $P$, while the Bowen-Series map $A_\Gamma$ plays the role of the external map.} We summarize the relevant parts of this theory from~\cite{LLM24}.

\subsubsection{Degenerate polynomial-like maps}\label{degen_poly_like map_subsubsec}

A set in $\widehat{\C}$ which is homeomorphic to a closed disk quotiented by a finite geodesic lamination, and which has a piecewise smooth boundary, is called a \textit{pinched polygon}.
The cut-points of a pinched polygon are called its \emph{pinched points}, and the non-cut points where two smooth local boundary arcs meet are called the \emph{corners} of the pinched polygon. 
\begin{figure}[ht]
\captionsetup{width=0.96\linewidth}
\begin{tikzpicture}
\node[anchor=south west,inner sep=0] at (0,0) {\includegraphics[width=0.8\textwidth]{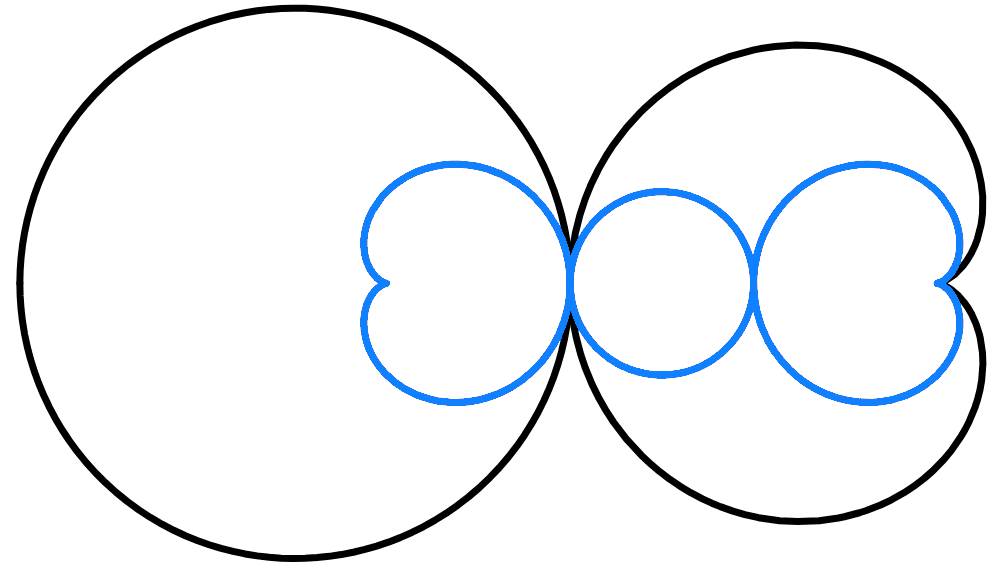}}; 
\node at (7,1.8) {\begin{small}$\partial P_1$\end{small}};
\node at (5.4,1) {\begin{small}$\partial P_2$\end{small}};
\end{tikzpicture}
\caption{Pictured is the domain and codomain of a degenerate polynomial-like map. {Here, the blue pinched polygon $P_1$ has two pinched points and two corners, while the black pinched polygon $P_2$ has one pinched point and one corner.}}
\label{degenerate_poly_like_fig}
\end{figure}

\begin{defn}\label{degenerate_poly_def}\cite[Definition~1.7]{LLM24}
Let $P_1, P_2\subset \widehat{\C}$ be pinched polygons such that 
\begin{enumerate}[leftmargin=*]
\item $P_1\subset P_2$,
\item the pinched points (respectively, corners) of $P_2$ are also pinched points (respectively, corners) of $P_1$, and
\item $\partial P_1\cap \partial P_2$ consists of the corners and pinched points of $P_2$.
\end{enumerate}
Suppose that there is a degree $d$ continuous map $g:P_1 \to P_2$ which is a proper holomorphic map from each component of $\Int{P_1}$ onto some component of $\Int{P_2}$, and such that the singular points of $\partial P_1$ are the $g$-preimages of the singular points of $\partial P_2$.

We then call the triple $(g, P_1, P_2)$ a \emph{degenerate polynomial-like map} of degree $d$.
\end{defn}
We refer the reader to \cite[\S 1.3, Appendix~A]{LLM24} for the notions of filled Julia sets/non-escaping sets, internal classes, external maps, and hybrid conjugacies of degenerate polynomial-like maps.

We remark that the map $F:\overline{F^{-1}(\Omega)}\to\overline{\Omega}$ (where $F$ is a conformal mating of $P$ and $A_\Gamma$) is a degenerate polynomial-like map of degree $2d-1$ which has $P$ as its internal class and $A_\Gamma$ as its external map.

\subsubsection{External/vertical fibers and a total space}\label{total_space_subsubsec}
With the above framework, the conformal matings between a fixed Fuchsian punctured sphere group $\Gamma$ and polynomials with connected Julia set lie in the following ambient space $\mathscr{B}_\Gamma$.
\begin{defn}\cite[\S 15]{LLM24}\label{vertical_fiber_def}
Let $\Gamma\in\mathrm{Teich}(S_{0,d+1})$. The space $\mathscr{B}_\Gamma$ is defined to be the collection of all degenerate polynomial-like maps with connected non-escaping set that admit the Bowen-Series map
$A_\Gamma$ as their external map. We call this space a \emph{external/vertical fiber} in the space of degenerate polynomial-like maps.
\end{defn}

The space $\mathscr{B}_\Gamma$ is equipped with an appropriate Carath{\'e}odory topology for holomorphic functions; see \cite[Definition~6.2, \S~15.3]{LLM24} for details.

The following result follows from \cite[Theorem~A.3 (part 2)]{LLM24}.

\begin{prop}\label{bdry_bijection_prop}
Let $\Gamma_1,\Gamma_2\in\mathrm{Teich}(S_{0,d+1})$. Then, there exists a dynamically natural bijection between $\mathscr{B}(\Gamma_1)$ and $\mathscr{B}(\Gamma_2)$.
\end{prop}

The bijection of Proposition~\ref{bdry_bijection_prop} is constructed by quasiconformally replacing the external map $A_{\Gamma_1}$ of degenerate polynomial-like maps in $\mathscr{B}(\Gamma_1)$ with the external map $A_{\Gamma_2}$. In holomorphic dynamics, such maps between spaces of polynomial-like maps are known as \emph{straightening maps}, which are known to be discontinuous at quasiconformally non-rigid parameters in many cases (cf. \cite{Ino09,IM21}). 

\begin{question}\label{str_discont_qn}
Let $\Gamma_1,\Gamma_2\in\mathrm{Teich}(S_{0,d+1})$. Is the dynamically natural bijection between $\mathscr{B}(\Gamma_1)$ and $\mathscr{B}(\Gamma_2)$ discontinuous?
\end{question}

Let us now look at the total space 
$$
\mathfrak{T}:=\bigsqcup_{\Gamma\in\mathrm{Teich}(S_{0,d+1})} \mathscr{B}_\Gamma
$$
of degenerate polynomial-like maps having Bowen-Series maps of arbitrary $(d+1)$-times punctured spheres as external maps. It is an ambient space for all conformal matings between complex polynomials and Bowen-Series maps of $(d+1)$-times punctured spheres. According to \cite[Theorem~A.3]{LLM24}:

\begin{prop}\label{total_space_cart_prod_prop}
There is a dynamically natural bijection between the space $\mathfrak{T}$ and the product $\mathrm{Teich}(S_{0,d+1})\times\mathscr{B}_{\Gamma_0}$, for any $\Gamma_0\in\mathrm{Teich}(S_{0,d+1})$. 
\end{prop}
\begin{remark}\label{cart_prod_rem}
The existence of this Cartesian product structure can be explained as follows. For each $F\in\mathscr{B}_{\Gamma_0}$, there exists a continuous injection 
$$
\Psi_F:\left(\mathrm{Teich}(S_{0,d+1}),\Gamma_0\right)\to\left(\mathfrak{T},F\right)
$$ 
such that $\Psi_F(\Gamma)$ is hybrid conjugate to $F$ and has $A_\Gamma$ as its external map. The map $\Psi_F$ is defined by quasiconformal deformation of the external map (as in Proposition~\ref{bdry_bijection_prop}), and its image is called the \emph{horizontal leaf} through $F$ in $\mathfrak{T}$ (see Figure~\ref{total_space_fig}). Such
a horizontal leaf intersects each external fiber $\mathscr{B}_\Gamma$ at exactly one point. Flowing along such leaves gives a bijective holonomy map $\Psi_\Gamma$ from the base external fiber $\mathscr{B}_{\Gamma_0}$ onto any other external fiber $\mathscr{B}_\Gamma$. 
These holonomy maps fit together to produce the desired Cartesian product structure of $\mathfrak{T}$. Possible discontinuity of the holonomy maps between external fibers leads us to believe that this dynamically natural Cartesian product structure cannot be promoted to a topological product structure.
\end{remark}
\begin{figure}[ht!]
\captionsetup{width=0.98\linewidth}
\begin{tikzpicture}
\node[anchor=south west,inner sep=0] at (0,0) {\includegraphics[width=1\linewidth]{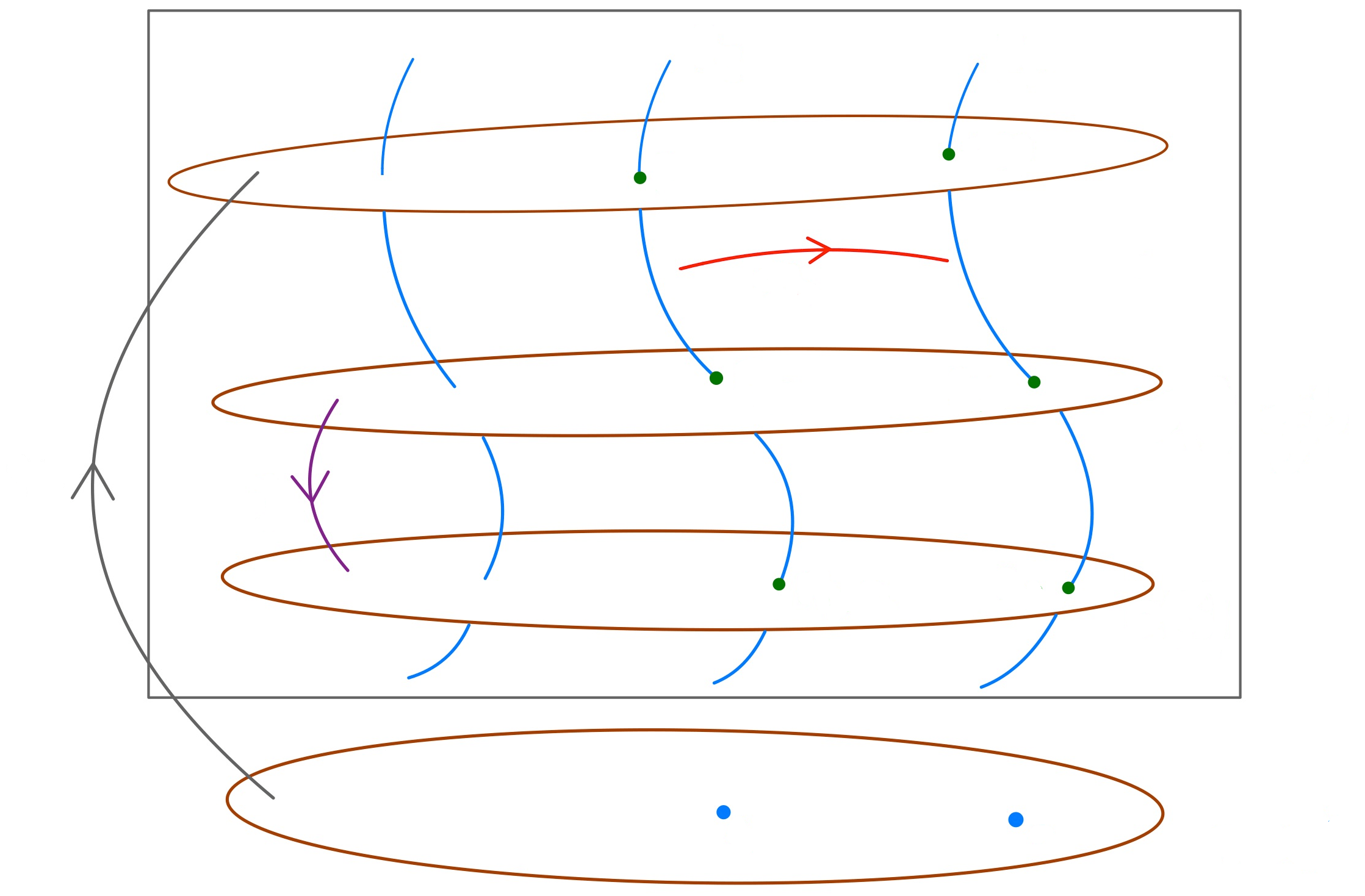}};
\node at (11.8,6.25) {\begin{Large}$\mathfrak{T}$\end{Large}};
\node at (4.4,0.75) {\begin{small}$\mathrm{Teich}(S_{0,d+1})$\end{small}};
\node at (0.48,3.28) {$\Psi_F$};
\node at (2.56,3.8) {$\chi$};
\node at (7,0.56) {\begin{scriptsize}$\Gamma_0$\end{scriptsize}};
\node at (9.6,0.56) {\begin{scriptsize}$\Gamma$\end{scriptsize}};
\node at (5.4,4.6) {\begin{small}$\mathfrak{B}(P_1)$\end{small}};
\node at (5.6,2.88) {\begin{small}$\mathfrak{B}(P_2)$\end{small}};
\node at (6.66,7.75) {$\mathscr{B}_{\Gamma_0}$};
\node at (9.42,7.75) {\begin{small}$\mathscr{B}_{\Gamma}$\end{small}};
\node at (7.75,5.66) {$\Psi_{\Gamma}$};
\node at (6.16,6.72) {\begin{small}$F$\end{small}};
\node at (9.36,6.9) {\begin{scriptsize}$\Psi_F(\Gamma)$\end{scriptsize}};
\node at (7.2,4.75) {\begin{scriptsize}$F_{\Gamma_0,P_1}$\end{scriptsize}};
\node at (9.2,4.75) {\begin{scriptsize}$F_{\Gamma,P_1}$\end{scriptsize}};
\node at (7.75,2.88) {\begin{scriptsize}$F_{\Gamma_0,P_2}$\end{scriptsize}};
\node at (9.4,2.88) {\begin{scriptsize}$F_{\Gamma,P_2}$\end{scriptsize}};
\end{tikzpicture}
\caption{Some vertical fibers and horizontal leaves in the total space $\mathfrak{T}$ are depicted. For $P_1,P_2\in\cH_{2d-1}$, the corresponding horizontal leaves are precisely the Bers slices $\mathfrak{B}(P_1), \mathfrak{B}(P_2)$. Lack of continuity of the bijective holonomy map $\Psi_\Gamma:\mathscr{B}_{\Gamma_0}\to\mathscr{B}_\Gamma$ represents discontinuity of straightening maps in holomorphic dynamics, and possible failure of continuous boundary extension of the homeomorphism $\chi:\mathfrak{B}(P_1)\to\mathfrak{B}(P_2)$ portrays the Kerckhoff-Thurston discontinuity phenomenon for Kleinian groups.}
\label{total_space_fig}
\end{figure}

\subsubsection{Identifying degenerate polynomial-like maps with correspondences}\label{identify_b_inv_corr_subsec}

It is easily seen that each $F\in\mathfrak{T}$ extends meromorphically to a B-involution in an inversive multi-domain (see \cite[Proposition~15.4]{LLM24}).
Thanks to the algebraic description of B-involutions given by Theorem~\ref{b_inv_thm}, the total space $\mathfrak{T}$
can be embedded into the character variety $\mathcal{CV}_{2d}$. {As the next result shows, under this embedding, the total space $\mathfrak{T}$ lives in a topologically well-behaved part of the character variety (see Definition~\ref{char_var_def}).}
{
\begin{prop}\label{total_space_regular_lem}
 The total space $\mathfrak{T}$ embeds in $\mathcal{CV}_{2d}^{reg}\cap\mathcal{CV}_{2d}^{simp}$.   
\end{prop}
\begin{proof}
Let $\pmb{R}:(\RT,\widehat{\C}^\RV)\to\widehat{\C}$ be the rational uniformizing map corresponding to a B-involution $(F:\overline{\Omega}\to\widehat{\C})\in\mathfrak{T}$. If $\vert\RV\vert=1$, then by \cite[Proposition~15.6 (part 3)]{LLM24}, the map $\pmb{R}$ has simple critical points at the fixed points $\pm 1$ of $\eta(z)=\frac1z$. Hence in this case, $\pmb{R}\in\Rat_{2d}^{reg}(\C)/\!\!\sim ~\subset\mathcal{CV}_{2d}^{reg}$. 

Now let $\pmb{R}\in\Rat_\kappa(\C)/\!\!\sim$ for some characteristic data $\kappa$ with $\vert\RV\vert>1$. 
According to \cite[Proposition~15.6 (part 3)]{LLM24}, the characteristic data $\kappa$ is simple. Hence, to prove that $\pmb{R}\in\Rat_\kappa^{reg}(\C)/\!\!\sim ~\subset\mathcal{CV}_{2d}^{reg}$, it suffices to verify condition~\eqref{reg:item:1} of regularity of characteristic data (see Definition~\ref{simple_sign_def} and Remark~\ref{rem:regular}). To this end, let $\cL$ be the coarse lamination associated with the B-involution $F:\overline{\Omega}\to\widehat{\C}$. By \cite[Proposition~15.6]{LLM24}, $\cL$ has no polygon. By \cite[\S 15.1]{LLM24}, two adjacent Riemann spheres of $\widehat{\C}^\RV$ correspond to a pair of touching components of $\Omega$, which in turn corresponds to two adjacent components of $\D-\cL$. Further, by \cite[Lemma~15.5]{LLM24}, the degrees of $\pmb{R}$ on two adjacent spheres $\widehat{\C}_a, \widehat{\C}_b$ (in $\widehat{\C}^\RV$) are both equal to $1$ if and only if for the corresponding adjacent components $\cG_a, \cG_b$ of $\D-\cL$, the intersections $\partial\cG_a\cap\partial\mathbb{S}^1, \partial\cG_b\cap\partial\mathbb{S}^1$ both have length $1/2d$ (where we identify $\mathbb{S}^1$ with $\R/\Z$). But the existence of such adjacent components is easily ruled out by the fact that $\cL$ has no polygon.
\end{proof}
}
Under the above embedding, the HSU locus $\mathfrak{G}(\mathrm{Teich}(S_{0,d+1})\times \cH_{2d-1})$ corresponds to the space of conformal matings between $A_\Gamma$, where $\Gamma\in\mathrm{Teich}(S_{0,d+1})$, and $P\in\cH_{2d-1}$.
In particular, the vertical Bers slice $\mathfrak{B}(\Gamma)$ (see Definition~\ref{bers_def}) can be identified with the principal hyperbolic component in the external fiber $\mathscr{B}_\Gamma$. Further, the horizontal Bers slice $\mathfrak{B}(P)$, where $P\in\cH_{2d-1}$, can be identified with the horizontal leaf in $\mathfrak{T}$ passing through the conformal mating of $P$ and $A_{\Gamma_0}$, for any $\Gamma_0\in\mathrm{Teich}(S_{0,d+1})$.

{
\begin{example}\label{quadratic_mating_vert_fiber_rem}
Let $\Sigma_0$ be the sphere with $2$ punctures and one order $2$ orbifold point. The Teichm{\"u}ller space of $\Sigma_0$ is a point, and hence the total space $\mathfrak{T}$ consists of a single vertical fiber; i.e., $\mathfrak{T}=\mathscr{B}_{\Sigma_0}$. This total space contains matings of quadratic polynomials and the Fuchsian group uniformizing $\Sigma_0$. There are two characteristic data for $\mathfrak{T}$: i) degree $3$ rational maps on a single Riemann sphere (which contains the HSU locus), and ii) rational maps $\pmb{R}:(\RT, \widehat\C^\RV)\to\widehat{\C}$, where $(\RT, \widehat\C^\RV)$ is a tree of two spheres,  $(\pmb{\tau},\pmb{\eta})$ is a M{\"o}bius involution fixing each of these two spheres, with $\pmb{R}$ having degree $2$ on one sphere and degree $1$ on the other (see Figure~\ref{mating_1_fig} for an example).   
\end{example}
}

\section{Quasiconformal surgery of B-involutions and correspondences}\label{qc_def_sec}

In this mostly technical section, we carry out a basic discussion of quasiconformal deformations and surgeries in the context of algebraic correspondences. This will be useful in our study of Bers slices and external fibers.

Let $\Omega=\bigsqcup_{a\in\RV}\Omega_a$ be an inversive multi-domain with associated B-involution $F:\overline{\Omega}\to\widehat{\C}$, where $\RV$ is a finite index set (see Definition~\ref{b_inv_def}). Let 
$$
\pmb{R}=\bigcup_{a\in\RV} R_a:(\RT, \widehat\C^\RV)\to\widehat{\C}
$$ 
be the uniformizing rational map for $F$. Let $\mathfrak{D}_a\subset\widehat{\C}_a$ be the Jordan domains given by Theorem~\ref{b_inv_thm}.  
In particular, $F\circ R_a= R_{\tau(a)}\circ\eta_a$ on $\mathfrak{D}_a$, $a\in\RV$ (where $\tau$ is an involution on $\RV$).

In this section, we will analyze the effect of quasiconformal deformations (respectively, surgeries) of the B-involution $F$ on the uniformizing rational maps $R_a$. It will turn out that the uniformizing rational maps of the new B-involution are related to those of the original B-involution by a non-dynamical quasiconformal deformation (respectively, surgery). We will also see how the algebraic correspondences associated with these B-involutions are related to each other.

\subsection{QC deformations of B-involutions and associated rational uniformizations}\label{qc_def_subsec}

Let $\mu$ be an $F$-invariant complex structure (i.e., measurable ellipse field) on $\widehat{\C}$, and $\Psi:\widehat{\C}\to\widehat{\C}$ a $K$-quasiconformal homeomorphism that sends $\mu$ to the standard complex structure. We define $\widetilde{F}:=\Psi\circ F\circ\Psi^{-1}:\overline{\widetilde{\Omega}}:=\Psi(\overline{\Omega})\to\widehat{\C}$. It is easy to see that $\widetilde{F}$ is also a B-involution in an inversive multi-domain. Our aim is to relate the rational uniformizing map for $\widetilde{F}$ to the rational map $\pmb{R}$. The description is implicit in the proof of \cite[Proposition~4.9]{MM2}. We present the core ideas of that proof here in a general form for later use.

{
\begin{lem}
Let the notation be as above. Then there exists a quasiconformal homeomorphism $\widetilde{\Psi}:(\RT, \widehat\C^\RV)\to(\RT, \widehat\C^\RV)$ such that the rational uniformizing map $\widetilde{\pmb{R}}$ of the $B$-involution $\widetilde{F}:=\Psi\circ F\circ\Psi^{-1}$ is given by $\widetilde{\pmb{R}}=\Psi\circ\pmb{R}\circ \widetilde{\Psi}^{-1}$.
\end{lem}
}
\begin{proof}
Define the complex structure
$$
\widetilde{\mu}:= \pmb{R}^*(\mu)
$$
on $(\RT, \widehat\C^\RV)$ as the pullback of $\mu$ under $\pmb{R}$. Let 
$$
\widetilde{\Psi}: (\RT, \widehat\C^\RV)\to(\RT, \widehat\C^\RV)
$$
be a $K$-quasiconformal homeomorphism where $\widetilde{\Psi}$ maps each $\widehat{\C}_a$ to itself, and straightens $\widetilde{\mu}$ to the standard complex structure on $(\RT, \widehat\C^\RV)$.
Then, the quasiregular map 
$$
\widetilde{\pmb{R}}:=\Psi\circ\pmb{R}\circ\widetilde{\Psi}^{-1}: (\RT, \widehat\C^\RV)\to\widehat{\C}
$$
preserves the standard complex structure. Hence, $\widetilde{\pmb{R}}$ is a rational map.

Note that the relation
$F\circ\pmb{R}=\pmb{R}\circ \pmb{\eta}$ on $\pmb{\mathfrak{D}}=\bigcup_{a\in\RV}\mathfrak{D}_a$, and $F$-invariance of $\mu$ together imply that the complex structure $\widetilde{\mu}$ on $(\RT, \widehat\C^\RV)$
is $\pmb{\eta}$-invariant. Therefore, 
$$
\widetilde{\Psi}\circ\pmb{\eta}\circ\widetilde{\Psi}^{-1}:(\RT, \widehat\C^\RV)\to(\RT, \widehat\C^\RV)
$$
is a conformal involution.
We can post-compose $\widetilde{\Psi}$ with a M{\"o}bius automorphism of $(\RT, \widehat\C^\RV)$ to ensure that $\widetilde{\Psi}\circ\pmb{\eta}\circ\widetilde{\Psi}^{-1}=\pmb{\eta}$.

We set $\widetilde{\pmb{\mathfrak{D}}}:=\widetilde{\Psi}(\pmb{\mathfrak{D}})\subset(\RT, \widehat\C^\RV)$, so $\widetilde{\Omega}=\Psi(\Omega)= \widetilde{\pmb{R}}(\widetilde{\mathfrak{D}})$. The fact that $\pmb{\mathfrak{D}}$ is mapped inside out by $\pmb{\eta}$ means that the same is true for $\widetilde{\pmb{\mathfrak{D}}}$. 
It is now easily checked from the above discussion that $\widetilde{\pmb{R}}$ is injective on $\widetilde{\pmb{\mathfrak{D}}}$, and
\begin{align*}
\widetilde{F} &= \Psi\circ F\circ\Psi^{-1}\\
&=\Psi\circ(\pmb{R}\circ\pmb{\eta}\circ(\pmb{R}\vert_{\pmb{\mathfrak{D}}})^{-{1}})\circ\Psi^{-1}\\
&= (\Psi\circ \pmb{R}\circ\widetilde{\Psi}^{-1})\circ(\widetilde{\Psi}\circ\pmb{\eta}\circ\widetilde{\Psi}^{-1})\circ(\widetilde{\Psi}\circ(\pmb{R}\vert_{\pmb{\mathfrak{D}}})^{-{1}}\circ\Psi^{-1})\\
&=\widetilde{\pmb{R}}\circ\pmb{\eta}\circ (\widetilde{\pmb{R}}\vert_{\widetilde{\pmb{\mathfrak{D}}}})^{-1},\quad \textrm{on}\ \overline{\widetilde{\Omega}}.
\end{align*} 
Thus, $\widetilde{\pmb{R}}:(\RT, \widehat\C^\RV)\to\widehat{\C}$ is the uniformizing rational map for the B-involution $\widetilde{F}$ (see \cite[Figure~9]{MM2} for an illustration).
\end{proof}

It also follows from the above analysis that if $\{\mu_\alpha:\alpha\in\Lambda\}$ (where $\Lambda$ is a topological/complex manifold) is a family of $F$-invariant complex structures depending continuously/holomorphically on $\alpha$, then the coefficients of the uniformizing rational map $\widetilde{\pmb{R}}_\alpha$ depend continuously/holomorphically on $\alpha$.

Finally, consider the algebraic correspondences $\mathfrak{C}$ and $\widetilde{\mathfrak{C}}$ defined by the branched coverings $\pmb{R}$ and $\widetilde{\pmb{R}}$ via Formula~\eqref{corr_eqn}.

{
\begin{cor}
   The $K$-quasiconformal homeomorphism $\widetilde{\Psi}:(\RT, \widehat\C^\RV)\to(\RT, \widehat\C^\RV)$ conjugates the correspondence $\mathfrak{C}$ to $\widetilde{\mathfrak{C}}$.
\end{cor} 
}

\subsection{Intrinsic QC deformation of correspondences}\label{qc_def_corr_subsec}

In this subsection, we will study quasiconformal deformations of a bi-degree $(n,n)$ algebraic correspondence $\mathfrak{C}$ defined by a branched covering $\pmb{R}:(\RT, \widehat\C^\RV)\to\widehat{\C}$ via Formula~\eqref{corr_eqn}, not necessarily coming from a B-involution. The purpose is to show that the conjugated correspondence is also of the same form.

Let $\widetilde{\mu}$ be a $\mathfrak{C}$-invariant almost complex structure on $(\RT, \widehat\C^\RV)$, and $\widetilde{\Psi}:(\RT, \widehat\C^\RV)\to(\RT, \widehat\C^\RV)$ be a quasiconformal homeomorphism straightening $\widetilde{\mu}$ to the standard complex structure. Then, $\widetilde{\mathfrak{C}}:=\widetilde{\Psi}\circ\mathfrak{C}\circ\widetilde{\Psi}^{-1}$ is a bi-degree $(n,n)$ holomorphic correspondence on $(\RT, \widehat\C^\RV)$. By Chow's theorem
{\cite[p.167]{GH78}}, $\widetilde{\mathfrak{C}}$ is an algebraic correspondence.

{
\begin{lem}\label{qc_def_corr_lem}
Let the notation be as above. Then there exist a quasiconformal homeomorphism $\Psi:\widehat\C\to\widehat\C$ and a rational map $\widetilde{\pmb{R}}:(\RT,\widehat{\C}^{\RV})\to\widehat{\C}$ such that $\widetilde{\pmb{R}}:=\Psi\circ\pmb{R}\circ\widetilde{\Psi}^{-1}$, and
$\widetilde{\mathfrak{C}}$ is defined by the rational map $\widetilde{\pmb{R}}$ via Formula~\eqref{corr_eqn}.
\end{lem}
}
\begin{proof}
Note that the local forward branches of $\mathfrak{C}$ are given by post-compositions of $\pmb{\eta}$ with local deck transformations of $\pmb{R}$. From this, it is easily checked that the grand orbits of $\mathfrak{C}$ contain $\pmb{\eta}$ and the local deck transformations of $\pmb{R}$. Hence, $\widetilde{\mu}$ is invariant under these maps. In particular, $\widetilde{\Psi}\circ\pmb{\eta}\circ\widetilde{\Psi}^{-1}$ is a conformal involution on $(\RT, \widehat\C^\RV)$. Possibly after post-composing $\widetilde{\Psi}$ with a M{\"o}bius automorphism, we can assume that $\widetilde{\Psi}$ conjugates $\pmb{\eta}$ to itself. Further, the fact that $\widetilde{\mu}$ is preserved by the local deck transformations of $\pmb{R}$ implies that $\mu:=\pmb{R}_*(\widetilde{\mu})$ is a well-defined almost complex structure on $\widehat{\C}$. (Note that the push-forward of a complex structure under a non-invertible map is in general not well-defined; the invariance of $\widetilde{\mu}$ under the local deck transformations of $\pmb{R}$ is a strong condition that makes this possible in the present situation.) Let $\Psi:\widehat{\C}\to\widehat{\C}$ be a quasiconformal homeomorphism straightening $\mu$ to the standard complex structure. Then, 
$$
\widetilde{\pmb{R}}:=\Psi\circ\pmb{R}\circ\widetilde{\Psi}^{-1}:(\RT, \widehat\C^\RV)\to\widehat{\C}
$$ 
is a holomorphic branched covering. Finally, as $\widetilde{\mathfrak{C}}=\widetilde{\Psi}\circ\mathfrak{C}\circ\widetilde{\Psi}^{-1}$, and
$$
\widetilde{\pmb{R}}^{-1}\circ\widetilde{\pmb{R}}\circ\pmb{\eta}=\widetilde{\Psi}\circ\left(\pmb{R}^{-1}\circ\pmb{R}\circ\pmb{\eta}\right)\circ\widetilde{\Psi}^{-1}
$$
(where we used the fact that $\widetilde{\Psi}\circ\pmb{\eta}=\pmb{\eta}\circ\widetilde{\Psi}$), it follows that the rational map $\widetilde{\pmb{R}}$ defines the correspondence $\widetilde{\mathfrak{C}}$ via Formula~\eqref{corr_eqn}.
\end{proof}

\subsection{QC surgery of B-involutions and associated rational uniformizations}\label{qc_surgery_subsec}

Let $X\subset\overline{\Omega}$ be a forward invariant subset of the B-involution $F:\overline{\Omega}\to\widehat{\C}$ {(with uniformizing rational map $\pmb{R}:(\RT,\widehat{\C}^{\RV})\to\widehat{\C}$)}. Assume further that $X\cap\partial\Omega$ is a finite set. {In the applications, the set $X$ will typically be chosen as the non-escaping set $\cK(F)$ of the B-involution $F$.} Let $\widetilde{F}:\overline{\widetilde{\Omega}}\to\widehat{\C}$ be a B-involution and $\Psi:\widehat{\C}\to\widehat{\C}$ a $K$-quasiconformal map such that $\Psi$ conjugates $F\vert_{\overline{\Omega}- X}$ to $\widetilde{F}\vert_{\overline{\widetilde{\Omega}}-\widetilde{X}}$, where $\widetilde{X}:=\Psi(X)$. By continuity, $\Psi$ conjugates $F\vert_{\partial X}$ to~$\widetilde{F}\vert_{\partial\widetilde{X}}$.

\begin{lem}\label{qc_surgery_lem}
Let the setup be as above. Then there exist a quasiconformal map $\widecheck{\Psi}:(\RT,\widehat{\C}^{\RV})\to(\RT,\widehat{\C}^{\RV})$ and a rational map $\widetilde{\pmb{R}}:(\RT,\widehat{\C}^{\RV})\to\widehat{\C}$ such that the following hold.
\begin{itemize}[leftmargin=*]
    \item $\widetilde{\pmb{R}}\equiv \Psi\circ\pmb{R}\circ\widecheck{\Psi}^{-1}$
outside $\widecheck{\Psi}(\pmb{R}^{-1}(X))$.
    \item $\widetilde{\pmb{R}}$ is the uniformizing rational map for the B-involution $\widetilde{F}$.
    \item The correspondences $\mathfrak{C}, \widetilde{\mathfrak{C}}$ defined by $\pmb{R}, \widetilde{\pmb{R}}$ (respectively) via Formula~\eqref{corr_eqn} are conjugate by $\widecheck{\Psi}$ outside of $\pmb{R}^{-1}(X)$.
\end{itemize}
\end{lem}
\begin{proof}
Define a $K$-quasiregular map
\begin{align*}
\widecheck{\pmb{R}}:(\RT, \widehat\C^\RV)\to\widehat{\C}, \hspace{1cm}
\widecheck{\pmb{R}}=
\begin{cases}
\Psi\circ \pmb{R}\quad \mathrm{on}\quad \overline{\pmb{\mathfrak{D}}},\\
\widetilde{F}\circ\Psi\circ \pmb{R}\circ\pmb{\eta} \quad \mathrm{on}\quad (\RT, \widehat\C^\RV)-\overline{\pmb{\mathfrak{D}}}.
\end{cases}
\end{align*}
The fact that the piecewise definitions given above match continuously follows from the fact that $F\circ\pmb{R}\circ\pmb{\eta}=\pmb{R}$ on $\partial\pmb{\mathfrak{D}}$, and $\Psi\circ F=\widetilde{F}\circ\Psi$ on~$\partial\Omega$. By construction, 
$$
\widetilde{F}\vert_{\widetilde{\Omega}}\equiv \widecheck{\pmb{R}}\circ\pmb{\eta}\circ\left(\widecheck{\pmb{R}}\vert_{\pmb{\mathfrak{D}}}\right)^{-1}.
$$
We now pull the standard complex structure on $\widehat{\C}$ back to $(\RT, \widehat\C^\RV)$ by $\widecheck{\pmb{R}}$, and call it $\widecheck{\mu}$. Since $\widetilde{F}$ is holomorphic by assumption, it follows that $\widecheck{\mu}$ is $\pmb{\eta}$-invariant.
Let $\widecheck{\Psi}:(\RT, \widehat\C^\RV)\to(\RT, \widehat\C^\RV)$ be a $K$-quasiconformal straightening map for $\widecheck{\mu}$. Then, 
$$
\widetilde{\pmb{R}}:(\RT, \widehat\C^\RV)\to\widehat{\C},\quad \widetilde{\pmb{R}}:=\widecheck{\pmb{R}}\circ\widecheck{\Psi}^{-1}$$ 
is a holomorphic branched covering. Moreover, $\widecheck{\Psi}\circ\pmb{\eta}\circ\widecheck{\Psi}^{-1}$ is a conformal involution of $(\RT, \widehat\C^\RV)$, which we can assume to be $\pmb{\eta}$ after appropriate normalizations.

Let us now set $\widetilde{\pmb{\mathfrak{D}}}:={\widecheck{\Psi}(\pmb{\mathfrak{D}})}$, and observe that 
$$
\widetilde{F}\vert_{\widetilde{\Omega}}\equiv \widetilde{\pmb{R}}\circ\pmb{\eta}\circ\left(\widetilde{\pmb{R}}\vert_{\widetilde{\pmb{\mathfrak{D}}}}\right)^{-1}.
$$
Thus, $\widetilde{\pmb{R}}$ is the uniformizing rational map for the B-involution $\widetilde{F}$.

Recall also that $\widetilde{F}\circ\Psi=\Psi\circ F$ outside $X$, and $F\circ\pmb{R}\circ\pmb{\eta}=\pmb{R}$ outside $\pmb{\mathfrak{D}}$. Hence, 
$$
\widetilde{\pmb{R}}\equiv \Psi\circ\pmb{R}\circ\widecheck{\Psi}^{-1}
$$ 
outside $\widecheck{\Psi}\left(\pmb{\eta}((\pmb{R}\vert_{\pmb{\mathfrak{D}}})^{-1}(X))\right)\subseteq \widecheck{\Psi}(\pmb{R}^{-1}(X))$ (where the containment follows from forward invariance of $X$ under $F$).

{The last statement follows from the definitions of $\mathfrak{C}, \widetilde{\mathfrak{C}}$ in terms of the rational maps $\pmb{R}$ and $\widetilde{\pmb{R}}$ (via Formula~\eqref{corr_eqn}), the equality $\widetilde{\pmb{R}}\equiv \Psi\circ\pmb{R}\circ\widecheck{\Psi}^{-1}$ outside $\widecheck{\Psi}(\pmb{R}^{-1}(X))$, and the conjugacy relation $\widecheck{\Psi}\circ\eta=\eta\circ\widecheck{\Psi}$.}
\end{proof}

\section{Homeomorphism between Bers slices and its regularity}\label{bers_slice_homeo_sec}

We will now show that the map $\mathfrak{G}$ (see \S~\ref{qf_bers_slice_subsec}) induces a homeomorphism between the topological product of the Teichm{\"u}ller space and principal hyperbolic component, and the HSU locus.

\begin{prop}\label{Xi_homeo_prop}
$\mathfrak{G}: \mathrm{Teich}(S_{0,d+1})\times \cH_{2d-1} \longrightarrow \mathrm{Rat}_{2d}^{reg}(\C)/\!\!\sim$ is a homeomorphism onto its image.
\end{prop}
\begin{proof}
As mentioned before, it follows from \cite[\S 6]{MM2} that $\mathfrak{G}$ is injective. We will now exhibit continuity and properness of the map $\mathfrak{G}$. Since $\Rat_{2d}^{reg}(\C)/\!\!\sim$ is Hausdorff (by Corollary~\ref{cor:HausdorffReg}), these properties will imply that $\mathfrak{G}$ is a homeomorphism onto its image.
\smallskip

\noindent\textbf{$\mathfrak{G}$ is continuous.} Let us fix {$[R_{\Gamma,P}]=\mathfrak{G}(\Gamma,P)$, for some $\Gamma\in\mathrm{Teich}(S_{0,d+1})$ and $P\in\cH_{2d-1}$. We denote the associated B-involution that is the mating of $A_\Gamma$ and $P$ by $F$.} For definiteness, let us assume that $P(z)=z^{2d-1}$. The conformal mating $\widetilde{F}$ of any pair $(\widetilde{\Gamma},\widetilde{P})\in\mathrm{Teich}(S_{0,d+1})\times\cH_{2d-1}$ can be constructed from $F$ by performing a quasiconformal deformation on the escaping set (where the dynamics is conformally conjugate to $A_{\Gamma}$) and a quasiconformal surgery on the non-escaping set (where the dynamics is conformally conjugate to $P$). Such a surgery deforms the conformal dynamics of $A_\Gamma:\cD_\Gamma\to\overline{\D}$ on the escaping set to that of $A_{\widetilde{\Gamma}}:\cD_{\widetilde{\Gamma}}\to\overline{\D}$, and replaces
the conformal dynamics of $z^{2d-1}\vert_{\overline{\D}}$ on the non-escaping set with that of $\widetilde{B}\vert_{\overline{\D}}$, where $\widetilde{B}$ is the unique normalized Blaschke product such that $\widetilde{B}\vert_{\overline{\D}}$ is conformally conjugate to  $\widetilde{P}\vert_{\cK(\widetilde{P})}$ (cf. \cite{McM88}, \cite[\S 4.1.3, \S 4.2]{BF14}). In fact, the above procedure can be carried out so that the resulting conformal mating $\widetilde{F}$ as well as the quasiconformal conjugacy $\Psi_{\widetilde{\Gamma},\widetilde{P}}$ between the quasiregular modification of $F$ and the conformal mating $\widetilde{F}$ depend continuously on $\widetilde{\Gamma}, \widetilde{P}$ (cf. \cite[\S 5]{Mil12}).

Let us denote the rational uniformizing map for $F$ by $R$.
To show that the rational uniformizing map $R_{\widetilde{\Gamma},\widetilde{P}}$ associated with $\widetilde{F}$ can also be chosen continuously, we will appeal to the  analysis of \S~\ref{qc_surgery_subsec}. In the notation of that subsection, the set $X$ can be chosen to be the non-escaping set of $F$. It follows from the preceding discussion that the quasiregular modification $\widecheck{R}_{\widetilde{\Gamma},\widetilde{P}}$ of $R$ constructed in the proof of Lemma~\ref{qc_surgery_lem} and the almost complex structure $\widecheck{\mu}_{\widetilde{\Gamma},\widetilde{P}}:=\widecheck{R}_{\widetilde{\Gamma},\widetilde{P}}^*(\mu_0)$ (where $\mu_0$ is the standard complex structure) depend continuously on $\widetilde{\Gamma}, \widetilde{P}$. By the parametric version of the Measurable Riemann Mapping Theorem, the quasiconformal homeomorphism $\widecheck{\Psi}_{\widetilde{\Gamma},\widetilde{P}}$ that straightens $\widecheck{\mu}_{\widetilde{\Gamma},\widetilde{P}}$ to the standard complex structure also depends continuously on $\widetilde{\Gamma},\widetilde{P}$. In light of the relation (see Lemma~\ref{qc_surgery_lem}) 
$$
R_{\widetilde{\Gamma},\widetilde{P}}=\widecheck{R}_{\widetilde{\Gamma},\widetilde{P}}\circ \widecheck{\Psi}_{\widetilde{\Gamma},\widetilde{P}}^{-1},
$$
we conclude that $R_{\widetilde{\Gamma},\widetilde{P}}$ depends continuously on $\widetilde{\Gamma}, \widetilde{P}$.
\smallskip

\noindent\textbf{$\mathfrak{G}$ is proper.}
By way of contradiction, suppose that $R_0=\mathfrak{G}(\Gamma_0,P_0)$ for some $(\Gamma_0,P_0)\in\mathrm{Teich}(S_{0,d+1})\times \cH_{2d-1}$, and $\{R_n=\mathfrak{G}(\Gamma_n,P_n)\}_{n\geq 1}\xrightarrow[n\to\infty]{} R_0$, for an escaping sequence $(\Gamma_n,P_n)$ in $\mathrm{Teich}(S_{0,d+1})\times \cH_{2d-1}$. For $n\geq 0$, let $F_n:\overline{\Omega}_n\to\widehat{\C}$ be the conformal mating of $A_{\Gamma_n}$ with $P_n$. Let $\mathfrak{D}_n$ be the Jordan domain such that $\eta(\overline{\mathfrak{D}}_n)=\widehat{\C}-\mathfrak{D}_n$, $R_n:\overline{\mathfrak{D}}_n\to\overline{\Omega}_n$ is a conformal homeomorphism, and $F_n\vert_{\overline{\Omega}_n}\equiv R_n\circ\eta\circ\left(R_n\vert_{\overline{\mathfrak{D}}_n}\right)^{-1}$ (see Theorem~\ref{b_inv_thm}). Further, $\partial\mathfrak{D}_n$ contains $2d$ simple critical points of $R_n$ which are  mapped by $R_n$ onto the non-singular points of $\partial\Omega_n$. We denote the correspondence on $\widehat{\C}$ associated with $R_n$ by $\mathfrak{C}_n$ (defined via Equation~\eqref{corr_eqn}). By \cite[\S 6.4]{MM2}, the domain $\mathfrak{D}_n$ as well as the B-involution $F_n$ can be determined uniquely from the correspondence $\mathfrak{C}_n$, or equivalently, from the rational map $R_n$.

For $n$ sufficiently large, the rational map $R_n$ is a small perturbation of $R_0$. The discussion in the previous paragraph now implies that the domain $\mathfrak{D}_n$ as well as the B-involution $F_n$ are small perturbations of $\mathfrak{D}_0, F_0$, respectively. Since $F_0$ has an attracting fixed point in $\cK(F_0)$, it follows that for $n$ large, the multiplier of the attracting fixed point of $F_n$ in $\cK(F_n)$ stays away from $\mathbb{S}^1$. As $F_n\vert_{\cK(F_n)}$ is conformally conjugate to $P_n\vert_{\cK(P_n)}$, we conclude that $P_n$ does not go to the boundary of $\cH_{2d-1}$.

The $2d$-gon $\widehat{\C}-\Omega_n$, with vertices at the marked critical values of $R_n$, is biholomorphic to a fundamental domain for the $\Gamma_n$-action on $\D$. For $n$ large, the domain $\mathfrak{D}_n$ is a small perturbation of $\mathfrak{D}_0$; hence, the extremal length of the path family (in $\widehat{\C}-\mathfrak{D}_n$) connecting any pair of non-adjacent sides of $\partial\mathfrak{D}_n$ is bounded away from $0$. So the marked fundamental domains of $\Gamma_n$ have the same property, and it follows that $\Gamma_n$ does not go to the boundary of $\mathrm{Teich}(S_{0,d+1})$. This contradicts the assumption that $(\Gamma_n,P_n)$ is an escaping sequence in $\mathrm{Teich}(S_{0,d+1})\times \cH_{2d-1}$.
\end{proof}

{
\begin{remark}
In the notation of Proposition~\ref{Xi_homeo_prop}, let $\widetilde{B}$ be the unique normalized Blaschke product conformally conjugate to $\widetilde{P}\vert_{\cK(\widetilde{P})}$ on $\overline{\D}$. The limit set of the conformal mating $\widetilde{F}$ of $A_{\widetilde{\Gamma}}$ and $\widetilde{P}$ is the welding curve for the circle homeomorphism conjugating $\widetilde{B}\vert_{\mathbb{S}^1}$ to the Bowen-Series map $A_{\widetilde{\Gamma}}\vert_{\mathbb{S}^1}$. These welding homeomorphisms depend continuously on $\widetilde{\Gamma}$ and $\widetilde{P}$ but are not quasisymmetric as they carry hyperbolic fixed points to parabolic ones (cf. \cite[Lemma~3.4]{MM2}, \cite[\S~5]{LMMN25}). Consequently, the parametric Measurable Riemann Mapping Theorem cannot be used directly to deduce the continuous dependence of $\widetilde{F}$ on $\widetilde{\Gamma}$ and $\widetilde{P}$. Instead, the continuity of $\mathfrak{G}$ is established by starting with a base conformal mating and realizing all others as its quasiconformal deformations.
\end{remark}}

As a consequence of the above result, we have that the complex dimension of each horizontal Bers slice is $d-2$ (cf. \S~\ref{dim_count_subsec}).

As in the classical Bers theory of Kleinian groups, the various Bers slices $\mathfrak{B}(P)$, $P\in\cH_{2d-1}$, are naturally homeomorphic via restrictions of the map $\mathfrak{G}$ on horizontal slices. In order to study boundaries of Bers slices, we will need more control on the regularity of such a homeomorphism. We recall the notation $\mathfrak{C}_{\Gamma,P}$, which is the correspondence arising as the mating of $\Gamma\in\mathrm{Teich}(S_{0,d+1})$ and $P\in\cH_{2d-1}$. 

\begin{theorem}\label{bers_slice_hoemo_thm}
Let $P_1,P_2\in\cH_{2d-1}$. Then, the map 
\begin{align*}
\chi:\mathfrak{B}(P_1)\longrightarrow\mathfrak{B}(P_2)\\
\mathfrak{C}_{\Gamma,P_1}\mapsto \mathfrak{C}_{\Gamma,P_2} & \quad \forall\ \Gamma\in\mathrm{Teich}(S_{0,d+1}),
\end{align*}
is a homeomorphism. Moreover, the correspondences $\mathfrak{C}_{\Gamma,P_1}$ and $\mathfrak{C}_{\Gamma,P_2}$ are $K$-quasiconformally conjugate 
{(for a constant $K>0$ that is independent of $\Gamma$ but depends on $P_1, P_2$)} in a neighborhood of their filled limit sets (preserving the marked fixed points on the limit sets) such that
\begin{enumerate}[leftmargin=*]
    \item the conjugacy is conformal on the filled limit set, and
    \item the {domains of definition of the quasiconformal conjugacies can be chosen as closed neighborhoods of the filled limit sets that depend continuously on $\Gamma$ in the Hausdorff topology on compact subsets of $\widehat{\C}$}.
\end{enumerate}
\end{theorem}

The first statement of Theorem~\ref{bers_slice_hoemo_thm} follows from the fact that 
$$
\mathfrak{G}:\mathrm{Teich}(S_{0,d+1})\times\{P_j\}\to\mathfrak{B}(P_j),\ j\in\{1,2\},
$$
are homeomorphisms (see Proposition~\ref{Xi_homeo_prop}). However, we need to furnish the details of the quasiconformal surgical procedure to pass from $\mathfrak{C}_{\Gamma,P_1}$ to $\mathfrak{C}_{\Gamma,P_2}$ in order to deduce the additional control on $\chi$. This extra regularity of $\chi$ will be important in studying its boundary behavior.

Note that each $P\in\cH_{2d-1}$ has a marked repelling fixed point on its Julia set that is the landing point of the dynamical ray at angle $0$. Thus, {$P\vert_{\Int{\cK(P)}}$} is conformally conjugate to a Blaschke product $B$ (restricted to $\D$) that has an attracting fixed point in $\D$ and has a marked repelling fixed point at $1$. 

\begin{lem}\label{blaschke_conjugacy_lem}
Let $B_1,B_2$ be degree $n\geq 2$ Blaschke products each having an attracting fixed point in $\D$ and a repelling fixed point at $1$. Then, there exists a quasiconformal homeomorphism $\mathfrak{h}:\overline{\D}\to\overline{\D}$ that conjugates $B_1$ to $B_2$ on a relative neighborhood of $\mathbb{S}^1$ in $\overline{\D}$, and sends $1$ to $1$.    
\end{lem}
\begin{proof}
For each $j\in\{1,2\}$, there exists $r_j\in(0,1)$ such that the ball $B(0,r_j)$ contains all $n-1$ critical points of $B_j$ (in $\D$), $B_j^{-1}(B(0,r_j))$ is a topological disk compactly containing $B(0,r_j)$, and $B_j:B_j^{-1}(B(0,r_j))\to B(0,r_j)$ is a degree $n$ branched cover. Consider a conformal isomorphism $\mathfrak{h}:B(0,r_1)\to B(0,r_2)$. It extends continuously to a diffeomorphism between the boundaries. Using the degree $n$ covering maps $B_j:\partial B_j^{-1}(B(0,r_j))\to \partial B(0,r_j)$, we lift $\mathfrak{h}:\partial B(0,r_1)\to \partial B(0,r_2)$ to a diffeomorphism $\mathfrak{h}:\partial B_1^{-1}(B(0,r_1))\to \partial B_2^{-1}(B(0,r_2))$. We then interpolate in the annulus to get a diffeomorphism $\mathfrak{h}:\overline{B_1^{-1}(B(0,r_1))}\to\overline{B_2^{-1}(B(0,r_2))}$. Since these annuli are fundamental domains for the Blaschke products, we can now use the dynamics to lift $\mathfrak{h}$ iteratively. This produces a quasiconformal diffeomorphism of $\D$ that extends to a quasisymmetry $\mathfrak{h}$ of $\mathbb{S}^1$. By construction, the map $\mathfrak{h}$ conjugates $B_1$ to $B_2$ on a relative neighborhood of $\mathbb{S}^1$ in $\overline{\D}$. The lifts can be chosen such that the map $\mathfrak{h}$ maps the fixed point $1$ of $B_1$ to the fixed point $1$ of $B_2$.
\end{proof}

\begin{proof}[Proof of Theorem~\ref{bers_slice_hoemo_thm}]
Let $B_1, B_2$ be degree $2d-1$ Blaschke products each having an attracting fixed point in $\D$ and a repelling fixed point at $1$, such that $P_j\vert_{\cK(P_j)}$ is conformally conjugate to $B_j\vert_{\overline{\D}}$, $j\in\{1,2\}$.

By Lemma~\ref{blaschke_conjugacy_lem}, there exist a quasiconformal homeomorphism $\mathfrak{h}:(\overline{\D},1)\to(\overline{\D},1)$ and relative closed neighborhoods $N_1, N_2$ of $\mathbb{S}^1$ in $\overline{\D}$ such that $\mathfrak{h}(N_1)=N_2$, and $\mathfrak{h}$ conjugates $B_1:B_1^{-1}(N_1)\to N_1$ to $B_2:B_2^{-1}(N_2)\to N_2$.

For $\Gamma\in\mathrm{Teich}(S_{0,d+1})$, let $F_{\Gamma, P_j}$ be the conformal mating of $A_\Gamma$ and $P_j$, $j\in\{1,2\}$. 
Then, there exist conformal maps $\mathfrak{g}_{\Gamma,j}:\overline{\D}\to \cK(F_{\Gamma, P_j})$ that conjugate $B_j$ to $F_{\Gamma, P_j}$, and carry the fixed point $1$ of $B_j$ to the marked fixed point of $F_{\Gamma, P_j}$ on their limit set.
We recall that since all the conformal matings $F_{\Gamma, P_j}$ are quasiconformally conjugate to any fixed $F_{\Gamma_0, P_j}$ with the quasiconformal conjugacies depending continuously on $\Gamma$, the conformal matings $F_{\Gamma, P_j}$ as well as the conformal maps $\mathfrak{g}_{\Gamma,j}$ depend continuously on $\Gamma$.

We will now construct the conformal mating $F_{\Gamma,P_2}$ from $F_{\Gamma,P_1}$ using quasiconformal surgery, for $\Gamma\in\mathrm{Teich}(S_{0,d+1})$. To this end, define a quasiregular map
\begin{align*}
&\widetilde{F}_{\Gamma,P_1}=
\begin{cases}
F_{\Gamma,P_1}\quad \mathrm{on}\quad \cT(F_{\Gamma,P_1}),\\
\mathfrak{g}_{\Gamma,1}\circ\mathfrak{h}^{-1}\circ B_2\circ\mathfrak{h}\circ\mathfrak{g}_{\Gamma,1}^{-1} \quad \mathrm{on}\quad \cK(F_{\Gamma,P_1}).
\end{cases}    
\end{align*}
Note that due to the equivariance property of $\mathfrak{h}$ with respect to $B_1$ and $B_2$, the piecewise definitions given above agree on $\mathfrak{g}_{\Gamma,1}(N_1)$. We define an $\widetilde{F}_{\Gamma,P_1}$-invariant almost complex structure $\mu_\Gamma$ on $\widehat{\C}$ by setting it to be the standard complex structure $\mu_0$ on $\cT(F_{\Gamma,P_1})$ and setting it to be $\left(\mathfrak{h}\circ\mathfrak{g}_{\Gamma,1}^{-1}\right)^*(\mu_0)$ on $\cK(F_{\Gamma,P_1})$. We now conjugate $\widetilde{F}_{\Gamma,P_1}$ by the quasiconformal map $\mathfrak{Q}_\Gamma$ that straightens the almost complex structure $\mu_\Gamma$ to the standard complex structure $\mu_0$, and obtain a holomorphic map $\mathfrak{Q}_\Gamma\circ\widetilde{F}_{\Gamma,P_1}\circ\mathfrak{Q}_\Gamma^{-1}$. 
Since $\widetilde{F}_{\Gamma,P_1}$ and $F_{\Gamma,P_1}$ agree on $\cT(F_{\Gamma,P_1})$ and $\mathfrak{Q}_\Gamma$ is conformal there, it follows that $\mathfrak{Q}_\Gamma\circ\widetilde{F}_{\Gamma,P_1}\circ\mathfrak{Q}_\Gamma^{-1}$ is conformally conjugate to the Bowen-Series map $A_{\Gamma}$ on $\mathfrak{Q}_\Gamma(\cT(F_{\Gamma,P_1}))$. On the other hand, the map $\mathfrak{Q}_\Gamma\circ\mathfrak{g}_{\Gamma,1}\circ\mathfrak{h}^{-1}:\overline{\D}\to\mathfrak{Q}_\Gamma(\cK(F_{\Gamma,P_1}))$ is a conformal conjugacy between $B_2$ and $\mathfrak{Q}_\Gamma\circ\widetilde{F}_{\Gamma,P_1}\circ\mathfrak{Q}_\Gamma^{-1}$. Therefore, we have that 
$$
F_{\Gamma,P_2}\equiv \mathfrak{Q}_\Gamma\circ\widetilde{F}_{\Gamma,P_1}\circ\mathfrak{Q}_\Gamma^{-1},\ \mathrm{and}\ \mathfrak{g}_{\Gamma,2}\equiv \mathfrak{Q}_\Gamma\circ\mathfrak{g}_{\Gamma,1}\circ\mathfrak{h}^{-1}.
$$
By construction, $F_{\Gamma,P_1}$ and $F_{\Gamma,P_2}$ are quasiconformally conjugate on their restrictions to the preimages of $\cT(F_{\Gamma,P_1})\cup\mathfrak{g}_{\Gamma,1}(N_1)$ and $\cT(F_{\Gamma,P_2})\cup\mathfrak{g}_{\Gamma,2}(N_2)$, respectively. We remark that $\mathfrak{Q}_\Gamma$ is $K$-quasiconformal, where $K$ only depends on the quasiconformality constant of $\mathfrak{h}$, and not on $\Gamma$.
Finally, the continuous dependence of $\widetilde{F}_{\Gamma,P_1}, \mu_\Gamma, \mathfrak{Q}_\Gamma$, and $\mathfrak{g}_{\Gamma,j}$ on $\Gamma$ guarantee that the domains of the quasiconformal conjugations between $F_{\Gamma,P_1}$ and $F_{\Gamma,P_2}$ also depend continuously on $\Gamma$.

Thanks to {Lemma~\ref{qc_surgery_lem}}, the above quasiconformal surgery between $F_{\Gamma,P_1}$ and $F_{\Gamma,P_2}$ can be lifted to the correspondence planes yielding a $K$-quasiconformal conjugacy between $\mathfrak{C}_{\Gamma,P_1}$ and $\mathfrak{C}_{\Gamma,P_2}$ in a neighborhood of their filled limit sets such that the conjugacy is conformal on the filled limit sets and the domains of the conjugacies vary continuously with $\Gamma\in\mathrm{Teich}(S_{0,d+1})$. To complete the proof, note that the marked fixed points of the correspondences on their limit sets come from the marked fixed points of the conformal matings $F_{\Gamma,P_j}$ corresponding to the marked fixed points (at $1$) of the Bowen-Series maps $A_\Gamma$. Since the surgery procedure carried out above preserves these marked fixed points of the conformal matings, the quasiconformal conjugacy between $\mathfrak{C}_{\Gamma,P_1}$ and $\mathfrak{C}_{\Gamma,P_2}$ also preserves the marked fixed points on the limit sets.
\end{proof}

\section{Pre-compactness of horizontal Bers slices}\label{pre_comp_bers_sec}

The main result of this section is the following:

\begin{theorem}\label{pre_comp_thm_1}
Let $P\in\cH_{2d-1}$. Then $\mathfrak{B}(P)$ is pre-compact in $\mathrm{Rat}_{2d}(\C)/\!\!\sim$. {Further, $\overline{\mathfrak{B}(P)}\subset \mathrm{Rat}_{2d}^{reg}(\C)/\!\!\sim$}.
\end{theorem}

\begin{remark}
The above theorem underscores the importance of algebraic correspondences. Indeed, according to \cite[Theorem~C]{MM1}, there are only finitely many quasiconformal conjugacy classes of Kleinian groups on the Bers boundary of $\mathrm{Teich}(S_{0,d+1})$ that admit continuous Bowen-Series maps. In other words, the single-valued mating framework (as described in \S~\ref{corr_mating_subsec}) does not cover most groups on the Bers boundary of $\mathrm{Teich}(S_{0,d+1})$. On the other hand, as the closure $\overline{\mathfrak{B}(P)}$ is compact in the {regular character variety $\mathrm{Rat}_{2d}^{reg}(\C)/\!\!\sim$} of algebraic correspondences, it is conceivable that such boundary points arise as matings of polynomials with a large collection of (all?) Bers boundary groups. 
\end{remark}

\subsection{Normalization}
Let $B:\D \longrightarrow \D$ be a marked Blaschke product of degree $2d-1$ with a (super-)attracting fixed point at $0$  such that $B\vert_{\overline{\D}}$ is conformally conjugate to $P\vert_{\cK(P)}$, where the conjugacy sends the marked repelling fixed point of $B$ to the landing point of the dynamical $0-$ray of $P$. For definiteness, we assume that $1\in \partial \D$ is the marked repelling fixed point of $B$.
For $\Gamma\in\mathrm{Teich}(S_{0,d+1})$, we denote the conformal mating between $A_\Gamma$ and $B\vert_{\overline{\D}}$ by $F_\Gamma:\overline{\Omega_\Gamma}\to\widehat{\C}$ (see Theorem~\ref{conf_mating_bs_poly_thm}). Theorem~\ref{b_inv_thm}, applied to the mating $F_\Gamma:\overline{\Omega_\Gamma}\to\widehat{\C}$, gives us a degree $2d$ rational map $R_\Gamma$ and a Jordan domain $\mathfrak{D}_\Gamma$ such that $\partial\mathfrak{D}_\Gamma$ is $\eta$-invariant, $1\in\partial\mathfrak{D}_\Gamma$, the map $R_\Gamma$ carries $\overline{\mathfrak{D}_\Gamma}$ homeomorphically onto $\overline{\Omega_\Gamma}$, and
\begin{align}
F_\Gamma\vert_{\overline{\Omega_\Gamma}}\equiv R_\Gamma\circ\eta\circ(R_\Gamma\vert_{\overline{\mathfrak{D}_\Gamma}})^{-1}.
\label{b_inv_formula_1}
\end{align}
We denote the algebraic correspondence given by~\eqref{corr_eqn}, associated with $R_\Gamma$, by $\mathfrak{C}_\Gamma$.

Note that the map $F_\Gamma$ has a unique (super-)attracting fixed point. 
Recall from \S~\ref{char_var_subsec} that the map $R_\Gamma$ is unique up to post-composition with an arbitrary M{\"o}bius map and pre-composition with a M{\"o}bius map commuting with $\eta$. 
Using this freedom of normalization, we can assume that
\begin{enumerate}[leftmargin=*]
    \item $\infty$ is the unique (super-)attracting fixed point of the conformal mating $F_\Gamma$, with its immediate basin of attraction denoted by $\mathcal{B}_\infty(F_\Gamma)$;
    \item $1 \in \partial {\mathfrak{D}_\Gamma}$ and $1$ is a simple critical point of $R_\Gamma$;
    \item the conformal conjugacy $\phi_\Gamma:(\D,0)\to(\mathcal{B}_\infty(F_\Gamma),\infty)$
    between $B$ to $F_\Gamma$ has the asymptotics $\phi_\Gamma(\zeta)=1/\zeta+O(\zeta)$ near $0$, and $\phi_\Gamma(1) = R_\Gamma(1)$;
    \item $\infty\in\mathfrak{D}_\Gamma$ with $R_\Gamma(\infty)=\infty$.
\end{enumerate}

\subsection{Degeneration}
To prove Theorem \ref{pre_comp_thm_1}, we will argue by contradiction and assume that $\mathfrak{B}(P)$ is not pre-compact.
Let $\mathfrak{C}_{n}$ be a sequence of correspondences in $\mathfrak{B}(P)$. Let $R_{n}$ be the corresponding rational maps with the above normalization.
Recall that a sequence $\mathfrak{C}_{n}$ \textit{diverges} in $\mathrm{Rat}_{2d}(\C)/\!\!\sim$ if it escapes every compact set of $\mathrm{Rat}_{2d}(\C)/\!\!\sim$; or equivalently, there is no subsequence of $\mathfrak{C}_{n}$ that converges in $\mathrm{Rat}_{2d}(\C)/\!\!\sim$.

\begin{lem}\label{lem:Rinfty}
    Suppose that $\mathfrak{C}_{n}$ diverges in $\mathrm{Rat}_{2d}(\C)/\!\!\sim$. Then $\vert R_n'(\infty)\vert \to \infty$.
\end{lem}
\begin{proof}
    Suppose that $\vert R_n'(\infty)\vert$ does not go to infinity. Then after passing to a subsequence, we may assume that $R_n'(\infty)$ converges. After passing to a further subsequence, we can assume that $R_n$ converges to a rational map $R$ of degree potentially smaller than $2d$. By Lemma \ref{comcon}, $R_n$ converges compactly to $R$ away from finitely many points.
    
    Consider the univalent map $\Phi_n: (\D,0) \longrightarrow (\widehat{\C}, \infty)$ by $\Phi_n = R_n^{-1} \circ \phi_n$, where $\phi_n:(\D,0)\to(\mathcal{B}_\infty(F_n),\infty)$ is the conformal map conjugating $B$ to $F_n=R_n\circ\eta\circ(R_n\vert_{\mathfrak{D}_n})^{-1}$, and $R_n^{-1}$ is the inverse of the conformal map $R_n: (\mathfrak{D}_n, \infty) \longrightarrow (\Omega_n, \infty)$ {(see Figure~\ref{two_rescalings_fig}, top two figures for an illustration)}. 
    Since $0, 1 \notin$ { $\mathfrak{D}_n$} and $\Phi_n(\D) \subseteq$ {$\mathfrak{D}_n$}, we have that $0, 1 \notin \Phi_n(\D)$. {Thus $\Phi_n$ has bounded spherical derivative at $0$. In particular, $\vert R_n'(\infty)\vert$ is bounded away from $0$. Moreover, as $\vert R_n'(\infty)\vert$ is assumed to be bounded, it follows that $\Phi_n$ has spherical derivative comparable to $1$ at $0$. Thus, $\Phi_n$ has the asymptotics $\Phi_n(\zeta)=a_n/\zeta+b_n+O(\zeta)$ near $0$, where $a_n$ is comparable to $1$ and $b_n$ is a bounded sequence (cf. \cite[Theorem~1.4]{Pom75}).   Therefore, after passing to a subsequence, $\Phi_n$ converges compactly to a univalent map $\Phi: (\D,0) \longrightarrow (\widehat{\C}, \infty)$ (see \cite[Theorem~1.7]{Pom75} or \cite[Theorem 2.7]{McM94}).
    Once again, due to our normalization of $\phi_n$, after passing to a subsequence, $\phi_n$ converges compactly to a univalent map $\phi: (\D,0) \longrightarrow (\widehat{\C}, \infty)$ as well.}
    Let $U \coloneqq \Phi(\D)$ and $\mathcal{B} \coloneqq \phi(\D)$. Then $R_n$ converges compactly on $U$ to a conformal map $R: (U, \infty) \longrightarrow (\mathcal{B}, \infty)$. Note that $R: (U, \infty) \longrightarrow (\mathcal{B}, \infty)$ is necessarily the restriction of the limiting rational map $R$ on $U$, which justifies the use of the notation.
    
    Let $U_n \coloneqq \Phi_n(\D)$, and $W_n \coloneqq\eta(U_n)$. Then $R_n: (W_n, 0) \longrightarrow (\mathcal{B}_\infty(F_n), \infty)$ is a sequence of degree $2d-1$ proper maps. Since $(W_n, 0)$ and $(\mathcal{B}_\infty(F_n), \infty)$ converge to $(W{=\eta(U)}, 0)$ and $(\mathcal{B}, \infty)$ respectively, after passing to a subsequence, $R_n$ converges compactly to a proper map $R: (W, 0) \longrightarrow (\mathcal{B}, \infty)$, potentially of lower degree, by Proposition \ref{prop:ConvergenceProperMap}. Similarly as before, the proper map $R: (W, 0) \longrightarrow (\mathcal{B}, \infty)$ is necessarily the restriction of the limiting rational map $R$ on $W$. Note that 
    $$
    B(z) = \phi_n^{-1} \circ R_n \circ \eta \circ R_n^{-1} \circ \phi_n(z)
    $$
    for $z \in \D$. So the set of critical values of the map $R_n: (W_n, 0) \longrightarrow (\mathcal{B}_\infty(F_n), \infty)$ equals  $\phi_n(V(B))$ where $V(B) \subseteq \D$ is the set of critical values of $B: \D \longrightarrow \D$. Since the set $V(B)$ is independent of $n$, { the maximal distance between points in $\phi_n(V(B))$ and $\infty = \phi_n(0)$ in the hyperbolic metric of $\mathcal{B}_\infty(F_n) = \phi_n(\D)$ is constant.} It follows that $\phi_n(V(B))$ is contained in some compact subset of $\mathcal{B}$ for all sufficiently large $n$. Therefore, the critical points of $R_n: (W_n, 0) \longrightarrow (\mathcal{B}_\infty(F_n), \infty)$ are contained in some compact subset of $W$ for all sufficiently large $n$. Thus, $R: (W, 0) \longrightarrow (\mathcal{B}, \infty)$ has degree $2d-1$ by Proposition \ref{prop:ConvergenceProperMap}.
    
    Since $W_n=\eta(U_n)$ is disjoint from $U_n$ for all $n$, we conclude that $U$ is disjoint from $W= \eta(U)$. Thus $R_n$ converges compactly to a degree $2d$ proper map $R: U \sqcup W \longrightarrow \mathcal{B}$. 
    Therefore, the limiting rational map $R$ has degree $2d$. This means that $R_n$ converges in $\mathrm{Rat}_{2d}(\C)$. This is a contradiction, which proves that $\vert R_n'(\infty)\vert \longrightarrow~\infty$.
\end{proof}

\subsection{The special case of $z^{2d-1}$}\label{special_case_subsec}
Before proving pre-compactness of $\mathfrak{B}(P)$ for a general polynomial $P\in\cH_{2d-1}$, let us consider the special case $P(z) = z^{2d-1}$. In this case, we can directly compute $R_n'(\infty)$ and appeal to Lemma~\ref{lem:Rinfty} to conclude boundedness of $\mathfrak{B}(P)$ as follows.
By our normalization, $R_n$ has a simple pole at $\infty$ and an order $2d-1$ pole at the origin. Together with the fact that the critical points of $R_n$ are of the form 
$$
\{\pm 1, \alpha_{1,n}^{\pm 1},\cdots,\alpha_{d-1,n}^{\pm 1}\},
$$
we can write $R_n$ in the following explicit form:
\begin{align*}
\frac{R_n(z)}{R'_n(\infty)}=  \ z+\frac{a_{1,n}}{z}+\cdots+\frac{a_{d-2,n}}{z^{d-2}}+\frac{a_{d,n}}{z^d}+\cdots+\frac{a_{2d-3,n}}{z^{2d-3}}+\frac{1}{(2d-1)\cdot z^{2d-1}},\nonumber
\label{rat_norm}
\end{align*}
where
\begin{equation*}
a_{2d-j-1}=-\ \frac{(j-1)}{(2d-j-1)} a_{j-1},\quad j\in\{2,\cdots,d-1\}.
\end{equation*}
We refer the reader to \cite[\S 4.4, \S 7.1]{MM2} for details of the above computation. 
Thus, the Laurent series expansion near $\infty$ of the conformal mating $F_n=R_n\circ\eta\circ(R_n\vert_{\mathfrak{D}_n})^{-1}$ is given by
$$
F_n(z)=\frac{z^{2d-1}}{(2d-1)(R_n'(\infty))^{2d-2}}+O(z^{2d-2}).
$$
Let $\phi_n^*(\zeta) = \phi_n(1/\zeta)$ be the conformal map from $(\D^*, \infty)\to(\mathcal{B}_\infty(F_n),\infty)$ which conjugates $z^{2d-1}$ to $F_n$. 
By our normalization, we have
$$
\phi_n^*(\zeta)=\zeta +O(1/\zeta).
$$
Since $\phi_n^*$ conjugates $z^{2d-1}$ to $F_n$, and {has asymptotics $\zeta +O(1/\zeta)$ as $\zeta\to\infty$}, we have $(2d-1)^{\frac{1}{2d-2}}R_n'(\infty) = 1$.
Therefore, $R_n'(\infty)$ is bounded and $\mathfrak{B}(z^{2d-1})$ is bounded by Lemma \ref{lem:Rinfty}.

{
\begin{example}\label{four_punc_bounded_example}
For $d=3$, the Bers slice $\mathfrak{B}(z^5)$ of $4$-times punctured spheres consists of rational maps $R_c(z)=z+\frac{c}{z}-\frac{c}{3z^3}+\frac{1}{5z^5}$, where $c$ runs over a bounded set in $\C$.    
\end{example}
}

\subsection{The general case}\label{gen_case_subsec}
Let us now assume that $B$ is a general degree $2d-1$ marked Blaschke product with an attracting fixed point at $0$ such that $B\vert_{\overline{\D}}$ is conformally conjugate to $P\vert_{\cK(P)}$, for some $P\in\cH_{2d-1}$. As mentioned earlier, we will prove Theorem~\ref{pre_comp_thm_1} by method of contradiction. To this end, we assume that $\mathfrak{C}_{n}\in\mathfrak{B}(P)$ diverges in $\mathrm{Rat}_{2d}(\C)/\!\!\sim$. To obtain a contradiction, we will look at two rescaling limits of the uniformizing rational maps $R_n$ via the change of coordinates $L_n(z) = \frac{z}{|R'_n(\infty)|}$ and $M_n(z) = |R'_n(\infty)|\cdot z$ (which zoom near $0$ and $\infty$ respectively). One of these rescaling limits will have degree $1$, while the other one will be of degree $2d-1$ with a fully ramified critical point. These and other mapping properties of the two rescaling limits will be used to show that in the limit, the fixed point $R_n(1)=\phi_n(1)\in\partial\mathcal{B}_\infty(F_n)$ (of the conformal mating $F_n$) becomes a superattracting fixed point of a limiting dynamical system. But this would contradict the
fact that the limit set $\partial\mathcal{B}_\infty(F_n)$ is a repeller for $F_n$ in the direction of the basin $\mathcal{B}_\infty(F_n)$.

It is convenient to use $\widehat{\C}_L$ and $\widehat{\C}_M$ to denote the Riemann spheres in the $L_n$ and $M_n$ coordinates. The subscripts allow us to keep track of the coordinates we are using. We shall also sometimes use $x_L$ or $x_M$ to denote the point in $L$ or $M$ coordinates respectively.
\begin{lem}\label{lem:tworescalinglimits}
    Suppose that $\mathfrak{C}_{n}$ diverges in $\mathrm{Rat}_{2d}(\C)/\!\!\sim$. Then after passing to a subsequence,
    \begin{itemize}[leftmargin=4mm]
        \item $R_n \circ M_n: \widehat{\C}_M \to \widehat{\C}$ converges to a degree $1$ map $G_1$ with a unique hole at~$0_M$;
        \item $R_n \circ L_n: \widehat{\C}_L \to \widehat{\C}$ converges to a degree $2d-1$ map $G_2$ with a unique hole at~$\infty_L$.
    \end{itemize}
\end{lem}
\begin{figure}[ht!]
\captionsetup{width=0.98\linewidth}
\begin{tikzpicture}
\node[anchor=south west,inner sep=0] at (0,0) {\includegraphics[width=0.99\linewidth]{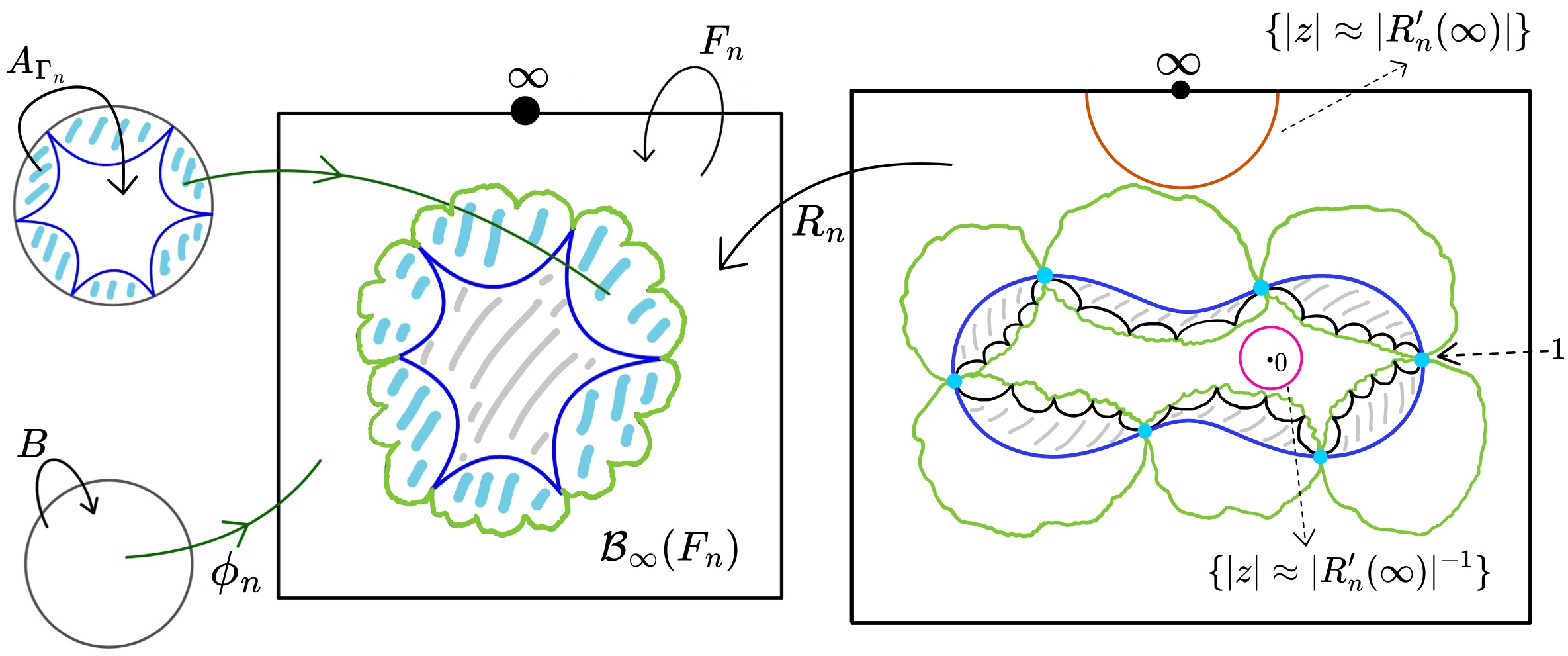}};
\node[anchor=south west,inner sep=0] at (1.5,-5.16) {\includegraphics[width=0.9\linewidth]{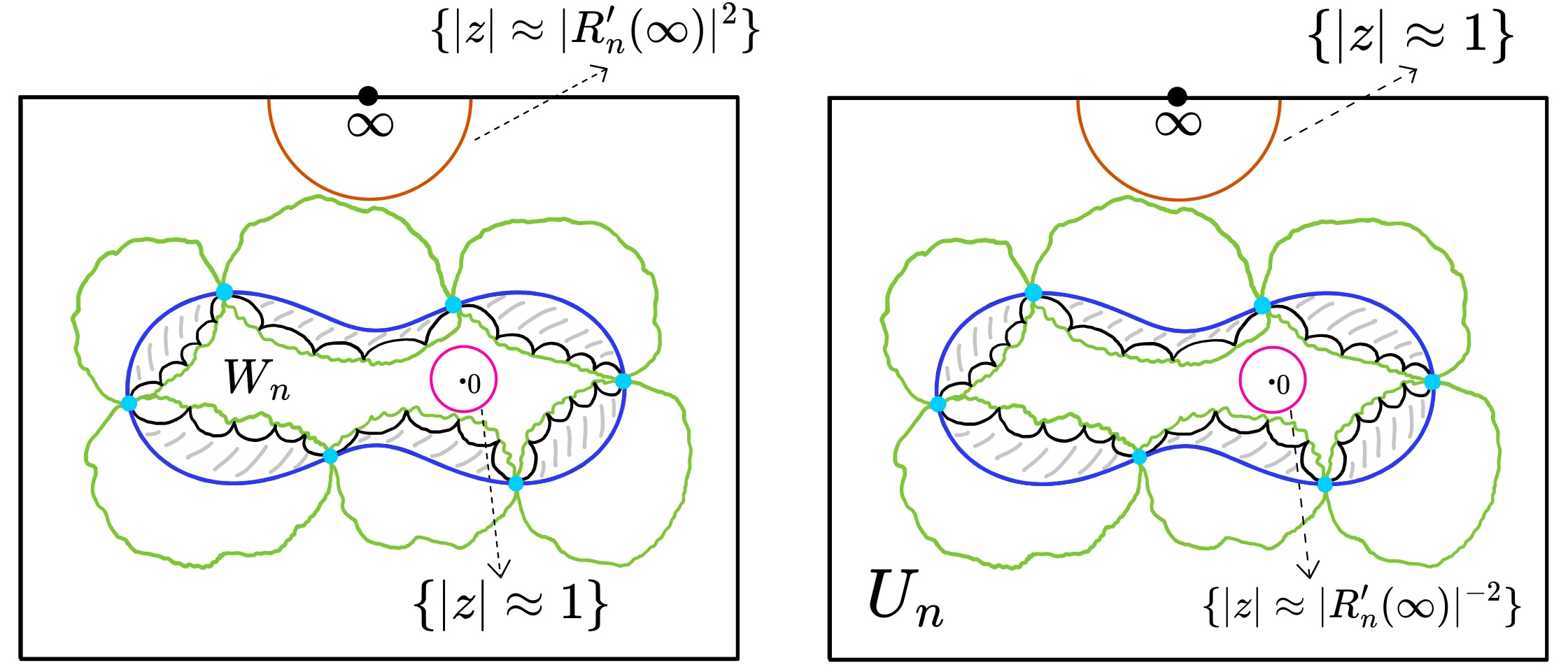}}; 
\node at (11.36,2.36)[circle,fill,inner sep=1.5pt]{};
\draw [->, line width=0.2mm] (8.8,-1.4) to (9.6,1);
\node at (9.56,-0.1) {\begin{small}$M_n$\end{small}};
\draw [->, line width=0.2mm] (6.6,-1.4) to (8.4,1);
\node at (7.96,-0.1) {\begin{small}$L_n$\end{small}};
\end{tikzpicture}
\caption{Illustrated are the two change of coordinates $L_n$ and $M_n$ appearing in Lemma~\ref{lem:tworescalinglimits}. On the top right figure, the exterior of the blue curve is the domain~$\mathfrak{D}_n$.}
\label{two_rescalings_fig}
\end{figure}
\begin{proof}
    After passing to a subsequence, we assume that $R_n \circ M_n$ and $R_n \circ L_n$ converge to some rational maps $G_1$ and $G_2$ of potentially lower degree.
    Consider the univalent map 
    $$
    {\Phi^M_n}: (\D,0) \longrightarrow (\widehat{\C}_M, \infty_M)
    $$ 
    defined as ${\Phi^M_n} = M_n^{-1}\circ R_n^{-1} \circ \phi_n$, where $\phi_n:(\D,0)\to(\mathcal{B}_\infty(F_n),\infty)$ is the {normalized} conformal conjugacy between $B$ and $F_n$, and $R_n^{-1}$ is the inverse of the conformal map $R_n: (\mathfrak{D}_n, \infty) \longrightarrow (\Omega_n, \infty)$.
    Since $M_n'(\infty) = \frac{1}{\vert R'_n(\infty)\vert}$, the spherical derivative of ${\Phi^M_n}$ at $0$ is $1$. {Further, as $0\notin U_n \coloneqq \Phi_n^M(\D)$, it follows that $\Phi_n^M(\zeta)=a_n/\zeta+b_n+O(\zeta)$ near $0$, where $\vert a_n\vert=1$ and $b_n$ is a bounded sequence (cf. \cite[Theorem~1.4]{Pom75}).} Thus $U_n$ converges to some domain $U$, and $G_1: U \longrightarrow \mathcal{B}$ is conformal, where $\mathcal{B}$ is the limit of $\mathcal{B}_\infty(F_n)$ (see Figure~\ref{two_rescalings_fig}). In particular, $G_1$ has degree at least $1$. 

    Let $W_n := L_n^{-1} \circ \eta \circ M_n(U_n) \subseteq \widehat{\C}_L$. Since $L_n^{-1} \circ \eta \circ M_n (z) = \frac{1}{z}$, the domain $W_n$ converges to $W = \eta(U) \subseteq \widehat{\C}_L$ (see Figure~\ref{two_rescalings_fig}). Note that $R_n \circ L_n$ is proper and has degree $2d-1$ on $W_n$.
    By Proposition \ref{prop:ConvergenceProperMap}, $R_n \circ L_n$ converges compactly on $W$ to $G_2$, and $G_2:  W \longrightarrow \mathcal{B}$ is proper and has degree $2d-1$.
    Therefore, $G_2$ has degree at least $2d-1$.
    Since $R_n$ {is a degenerating sequence in $\Rat_{2d}(\C)$}, we conclude that $G_1$ and $G_2$ have degree $1$ and $2d-1$ respectively by Lemma \ref{lem:rl}.
 
    {We will now prove the statements about the holes of $G_1$ and $G_2$.} Since $0 \notin U_n$ for all $n$, we conclude that $0 \notin U$. Thus $\infty = \eta(0) \notin W = \eta(U)$.
    Let $K$ be any compact subset of $\mathcal{B}$. Let $K_i \coloneqq G_i^{-1}(K)$, $i = 1, 2$. Note that $K_2$ is a bounded subset of $\C_L$; {indeed, as $G_2: W \longrightarrow \mathcal{B}$ is a proper map of degree $2d-1$, which is the global degree of $G_2$, we have that $K_2$ is a compact subset of $W$.}
    For all sufficiently large $n$, $K \subseteq \mathcal{B}_{\infty}(F_n)$. Let $K_{1, n} \coloneqq (R_n \circ M_n)^{-1}(K) \cap U_n$ and $K_{2, n} \coloneqq (R_n \circ L_n)^{-1}(K) \cap W_n$.
    Note that for all large $n$,
    {$$
    (R_n \circ M_n)^{-1}(K) = K_{1, n}  \sqcup \widehat{K}_{1,n},
    $$
    where $K_{1, n}$ is mapped univalently and $\widehat{K}_{1,n}:=(M_n^{-1}\circ L_n)(K_{2, n})$ is mapped with degree $2d-1$ onto $K$ under $R_n\circ M_n$.}
    Since $K_2$ is a bounded subset of $\C_L$, and $K_{2,n}$ converges to $K_2$, it follows that $K_{2, n}$ are uniformly bounded for all sufficiently large $n$. Thus, $\widehat{K}_{1,n}$ converges to $0_M$ {(recall that by Lemma~\ref{lem:Rinfty}, $\vert R_n'(\infty)\vert\to+\infty$)}. 
    Now let $a$ be a hole in $\widehat{\C}_M$ for $G_1$. By Lemma \ref{lem:ns}, {for any neighborhood $N$ of $a$, we have $R_n \circ M_n(N) = \widehat{\C}$ for $n$ large enough. In particular, $K\subseteq R_n \circ M_n(N)$ and hence $N\cap \widehat{K}_{1,n}\neq\emptyset$ for all sufficiently large $n$.}
    Thus  $a = 0$.
    A similar argument shows that $\infty_L$ is the unique hole in $\widehat{\C}_L$ for $G_2$.
\end{proof}

\begin{lem}\label{lem:2d-2}
    Let $G_1, G_2$ be the rational maps in Lemma \ref{lem:tworescalinglimits}. Then $G_1(0) = G_2(\infty)$, and $\infty$ is a critical point of multiplicity $2d-2$ for $G_2$.
\end{lem}
\begin{proof}
    Suppose that $G_1(0) \neq G_2(\infty)$. Choose an open set $V \subseteq \widehat{\C} - \{G_1(0), G_2(\infty)\}$.
    Choose small simple closed curves $\widetilde{\alpha}, \widetilde{\beta}$ around $G_1(0)$ and $G_2(\infty)$ respectively that are disjoint from $V$. There exist small simple closed curves $\alpha \subseteq \widehat{\C}_M$ and $\beta \subseteq \widehat{\C}_L$ around $0_M$ and $\infty_L$ that are mapped to $\alpha, \beta$ by $G_1$ and $G_2$ respectively.
    Since $0_M$ and $\infty_L$ are the only holes of $G_1$ and $G_2$ respectively, for all sufficiently large $n$, there are approximations $\alpha_n$ and $\beta_n$ of $\alpha$ and $\beta$ so that
    $$
    R_n\circ M_n(\alpha_n) = \widetilde{\alpha} \text{ and } R_n \circ L_n(\beta_n) = \widetilde{\beta}.
    $$
    Note that $M_n(\alpha_n)$ and $L_n(\beta_n)$ divide the Riemann spheres into $3$ components. Denote these by $\mathcal{V}_{1,n}, \mathcal{A}_n, \mathcal{V}_{2,n}$ where 
    \begin{enumerate}[leftmargin=*]
        \item $\mathcal{V}_{1,n}$ $\mathcal{V}_{2,n}$ are topological disks with boundaries $M_n(\alpha_n)$, $L_n(\beta_n)$ respectively;
        \item $\mathcal{A}_n$ is the annulus bounded by $M_n(\alpha_n)$ and $L_n(\beta_n)$.
    \end{enumerate}
    Since $R_n\circ M_n$ converges to $G_1$ compactly away from $0$, and $G_1$ has degree $1$, we conclude that for sufficiently large $n$, $V$ has one preimage {(under $R_n$)} in $\mathcal{V}_{1,n}$ counted with multiplicity. Similarly, $V$ has $2d-1$ preimages in $\mathcal{V}_{2,n}$ counted with multiplicity.
    On the other hand, the boundary of the annulus $\mathcal{A}_n$ is mapped { by $R_n$} to the union of $\widetilde{\alpha}$ and $\widetilde{\beta}$ under $R_n$. Thus, there exists at least 1 preimage of $V$ in $\mathcal{A}_n$. This is a contradiction, as $R_n$ has degree $2d$. Therefore, $G_1(0) = G_2(\infty)$.
    
    Since $G_2$ has degree $2d-1$, it has $4d-4$ critical points counted with multiplicity.
    Since $G_2: W \longrightarrow \mathcal{B}$ is a proper map of degree $2d-1$, there are $2d-2$ critical points of $G_2$ in $W$ { (see the proof of Lemma~\ref{lem:tworescalinglimits} for the definitions of the sets $W, W_n$). We claim that the only critical point of $G_2$ outside $W$ is at $\infty$, whence the desired multiplicity of this critical point would follow.
    By way of contradiction,} suppose that $c$ is a critical point of $G_2$ in $\C - W$. Since $\infty$ is the only hole of $G_2$, there exists a sequence $c_n$ of critical points of $R_n$ so that $L_n^{-1}(c_n) \to c$. Since $c$ is not contained in $W$, $c_n$ is not in $W_n$ for $n$ sufficiently large. {Thus, $c_n\in\partial\mathfrak{D}_n$, and hence $\eta(c_n)$ is also a critical point of $R_n$ (cf. \cite[Proposition~4.14, Corollary~4.15]{MM2}}. Since $L_n^{-1}(c_n) \to c$, it follows that
    $$
    M_n^{-1}(\eta(c_n)) = M_n^{-1}\circ\eta\circ L_n (L_n^{-1}(c_n)) = \eta(L_n^{-1}(c_n)),
    $$
    {which is a critical point of $R_n\circ M_n$,} converges to $\eta(c) \in \widehat{\C}_M$. Further, $\eta(c) \neq 0$ as $c \neq \infty$. Since the only hole of $G_1$ is $0$, we conclude that $\eta(c)$ is a critical point of $G_1$. This is a contradiction, as $G_1$ has degree $1$.
    This completes the proof.
\end{proof}

\begin{lem}\label{lem:ld}
    {Let $G_1, G_2$ be the rational maps in Lemma \ref{lem:tworescalinglimits}, and $A \coloneqq G_1(0) = G_2(\infty)$.} Then {$R_n(1)\to A$, and} for all sufficiently small $r > 0$, there exists $N$ so that for all $n \geq N$, we have $F_n(\Omega_n \cap \partial B(A, r)) \subseteq B(A, r/2)$.
\end{lem}
\begin{proof}
{By the description of the critical points of $G_2$ given in Lemmas~\ref{lem:tworescalinglimits} and~\ref{lem:2d-2}, all the critical points of $R_n\circ L_n$ in $L_n^{-1}(\partial\mathfrak{D}_n)$ converge to the unique hole $\infty$ of $G_2$. In particular, $L_n^{-1}(1)\to\infty$. By way of contradiction, suppose that the associated critical value $R_n(1)$ does not converge to $A=G_2(\infty)$. After possibly passing to a subsequence, we may assume that $R_n(1)\to A'\neq A$. 
Note that as $R_n(1)\notin\mathcal{B}_\infty(F_n)$, it follows that $A'\notin\mathcal{B}$ (where $(\mathcal{B},\infty)$ is the Carath{\'e}odory limit of the pointed disks $(\mathcal{B}_\infty(F_n),\infty)$). Hence, $A'$ is not a critical value of $G_2$. So, for $\epsilon>0$ small enough, $G_2^{-1}(B(A',\epsilon))$ consists of $2d-1$ disjoint open sets $\mathcal{W}_1,\cdots,\mathcal{W}_{2d-1}$ in $\C$ (each of which is mapped homeomorphically). Since there is no hole of $G_2$ in any $\mathcal{W}_i$, we conclude that for $n$ sufficiently large, there are $2d-1$ distinct preimages of $R_n(1)$ under $R_n\circ L_n$ in $\mathcal{W}_1,\cdots,\mathcal{W}_{2d-1}$.
But for $n$ large, the critical point $L_n^{-1}(1)$ of $R_n\circ L_n$ is close to $\infty$ (i.e., it lies outside $\cup \mathcal{W}_i$) and maps to $R_n(1)$. This gives $2d+1$ preimages of $R_n(1)$ under $R_n\circ L_n$ (counted with multiplicities), which is a contradiction. Hence, $R_n(1)\to A$.

The above discussion and the fact that $R_n(1)\in\partial\mathcal{B}_\infty(F_n)\cap\partial\Omega_n$ together imply that for all small $r>0$, the ball $B(A,r)$ intersects $\Omega_n$ as well as $\mathcal{B}_\infty(F_n)$ non-trivially.}

    Let $\alpha_r = \partial B(A, r)$. By Lemma \ref{lem:2d-2}, {the map $\widehat{G}:=G_2 \circ \eta \circ G_1^{-1}$ has a superattracting fixed point of local degree $2d-1$ at $A$; i.e., for $z$ close to $A$, we have
    $$
    \widehat{G} (z) = A + C(z-A)^{2d-1} + O(|z-A|^{2d}),
    $$    
    for some non-zero constant $C$.}
    Thus for sufficiently small $r$, we have $\widehat{G} (\alpha_r) \subseteq B(A, r/2)$.
    Note also that 
    $$
    \widehat{G} = G_2 \circ L_n^{-1} \circ \eta \circ M_n \circ G_1^{-1}.
    $$
    Since $R_n \circ M_n$ and $R_n \circ L_n$ converge compactly away from holes to $G_1$ and $G_2$ respectively, we conclude that $R_n \circ \eta \circ R_n^{-1}$ converges compactly to $\widehat{G}$ {on $\Omega_n\cap\overline{B(A,r)}$ away from $A$}.
    Thus, for sufficiently small $r>0$, if $n$ is large enough, we have $R_n \circ \eta \circ R_n^{-1}(\Omega_n\cap\alpha_r) \subseteq B(A, r/2)$. Since $F_n\equiv R_n\circ\eta\circ R_n^{-1}$, it follows that $F_n(\Omega_n \cap \partial B(A, r)) \subseteq B(A, r/2)$ for all large $n$.
\end{proof}

\begin{proof}[Proof of Theorem \ref{pre_comp_thm_1}]
Recall that under our normalization, $1 \in \partial \D$ is a repelling fixed point of the Blaschke product $B$ and $\phi_n(1) = R_n(1)$, where $1$ is a critical point of $R_n$.
Let $\gamma \subseteq \D$ be an invariant ray of $B$ landing at the repelling fixed point $1$. More precisely, this means that $\gamma$ is an arc with an end point at $1$, and $\gamma \subseteq \widetilde{\gamma} \coloneqq B(\gamma)$.
Denote $\partial \gamma = \{1, a\}$.
Let $\gamma_n \coloneqq \phi_n(\gamma)$, and let $\widetilde{\gamma}_n  \coloneqq \phi_n(\widetilde{\gamma})$
Let $r> 0$ be a sufficiently small radius as in Lemma \ref{lem:ld} and $\alpha_r = \partial B(A, r)$. We choose $r$ small enough so that $\phi_n(a)$ and $A$ are in different components of $\widehat{\C} - \alpha_r$. Then, {by Lemma~\ref{lem:ld}}, $\gamma_n$ must intersect $\alpha_r$ for all large enough $n$.
Let $\delta_n$ be the component of $\gamma_n - \alpha_r$ that contains $\phi_n(a)$.
Similarly, let $\widetilde{\delta}_n$ be the component of $\widetilde{\gamma}_n - \alpha_r$ that contains $\phi_n(B(a))$.
Since $\phi_n$ is a conjugacy and $\gamma$ is invariant under $B$, we conclude that $F_n(\delta_n) \subseteq \widetilde{\delta}_n$ for all sufficiently large $n$.
On the other hand, by Lemma \ref{lem:ld}, $F_n(\delta_n)$ intersects $B(A, r/2)$. This is a contradiction, and the {first part of the} theorem follows.

{It remains to show that $\overline{\mathfrak{B}(P)}\subset\Rat_{2d}^{reg}(\C)/\!\!\sim$. Each $R\in\mathfrak{B}(P)$ has two simple critical points at the fixed points $\pm 1$ of $\eta$. Hence, $\mathfrak{B}(P)\subset\Rat_{2d}^{reg}(\C)/\!\!\sim$. 

Now suppose that a sequence $R_n\in\mathfrak{B}(P)$ converges to $R\in\partial\mathfrak{B}(P)$. We will use the notation $U, W=\eta(U), \mathcal{B}$ introduced in Lemma~\ref{lem:Rinfty}. By the proof of that lemma, $R:U\to\mathcal{B}$ is a conformal isomorphism, $R:W\to\mathcal{B}$ is a degree $2d-1$ branched covering, and $\mathfrak{b}:=(R\vert_U)^{-1}\circ R\vert_W\circ \eta:U\to U$ is conformally conjugate to $P\vert_{\Int{\cK(P)}}$. 
In particular, $R$ is injective on $B(\pm 1,\epsilon)\cap U$ and $B(\pm 1,\epsilon)\cap W$ for $\epsilon>0$ small, and hence $R$ has simple critical points at $\pm 1$. Therefore, $R\in\Rat_{2d}^{deg}(\C)/\!\!\sim$.}
\end{proof}

\begin{remark}\label{rmk:orbifoldModification}
    We remark that the proof in the genus 0 orbifold case is nearly identical to the manifold case. In the general orbifold setting, additional critical points may appear in the tiling set (cf. \cite[Proposition~4.14]{MM2} or \cite[Proposition~15.6]{LLM24}), requiring a slight modification to the proof of Lemma~\ref{lem:2d-2}. An argument similar to that of Lemma~\ref{lem:2d-2}, combined with a simple counting argument, shows that all cusp critical points, along with at least one critical point from the tiling set, converge to $\infty_L \in \widehat{\C}_L$. This shows that $\infty_L$ is a critical point for $G_2$ (cf.\ Lemma \ref{lem:tworescalinglimits}). This fact allows one to establish Lemma~\ref{lem:ld} and complete the proof of the analog of Theorem~\ref{pre_comp_thm_1} in the orbifold setting.
\end{remark}

\section{Bers boundaries}\label{bers_bdry_sec}

{Recall from \S~\ref{qf_bers_slice_subsec} that for $P\in\cH_{2d-1}$ and $\Gamma\in\mathrm{Teich}(S_{0,d+1})$, there exists a degree $2d$ rational map $R_{\Gamma,P}:\widehat{\C}\to\widehat{\C}$ such that the associated correspondence $\mathfrak{C}_{\Gamma,P}$ (defined by Equation~\eqref{corr_eqn}, where $(\mathscr{T},\widehat{\C}^{\mathscr{V}})$ is a single Riemann sphere) is a mating of $\Gamma$ and $P$. Further, for a fixed polynomial $P$ in $\cH_{2d-1}$, the Bers slice $\mathfrak{B}(P)$ is defined as:
$$
\mathfrak{B}(P)= \{[R_{\Gamma,P}]: \Gamma\in\mathrm{Teich}(S_{0,d+1})\}.
$$
By Theorem~\ref{pre_comp_thm_1}, the closure $\overline{\mathfrak{B}(P)}\subset \mathrm{Rat}_{2d}^{reg}(\C)/\!\!\sim$ is compact. 

We denote the boundary of $\mathfrak{B}(P)$  by $\partial\mathfrak{B}(P)$. For rational maps $R\in\partial\mathfrak{B}(P)$, the associated bi-degree $(2d-1,2d-1)$ algebraic correspondences 
$$
\mathfrak{C} = \mathfrak{C}_{R}:= \{(x,y) \in \widehat\C \times \widehat\C: \frac{R(x) - R(\eta(y))}{x-\eta(y)} = 0\}
$$
are called \emph{Bers boundary correspondences}. Abusing notation, we will identify the equivalence class of $R$ in $\mathrm{Rat}_{2d}^{reg}(\C)/\!\!\sim$ with the correspondence $\mathfrak{C}_R$.}

The goal of this section is to study the relationship between various Bers compactifications $\overline{\mathfrak{B}(P)}$, for $P\in\cH_{2d-1}$. This will involve a surgery construction that replaces the dynamics of the correspondences on their $P$-components, and an analysis of the boundary behavior of the homeomorphism $\chi$ between a pair of Bers slices (constructed in Theorem~\ref{bers_slice_hoemo_thm}.

In light of the discontinuity phenomenon of \cite{KT90}, we conjecture the following:

\begin{conj}
Let $P_1, P_2\in\cH_{2d-1}$. Then, the homeomorphism $\chi:\mathfrak{B}(P_1)\to\mathfrak{B}(P_2)$ does not extend continuously to the boundary.   
\end{conj}

However, we will show in subsection~\ref{four_punc_sphere_subsec} that the situation is very special when the Teichm{\"u}ller space has $\C$-dimension one (for $4$-times punctured sphere groups). We refer the reader to \cite{Ber81} for the corresponding phenomenon in the Kleinian groups world, and note that our proof is essentially different from that of \cite{Ber81}. Indeed, the proof given there uses Sullivan's theorem on non-existence of invariant line fields on  limit sets of Kleinian groups, while such a result is not available for the correspondences that we are dealing with in this paper.

\subsection{The class $\cG_P$ of correspondences}\label{corr_class_subsec}

It follows from the proof of Lemma~\ref{lem:Rinfty} that all correspondences in $\overline{\mathfrak{B}(P)}$ have the following common feature. If the degree $2d$ rational map $R$ defines $\mathfrak{C}\in\overline{\mathfrak{B}(P)}$ via Formula~\eqref{corr_eqn}, then there exist two topological disks $\cU^+, \cU^-$ such that 
\begin{enumerate}[leftmargin=*]
\item $\mathcal{U}^-=\eta(\mathcal{U}^+)$,
\item $R(\cU^+)=R(\cU^-)$, {(See Figure~\ref{corr_fig}, where {$\cU^\pm=\Int{\widetilde{\cK}^\pm}$} are the two connected components of the interior of $\widetilde{\mathcal K}$ in the left figure {with $\cU^+$ being the bounded component}, and $R(\cU^+)=R(\cU^-)$ is the interior of 
${\mathcal K} (F)$ in the right figure.)}
\item $R:\cU^+\to R(\cU^+)$ is a homeomorphism,
\item $R:\cU^-\to R(\cU^-)$ is a branched covering of degree $2d-1$, and
\item $\cU:=\cU^+\cup\cU^-$ is completely invariant under the correspondence.
\end{enumerate}
It follows from the above that $\mathcal{U}^+$ is backward invariant and $\cU^-$ is forward invariant under $\mathfrak{C}$. Further, there exists a forward branch
$$
\mathfrak{b}:=\left(R\vert_{\cU^+}\right)^{-1}\circ R\vert_{\cU^-}\circ\eta: \cU^+\to\cU^+
$$
(respectively, a backward branch $\eta\circ\mathfrak{b}\circ\eta=\eta\circ\left(R\vert_{\cU^+}\right)^{-1}\circ R\vert_{\cU^-}: \cU^-\to\cU^-$)
of $\mathfrak{C}$ on $\cU^+$ (respectively, on $\cU^-$) that is conformally conjugate to $P\vert_{\Int{\cK(P)}}$. 
We will refer to $\cU^\pm$ as the \emph{$P$-components}, the forward branch $\mathfrak{b}$ as the \emph{distinguished branch}, and the complement $\widehat{\C}-\cU$ of the $P$-components as the \emph{filled limit set} of the correspondence $\mathfrak{C}$. Note that the filled limit set is also completely invariant under the correspondence.
We also note that $R$ has a simple critical point at $1$, and it is a marked boundary fixed point of the distinguished branch $\mathfrak{b}$ preserving $\cU^+$. Under the conformal conjugacy between $\mathfrak{b}$ and $P$, this boundary fixed point corresponds to the landing point of the dynamical $0$-ray of $P$ on {its} Julia set.

We denote the collection of all correspondences defined by degree $2d$ rational maps (via Formula~\eqref{corr_eqn}) satisfying the properties listed in the previous paragraph by $\mathcal{G}$. For $P\in\cH_{2d-1}$, the subspace $\cG_P$ of $\cG$ consists of those correspondences whose distinguished branch $\mathfrak{b}:\cU^+\to\cU^+$ is conformally conjugate to $P\vert_{\Int{\cK(P)}}$ such that the conformal conjugacy sends the marked boundary fixed point of $\mathfrak{b}$ to the landing point of the dynamical $0$-ray of $P$. Clearly, $\overline{\mathfrak{B}(P)}\subset\cG_P$.

Each correspondence $\mathfrak{C}\in\mathcal{G}$ is \emph{reversible} in the sense that the conformal involution $\eta$ conjugates the forward branches of $\mathfrak{C}$ to its backward branches.
{By Lemma~\ref{qc_def_corr_lem}}, the family $\mathcal{G}$ of bi-degree $(2d-1,2d-1)$ algebraic  correspondences is closed under quasiconformal deformations.

\subsection{Some technical results}\label{chi_tools_subsec}

We start with a useful definition that has its roots in the theory of polynomial-like maps in holomorphic dynamics.
 
\begin{defn}\label{hybrid_conj_def}
We say that two correspondences $\mathfrak{C}_1,\mathfrak{C}_2\in\cG$ are \emph{hybrid conjugate} if they are quasiconformally conjugate in some neighborhoods of their filled limit sets where the conjugacy is conformal on the filled limit sets and sends the marked fixed point on $\partial\cU_1^+$ to that on $\partial\cU_2^+$.     
\end{defn}

\begin{lem}\label{conjugacy_extension_lem}
Let $P\in\cH_{2d-1}$, and $\mathfrak{C}_1,\mathfrak{C}_2\in\mathcal{G}_P$ be quasiconformally conjugate in some neighborhoods of their filled limit sets where the conjugacy sends the marked fixed point on $\partial\cU_1^+$ to that on $\partial\cU_2^+$.
Then, $\mathfrak{C}_1$ and $\mathfrak{C}_2$ are quasiconformally conjugate on $\widehat{\C}$, and the conjugacy can be chosen to be conformal on their $P$-components.
\end{lem}
\begin{proof}
Suppose that $\mathfrak{C}_j$ is defined by the rational map $R_j$, $j\in\{1,2\}$.
Let $\phi$ be a quasiconformal homeomorphism from a neighborhood of the filled limit set of $\mathfrak{C}_1$ to a neighborhood of the filled limit set of $\mathfrak{C}_2$, that conjugates the restrictions of the correspondences there, and preserves the marking on $\partial\cU_j^+$. By our hypothesis, the distinguished branches
$$
\mathfrak{b}_j:=\left(R\vert_{\cU^+_j}\right)^{-1}\circ R\vert_{\cU_j^-}\circ\eta:\cU_j^+\to\cU_j^+
$$ 
of $\mathfrak{C}_j$ are conformally conjugate to $P\vert_{\Int{\cK(P)}}$ via conformal maps $\zeta_j:\cU_j^+\to\Int{\cK(P)}$, $j\in\{1,2\}$. Recall that these conjugacies send the marked fixed point on the boundary of $\partial \cU_j^+$ to the marked repelling fixed point on the Julia set of $P$ (this is the landing point of the dynamical $0$-ray). Thus, the conformal map $\zeta:=\zeta_2^{-1}\circ\zeta_1:\cU_1^+\to\cU_2^+$ conjugates the branches $\mathfrak{b}_j$ of the correspondences $\mathfrak{C}_j$, and preserves the boundary marking.
As $\zeta$ and $\psi$ act the same way on the $\mathfrak{b}_1$-invariant prime ends of $\partial\cU_1^+$, the argument of
\cite[\S 1, Lemma~1]{DH85} implies that $\phi$ and $\zeta$ match continuously along $\partial\cU_1^+$.
\begin{figure}[ht]
\captionsetup{width=0.98\linewidth}
\begin{tikzpicture}
\node[anchor=south west,inner sep=0] at (0,0) {\includegraphics[width=1\linewidth]{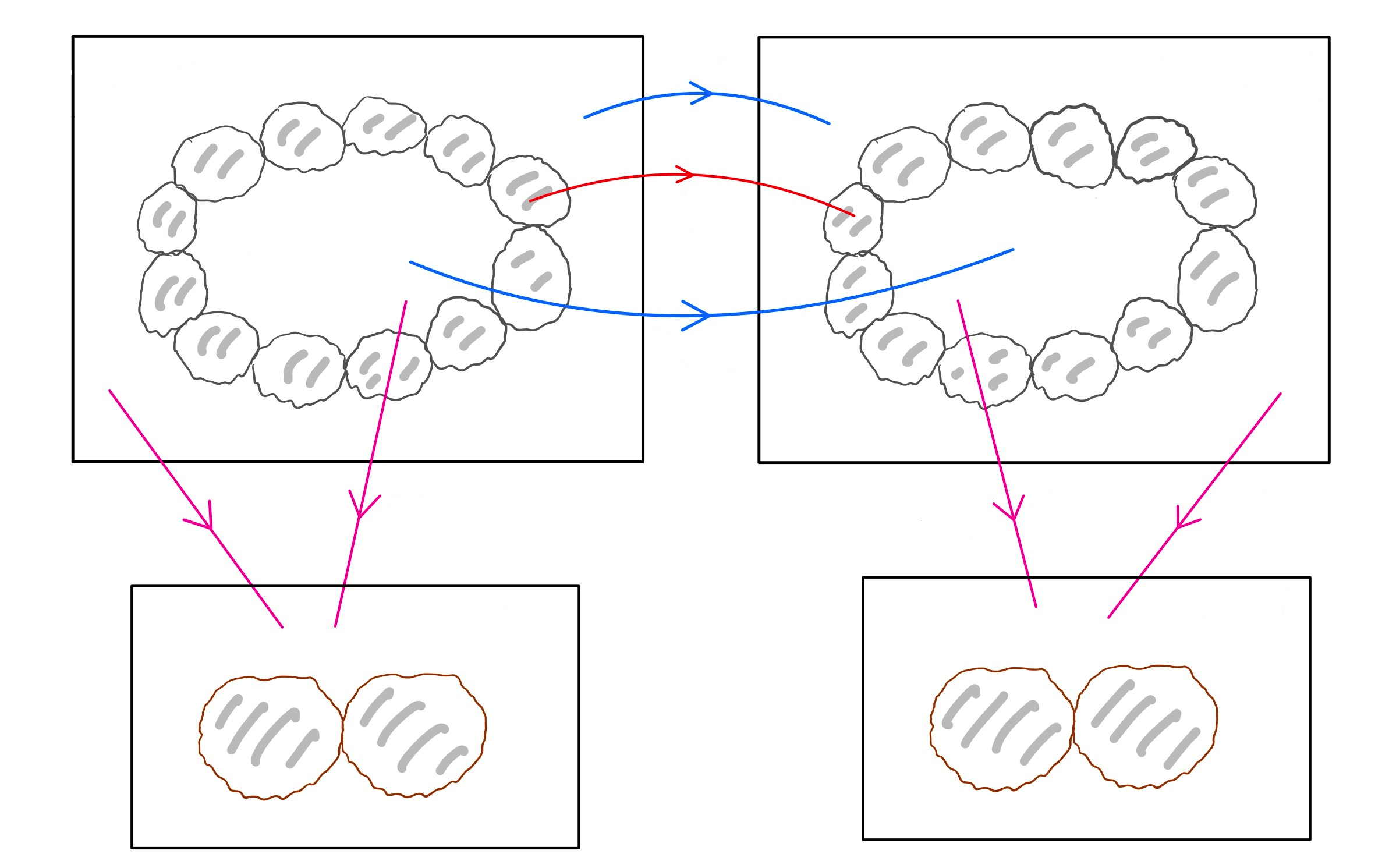}};
\node at (0.36,7.2) {$\mathfrak{C}_1$};
\node at (2,7.32) {\begin{scriptsize}$\cU^+_1=\eta(\cU_1^-)$\end{scriptsize}};
\node at (2.7,5.6) {\begin{small}$\cU_1^-$\end{small}};
\node at (1.4,3.2) {\begin{scriptsize}$1$:$1$\end{scriptsize}};
\node at (2.2,3.2) {\begin{scriptsize}$R_1$\end{scriptsize}};
\node at (3,3.2) {\begin{scriptsize}$R_1$\end{scriptsize}};
\node at (4.1,3.2) {\begin{scriptsize}$(2d-1)$:$1$\end{scriptsize}};
\node at (4.4,2.36) {\begin{tiny}$R_1(\cU_1^+)$\end{tiny}};
\node at (4.32,2.05) {\begin{tiny}$=R_1(\cU_1^-)$\end{tiny}};
\node at (6.28,7.4) {\begin{small}$\zeta$\end{small}};
\node at (6.28,6) {\begin{small}$\phi$\end{small}};
\node at (6.38,4.8) {\begin{tiny}$\eta\circ\zeta\circ\eta$\end{tiny}};
\node at (12.4,7.2) {$\mathfrak{C}_2$};
\node at (11,7.32) {\begin{scriptsize}$\cU^+_2=\eta(\cU_2^-)$\end{scriptsize}};
\node at (10,5.6) {\begin{small}$\cU_2^-$\end{small}};
\node at (11.2,3.2) {\begin{scriptsize}$1$:$1$\end{scriptsize}};
\node at (10.48,3.2) {\begin{scriptsize}$R_2$\end{scriptsize}};
\node at (9.56,3.2) {\begin{scriptsize}$R_2$\end{scriptsize}};
\node at (8.4,3.2) {\begin{scriptsize}$(2d-1)$:$1$\end{scriptsize}};
\node at (11.2,2.42) {\begin{tiny}$R_2(\cU_2^+)$\end{tiny}};
\node at (11.1,2.1) {\begin{tiny}$=R_2(\cU_2^-)$\end{tiny}};
\end{tikzpicture}
\caption{A schematic diagram of the construction of a global quasiconformal conjugacy between $\mathfrak{C}_1$ and $\mathfrak{C}_2$. The shaded regions are the filled limit sets of the correspondences.}
\label{conj_ext_fig}
\end{figure}

As the forward branches of $\mathfrak{C}_j$ are conjugate to their backward branches via $\eta$, we can apply the above argument to the inverse branches of $\mathfrak{C}_j$ on $\cU_j^-$ (which are also conformally conjugate to $P$) and conclude that $\phi$ and $\eta\circ\zeta\circ\eta$ match continuously along $\partial\cU_1^-$.

By Rickman's lemma (cf. \cite[\S 1, Lemma~2]{DH85}), the map 
\begin{align*}
\widecheck{\phi}:\widehat{\C}\to\widehat{\C}, \hspace{1cm}
\widecheck{\phi}=
\begin{cases}
\zeta\hspace{1.5cm} \mathrm{on}\quad \cU_1^+,\\
\eta\circ\zeta\circ\eta\quad \mathrm{on}\quad \cU_1^-,\\
\phi\hspace{1.2cm} \mathrm{elsewhere},
\end{cases}
\end{align*}
is a quasiconformal homeomorphism (see Figure~\ref{conj_ext_fig}). By design, $\widecheck{\psi}$ conjugates $\mathfrak{C}_1$ to $\mathfrak{C}_2$ on the filled limit sets. It remains to show that the same is true on the $P$-components.

By construction, $\widecheck{\phi}$ conjugates the branches $\mathfrak{b}_j:\cU_j^+\to\cU_j^+$, $j\in\{1,2\}$. It also conjugates $\eta\vert_{\cU_1^+\cup\cU_1^-}$ to $\eta\vert_{\cU_2^+\cup\cU_2^-}$. To show that the $\widecheck{\phi}$ conjugates the other forward branches of $\mathfrak{C}_j$ on $\cU_j^+$, it suffices to show that it conjugates the local deck transformations of $R_1:\cU_1^-\to R(\cU_1^-)$ to those of $R_2:\cU_2^-\to R(\cU_2^-)$. But this is a straightforward consequence of the facts that 
\begin{itemize}
\item $\widecheck{\phi}$ conjugates $\left(R_1\vert_{\cU_1^+}\right)^{-1}\circ R_1\vert_{\cU_1^-}$ to $\left(R_2\vert_{\cU_2^+}\right)^{-1}\circ R_2\vert_{\cU_2^-}$, and
\item the local deck transformations of $R_j\vert_{\cU_j^-}$ are also the local deck transformations of $\left(R_j\vert_{\cU_j^+}\right)^{-1}\circ R_j\vert_{\cU_j^-}$, for $j\in\{1,2\}$.
\end{itemize}

We next observe that $\cU_j^+$ is backward invariant under $\mathfrak{C}_j$, and the backward branches are given by the local inverse branches of $\mathfrak{b}_j$, $j\in\{1,2\}$. Since $\widecheck{\psi}$ conjugates $\mathfrak{b}_1$ to $\mathfrak{b}_2$, it also conjugates the local inverse branches $\mathfrak{b}_1$ to those of $\mathfrak{b}_2$. Therefore, $\widecheck{\phi}$ conjugates the backward branches of $\mathfrak{C}_1$ on $\cU_1^+$ to those of $\mathfrak{C}_2$ on $\cU_2^+$. Once again, the $\eta$-symmetry between the forward and backward branches of $\mathfrak{C}_j$ implies that $\widecheck{\phi}$ conjugates the forward branches of $\mathfrak{C}_1$ on $\cU_1^-$ to those of $\mathfrak{C}_2$ on $\cU_2^-$. Thus, $\widecheck{\phi}$ conjugates all the forward branches of $\mathfrak{C}_1$ to those of $\mathfrak{C}_2$.
\end{proof}

\begin{lem}\label{surgery_replace_blaschke_lem}
Let $P_1,P_2\in\cH_{2d-1}$, and $\mathfrak{C}\in\cG_{P_1}$. Then, there exists a unique correspondence $\mathfrak{C}'\in\cG_{P_2}$ that is hybrid conjugate to $\mathfrak{C}$. 
\end{lem}
\begin{remark}
The main idea of the proof of this lemma is to replace the conformal dynamics of $P_1$ with that of $P_2$ in the dynamical plane of $\mathfrak{C}$, and is reminiscent of the {proof of Lemma~\ref{qc_surgery_lem}}. However, since we do not have a conformal mating plane (that is covered by the correspondence plane) at our disposal in the current setting, the construction is more intricate.  
\end{remark}
\begin{proof}
\noindent\textbf{Uniqueness.} Suppose that $\mathfrak{C}'$ and $\mathfrak{C}''$ are two correspondences in $\cG_{P_2}$, which are hybrid conjugate to $\mathfrak{C}$. Then, they are hybrid conjugate to each other. Using Lemma~\ref{conjugacy_extension_lem}, we can extend this hybrid conjugacy to a global conformal (i.e., M{\"o}bius) conjugacy between $\mathfrak{C}'$ and $\mathfrak{C}''$.
\medskip

\noindent\textbf{Existence.} 
Recall that for $j\in\{1,2\}$, there exist Blaschke products $B_j$, each having an attracting fixed point in $\D$ and a marked repelling fixed point at $1$, such that $P_j\vert_{\cK(P_j)}$ is conformally conjugate to $B_j\vert_{\overline{\D}}$. Further, such a conjugacy carries the marked repelling fixed point on the Julia set of $P_j$ to the marked fixed point of $B_j$. By Lemma~\ref{blaschke_conjugacy_lem}, there exists a quasiconformal homeomorphism $\mathfrak{h}:(\overline{\D},1)\to(\overline{\D},1)$ that conjugates $B_1:B_1^{-1}(N_1)\to N_1$ to $B_2:B_2^{-1}(N_2)\to N_2$, where $N_1, N_2$ are relative neighborhoods of $\mathbb{S}^1$ in $\overline{\D}$.  

Let $R$ be the rational map defining $\mathfrak{C}$ via Equation~\eqref{corr_eqn}. Define 
$$
\mathscr{U}:=R(\cU^+)=R(\cU^-),\quad \mathrm{and}\quad F:=R\vert_{\cU^-}\circ\eta\circ\left(R\vert_{\cU^+}\right)^{-1}:\mathscr{U}\to\mathscr{U}.
$$
The marked boundary fixed point of $\mathfrak{b}$ on $\cU^+$ determines a marked boundary fixed point on $F$ on $\mathscr{U}$.
Evidently, $R\vert_{\cU^+}$ is a conformal conjugacy between the distinguished branch $\mathfrak{b}$ of $\mathfrak{C}$ and the map $F$. Thus, we have a conformal map $\mathfrak{g}:\D\to\mathscr{U}$ that conjugates $B_1$ to $F$, and respects the boundary markings. 

Set $\mathfrak{f}:=\mathfrak{h}\circ\mathfrak{g}^{-1}:\mathscr{U}\to\D$. We next define an essentially bounded almost complex structure $\mu$ on $\widehat{\C}$, where $\mu$ is given by $\mathfrak{f}^*(\mu_0\vert_{\D})$ on $\mathscr{U}$ and the standard complex structure elsewhere. Then, there exists a quasiconformal homeomorphism $\Psi$ of $\widehat{\C}$ that straightens $\mu$ to the standard complex structure $\mu_0$. Observe that as $B_2$ is holomorphic, $\mu\vert_{\mathscr{U}}$ is $\mathfrak{f}^{-1}\circ B_2\circ\mathfrak{f}$-invariant. Hence, 
$$
\widetilde{F}:=\Psi\circ\left(\mathfrak{f}^{-1}\circ B_2\circ\mathfrak{f}\right)\circ\Psi^{-1}:\Psi(\mathscr{U})\to\Psi(\mathscr{U})
$$ 
is holomorphic and conjugate to $B_2\vert_{\D}$ via $\mathfrak{f}\circ\Psi^{-1}$ (see Figure~\ref{blaschke_surgery_fig}). We note that the equivariance properties of $\mathfrak{g}$ and $\mathfrak{h}$ imply that $\widetilde{F}$ coincides with $\Psi\circ F\circ\Psi^{-1}$ near the boundary of $\Psi(\mathscr{U})$.
\begin{figure}[ht]
\captionsetup{width=0.98\linewidth}
\begin{tikzpicture}
\node[anchor=south west,inner sep=0] at (0,0) {\includegraphics[width=1\linewidth]{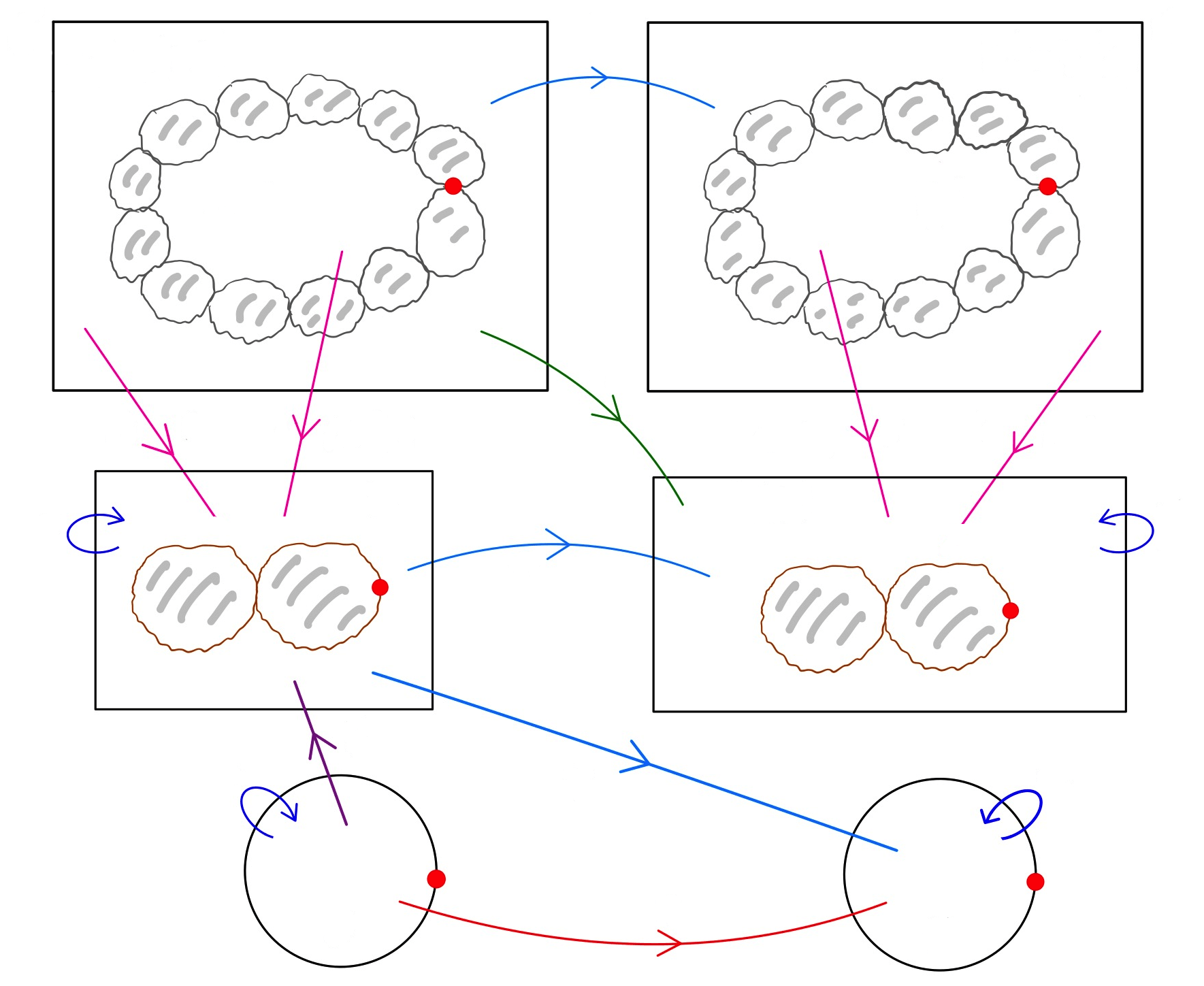}};
\node at (0.28,9.2) {$\mathfrak{C}$};
\node at (1.8,9.8) {\begin{small}$\cU^+=\eta(\cU^-)$\end{small}};
\node at (2.8,8.2) {$\cU^-$};
\node at (1.24,5.84) {\begin{small}$1$:$1$\end{small}};
\node at (2,5.84) {\begin{small}$R$\end{small}};
\node at (2.8,5.84) {\begin{small}$R$\end{small}};
\node at (4.2,5.84) {\begin{small}$(2d-1)$:$1$\end{small}};
\node at (4,5) {$\mathscr{U}$};
\node at (6.28,9.9) {$\widecheck{\Psi}$};
\node at (6.28,6.6) {$\widecheck{R}$};
\node at (5.54,4.36) {$\Psi$};
\node at (12.36,9.2) {$\mathfrak{C}_2$};
\node at (11,9.8) {\begin{small}$\widetilde{\cU^+}=\eta(\widetilde{\cU^-})$\end{small}};
\node at (9.9,8.2) {$\widetilde{\cU^-}$};
\node at (11.2,5.84) {\begin{small}$1$:$1$\end{small}};
\node at (10.4,5.84) {\begin{small}$\widetilde{R}$\end{small}};
\node at (9.42,5.84) {\begin{small}$\widetilde{R}$\end{small}};
\node at (8.12,5.84) {\begin{small}$(2d-1)$:$1$\end{small}};
\node at (8.24,5) {$\Psi(\mathscr{U})$};
\node at (0.5,4.75) {$F$};
\node at (12.28,4.8) {$\widetilde{F}$};
\node at (3.6,0.8) {$\D$};
\node at (10,0.8) {$\D$};
\node at (2.32,2.16) {\begin{small}$B_1$\end{small}};
\node at (11.12,2.16) {\begin{small}$B_2$\end{small}};
\node at (3.7,2.56) {$\mathfrak{g}$};
\node at (6,2.32) {$\mathfrak{f}$};
\node at (6.6,0.72) {$\mathfrak{h}$};
\end{tikzpicture}
\caption{Illustrated is the proof of Lemma~\ref{surgery_replace_blaschke_lem}.}
\label{blaschke_surgery_fig}
\end{figure}

Now define a quasiregular map
\begin{align*}
\widecheck{R}:\widehat{\C}\to\widehat{\C}, \hspace{1cm}
\widecheck{R}=
\begin{cases}
\widetilde{F}\circ\Psi\circ R\circ\eta \quad \mathrm{on}\quad \cU^-,\\
\Psi\circ R\hspace{1.5cm}  \mathrm{elsewhere}.
\end{cases}
\end{align*}
The continuous matching of the piecewise definition of $\widecheck{R}$ follows from the facts that $\widetilde{F}\equiv \Psi\circ F\circ\Psi^{-1}$ near the boundary of $\Psi(\mathscr{U})$ and $F\circ R\circ\eta=R$ on $\cU^-$.
By construction, 
$$
\widetilde{F}\vert_{\Psi(\mathscr{U})}\equiv \widecheck{R}\circ\eta\circ\left(\widecheck{R}\vert_{\cU^+}\right)^{-1}.
$$
Set $\widecheck{\mu}:=\widecheck{R}^*(\mu_0)$ on $\widehat{\C}$. Since $\Psi$ is conformal outside $\mathscr{U}$, it follows that $\widecheck{\mu}$ is the standard complex structure on the filled limit set of $\mathfrak{C}$.
The holomorphicity of $\widetilde{F}$ translates to $\eta$-invariance of $\widecheck{\mu}$ on $\cU^+\cup\cU^-$, and hence $\widecheck{\mu}$ is $\eta$-invariant on all of $\widehat{\C}$.
If $\widecheck{\Psi}$ is a quasiconformal homeomorphism of the sphere straightening $\widecheck{\mu}$ to the standard complex structure, then, 
$$
\widetilde{R}:\widehat{\C}\to\widehat{\C},\quad \widetilde{R}:=\widecheck{R}\circ\widecheck{\Psi}^{-1}
$$ 
is a degree $2d$ rational map. Further, possibly after post-composing $\widetilde{\Psi}$ with a M{\"o}bius map, we can assume that the conformal involution $\widecheck{\Psi}\circ\eta\circ\widecheck{\Psi}^{-1}$ equals $\eta$.

Let $\mathfrak{C}'$ be the correspondence defined by $\widetilde{R}$ via Formula~\eqref{corr_eqn}.
It is easily seen that $\widetilde{\cU^+}:=\widecheck{\Psi}(\cU^+)$ and $\widetilde{\cU^-}:=\widecheck{\Psi}(\cU^-)$ are $\eta$-symmetric topological disks that are mapped with degree $1$ and $2d-1$ onto the same image under $\widetilde{R}$. Further, $\widetilde{\cU^+}\cup\widetilde{\cU^-}$ is completely invariant under $\mathfrak{C}'$, and the distinguished branch $\left(\widetilde{R}\vert_{\widetilde{\cU^+}}\right)^{-1}\circ \widetilde{R}\vert_{\widetilde{\cU^-}}\circ\eta:\widetilde{\cU^+}\to\widetilde{\cU^+}$ of $\mathfrak{C}'$ is conjugate via the conformal map $\mathfrak{f}\circ R\circ\widecheck{\Psi}^{-1}$ to the Blaschke product $B_2$ (in a marking-preserving way). It now follows that $\mathfrak{C}'\in\cG_{P_2}$. Its filled limit set is given by $\widehat{\C}-\left(\widetilde{\cU^+}\cup\widetilde{\cU^-}\right)$.

Moreover, the rational map $\widetilde{R}$ agrees with $\Psi\circ R\circ\widecheck{\Psi}^{-1}$ in a neighborhood of the filled limit set of $\mathfrak{C}'$. Hence, $\mathfrak{C}$ and $\mathfrak{C}'$ are quasiconformally conjugate (via $\widecheck{\Psi}$) in some neighborhoods of their filled limit sets. The marking is preserved by the definition of $\mathfrak{C}'$. Finally, as $\widecheck{\mu}$ is trivial on the filled limit set of $\mathfrak{C}$, we have that the conjugacy $\widecheck{\Psi}$ is conformal on the filled limit sets. In summary, the correspondences $\mathfrak{C}\in\cG_{P_1}$ and $\mathfrak{C}'\in\cG_{P_2}$ are hybrid conjugate.
\end{proof}

\begin{lem}\label{limit_of_str_corr_lem}
Let $P_1,P_2\in\cH_{2d-1}$, and $\chi:\mathfrak{B}(P_1)\longrightarrow\mathfrak{B}(P_2)$ be the homeomorphism of Theorem~\ref{bers_slice_hoemo_thm}.
Let $\mathfrak{C}_\infty\in\partial\mathfrak{B}(P_1)$. Then, all possible subsequential limits of $\chi$-images of sequences in $\mathfrak{B}(P_1)$ converging to $\mathfrak{C}_\infty$ are quasiconformally conjugate on $\widehat{\C}$. 
Moreover, such a conjugacy can be chosen to be conformal on the $P_2$-components.
\end{lem}
\begin{proof}
Suppose that $\{\mathfrak{C}_n\}\subset\mathfrak{B}(P_1)$ converges to $\mathfrak{C}_\infty$, and $\mathfrak{C}'\in\partial\mathfrak{B}(P_2)$ is a subsequential limit of $\chi(\mathfrak{C}_n)$.
By Theorem~\ref{bers_slice_hoemo_thm}, each $\mathfrak{C}_n$ is conjugate to $\chi(\mathfrak{C}_n)\in\mathfrak{B}(P_2)$ in neighborhoods of their filled limit sets via a $K$-quasiconformal homeomorphism $\mathfrak{H}$ of $\widehat{\C}$ (where $K>0$ is independent of $n$). Possibly after passing to a further subsequence, we may assume that $\mathfrak{H}_n$ converges to a $K$-quasiconformal homeomorphism $\mathfrak{H}_\infty$ of $\widehat{\C}$. Further, the neighborhoods of the filled limit sets of $\mathfrak{C}_n$ and $\chi(\mathfrak{C}_n)$ where the conjugation takes place, converge to neighborhoods of the filled limit sets of $\mathfrak{C}_\infty$ and $\mathfrak{C}'$. This follows from the fact that the uniformizations of the $P_1$-components (respectively, $P_2$-components) of $\mathfrak{C}_n$ (respectively, $\chi(\mathfrak{C}_n)$) converge uniformly on compact sets to the uniformizations of the $P_1$-components (respectively, $P_2$-components) of $\mathfrak{C}_\infty$ (respectively, of $\mathfrak{C}'$). 
Therefore, the global quasiconformal map $\mathfrak{H}_\infty$ conjugates $\mathfrak{C}_\infty$ to $\mathfrak{C}'$ in neighborhoods of their filled limit sets. 

We conclude that all subsequential limits of $\{\chi(\mathfrak{C}_n)\}$, where $\{\mathfrak{C}_n\}$ ranges over all sequences in $\mathfrak{B}(P_1)$ converging to $\mathfrak{C}_\infty$, are quasiconformally conjugate to $\mathfrak{C}_\infty$, and hence to each other, in neighborhoods of their filled limit sets. The extension of this conjugacy to the whole sphere in such a way that it is conformal on the $P_2$-components follows from Lemma~\ref{conjugacy_extension_lem}.
\end{proof}

\subsection{The case of the $4$-times punctured spheres}\label{four_punc_sphere_subsec}

The next result is one of the vital ingredients in the proof of the main theorem of this section, where planarity of the ambient parameter space (which contains the Bers slice $\mathfrak{B}(z^5)$) plays a crucial role.

\begin{lem}\label{special_plane_lem}
Let $\mathfrak{C}\in\overline{\mathfrak{B}(z^5)}$. Then the following are true.
\begin{enumerate}[leftmargin=*]
\item Any correspondence $\widetilde{\mathfrak{C}}$ that is quasiconformally conjugate to $\mathfrak{C}$ is defined by a rational map $\widetilde{R}$ of the form
$$
\widetilde{R}(z)=z+\frac{c}{z}-\frac{c}{3z^3}+\frac{1}{5z^5},
$$
for some $c\in\C$, via Formula~\eqref{corr_eqn}.

\item If $\mathfrak{C}\in\partial\mathfrak{B}(z^5)$, then it admits no non-trivial quasiconformal deformation.
\end{enumerate} 
\end{lem}
{
\begin{remark}
 As we will see shortly, complex one-dimensionality of the Bers slice plays a key role in the proof of the second statement of the lemma. In fact, the proof is an adaptation of the fact that quadratic polynomials lying on the boundary of the Mandelbrot set do not admit quasiconformal deformations (cf. \cite[Chapter~I, \S~6, Proposition~7]{DH85}).
\end{remark}}
\begin{proof}
Let $R$ be a degree $6$ rational map defining the correspondence $\mathfrak{C}\in\overline{\mathfrak{B}(z^5)}$. By the normalization and proof of pre-compactness given in \S~\ref{special_case_subsec}, the map $R$ can be chosen to be $R(z)\equiv R_a(z)=z+a/z-a/(3z^3)+1/(5z^5)$, for some $a\in\C$.

1) Suppose that $\widetilde{\mathfrak{C}}$ is quasiconformally conjugate to $\mathfrak{C}$. {By Lemma~\ref{qc_def_corr_lem}}, there exists a rational map $\widetilde{R}$ which defines the correspondence $\widetilde{\mathfrak{C}}$ (up to M{\"o}bius conjugation) via Formula~\eqref{corr_eqn}. Moreover, $\widetilde{R}=\Psi\circ R\circ\widetilde{\Psi}^{-1}$, where $\Psi,\widetilde{\Psi}$ are quasiconformal maps. Here, $\widetilde{\Psi}$ commutes with $\eta$ and conjugates $\mathfrak{C}$ to $\widetilde{\mathfrak{C}}$. 

{Recall that $C(\eta)$ denotes the centralizer of $\eta(z)=1/z$.} Possibly after post-composing $\Psi$ with a M{\"o}bius map and post-composing $\widetilde{\Psi}$ with an element of $C(\eta)$ (the latter amounts to conjugating $\widetilde{\mathfrak{C}}$ by an element of $C(\eta)$), we can assume that $\Psi(\infty)=\widetilde{\Psi}(\infty)=\infty$. Since $\widetilde{\Psi}$ commutes with $\eta$, it follows that $\widetilde{\Psi}(0)=0$. Hence $\widetilde{R}$ has a simple pole at $\infty$ and a $5$-fold pole at the origin. Therefore, $\widetilde{R}(z)=\frac{R_1(z)}{z^5}$, for a degree $6$ polynomial $R_1$. Post-composing $\Psi$ with an affine map, we may further assume that $R_1$ is monic and centered.
Finally, the commutation of $\widetilde{\Psi}$ with $\eta$ also implies that $\widetilde{R}$ has six of its distinct critical points at $\pm 1, c_1^{\pm 1}$, and $c_2^{\pm 1}$, for some $c_1,c_2\in\C^*$. As in \S~\ref{special_case_subsec}, one can now use Vieta's formulas to deduce that $\widetilde{R}$ has the desired form.

2) Suppose that $\mathfrak{C}\in\partial\mathfrak{B}(z^5)$ is quasiconformally conjugate to some correspondence $\widetilde{\mathfrak{C}}$ (that is not M{\"o}bius equivalent to $\mathfrak{C}$), via a quasiconformal map $\widetilde{\Psi}$. Set $\widetilde{\mu}:=\overline{\partial}\widetilde{\Psi}/\partial\widetilde{\Psi}$. Then, all Beltrami coefficients in the Beltrami disk $\{\lambda\widetilde{\mu}:\vert \lambda\vert<1/\vert\vert\widetilde{\mu}\vert\vert_\infty\}$ are $\mathfrak{C}$-invariant. If $\widetilde{\Psi}_\lambda$ is the (normalized) quasiconformal map straightening $\lambda\widetilde{\mu}$, then by the parametric version of the Measurable Riemann Mapping Theorem and part (1) of this lemma,
$$
\lambda\mapsto \widetilde{\mathfrak{C}}_\lambda:=\widetilde{\Psi}_\lambda\circ\mathfrak{C}\circ\widetilde{\Psi}_\lambda^{-1}
$$
defines a non-constant holomorphic map from the disk $B(0,1/\vert\vert\widetilde{\mu}\vert\vert_\infty)$ into the complex one-dimensional family of rational maps 
$$
\mathfrak{F}:=\bigg\{\widetilde{R}_c(z)=z+\frac{c}{z}-\frac{c}{3z^3}+\frac{1}{5z^5}:c\in\C\bigg\}.
$$
Hence, the rational map $R_a$ that defines $\mathfrak{C}$ lies in a conformal disk $\mathfrak{W}\subset\mathfrak{F}$, such that each rational map in $\mathfrak{W}$ defines (via Formula~\eqref{corr_eqn}) a correspondence which is quasiconformally conjugate to $\mathfrak{C}$.
But this is impossible because there exists $R_c\in\mathfrak{W}$, arbitrarily close to $R_a$, such that the correspondence generated by $R_c$ lies in the Bers slice $\mathfrak{B}(z^5)$ and hence is not quasiconformally conjugate to $\mathfrak{C}$.
\end{proof}

\begin{lem}\label{bers_slice_cont_ext_lem}
Let $P\in\cH_{5}$. Then, the dynamically natural homeomorphism $\chi:\mathfrak{B}(P)\longrightarrow\mathfrak{B}(z^5)$ extends continuously to the closure. 
\end{lem}
\begin{proof}
Pick $\mathfrak{C}_\infty\in\partial\mathfrak{B}(P)$. By Lemma~\ref{limit_of_str_corr_lem}, all possible subsequential limits of $\{\chi(\mathfrak{C}_n)\}$ on $\partial\mathfrak{B}(z^5)$ (of which at least one exists due to compactness of $\overline{\mathfrak{B}(z^5)}$), as $\{\mathfrak{C}_n\}$ runs over all sequences in $\mathfrak{B}(P)$ converging to $\mathfrak{C}_\infty$, are globally quasiconformally conjugate. By Lemma~\ref{special_plane_lem}, all such subsequential limits are the same. This allows us to extend $\chi$ to the boundary as this unique limit. This yields the desired continuous extension $\chi:\overline{\mathfrak{B}(P)}\to~\overline{\mathfrak{B}(z^5)}$.
\end{proof}

Recall from Theorem~\ref{bers_slice_hoemo_thm} that the map $\chi:\mathfrak{B}(P)\to\mathfrak{B}(z^{2d-1})$ (for arbitrary $d\geq 3$ and $P\in\cH_{2d-1}$) preserves the conformal class of the dynamics on the filled limit sets. However, the boundary extension of $\chi$ given by Lemma~\ref{bers_slice_cont_ext_lem}, in the special case of $d=3$, a priori only preserves the quasiconformal class of the dynamics on the filled limit sets. This is related to possible discontinuities of filled limit sets as the correspondences approach the boundaries of Bers slices.
We will now exploit the quasiconformal rigidity of parameters on the boundary of $\mathfrak{B}(z^5)$ to demonstrate that the map $\chi:\partial\mathfrak{B}(P)\to\partial\mathfrak{B}(z^5)$ actually preserves the conformal class of the dynamics on the filled limit sets.

\begin{lem}\label{chi_conf_conj_lem}
Let $P\in\cH_{5}$, and $\mathfrak{C}\in\partial\mathfrak{B}(P)$. Then, $\chi(\mathfrak{C})\in\partial\mathfrak{B}(z^5)$ is the unique correspondence in $\cG_{z^5}$ that is hybrid conjugate to $\mathfrak{C}$.  
\end{lem}
\begin{proof}
According to Lemma~\ref{surgery_replace_blaschke_lem}, for each $\mathfrak{C}\in\partial\mathfrak{B}(P)$, there exists a unique correspondence $\mathfrak{C}'\in\cG_{z^{5}}$ such that $\mathfrak{C}$ and $\mathfrak{C}'$ are hybrid conjugate. To prove the lemma, it suffices to show that $\mathfrak{C}'=\chi(\mathfrak{C})$.

By definition of $\mathfrak{C}'$ and the continuous extension of the map $\chi$, we have that $\mathfrak{C}',\chi(\mathfrak{C})\in\cG_{z^{5}}$ are quasiconformally conjugate in some neighborhoods of their filled limit sets and the markings are preserved by the conjugacy. 
Lemma~\ref{conjugacy_extension_lem} now implies that $\mathfrak{C}'$ and $\chi(\mathfrak{C})$ are globally quasiconformally conjugate. Since $\chi(\mathfrak{C})\in\partial\mathfrak{B}(z^5)$, it follows from Lemma~\ref{special_plane_lem} that $\mathfrak{C}'=\chi(\mathfrak{C})$.
\end{proof}

\begin{lem}\label{qc_rigidity_gen_lem}
Let $P\in\cH_{5}$. If $\mathfrak{C}\in\partial\mathfrak{B}(P)$, then it admits no non-trivial quasiconformal deformation supported on its filled limit set. In particular, $\mathfrak{C}$ is not quasiconformally conjugate to any other correspondence on $\partial\mathfrak{B}(P)$. 
\end{lem}
\begin{proof}
By way of contradiction, assume that $\mathfrak{C}\in\partial\mathfrak{B}(P)$ admits a non-trivial quasiconformal deformation supported on its filled limit set.  Then, there exists a correspondence $\mathfrak{C}_1$, not M{\"o}bius conjugate to $\mathfrak{C}$, that is quasiconformally conjugate to $\mathfrak{C}$ on $\widehat{\C}$ such that the conjugacy is conformal on the $P$-components of $\mathfrak{C}$. Consequently, the distinguished forward branch of $\mathfrak{C}_1$ is also conformally conjugate to $P$, and hence $\mathfrak{C}_1\in\cG_P$.

We apply Lemma~\ref{surgery_replace_blaschke_lem} to the correspondences $\mathfrak{C},\mathfrak{C}_1\in\cG_P$ to obtain unique correspondences $\mathfrak{C}', \mathfrak{C}_1'\in\cG_{z^5}$ that are hybrid conjugate to $\mathfrak{C},\mathfrak{C}_1$, respectively. By Lemma~\ref{chi_conf_conj_lem}, we have $\mathfrak{C}'=\chi(\mathfrak{C})\in\partial\mathfrak{B}(z^5)$. Since $\mathfrak{C}$ and $\mathfrak{C}_1$ are quasiconformally conjugate, it follows from the construction (and Lemma~\ref{conjugacy_extension_lem}) that $\chi(\mathfrak{C})$ and $\mathfrak{C}_1'$ are also quasiconformally conjugate. So Lemma~\ref{special_plane_lem} forces $\mathfrak{C}_1'$ and $\chi(\mathfrak{C})$ to be the same. But this in turn implies that
$\mathfrak{C}$ and $\mathfrak{C}_1$ are hybrid conjugate. Using Lemma~\ref{conjugacy_extension_lem}, we can extend this to a global conformal conjugacy between $\mathfrak{C}$ and $\mathfrak{C}_1$. This is a contradiction.
\end{proof}

\begin{lem}\label{bers_slice_bdry_hoemo_lem}
Let $P\in\cH_{5}$. Then $\chi:\partial\mathfrak{B}(P)\longrightarrow\partial\mathfrak{B}(z^5)$ is a homeomorphism.
\end{lem}
\begin{proof}
Thanks to the quasiconformal rigidity statement of Lemma~\ref{qc_rigidity_gen_lem}, the proof of Lemma~\ref{bers_slice_cont_ext_lem} applies verbatim to the straightening map $\widetilde{\chi}:\mathfrak{B}(z^5)\to\mathfrak{B}(P)$, and produces a continuous extension $\widetilde{\chi}:\overline{\mathfrak{B}(z^5)}\to\overline{\mathfrak{B}(P)}$. In particular, for each $\mathfrak{C}\in\partial\mathfrak{B}(z^5)$, the correspondences $\mathfrak{C}$ and $\widetilde{\chi}(\mathfrak{C})$ are quasiconformally conjugate to each other on some neighborhoods of their filled limit sets preserving the marking. Thus, $\chi(\widetilde{\chi}(\mathfrak{C})$ is quasiconformally conjugate to $\mathfrak{C}$ on some neighborhoods of their filled limit sets preserving the marking. By Lemma~\ref{conjugacy_extension_lem}, $\chi(\widetilde{\chi}(\mathfrak{C})$ is quasiconformally conjugate to $\mathfrak{C}$ on $\widehat{\C}$. Lemma~\ref{special_plane_lem} now implies that $\chi(\widetilde{\chi}(\mathfrak{C})=\mathfrak{C}$; i.e., $\chi\circ\widetilde{\chi}$ is the identity map on $\mathfrak{B}(z^5)$. Similarly, using Lemma~\ref{qc_rigidity_gen_lem}, one deduces that $\widetilde{\chi}\circ\chi$ is the identity map on $\mathfrak{B}(P)$. Hence, $\chi:\partial\mathfrak{B}(P)\longrightarrow\partial\mathfrak{B}(z^5)$ is a homeomorphism with $\widetilde{\chi}$ as its inverse.
\end{proof}

The main theorem of this section is now immediate.

\begin{theorem}\label{bers_slice_closure_hoemo_thm}
For $P_1, P_2\in\cH_5$, there is a dynamically natural homeomorphism $\chi_{P_1,P_2}:\overline{\mathfrak{B}(P_1)}\longrightarrow\overline{\mathfrak{B}(P_2)}$ between the corresponding Bers slice closures. 
Moreover, if $\mathfrak{C}\in\overline{\mathfrak{B}(P_1)}$, then $\chi_{P_1,P_2}(\mathfrak{C})$ is the unique correspondence in $\overline{\mathfrak{B}(P_2)}$ that is hybrid conjugate to $\mathfrak{C}$.  
\end{theorem}

{
Since the Bers slice $\mathfrak{B}(z^5)$ can be identified with a bounded subset of an explicit copy of the complex plane (see Example~\ref{four_punc_example} and Example~\ref{four_punc_bounded_example}), and $\partial\mathfrak{B}(P)$ is naturally homeomorphic to $\partial\mathfrak{B}(z^5)$ for each $P\in\cH_{5}$, it may be fruitful to study this special slice in detail.
Item (2) of Question~\ref{qn2} reduces to the following question in this setting (cf. \cite{Min99}).

\begin{question}\label{jordan_curve_qstn}
Is $\partial\mathfrak{B}(z^5)$ canonically homeomorphic to the classical Bers boundary of the Teichm{\"u}ller space of $4$-times punctured spheres? In particular, is $\partial\mathfrak{B}(z^5)$ a Jordan curve?    
\end{question}
}

We end with another question inspired by Theorem~\ref{bers_slice_closure_hoemo_thm}.

\begin{question}\label{bdry_bijection_qstn}
Let $P_1,P_2\in\cH_{2d-1}$. Does there exist a dynamically natural bijection $\overline{\mathfrak{B}(P_1)}\longrightarrow\overline{\mathfrak{B}(P_2)}$?
\end{question}

Thanks to Lemma~\ref{surgery_replace_blaschke_lem}, given any $\mathfrak{C}\in\partial\mathfrak{B}(P_1)$, we can construct a unique correspondence $\mathfrak{C}'\in\cG_{P_2}$ that is hybrid conjugate to $\mathfrak{C}$. However, we do not know if $\mathfrak{C}'$ lies on the boundary of the Bers slice $\mathfrak{B}(P_2)$. Indeed, we had used quasiconformal rigidity of Bers boundary correspondences (which relied heavily on one-dimensionality of the parameter spaces) to deduce this fact in the case of $4$-times punctured spheres. In general, this question can be viewed as an analog of the Bers density theorem 
{\cite[p.3]{minsky-elc2}} in the context of correspondences.

\section{Compactness of external fibers}\label{vertical_compact_subsec}

In this section, we prove a counterpart of Theorem~\ref{pre_comp_thm_1} for the external/vertical fibers $\mathscr{B}_\Gamma$, $\Gamma\in\mathrm{Teich}(S_{0,d+1})$ (see Definition~\ref{vertical_fiber_def}). 

\begin{theorem}\label{vert_slice_comp_thm}
 Let $\Gamma\in\mathrm{Teich}(S_{0,d+1})$. Then the external fiber $\mathscr{B}_\Gamma$ is compact. In particular, the principal hyperbolic component $\mathfrak{B}(\Gamma)$ of $\mathscr{B}_\Gamma$ is pre-compact in the { regular character variety $\mathcal{CV}^{reg}_{2d} \subseteq \mathcal{CV}_{2d}$}.
\end{theorem}

Recall from \S~\ref{identify_b_inv_corr_subsec} that an external fiber $\mathscr{B}_\Gamma$ can either be regarded as the collection of degenerate polynomial-like maps $F$ (with connected non-escaping set) admitting the Bowen-Series map $A_\Gamma$ as their external map, or can be thought of as a space of bi-degree $(2d,2d)$ algebraic correspondences $\mathfrak{C}$ of the form Equation~\eqref{corr_eqn} (where $\pmb{R}$ is the uniformizing rational map of the B-involution obtained as a meromorphic extension of $F$). The latter perspective allows us to embed $\mathscr{B}_\Gamma$ into the {regular character variety $\mathcal{CV}^{reg}_{2d}$ (see Proposition~\ref{total_space_regular_lem})}.

\subsection{Connections between various boundedness results}\label{comparison_proof_subsec}

\subsubsection{External fiber vs horizontal Bers slice}
 Although Theorems~\ref{pre_comp_thm_1} and~\ref{vert_slice_comp_thm} are boundedness results in the character variety, their proofs are not quite parallel. Each correspondence in a horizontal Bers slice has the trivial characteristic data associated with a single sphere, and much of the proof of Theorem~\ref{pre_comp_thm_1} was devoted to disallowing degeneration of the associated uniformizing rational maps (defined on a single sphere). In particular, there is no change of characteristic data in the entire closure of a horizontal Bers slice.

On the other hand, a general sequence $\mathfrak{C}_n\in\mathscr{B}_\Gamma$ is defined on a tree of spheres, and as we shall see, it converges to some correspondence $\mathfrak{C}_\infty$, whose characteristic data is potentially `larger' than that of $\mathfrak{C}_n$. In other words, the rational maps defining $\mathfrak{C}_n$ must be allowed to degenerate to rational maps on trees of spheres, and we need to organize these rational maps on various trees of spheres appropriately to define the dominating characteristic data.

Also note that in order to prove compactness of $\mathscr{B}_\Gamma$, we need to justify that $\mathfrak{C}_\infty$ lies in the external fiber $\mathscr{B}_\Gamma$. This demands good control of the dynamical plane of the limiting correspondence $\mathfrak{C}_\infty$ and a significant part of the proof of Theorem~\ref{vert_slice_comp_thm} will be devoted to achieving this goal.

\subsubsection{Holomorphic vs antiholomorphic setting}
In \cite[Theorem~6.4]{LLM24}, an antiholomorphic analog of
Theorem~\ref{vert_slice_comp_thm} was proved; specifically, it was shown that the space of degenerate anti-polynomial-like maps with connected non-escaping set admitting the Nielsen map of a regular ideal polygon reflection group as their external map (such degenerate anti-polynomial-like maps turn out to be Schwarz reflection maps in quadrature domains) is compact.
It is worth pointing out some key differences between the holomorphic and antiholomorphic settings.

The fact that boundaries of quadrature domains are real-algebraic gives immediate control on the geometry of the limiting dynamical plane and allows one to count the number of critical points of the limiting rational maps. Since this is not true for inversive domains, we need to employ a combination of hyperbolic geometric techniques and algebraic arguments to study the limit of the inversive domains and to count the number of critical points.

Further, the rational uniformizations of quadrature domains are injective on the round disk, so one can use normal family arguments to construct the limiting rational maps. In the holomorphic case, the domains of univalence of the uniformizing rational maps may vary (making normal family arguments untenable). Hence, we need to consider rescaling limits on the correspondence planes to manufacture the limiting rational maps.

\subsection{Compactness of $\mathscr{B}_\Gamma$}\label{comp_vert_fiber_subsec}
Throughout this section, we assume that $\mathfrak{C}_n \in \mathscr{B}_\Gamma$ is a sequence of correspondences. After passing to a subsequence, we may assume that all the correspondences $\mathfrak{C}_n$ have the same characteristic data 
$$
\kappa=\left((\pmb{\tau}, \pmb{\eta}): (\RT, \widehat\C^\RV) \longrightarrow (\RT, \widehat\C^\RV), \delta, (\overline{\delta}_a)_{a\in \RV}\right).
$$

\begin{prop}[Pre-compactness of $\mathscr{B}_\Gamma$ in {$\mathcal{CV}^{reg}_{2d}$}]\label{limit_corr_prop}
After possibly passing to a subsequence, $\mathfrak{C}_n$ converges to a correspondence { $\mathfrak{C}_\infty\in\mathcal{CV}^{reg}_{2d}$} with characteristic data $\widetilde{\kappa}$ which dominates $\kappa$.
\end{prop}
\noindent {Since the proof of the proposition is rather long, we sketch a road map of the proof for the reader's convenience. 
\begin{itemize}[leftmargin=6mm]
    \item We start with the construction of the rescaling limits of the sequences of rational maps $(R_{a, n}=\pmb{R}_n\vert_{\widehat{\C}_a})_n$, for $a\in\RV$ (where $\pmb{R}_n$ is the uniformizing rational map for $\mathfrak{C}_n$). We then arrange these rescaling limits to produce a rational map $\pmb{R}_{a,\infty}$ from a graph of spheres $(\RT_a, \widehat{\C}^{\RV_a})$ to the Riemann sphere $\widehat{\C}$. Under the same rescalings, the involution $\eta(z)=1/z$ of $\widehat{\C}_a$ converges to a conformal involution on $(\RT_a, \widehat{\C}^{\RV_a})$. We also ensure that the associated rescalings commute with this limiting involution (which is guaranteed by Lemma~\ref{lem_1}).

    \item The next step is to organize all the rational maps $\pmb{R}_{a,\infty}:(\RT_a, \widehat{\C}^{\RV_a})\longrightarrow\widehat{\C}$, $a\in\RV$, together to produce a single rational map $\pmb{R}_\infty$ from a graph of spheres $(\widetilde{\RT}, \widehat{\C}^{\widetilde{\RV}})$ to $\widehat{\C}$. The conformal involutions on $(\RT_a, \widehat{\C}^{\RV_a})$ also fit together to yield a conformal involution $(\widetilde{\tau},\widetilde{\eta})$ on $(\widetilde{\RT}, \widehat{\C}^{\widetilde{\RV}})$. 
    
    \item The desired limiting correspondence $\mathfrak{C}_\infty$ is defined via Equation~\eqref{corr_eqn} using the rational map $\pmb{R}_\infty$ and the involution $(\widetilde{\pmb{\tau}},\widetilde{\pmb{\eta}})$.

    \item We proceed to show that the graph $\widetilde{\RT}$ is actually a tree; i.e., more than two spheres cannot touch at a node. This allows us to associate a characteristic data $\widetilde{\kappa}$ with the `limiting' correspondence $\mathfrak{C}_\infty$ such that $\widetilde{\kappa}$ dominates the original characteristic data $\kappa$. We also demonstrate simplicity of the characteristic data $\widetilde{\kappa}$ (see Definition~\ref{simple_sign_def}). This is a necessary condition for a correspondence to lie in the vertical fiber $\mathscr{B}_\Gamma$. These facts are proved in Lemma~\ref{lem_2}.

    \item At this point, it remains to argue that the sequence $\mathfrak{C}_n$ converges to $\mathfrak{C}_\infty$ (or equivalently, $\pmb{R}_n\to\pmb{R}_\infty$) in the topology of the character variety and to obtain sufficient geometric/dynamical control on the limiting correspondence $\mathfrak{C}_\infty$ that will eventually allow us to conclude that $\mathfrak{C}_\infty\in\mathscr{B}_\Gamma$. 

    We address the latter task first by showing that the B-involutions $F_n$ associated with the correspondences $\mathfrak{C}_n$ converge in the Carath{\'e}odory topology to a B-involution $F_\infty$ (see Lemma~\ref{lem_3} and the ensuing discussion).

    Finally, we justify the convergence $\mathfrak{C}_n\to\mathfrak{C}_\infty$ by showing that { limiting characteristic data $\widetilde{\kappa}$ is regular.}
    This is done in Lemma~\ref{lem_4}.
\end{itemize} }

\begin{proof}[Proof of Proposition~\ref{limit_corr_prop}]
Let 
\begin{enumerate}[leftmargin=*]
    \item $\pmb{R}_n: (\RT, \widehat{\C}^\RV) \longrightarrow \widehat{\C}$ be the corresponding rational uniformization map in $\Rat_\kappa(\C)$ (we remind the reader that the markings on the tree of Riemann spheres $(\RT, \widehat{\C}^\RV)$ depend on $\pmb{R}_n$),
    \item $F_n:\overline{\Omega_n}=\overline{\displaystyle\bigsqcup_{a\in\RV}\Omega_{a,n}}\to\widehat{\C}$ be the associated B-involution (where $\Omega_n$ is an inversive multi-domain), and
    \item $\pmb{\mathfrak{D}}_n=\displaystyle\bigsqcup_{a\in\RV}\mathfrak{D}_{a,n}\subset\widehat{\C}^\RV$ be the disjoint union of Jordan domains such that
    \begin{itemize}
        \item $\pmb{\eta}(\pmb{\mathfrak{D}}_n)=\widehat{\C}^\RV-\overline{\pmb{\mathfrak{D}}_n}$ and
        \item $\pmb{R}_n:\pmb{\mathfrak{D}}_n\to\Omega_n$ is a conformal isomorphism (see Theorem~\ref{b_inv_thm}).
    \end{itemize} 
\end{enumerate}   
In particular, $F_n\equiv \pmb{R}_n\circ\pmb{\eta}\circ\left(\pmb{R}_n\vert_{\pmb{\mathfrak{D}}_n}\right)^{-1}$.
By definition, the degree of {$R_{a, n}=\pmb{R}_n\vert_{\widehat{\C}_a}$} is equal to $\delta(a)$, for $a\in\RV$.
As in Definition~\ref{b_inv_def}, we denote the singular set of $\partial\Omega_n$ by $\mathfrak{S}_n$, and set $Q_n = \widehat{\C} - \left(\Omega_n\cup\mathfrak{S}_n\right)$. Further, let $\Delta_n \supseteq Q_n$ be the escaping/tiling set of $F_n$. Since the group dynamics is fixed, we have a sequence of uniformization map $\Phi_n:(\D, \Pi) \longrightarrow (\Delta_n, Q_n)$, where $\Pi$ is the preferred fundamental domain for the $\Gamma-$action on $\D$ (see \S~\ref{corr_mating_subsec}). By post-composition with suitable M\"obius transformations and after passing to a subsequence, we may assume that $\Phi_n$ converges compactly to a conformal map $\Phi_\infty:(\D, \Pi) \longrightarrow (\Delta_\infty, Q_\infty)$.

For simplicity of the presentation, we assume that $\tau(a) = a$ (where $\tau$ is the restriction of $\pmb{\tau}$ on the vertex set $\RV$).
Then for each $n$, there are exactly $\delta(a)$ preimages of $(\Delta_n, Q_n)$ under $R_{a,n}$ (each of which is mapped conformally onto $(\Delta_n, Q_n)$ by $R_{a,n}$). We denote them by $(\Delta_{a, n}^i, Q_{a, n}^i)$, $i =1,\cdots, \delta(a)$. Since $(\Delta_{a, n}^i, Q_{a, n}^i)$ is conformally equivalent to $(\D, \Pi)$, we can normalize it by some M\"obius map $\widehat{M}_{a,n}^i \in \PSL_2(\C)$ so that $\widehat{M}_{a,n}^i((\Delta_{a, n}^i, Q_{a, n}^i))$ converges (after possibly passing to a subsequence) to $(\widehat \Delta_{a, \infty}^i, \widehat Q_{a, \infty}^i)$ which is conformally equivalent to $(\D, \Pi)$. We will now arrange that the rescalings commute with the involution on the desired limiting tree of spheres.

Note that $\eta$ induces an involution on the collection $\{\Delta_{a, n}^i: i =1,\cdots, \delta(a)\}$.
Hence, we have a dichotomy 
$$
\eta(\Delta_{a, n}^i) = \Delta_{a, n}^i\qquad \textrm{or}\qquad \eta(\Delta_{a, n}^{i_1}) = \Delta_{a, n}^{i_2}\ \mathrm{and}\ \eta(\Delta_{a, n}^{i_2}) = \Delta_{a, n}^{i_1}.
$$

\begin{lem}[Compatibility between involution and rescalings]\label{lem_1}
\noindent\begin{enumerate}[leftmargin=*]
    \item If $\eta(\Delta_{a, n}^i) = \Delta_{a, n}^i$, then $\widehat{M}_{a,n}^i$ can be replaced by an equivalent M{\"o}bius map $M_{a,n}^i\in C(\eta)$.
    
    \item If $\eta(\Delta_{a, n}^{i_1}) = \Delta_{a, n}^{i_2}$ and $\eta(\Delta_{a, n}^{i_2}) = \Delta_{a, n}^{i_1}$, then $\widehat{M}_{a,n}^{i_1}, \widehat{M}_{a,n}^{i_2}$ can be replaced by equivalent M{\"o}bius maps $M_{a,n}^{i_1}, M_{a,n}^{i_2}$ satisfying $M_{a,n}^{i_1}\circ\eta = \eta\circ M_{a,n}^{i_2}$.
\end{enumerate}
\end{lem}
\begin{proof}
Let us first assume that $\eta(\Delta_{a, n}^i) = \Delta_{a, n}^i$. Then $(\widehat{M}_{a,n}^i\circ \eta)((\Delta_{a, n}^i, Q_{a, n}^i))$ also converges, and thus after passing to a subsequence, $L_n:=(\widehat{M}_{a,n}^i)\circ\eta\circ (\widehat{M}_{a,n}^i)^{-1}$ converges to some M\"obius map $L$. Note that $L^2$ is the identity map, so $L$ is conjugate to $\eta$ by some M\"obius map $M$; i.e., $M \circ L \circ M^{-1} = \eta$. Therefore $M\circ L_n\circ M^{-1}$ converges to $\eta$ and hence the fixed point set $M\circ \widehat{M}_{a,n}^i(\{1,-1\})$ of $M\circ L_n\circ M^{-1}$ converges to the fixed point set $\{1,-1\}$ of $\eta$. Thus, we can choose a sequence of M{\"o}bius maps $M_n$ converging  to the identity map such that $M_n$ carries $M\circ \widehat{M}_{a,n}^i(\pm 1)$ to $\pm 1$.
We now set $M_{a,n}^i \coloneqq M_n\circ M \circ \widehat{M}_{a,n}^i$, and conclude that $M_{a,n}^i \in C(\eta)$  gives the required M{\"o}bius map (observe that $M_{a,n}^i$ fixes $\pm 1$).

Next suppose that $\eta(\Delta_{a, n}^{i_1}) = \Delta_{a, n}^{i_2}$ and $\eta(\Delta_{a, n}^{i_2}) = \Delta_{a, n}^{i_1}$. Then a similar argument as above shows that
$$
L_{n}:=(\widehat{M}_{a,n}^{i_1})\circ\eta\circ (\widehat{M}_{a,n}^{i_2})^{-1}\to L\in\PSL_2(\C).
$$
We set $M_n:= \eta\circ L_{n}^{-1}$ and note that $M_n$ is a bounded sequence of M{\"o}bius maps. It is now easily seen that the desired M{\"o}bius maps are given by $M_{a,n}^{i_2}:=\widehat{M}_{a,n}^{i_2}$ and $M_{a,n}^{i_1}:=M_n\circ \widehat{M}_{a,n}^{i_1}$.
\end{proof}

After passing to a subsequence if necessary, we may assume that for $i_1, i_2 \in \{1,\cdots, \delta(a)\}$, $M_{a, n}^{i_1} , M_{a, n}^{i_2}$ are either equivalent or independent. Let $\RV_a$ be the indexing set of equivalence classes of $M_{a, n}^{i}$.
For each $p\in\RV_a$, we pick a representative $M_{a,n}^i$ in the corresponding equivalence class, and define the associated rescaling $\mathscr{M}_{a,n}^p:= \left(M_{a, n}^i\right)^{-1}$.
\smallskip

\noindent\textbf{Constructing a limiting rational map $\pmb{R}_{a,\infty}: (\RT_a, \widehat{\C}^{\RV_a}) \longrightarrow \widehat{\C}$ for each $a\in\RV$.} In order to manufacture a rescaling limit $\pmb{R}_{a,\infty}$ of the rational maps $R_{a,n}$, we need to organize the domain spheres of the rescalings $\mathscr{M}_{a,n}^p$, $p\in\RV_a$, in a tree structure. However, it will be convenient to first arrange them as a graph of spheres and then argue that the graph is indeed a tree.

Consider a collection of Riemann spheres $\widehat{\C}_p$, $p\in\RV_a$. 
By construction, for each $p, q \in\RV_a$ ($p\neq q$), the sequence of M{\"o}bius maps $\left(\mathscr{M}_{a,n}^p\right)^{-1}\circ\mathscr{M}_{a,n}^q$ converges to a constant in $\widehat{\C}_p$, which we denote by $x_{q\to p}$ (we call $x_{q\to p}$ the \emph{projection} of $\widehat{\C}_q$ to $\widehat{\C}_p$). For $p\in\RV_a$, we define 
$$
\widecheck{\Xi}_p:=\{x_{q\to p}: q\in \RV_a-\{p\}\}.
$$
Now for $p, q \in\RV_a$, $p\neq q$, we say they are {\em non-adjacent} if there exists $r\in\RV_a-\{p,q\}$ such that $x_{p,r}\neq x_{q,r}$. They are called {\em adjacent} otherwise.
Let $\RT_a$ be the graph by obtained by adding an edge to every pair of adjacent vertices.
Note that if the maximal cliques in $\RT_a$ (with at least $3$ vertices) are replaced by vertices, then we get a quotient tree.
 
\begin{remark}
Thanks to the last observation, it is equivalent to work with a tree that is obtained by replacing each $m$-vertex clique ($m\geq 3$) in $\RT_a$ with a star-tree that has $m$ tips and a new auxiliary base vertex (which will not correspond to a Riemann sphere).
\end{remark}

Following the notion of a marked tree of spheres (see Definition~\ref{defn:trs} and Definition~\ref{def-tangentdirn}), we construct a marked graph of spheres $(\RT_a,\widehat{\C}^{\RV_a})$, where the marking sends the tangent direction $v\in T_p\RT_a$, $p\in\RV_a$, associated with an edge $[p,q]\subset\RT_a$ to the point $x_{p\to q}$. 
It now follows from the construction and Corollary~\ref{cor:caratheodorylimitandrescaling} that with respect to the rescalings $\mathscr{M}_{a, n}^p$, $p \in \RV_a$, the rational map $R_{a,n}$ converges to a degree $\delta(a)$ rational map $\pmb{R}_{a,\infty}: (\RT_a, \widehat{\C}^{\RV_a}) \longrightarrow \widehat{\C}$ on the (marked) graph of spheres constructed above (the notion of convergence introduced in Definition~\ref{defn:cvrationalmap} applies to this setting). We remark that if $[p,q]$ is an edge in $\RT_a$, then $\pmb{R}_{a,\infty}\vert_{\widehat{\C}_p}(x_{q\to p})=\pmb{R}_{a,\infty}\vert_{\widehat{\C}_q}(x_{p\to q})$.

It also follows from the construction of $(\RT_a, \widehat{\C}^{\RV_a})$ that the action of $\eta$ on $\{\Delta_{a, n}^i: i = 1,\cdots, \delta(a)\}$ induces a map $(\pmb{\tau}^a,\pmb{\eta}^a): (\RT_a, \widehat{\C}^{\RV_a})\longrightarrow (\RT_a, \widehat\C^{\RV_a})$. Here, $\pmb{\tau}^a$ is an involution on the underlying graph (whose restriction to the vertex set $\RV_a$ is denoted by $\tau^a$), and for $p\in\RV_a$, the map $\pmb{\eta}^a$ sends the sphere $\widehat{\C}_p$ in $\widehat{\C}^{\RV_a}$ to the sphere $\widehat{\C}_{\tau^a(p)}$ as $z\mapsto 1/z$. By Lemma~\ref{lem_1}, $\{\mathscr{M}_{a,n}^p: p\in\RV_a\}\in \Aut_{(\pmb{\tau}^a,\pmb{\eta}^a)}(\RT_a, \widehat\C^{\RV_a})$. Further, the map $\eta:\widehat{\C}_a\to\widehat{\C}_a$ converges to the map $(\pmb{\tau}^a,\pmb{\eta}^a): (\RT_a, \widehat{\C}^{\RV_a})\longrightarrow (\RT_a, \widehat\C^{\RV_a})$ with respect to the rescalings $\mathscr{M}_{a, n}^p$, $p \in \RV_a$.
\smallskip

\noindent{\textbf{Piecing various graphs of spheres $(\RT_a, \widehat{\C}^{\RV_a})$ together.}
Our next goal is to attach the various graphs $\RT_a$, $a\in\RV$, appropriately to form a graph of spheres $(\widetilde{\RT},\widehat{\C}^{\widetilde{\RV}})$. This graph of spheres will be the domain of the desired limiting rational map $\pmb{R}_\infty:(\widetilde{\RT},\widehat{\C}^{\widetilde{\RV}})\to\widehat{\C}$ and will also be naturally equipped with a conformal involution. The candidate limiting correspondence will then be defined in terms of this rational map and involution.}

To this end, we start with the disjoint unions 
$$
\widecheck{\RT}:=\bigcup_{a\in\RV}\RT_a,\quad \widetilde{\RV}:=\bigcup_{a\in\RV}\RV_a,\quad \mathrm{and}\quad \widehat{\C}^{\widetilde{\RV}}:=\bigcup_{a\in\RV}\widehat{\C}^{\RV_a}.
$$ 
Now pick an edge $E = [a_1,a_2]$ in $\RT$. We denote the corresponding singular points for $\pmb{R}_n$ by $\xi_{E, a_i}^n \in \widehat{\C}_{a_i}$.
Let $p \in \RV_{a_i}$. After passing to a subsequence, we can define the projection of the edge $E = [a_1, a_2]$ to $\widehat{\C}_p$ as
$$
x_{E \to p} := \lim_n\  (\mathscr{M}_{a_i,n}^p)^{-1}(\xi_{E, a_i}^n).
$$
In this way, we generalize the definition of the projection of $\widehat{\C}_q$ to $\widehat{C}_p$ for $p \in \RV_a, q \in \RV_b$, $a\neq b$, by setting
$$
x_{q \to p} = x_{E, p},
$$
where $E$ is the edge of $\RT$ adjacent to $a$ that is in the same connected component of $\RT - \{a\}$ as $b$.

We say that $p \in \RV_{a_i}$ is {\em non-adjacent} to the edge $E = [a_1, a_2]$ if there exists a vertex $r \in \RV_{a_i} - \{p\}$ such that $x_{p, r} \neq x_{E, r}$. Otherwise, we say that $p$ is {\em adjacent} to $E$.
Let $p \in \RV_{a_1}$ and $q \in \RV_{a_2}$. We say they are {\em adjacent} if $p$ and $q$ are both adjacent to $E$.

Let $\widetilde{\RT}$ be the graph obtained by adding an edge to every pair of adjacent vertices $p \in \RV_{a_1}, q \in \RV_{a_2}$, where $[a_1, a_2]$ ranges  over all edges  of $\RT$.
Once again, if the maximal cliques in $\widetilde{\RT}$ (with at least $3$ vertices) are replaced by vertices, then we get a quotient tree. 
\begin{figure}[ht]
\captionsetup{width=0.98\linewidth}
\begin{tikzpicture}
\node[anchor=south west,inner sep=0] at (0,0) {\includegraphics[width=0.9\textwidth]{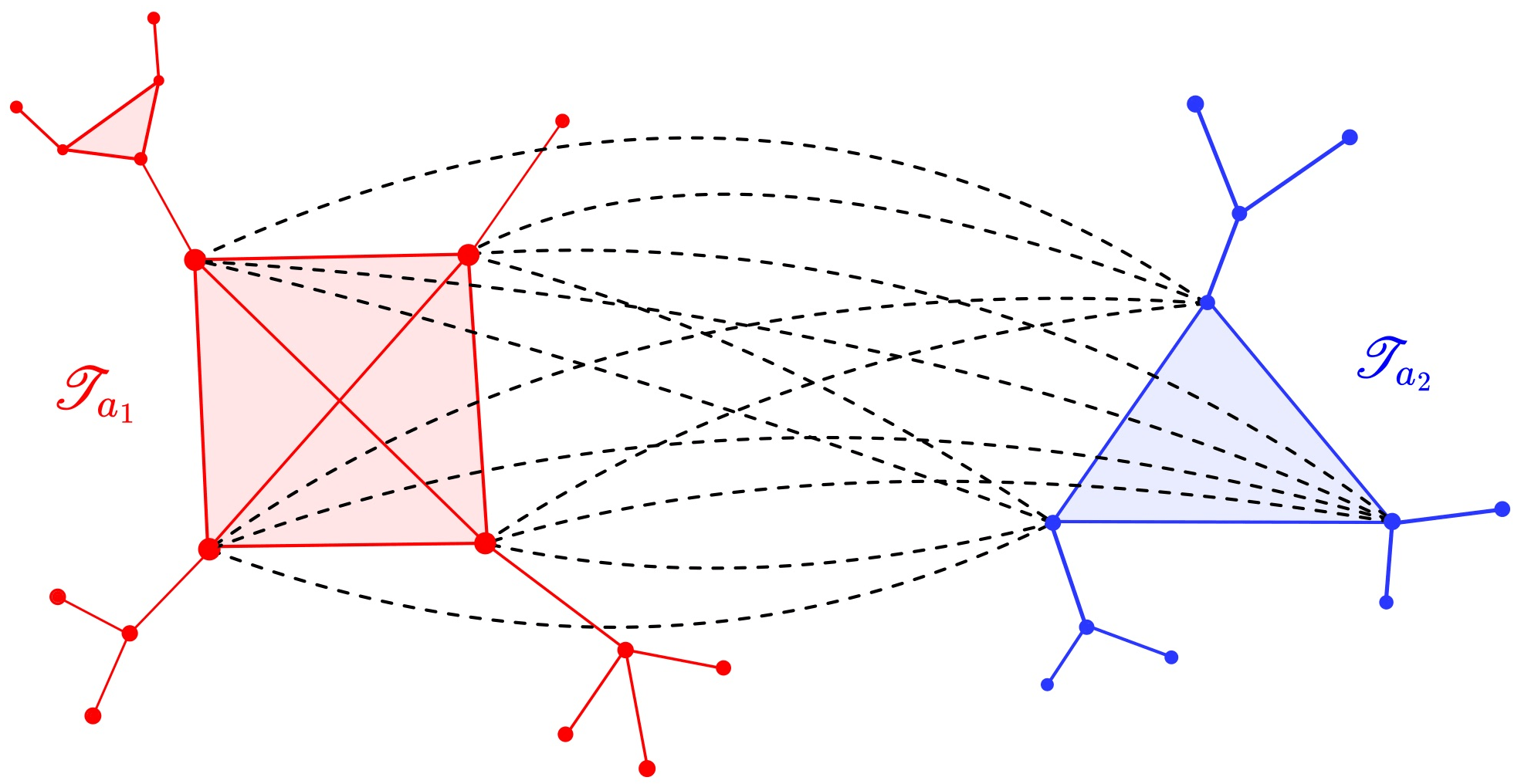}};
\end{tikzpicture}
\caption{Illustrated is an a priori possible configuration for $\widetilde{\RT}$ that is ruled out by Lemma~\ref{lem_2}. The dotted black edges connect vertices in $\RT_{a_1}$ and $\RT_{a_2}$ that are adjacent to each other.}
\label{clique_connection_fig}
\end{figure}
\smallskip

\noindent\textbf{The candidate limiting rational map and correspondence.} We define a rational map $\pmb{R}_\infty:(\widetilde{\RT}, \widehat{\C}^{\widetilde{\RV}})\to\widehat{\C}$ that acts as $\pmb{R}_{a,\infty}$ on $\RT_a$. It is easily checked that the map $\pmb{R}_\infty$ is compatible with the marking on $(\widetilde{\RT}, \widehat{\C}^{\widetilde{\RV}})$ defined above. Further, we have an induced involution $(\widetilde{\pmb{\tau}},\widetilde{\pmb{\eta}})$ on $(\widetilde{\RT}, \widehat{\C}^{\widetilde{\RV}})$ that acts piecewise as the maps $(\pmb{\tau}^a,\pmb{\eta}^a)$, $a\in\RV$. 
{Following Equation~\eqref{corr_eqn}, we define the candidate limiting correspondence $\mathfrak{C}_\infty$ by the equation
\begin{equation}
\mathfrak{C}_{\infty}:= \bigg\{(x,y) \in \widehat\C^{\widetilde{\RV}} \times \widehat\C^{\widetilde{\RV}}: \frac{\pmb{R}_\infty(x) - \pmb{R}_\infty\circ\widetilde{\pmb{\eta}}(y)}{x-\widetilde{\pmb{\eta}}(y)} = 0\bigg\}.
\end{equation}}
\smallskip

\noindent\textbf{$\widetilde{\RT}$ is a tree and the associated characteristic data is simple.}
We will now show that $\widetilde{\RT}$ is a tree.
Once this is proved, we can construct  characteristic data 
$$
\widetilde{\kappa}= \left((\widetilde{\pmb{\tau}}, \widetilde{\pmb{\eta}}): (\widetilde{\RT}, \widehat\C^{\widetilde{\RV}}) \longrightarrow (\widetilde{\RT}, \widehat\C^{\widetilde{\RV}}), \widetilde{\delta}, (\overline{\widetilde{\delta}}_p)_{p\in \widetilde{\RV}}\right)
$$ 
dominating $\kappa$ via $\pi$ (see \S~\ref{singnature_order_subsec}), where the functions $\widetilde{\delta}$, $\overline{\widetilde{\delta}}_p$, $p\in\widetilde{\RV}$, are defined using the global and local degrees of the rational maps {$R_p:=\pmb{R}_\infty\vert_{\widehat{\C}_p}$}, so that they satisfy the conditions of Definition~\ref{rat_kappa_def}.

We will also show that $\widetilde{\kappa}$ satisfies the condition of simplicity introduced in Definition~\ref{simple_sign_def}. The key idea of the proof is that the map $\eta(z)=1/z$ essentially acts as a rotation (and hence an almost isometry) near its fixed points $\pm 1$.

Let $\{p_1, p_2,..., p_m\} \subseteq \widetilde{\RV}$ be a clique in $\widetilde{\RT}$ consisting of pairwise adjacent vertices.
Then for any $i, j, k$, we have $x_{p_i\to p_k} = x_{p_j\to p_k}$. We denote this point by $x_k$ and the local degree of $R_{p_k}$ at $x_k$ by $\delta_k$.

\begin{lem}\label{lem_2}
    Let $\{p_1, p_2,..., p_m\} \subseteq \widetilde{\RV}$ be a clique in $\widetilde{\RT}$ consisting of pairwise adjacent vertices.
    Then $m = 2$ and $\delta_1 = \delta_2 = 1$.

    In particular, $\widetilde{\RT}$ is a tree and the characteristic data $\widetilde{\kappa}$ is simple.
\end{lem}
\begin{proof}
    For simplicity of the presentation, we assume that $p_1, ..., p_m$ are fixed under $\widetilde{\pmb{\tau}}$. The other cases can be treated similarly. 
    
    Let $x_i := x_{p_j\to p_i}$ for (any) $i \neq j$, and let $\delta_i$ be the local degree of $R_{p_i}$ at $x_i$.
    Note that by compatibility, we have $\eta(x_i) = x_i$ and $R_{p_i}(x_i) = R_{p_j}(x_j)$. Denote this point by $y := R_{p_i}(x_i)$.
    We will show that $\sum_{i=1}^m \delta_i \leq 2$.

    Let $A = B(y, 2\epsilon) - \overline{B(y, \epsilon)}$ be a small annulus around $y$. Then there exists some small annulus $A_i$ around $x_i$ so that $R_{p_i}: A_i \longrightarrow A$ is a degree $\delta_i$ covering.
    Let us endow $A$ with the measure $\mu_A$ coming from the flat metric $|dz|/|z-y|$. Similarly, we endow $A_i$ with the measure $\mu_{A_i}$ coming from the metric $|dz|/|z-x_i|$. Note that $R_{p_i}$ expands the flat metric by a factor of $\delta_i$.
    
    Recall that $\pmb{R}_n$ converges to $R_{p_i}$ under the appropriate rescaling, which we assume to be the identity map for simplicity of notation. Hence, we can choose a component $A_{i, n}$ of $\pmb{R}_n^{-1}(A)$ that approximates $A_i$ such that $\pmb{R}_n: A_{i,n} \longrightarrow A$ converges uniformly to $R_{p_i}: A_i \longrightarrow A$.
    Since $\pmb{R}_n$ is univalent on $\pmb{\mathfrak{D}}_n \longrightarrow \Omega_n$, we conclude that 
    $$
    \mu_A(A) \geq \mu_A(\Omega_n\cap A) \succeq \sum_{i=1}^m \delta_i^2 \mu_{A_{i,n}}(\pmb{\mathfrak{D}}_n \cap A_{i,n}).
    $$
    On the other hand, since $\eta$ almost preserves $A_{i,n}$ and the corresponding flat metric, we have
    $$
    \mu_{A_{i,n}}(\pmb{\mathfrak{D}}_n \cap A_{i,n}) \asymp \mu_{A_{i,n}}(\eta(\pmb{\mathfrak{D}}_n) \cap A_{i,n}) \asymp \frac{1}{2} \mu_{A_{i,n}}(A_{i,n}) \asymp \frac{1}{2\delta_i}\mu_A(A),
    $$
where the last asymptotics follows from the fact that $\pmb{R}_n: A_{i,n} \longrightarrow A$ expands the flat metric by $\delta_i$ thus the area by $\delta_i^2$, and the fact that its image covers the annulus $A$ $\delta_i$-times.
    Thus,
    $$
    \mu_A(A) \succeq  \sum_{i=1}^m \delta_i^2 \mu_{A_{i,n}}(\pmb{\mathfrak{D}}_n \cap A_{i,n}) \asymp \sum_{i=1}^m \frac{\delta_i}{2}\mu_A(A).
    $$
    This implies that $\sum_{i=1}^m \delta_i \leq 2$. Since each clique contains at least two vertices, we conclude that $m = 2$ and $\delta_1 = \delta_2 = 1$.
\end{proof}

\noindent\textbf{The limit of the B-involutions $F_n$ is a B-involution.}
{We will now study the geometry of the limit of the pinched polygons $\overline{\Omega_n}$, which are the domains of definition of the B-involutions $F_n$ (see \S~\ref{degen_poly_like map_subsubsec} for the notion of a pinched polygon).}

We recall the conformal maps $\Phi_\bullet:(\D, \Pi) \longrightarrow (\Delta_\bullet, Q_\bullet)$, $\bullet=n,\infty$, introduced before Lemma~\ref{lem_1}, and observe that due to conformality of $\Phi_\infty$, the set $\partial Q_\infty\cap \partial\Delta_\infty$ is a disjoint union of non-singular real-analytic arcs. We will show that the ends of each of these arcs land on $\partial\Delta_\infty$. {This will, in particular, imply that the limiting compact set of $\overline{\Omega_n}$ is also a pinched polygon with piecewise real-analytic boundary.}

Let $\gamma$ be one of the edges (in particular, a hyperbolic geodesic) of $\Pi \subseteq \D$. We parametrize the geodesic by arclength: $\gamma(t), t\in (-\infty, \infty)$.
Let $\gamma_n(t) := \Phi_n\circ \gamma(t)$ and $\gamma_\infty(t) := \Phi_\infty\circ \gamma(t)$, so the arcs $\gamma_\bullet\subset\partial Q_\bullet$ are hyperbolic geodesics in $\Delta_\bullet$ ($\bullet=n,\infty)$).
A point $x$ is said to lie in the \emph{(positive) accumulation set} of $(\gamma_n(t))_n$ if there exist subsequences $n_k, t_{n_k} \to \infty$ so that
$$
x = \lim_{k\to\infty} \gamma_{n_k}(t_{n_k}).
$$
We also define a \emph{cusp} of $\partial Q_\bullet$ as a singular point of of $\partial Q_\bullet$ that is not a cut-point.
\begin{lem}[Landing of limiting geodesics]\label{lem_3}
    After passing to a subsequence if necessary, the (positive) accumulation set of $(\gamma_n(t))_n$
    is a single point.
    In particular, $\gamma_\infty$ has a unique limit point on each of its two ends (where the two limit points may coincide) and the Hausdorff limit of $\overline{\gamma_n}$ is $\overline{\gamma_\infty}$.
\end{lem}
\begin{proof}
Let $\gamma^{+} = \gamma$ and let $\gamma^{-}$ be the adjacent edge of $\Pi \subseteq \D$ such that $\gamma^{-}$ has an endpoint at $\lim_{t\to+\infty}\gamma(t)$. 
Let us assume there exists a parabolic element $g\in \Gamma$ so that $g(\gamma^+) = \gamma^-$. The other case can be proved similarly.
We parametrize the geodesics $\gamma^\pm$ by arclength, $\gamma^{\pm}(t), t\in (-\infty, \infty)$, so that $g(\gamma^{+}(t)) = \gamma^{-}(t)$, and $\lim_{t\to\infty} \gamma^{+}(t) = \lim_{t\to\infty} \gamma^{-}(t) \in \partial \D$.
Similarly, let $\alpha^-(t) = g(\gamma^{-}(t))$ and $\alpha^+(t) = g^{-1}(\gamma^{+}(t))$ be parametrizations of the adjacent edges of the neighboring fundamental domains $g(\Pi), g^{-1}(\Pi)$ (see Figure~\ref{fig:Pi}).
Note that 
\begin{equation}
    \lim_{t\to\infty} d_\D(\alpha^{+}(t), \gamma^{+}(t)) = 
\lim_{t\to\infty} d_\D(\gamma^{+}(t), \gamma^{-}(t)) = \lim_{t\to\infty} d_\D(\gamma^{-}(t), \alpha^{-}(t)) = 0.\label{eqn:hd}
\end{equation}
Let 
$$
\gamma^\pm_\bullet(t) := \Phi_\bullet(\gamma^\pm(t)),\ \alpha^\pm_\bullet(t) := \Phi_\bullet(\alpha^\pm(t)),\quad \bullet= n, \infty.
$$
By equation~\eqref{eqn:hd} and by comparing the Euclidean distance and hyperbolic distance, we conclude that there exists $\delta_t \to 0$ as $t \to \infty$ so that
\begin{equation}
d_\C(\alpha^+_n(t), \gamma^{+}_n(t)), d_\C(\gamma^{+}_n(t), \gamma^{-}_n(t)), d_\C(\gamma^{-}_n(t), \alpha^-_n(t)) \leq \delta_t.\label{eqn:ed}
\end{equation}
\begin{figure}[ht]
\captionsetup{width=0.98\linewidth}
\begin{tikzpicture}
\node[anchor=south west,inner sep=0] at (0,0) {\includegraphics[width=0.96\textwidth]{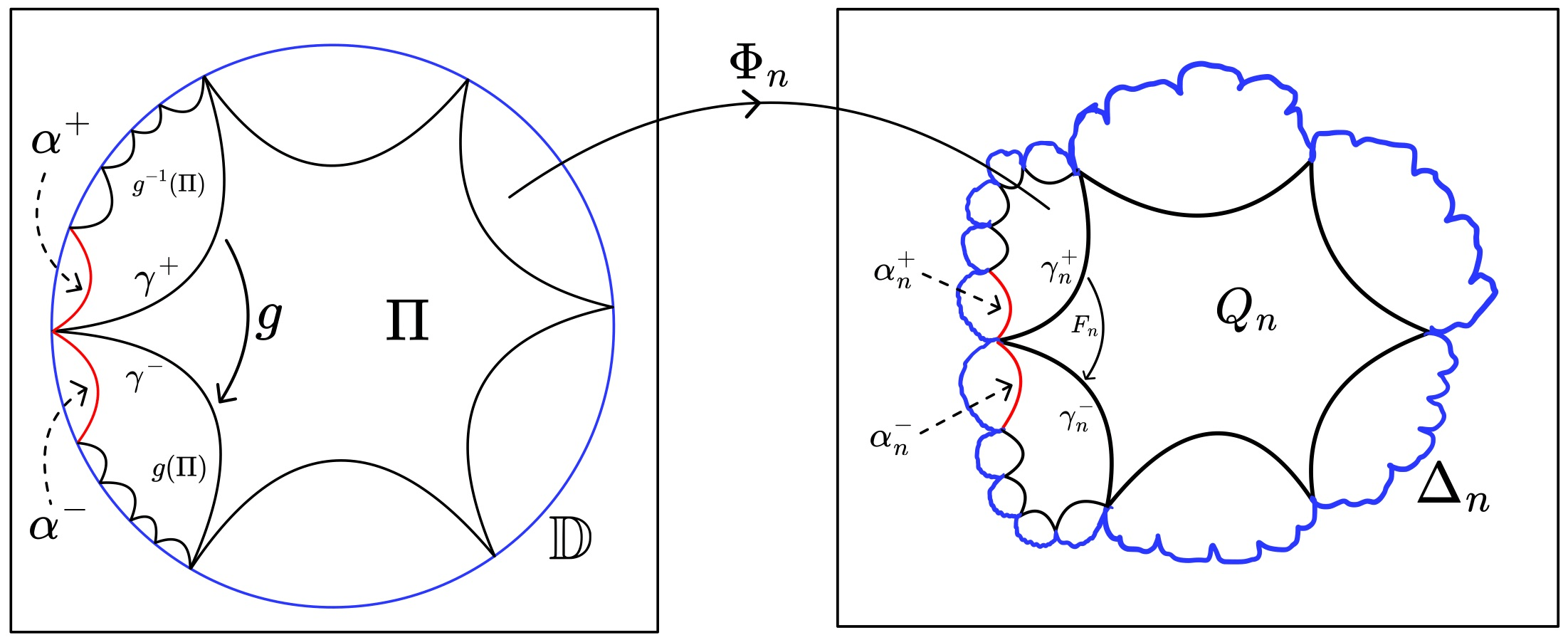}};
\node[anchor=south west,inner sep=0] at (2.8,-6.6) {\includegraphics[width=0.75\textwidth]{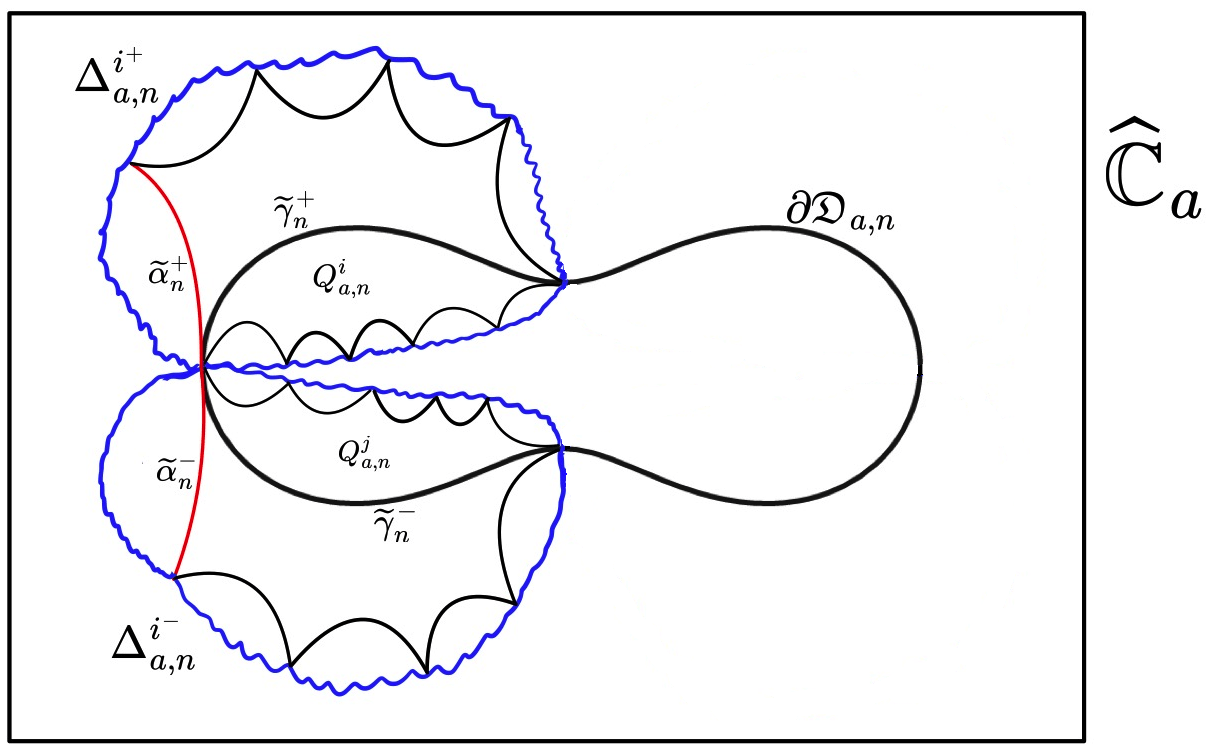}};
\draw [->, line width=0.4mm] (7.2,-1.2) to (8,0.4);
\node at (8.12,-0.4) {$R_{a,n}$};
\end{tikzpicture}
\caption{Illustrated is the proof of Lemma~\ref{lem_3}. The domain $\mathfrak{D}_{a,n}$ is the exterior of the black curve marked as $\partial\mathfrak{D}_{a,n}$.}
\label{fig:Pi}
\end{figure}

By equation~\eqref{eqn:ed}, the accumulation sets of $(\gamma^{\pm}_n(t))_n$ and $(\alpha^\pm_n(t))_n$ are all the same.
After passing to a subsequence, we assume that there exists some $t_n \to \infty$ so that $\gamma^{\pm}_{n}(t_{n})$ converges.

Let $\widetilde{\alpha}^\pm_n(t) = \left(\pmb{R}_n|_{\pmb{\mathfrak{D}}_n}\right)^{-1}(\alpha^\pm_n(t))$.
Similarly, let $\widetilde{\gamma}^\pm_n\subset\partial\pmb{\mathfrak{D}}_n$ be the component of $\pmb{R}_n^{-1}(\gamma^\pm_n)$ that is contained in the same component of $\pmb{R}_n^{-1}(\Delta_n)$ as $\widetilde{\alpha}^\pm_n$ (see Figure~\ref{fig:Pi}).
The arcs $\widetilde{\alpha}^\pm_\infty(t)$ and $\widetilde{\gamma}^\pm_\infty(t)$ are defined similarly.

Let us assume that $\widetilde{\alpha}^\pm_n$ are contained in the same Riemann sphere $\widehat{\C}_a, a\in \RV$. The other cases can be proved similarly.
    Let $\Delta_{a,n}^{i^\pm}$ be the components of $\pmb{R}_n^{-1}(\Delta_n)$ containing $\widetilde{\alpha}^\pm_n$, where $i^\pm\in\{1,\cdots,\delta(a)\}$.
    We consider two cases.
\smallskip

\noindent\textbf{Case (a):} $\widetilde{\alpha}^\pm_\infty$ \textbf{are contained in the same Riemann sphere} $\widehat{\C}_p$, $p \in\RV_a\subset \widetilde{\RV}$. 

\noindent Note that under the rescaling $\mathscr{M}_{a,n}^p$ associated with $\widehat{\C}_p$, the sequence $R_{a,n}$ converges to $R_p\equiv\pmb{R}_\infty\vert_{\widehat{\C}_p}$. To simplify the notation, we assume that $\mathscr{M}_{a,n}^p(z) = z$. 
Denote the accumulation set of $(\widetilde{\alpha}^\pm_n(t))_n$ by $L^\pm$. Then the accumulation sets of $(\widetilde{\gamma}^{\pm}_n(t))_n$ are $L^\pm$ as well. 
    Note that $\eta(L^\pm) = L^\mp$, as $\eta(\widetilde{\gamma}^{\pm}_n(t)) = \widetilde{\gamma}^{\mp}_n(t)$.
    Recall that $A_\Gamma$ is the Bowen-Series map.
    Since $A_\Gamma(\alpha^+) = A_\Gamma(\gamma^-) = \gamma^+$, and hence $F_n(\alpha_n^+) = F_n(\gamma_n^-) = \gamma_n^+$, we conclude that 
    $$
    R_{a,n}(\widetilde{\gamma}^{+}_n(t))= \gamma^+_n(t) = R_{a,n}(\eta(\widetilde{\alpha}^{+}_n(t))).
    $$
Let $x = \lim_{k}\widetilde{\gamma}^{+}_{n_k}(t_{n_k}) = \lim_{k}\widetilde{\alpha}^{+}_{n_k}(t_{n_k}) \in L^+$, and $\eta(x) = \lim_{k}\eta(\widetilde{\alpha}^{+}_{n_k}(t_{n_k}))$. Thus, we have $R_p(x) = R_p(\eta(x))$. 

We now argue by contradiction.
    Suppose  that $L^+$ is not a single point.
    Since $L^+$ is closed and connected, $L^+- (\Xi_p \cup \eta(\Xi_p))$ has an accumulation point. Thus, $R_p(z) = R_p(\eta(z))$ for all $z \in \widehat{\C}_p$. But this is not possible as $\eta(\widetilde{\gamma}^+_\infty) = \widetilde{\gamma}^-_\infty$, but their images $R_p(\widetilde{\gamma}^\pm_\infty) = \gamma_\infty^\pm$ are different.
    Thus, the accumulation set $L^+$ is a singleton set, denoted by $\{c\}$. 
    Since $L^+ \cup L^-$ is connected, $L^+ = L^- = \{c\}$ and $c = \eta(c) \in\partial\Delta_{a,\infty}^{i^+}\cap\partial\Delta_{a,\infty}^{i^-}$, where $\Delta_{a,\infty}^{i^\pm}$ are Carath{\'e}odory limits of $(\mathscr{M}_{a,n}^p)^{-1}(\Delta_{a,n}^{i^\pm})$. Note that $\Delta_{a,\infty}^{i^+}$ and $\Delta_{a,\infty}^{i^-}$ are mapped univalently onto $\Delta_\infty$ under $R_p$. It follows that the local degree of the rational map $R_p$ at $x$ is at least $2$. Thanks to Lemma~\ref{lem_2}, we conclude that $x\notin \Xi_p$.
    Thus, $c= \lim_n c_n$, where $\{c_n\} = \overline{\widetilde{\gamma}^{+}_n} \cap \overline{\widetilde{\gamma}^{-}_n}$ is the critical point associated with the corresponding cusp $R_{a,n}(c_n)$ on $\partial Q_n$.
    Therefore, the accumulation set of $(\gamma^\pm_n(t))_n$ is also a single point; namely, $R_{p}(c)$.
\smallskip

\noindent\textbf{Case (b):} $\widetilde{\alpha}^\pm_\infty$ \textbf{are contained in two different Riemann spheres} $\widehat{\C}_{p^\pm}, {p^\pm} \in \widetilde{\RV}$.

\noindent   Let $x^\pm \in \Xi_{p^\pm}$ be the singular points associated with the spheres $\widehat{\C}_{p^\mp}$ respectively.
    Let $L^\pm$ be the accumulation set of $(\widetilde{\gamma}^{\pm}_n(t))_n$ (with respect to the appropriate rescalings) in $\widehat{\C}_{p^\pm}$ .
    Since $\overline{\widetilde{\gamma}^{+}_n} \cup \overline{\widetilde{\gamma}^{+}_n}$ is connected, $x^\pm \in L^\pm$.

    We again argue by contradiction.
    Suppose  that $L^\pm \neq \{x^\pm\}$. Then a similar argument shows that for all $z \in \widehat{\C}_{p^\pm}$,
    $$
    R_{p^\mp} \circ \eta(z) = R_{p^\pm}(z).
    $$
    But this is a contradiction as $\eta(\widetilde{\gamma}^+_\infty) = \widetilde{\gamma}^-_\infty$, but their images $R_{p^\pm}(\widetilde{\gamma}^\pm_\infty) = \gamma_\infty^\pm$ are different.
    Therefore, the accumulation set of $(\gamma^\pm_n(t))_n$ is also a single point; namely, $R_{p^+}(x^+)=R_{p^-}(x^-)$.
\end{proof}

As a consequence of Lemma~\ref{lem_3}, we have the following facts. Recall that for $a\in\RV$, $p\in\RV_a$, the M{\"o}bius maps $\mathscr{M}_{a,n}^p$ are rescalings associated with~$\RT_a$.
\begin{enumerate}[leftmargin=*]
    \item The sets $\overline{\Omega_n}$ converge in the Hausdorff topology to some pinched polygon $\overline{\Omega_\infty}$ with piecewise (non-singular) real-analytic boundary, where $\Omega_\infty=\Int{\overline{\Omega_\infty}}$ is the union of all Carath{\'e}odory limits of the components of~$\Omega_n$. 

    \item For $a\in\RV$, $p\in\RV_a$, the closed Jordan disks $\left(\mathscr{M}_{a,n}^p\right)^{-1}\left(\overline{\mathfrak{D}_{a,n}}\right)$ converge to a closed Jordan disk $\overline{\mathfrak{D}_{p,\infty}}\subset\widehat{\C}_p\subset\widehat{\C}^{\widetilde{\RV}_a}$ with piecewise (non-singular) real-analytic boundary. The union of these pinched polygons $\overline{\mathfrak{D}_{p,\infty}}$, $p\in\widetilde{\RV}$, is a pinched polygon $\overline{\pmb{\mathfrak{D}}_\infty}\subset\widehat{\C}^{\widetilde{\RV}}$, where $\overline{\mathfrak{D}_{p,\infty}}, \overline{\mathfrak{D}_{q,\infty}}$ ($p\neq q\in\widetilde{\RV}$) can touch only at a singular point of $\widehat{\C}^{\widetilde{\RV}}$.

    \item $\widetilde{\pmb{\eta}}(\pmb{\mathfrak{D}}_\infty)=\widehat{\C}^{\widetilde{\RV}}-\overline{\pmb{\mathfrak{D}}_\infty}$ and $\pmb{R}_\infty:\overline{\pmb{\mathfrak{D}}_\infty}\to\overline{\Omega_\infty}$ is a homeomorphism (recall that $\pmb{\eta}(\pmb{\mathfrak{D}}_n)=\widehat{\C}-\overline{\pmb{\mathfrak{D}_n}}$ and $\pmb{R}_n:\overline{\pmb{\mathfrak{D}}_n}\to\overline{\Omega_n}$ is a homeomorphism).

    \item The B-involutions $F_n\equiv \pmb{R}_n\circ\pmb{\eta}\circ\left(\pmb{R}_n\vert_{\overline{\pmb{\mathfrak{D}}_n}}\right)^{-1}:\overline{\Omega_n}\to\widehat{\C}$ converge compactly to the B-involution $F_\infty:= \pmb{R}_\infty\circ\widetilde{\pmb{\eta}}\circ\left(\pmb{R}_\infty\vert_{\overline{\pmb{\mathfrak{D}}_\infty}}\right)^{-1}:\overline{\Omega_\infty}\to\widehat{\C}$.
\end{enumerate}
\smallskip

\noindent{\textbf{Regularity of $\widetilde{\kappa}$.}} 
Note that the degree of $\kappa$ and $\widetilde{\kappa}$ are both equal to $2d$.
We say that a critical point of $\pmb{R}_\bullet$ is {\em free} if it is not in $\overline{\pmb{\mathfrak{D}}_\bullet}$.
We call the other critical points \emph{cusp critical points} as they are mapped to cusps of $\partial\Omega_\bullet$ under $\pmb{R}_\bullet$ (recall that by Lemma~\ref{lem_2}, the singular points of $\widehat{\C}^{\widetilde{\RV}}$ are not critical points for any of the rational maps $R_p=\pmb{R}_\infty\vert_{\widehat{\C}_p}:\widehat{\C}_p\to\widehat{\C}$, $p\in\widetilde{\RV}$).

Let us list the free critical points of $\pmb{R}_n$ by $c_{1,n}, \cdots, c_{2d-2,n}$, and the cusp critical point by $\widehat{c}_{1,n},\cdots, \widehat{c}_{k,n}$, where $k = 2(d+1 - |\RV|)$ (by the proof of \cite[Proposition~15.6]{LLM24}).
After passing to a subsequence, we may assume that for all $a \in \RV$ and $p\in\RV_a\subset\widetilde{\RV}$, the limits 
$$
\lim_p c_{i,n}\coloneqq \lim_{n\to\infty} \left(\mathscr{M}_{a,n}^p\right)^{-1}(c_{i,n}) \text{ and }\lim_p \widehat{c}_{j,n}\coloneqq \lim_{n\to\infty} \left(\mathscr{M}_{a,n}^p\right)^{-1}(\widehat{c}_{j,n})  
$$
exist in $\widehat{\C}_p$ whenever $c_{i,n}, \widehat{c}_{j,n}\in\widehat{\C}_a$.
Note that by Lemma~\ref{lem_2}, we have for each sequence $c_{i,n}$ (or $\widehat{c}_{j,n}$) that
\begin{enumerate}[leftmargin=*]
    \item either there exists a unique $p \in \widetilde{\RV}$ such that $\displaystyle\lim_p c_{i,n}\in\widehat{C}_p - \Xi_p$; or 
    \item there exists a unique pair of adjacent vertices $p, q \in \widetilde{\RV}$ so that $\displaystyle\lim_p c_{i,n} = \xi_p(v)$ and $\displaystyle\lim_q c_{i,n} = \xi_q(w)$, where $v \in T_p\widetilde{\RT}$ is the tangent vector in the direction of $q$ and $w \in T_q\widetilde{\RT}$ is the tangent vector in the direction of $p$.
\end{enumerate}
We say $c_{i,n}$ (or $\widehat{c}_{j,n}$) converges to a non-singular point in $\widehat{\C}_p$ in the former case, and converges to the edge $[p,q]$ in the latter case.

Let $[p,q] \subseteq \widetilde{\RT}$ be an edge that projects to a single vertex of $\RT$ under~$\pi$. Then by Lemma~\ref{lem_2}, there are exactly two sequences of critical points converging to the edge $[p,q]$. Indeed, setting $y:=R_p(\xi_p(v))=R_q(\xi_q(w))\in\partial\Omega_\infty$, we see by Lemma~\ref{lem_2} that for $\epsilon>0$ small, there is a component $W_p\ni\xi_p(v)$ of $R_p^{-1}(B(y,\epsilon))$ and a component $W_q\ni\xi_q(w)$ of $R_q^{-1}(B(y,\epsilon))$ that map conformally onto $B(y,\epsilon)$ under $R_p$ and $R_q$, respectively. By Lemma~\ref{lem_2}, for $n$ large enough, $R_n^{-1}(B(y,\epsilon))$ has a component that corresponds to $W_p\cup W_q$ (under the rescalings), and this component is a conformal annulus that maps onto $B(y,\epsilon)$ as a $2:1$ branched cover under $R_n$. Hence, this component contains exactly two critical points of $R_n$ that converge to $\xi_p(v), \xi_q(w)$ under the rescalings.
By the proof of Lemma~\ref{lem_3}, these two sequences of critical points must be cusp critical points.
Therefore, we have that the number of free critical points and cusp critical points for $\pmb{R}_\infty$ are $2d-2$ and $2(d+1 - |\widetilde{\RV}|)$ respectively.

\begin{lem}\label{lem_4}
$\widetilde{\kappa}$ is regular in the sense of Definition~\ref{simple_sign_def}.
\end{lem}
\begin{proof}
    {Suppose that $\widetilde{\kappa} = \kappa$, then it is regular by Proposition~\ref{total_space_regular_lem}.
    Thus, we assume $\widetilde{\kappa}\neq \kappa$.
    By Remark~\ref{rem:regular} and Lemma~\ref{lem_2}, it suffices to verify  condition~\eqref{reg:item:1} in the definition of regularity. }
    By way of contradiction, suppose that there exists $a \in \RV$ such that $\RV_a\subset\widetilde{\RV}$ consists of exactly two vertices $p, q$; and $\overline{\widetilde{\delta}}_{p}(v_q) = \widetilde{\delta}(p)$ and $\overline{\widetilde{\delta}}_{q}(v_p) = \widetilde{\delta}(q)$, where $v_q \in T_{p}\RT_a$ is the tangent vector in the direction of $q$, and $v_p \in T_{q}\RT_a$ is the tangent vector in the direction of $p$. By Lemma~\ref{lem_2}, $\overline{\widetilde{\delta}}_{q}(v_p)=\overline{\widetilde{\delta}}_{q}(v_p)=1$, and hence $\deg(R_p)=\widetilde{\delta}(p)=\widetilde{\delta}(q)=\deg(R_q)=1$. It follows that $R_{a,n}:\widehat{\C}_a\to\widehat{\C}$ is a degree $2$ rational map that degenerates to a pair of M{\"o}bius maps $\widehat{\C}_p\to\widehat{\C}$ and $\widehat{\C}_q\to\widehat{\C}$. Since $\deg(R_{a,n})=2$, it follows from the analysis of \cite[\S 15.1, \S 15.2]{LLM24} that $R_{a,n}$ has at most one cusp critical point. But this contradicts the discussion preceding this lemma. { It follows that $\widetilde{\kappa}$ is regular.}
\end{proof}

\noindent\textbf{Finishing the proof of Proposition~\ref{limit_corr_prop}.}
{ Since $\widetilde{\kappa}$ is regular (by Lemma~\ref{lem_4}), if $\widetilde{\RT}$ is non-trivial, then $\pmb{R}_\infty$ is regular. Otherwise, both $\RT$ and $\widetilde{\RT}$ are trivial, and the same argument in the last paragraph of the proof of Theorem \ref{pre_comp_thm_1} gives that the two fixed points $\pm 1$ of $\eta$ are simple critical points for the limiting map $\pmb{R}_\infty$, proving regularity of $\pmb{R}_\infty$. Thus, it now follows from the construction of the rational map $\pmb{R}_\infty:(\widetilde{\RT}, \widehat{\C}^{\widetilde{\RV}})\to\widehat{\C}$ that the sequence $\pmb{R}_n\in\Rat_{\kappa}^{reg}(\C)$ converges to $\pmb{R}_\infty\in\Rat_{\widetilde{\kappa}}^{reg}(\C)\subset\mathcal{CV}^{reg}_{2d}$ (see Definition~\ref{char_var_top_def}  for the definition of the convergence). By Proposition~\ref{prop:convergenceCorrespondence} and the construction of the correspondence $\mathfrak{C}_\infty$, we now have that $\mathfrak{C}_n\to\mathfrak{C}_\infty$.}
\end{proof}

{At this point, the main result of this section is only a short distance away.}

\begin{proof}[Proof of Theorem~\ref{vert_slice_comp_thm}]
We continue to use the notation of Proposition~\ref{limit_corr_prop}.
To complete the proof of the theorem, we need to justify that the subsequential limit { $\mathfrak{C}_\infty\in\mathcal{CV}^{reg}_{2d}$} of $\{\mathfrak{C}_n\}$ lies in $\mathscr{B}_\Gamma$.
By definition, this task is equivalent to showing that the B-involution $F_\infty:= \pmb{R}_\infty\circ\widetilde{\pmb{\eta}}\circ\left(\pmb{R}_\infty\vert_{\overline{\pmb{\mathfrak{D}}_\infty}}\right)^{-1}:\overline{\Omega_\infty}\to\widehat{\C}$ lies in $\mathscr{B}_\Gamma$ (see \S~\ref{identify_b_inv_corr_subsec}).

To this end, recall that $F_\infty$ is the limit of the B-involutions
$F_n$ (in the Carath{\'e}odory topology on holomorphic functions). Hence, the conformal conjugacies $\Phi_n:(\D,\Pi)\to(\Delta_n,Q_n)$ between the Bowen-Series map $A_\Gamma$ and $F_n$ converge compactly to a conformal conjugacy $\Phi_\infty:(\D,\Pi)\to(\Delta_\infty,Q_\infty)$ between $A_\Gamma$ and $F_\infty$ (where $\Delta_\bullet$ is the escaping set of $F_\bullet$ and $Q_\bullet$ is the `desingularized' complement of the inversive multi-domain $\Omega_\bullet$). In other words, the restriction $F_\infty:\overline{F_\infty^{-1}(\Omega_\infty)}\to\overline{\Omega_\infty}$ is a degenerate polynomial-like map having $A_\Gamma$ as its external map. This completes the proof.
\end{proof}

\begin{remark}
{Recall from Remark~\ref{fixed_ray_lami_rem} that a finite geodesic lamination, termed as the coarse lamination, models the topology of the domain of definition of a B-involution in $\mathscr{B}_\Gamma$ (cf. \cite[\S 15.1]{LLM24}).} It is easy to see that the tree $\widetilde{\RT}$ (associated with the characteristic data $\widetilde{\kappa}$ of the limiting correspondence $\mathfrak{C}_\infty$) constructed in the proof of Proposition~\ref{limit_corr_prop} is dual to the coarse lamination of the B-involution~$F_\infty$.
\end{remark}

\begin{remark}\label{rmk:orbifoldModification_1}
The proof of Theorem~\ref{vert_slice_comp_thm} goes through unchanged in the genus $0$ orbifold case and yields compactness of the corresponding external fibers (see \cite[\S 15]{LLM24} for the definition of external fibers of general genus $0$ orbifolds).
\end{remark}


\begin{thebibliography}{BLLM24}

\bibitem[Ago04]{Ago04}
I.~Agol.
\newblock {\em Tameness of hyperbolic 3-manifolds}.
\newblock \url{https://arxiv.org/abs/math/0405568}, 2004.

\bibitem[Arf17]{arfeux}
M.~Arfeux. 
\newblock {\em Dynamics on trees of spheres.}
\newblock J. Lond. Math. Soc. (2) 95, no. 1, 177--202, 2017. 

\bibitem[BCM12]{minsky-elc2}
J.~Brock, R.~Canary, Y.~Minsky. 
\newblock {\em The classification of Kleinian surface groups, II: The ending lamination conjecture.}
\newblock Ann. of Math. (2) 176, no. 1, 1--149, 2012. 

\bibitem[BR10]{BR10}
M.~Baker and R.~Rumely.
\newblock {\em Potential theory and dynamics on the {B}erkovich projective line}.
\newblock Math. Surveys Monogr., 159. American Mathematical Society, Providence, RI, 2010. xxxiv+428 pp.

\bibitem[Ber60]{Ber60}
L.~Bers.
\newblock{Simultaneous uniformization}.
\newblock {\em Bull. Amer. Math. Soc.}, 66 94--97, 1960.

\bibitem[Ber70]{Ber70}
L.~Bers.
\newblock{On boundaries of {T}eichm\"{u}ller spaces and on {K}leinian groups. {I}}.
\newblock {\em Ann. of Math. (2)}, 91:570--600, 1970.

\bibitem[Ber81]{Ber81}
L.~Bers.
\newblock The action of the modular group on the complex boundary.
\newblock In {\em Riemann surfaces and related topics: Proceedings of the 1978 Stony Brook Conference (State Univ. New York, Stony Brook, N.Y., 1978)}, pp. 33-–52,
Ann. of Math. Stud., No. 97, Princeton University Press, Princeton, NJ, 1981.

\bibitem[BJ97]{BJ97}
C.~Bishop and P.~Jones.
\newblock Hausdorff dimension and Kleinian groups.
\newblock {\em Acta Math.}, 179:1--39, 1997.

\bibitem[BS79]{BS79}
R.~Bowen and C.~Series.
\newblock Markov maps associated with {F}uchsian groups.
\newblock {\em Inst. Hautes \'{E}tudes Sci. Publ. Math.}, 50:153--170, 1979.

\bibitem[BF14]{BF14} 
B.~Branner and N.~Fagella, 
\newblock \emph{Quasiconformal Surgery in Holomorphic Dynamics}. 
\newblock Cambridge Studies in Advanced Mathematics, vol.\ 141, Cambridge University Press, Cambridge, 2014. 

\bibitem[BCM12]{BCM12}
J.~F. Brock, R.~D. Canary, and Y.~N. Minsky.
\newblock The classification of {K}leinian surface groups {II}: {T}he
  ending lamination conjecture.
\newblock {\em Ann. of Math. (2)}, 176:1--149, 2012.

\bibitem[Br11]{Br11}
K.~Bromberg.
\newblock The space of Kleinian punctured torus groups is not locally connected. 
\newblock {\em Duke Math. J.}, (156) 387--427, 2011. 

\bibitem[BF03]{BF03}
S.~Bullett and M.~Freiberger.
\newblock Hecke groups, polynomial maps and matings.
\newblock {\em Int. J. Mod. Phys. B}, 17:3922--3931, 2003.

\bibitem[BF05]{BF05}
S.~Bullett and M.~Freiberger.
\newblock Holomorphic correspondences mating {C}hebyshev-like maps with {H}ecke groups.
\newblock {\em Ergodic Theory Dynam. Systems}, 25:1057--1090, 2005.

\bibitem[BH07]{BH07}
S.~Bullett and P.~Ha{\"i}ssinsky.
\newblock Pinching holomorphic correspondences. 
\newblock {\em Conform. Geom. Dyn.}, 11:65--89, 2007.

\bibitem[BH00]{BH00}
 S. Bullett, and W. Harvey.
\newblock Mating quadratic maps with Kleinian groups via quasiconformal surgery.
\newblock {\em Electronic Research Announcements of the AMS}, 6: 21--30, 2000.     
     
\bibitem[BL20]{BL20}
S.~Bullett and L.~Lomonaco.
\newblock Mating quadratic maps with the modular group {II}.
\newblock {\em Invent. Math.}, 220:185--210, 2020.

\bibitem[BL24]{BL24} 
S.~Bullett and L. Lomonaco.
\newblock Mating quadratic maps with the modular group III: the modular Mandelbrot set.
\newblock {\em Adv. Math.}, 458:Paper No. 109956, 48 pp., 2024.

\bibitem[BLLM24]{BLLM24}
S.~Bullett, L.~Lomonaco, M.~Lyubich, and S.~Mukherjee.
\newblock Mating parabolic rational maps with Hecke groups.
\newblock {\em Proc. London Math. Soc. (3)}, vol. 132, no. 3, e70132, 2026.

\bibitem[BP94]{BP94}
S.~Bullett and C.~Penrose.
\newblock Mating quadratic maps with the modular group.
\newblock {\em Invent. Math.}, 115:483--511, 1994.

\bibitem[CG06]{CG06}
D.~Calegari and D.~Gabai.
\newblock Shrinkwrapping and the taming of hyperbolic 3-manifolds.
\newblock {\em J. Amer. Math. Soc.}, 19:385--446, 2006.

\bibitem[CT07]{CTpub}
J.~Cannon and W.~P. Thurston.
\newblock Group invariant {P}eano curves.
\newblock {\em Geom. Topol.} 11:1315--1355, 2007.

\bibitem[DeM05]{DeM05}
L.~DeMarco. 
\newblock Iteration at the boundary of the space of rational maps. 
\newblock {\em Duke Math. J.}, 130:169--197, 2005.

\bibitem[DH85]{DH85}
A.~Douady and J.~H.~Hubbard.
\newblock On the dynamics of polynomial-like mappings.
\newblock {\em Ann. Sci. {\'E}cole Norm. Sup. (4)}, 18:287--343, 1985.

\bibitem[DF14]{DF14}
L.~DeMarco and X.~Faber.
\newblock Degenerations of complex dynamical systems.
\newblock {\em Forum Math. Sigma}, 2:e6, 2014.

\bibitem[Fat19-20]{Fat19}
P.~Fatou.
\newblock Sur les \'{e}quations fonctionnelles.
\newblock {\em Bull. Soc. Math. France}, 47:161--271, 1919; 48:33--94 and 48:208--314, 1920.

\bibitem[Fat26]{Fat26}
P.~Fatou.
\newblock Sur l'it\'{e}ration des fonctions transcendantes {E}nti\`eres.
\newblock {\em Acta Math.}, 47(4):337--370, 1926.

\bibitem[Fat29]{Fat29}
P.~Fatou.
\newblock Notice sur les travaux scientifiques de {M}. {P}. {F}atou.
\newblock Astronome titulaire de l'observatoire de Paris, 5--29, 1929,
\url{https://www.math.purdue.edu/~eremenko/dvi/fatou-b.pdf}.

\bibitem[GH78]{GH78}
P.~Griffiths and J.~Harris, Joseph. 
\newblock Principles of algebraic geometry. 
\newblock Pure and Applied Mathematics. Wiley-Interscience [John Wiley \& Sons], New York, 1978. 

\bibitem[HLP25]{HLP25}
Y.~M.~He, H.~Lee, and I.~Park.
\newblock Analytic theory on the space of Blaschke products: simultaneous uniformization and pressure metric.
\newblock \url{https://arxiv.org/abs/2507.17077}, 2025.

\bibitem[Ino09]{Ino09} 
H. Inou. 
\newblock Combinatorics and topology of straightening maps II: Discontinuity.
\newblock \url{http://arxiv.org/abs/0903.4289}, 2009.

\bibitem[IM21]{IM21}
H. Inou and S. Mukherjee. 
\newblock Discontinuity of straightening in anti-holomorphic
dynamics: I. 
\newblock {\em Trans. Amer. Math. Soc.}, 374:6445--6481, 2021.

\bibitem[Jul18]{Jul18}
G.~ Julia.
\newblock M{\'e}moire sur l'it{\'e}ration des fonctions rationnelles.
\newblock {\em Journal de Math{\'e}matiques Pures et Appliqu{\'e}es}, 1:47--246, 1918.

\bibitem[Jul22]{Jul22}
G.~Julia.
\newblock M\'{e}moire sur la permutabilit\'{e} des fractions rationnelles.
\newblock {\em Ann. Sci. \'{E}cole Norm. Sup. (3)}, 39:131--215, 1922.

\bibitem[Kap01]{kapovich-book}
M.~Kapovich.
\newblock Hyperbolic manifolds and discrete groups.
\newblock {\em Progress in Mathematics}, 183. Birkhauser Boston, Inc., Boston, MA. xxvi+467 pp., 2001.

\bibitem[Kiw15]{Kiw15}
J. Kiwi.
\newblock Rescaling limits of complex rational maps.
\newblock {\em Duke. Math. J.}, 164(7): 1437-1470, 2015.

\bibitem[KT90]{KT90}
S.~P.~Kerckhoff and W.~P.~Thurston. 
\newblock Non-continuity of the action of the modular group at Bers' boundary of Teichm{\"u}ller space. 
\newblock {\em Invent. Math.} 100:25--47, 1990.

\bibitem[Kle83]{Kle83}
F.~Klein.
\newblock Neue beitr{\"a}ge zur {R}iemann'schen functionentheorie.
\newblock {\em Math. Ann.}, 21:141--218, 1883.

\bibitem[Lau24]{Lau24}
M.~R.~Laude.
\newblock Continuity of matings of {K}leinian groups and polynomials.
\newblock \url{https://arxiv.org/abs/2411.08748}, 2024.

\bibitem[LMM21]{LMM21}
K.~Lazebnik, N.~G. Makarov, and S.~Mukherjee.
\newblock Univalent polynomials and {H}ubbard trees.
\newblock {\em Trans. Amer. Math. Soc.}, 374(7): 4839--4893, 2021.

\bibitem[LMM22]{LMM22}
K.~Lazebnik, N.~G. Makarov, and S.~Mukherjee.
\newblock Bers slices in families of univalent maps.
\newblock {\em Math. Z.}, 300:2771--2808, 2022.

\bibitem[LLMM21]{LLMM3}
S.-Y. Lee, M.~Lyubich, N.~G. Makarov, and S.~Mukherjee.
\newblock Schwarz reflections and anti-holomorphic correspondences.
\newblock {\em Adv. Math.}, 385:Paper No. 107766, 88 pp., 2021.

\bibitem[LLMM23]{LLMM1}
S.-Y. Lee, M.~Lyubich, N.~G. Makarov, and S.~Mukherjee.
\newblock Dynamics of {S}chwarz reflections: the mating phenomena.
\newblock {\em Ann. Sci. {\'E}c. Norm. Sup{\'e}r. (4)}, 56:1825--1881, 2023.

\bibitem[LLMM25]{LLMM2}
S.-Y. Lee, M.~Lyubich, N.~G. Makarov, and S.~Mukherjee.
\newblock {S}chwarz reflections and the {T}ricorn.
\newblock \emph{Ann. Inst. Fourier (Grenoble)}, 75:1987--2100, 2025.

\bibitem[LM16]{LM16}
S.~Y. Lee and N.~G. Makarov.
\newblock Topology of quadrature domains.
\newblock {\em J. Amer. Math. Soc.}, 29(2):333--369, 2016, arXiv:1307.0487.  



\bibitem[Luo21]{Luo21}
Y. Luo.
\newblock Limits of rational maps, $\R$-trees and barycentric extension.
\newblock {\em Adv. Math.}, 394: 108075, 2021.

\bibitem[Luo22a]{Luo22a}
Y. Luo.
\newblock Trees, length spectra for rational maps via barycentric extensions and Berkovich spaces.
\newblock {\em Duke. Math. J.}, 171(14): 2943--3001, 2022.

\bibitem[Luo22b]{Luo22b}
Y. Luo.
\newblock On geometrically finite degenerations {II}: convergence and divergence.
\newblock {\em Trans. Amer. Math. Soc.}, 375: 3469--3527, 2022.

\bibitem[LLM24]{LLM24}
Y.~Luo, M.~Lyubich, and S.~Mukherjee.
\newblock A general dynamical theory of Schwarz reflections, B-involutions, and algebraic correspondences.
\newblock \url{https://arxiv.org/abs/2408.00204}, 2024.

\bibitem[LMM24]{LMM24}
M.~Lyubich, J.~Mazor, and S.~Mukherjee.
\newblock Antiholomorphic correspondences and mating I: realization theorems.
\newblock \emph{Comm. Amer. Math. Soc.}, 4:495--547, 2024.

\bibitem[LMMN25]{LMMN25}
M.~Lyubich, S.~Merenkov, S.~Mukherjee, and D.~Ntalampekos.
\newblock David extension of circle homeomorphisms, welding, mating, and removability.
\newblock \emph{Mem. Amer. Math. Soc.}, vol. 313, no. 1588, pp. v+110, 2025.


\bibitem[Mar16]{Mar16} 
A.~Marden.
\newblock \emph{Hyperbolic manifolds, an introduction in 2 and 3 dimensions}.
\newblock Cambridge University Press, 2016.

\bibitem[Mas70]{Mas70}
B.~Maskit.
\newblock On boundaries of Teichm{\"u}ller spaces and on Kleinian groups. II.
\newblock {\em Ann. of Math. (2)}, 91:607--639, 1970.

\bibitem[McM88]{McM88}
C.~McMullen,
\newblock Automorphisms of rational maps.
\newblock In {\em Holomorphic functions and moduli, vol. I (Berkeley, CA, 1986)}, 31--60, Math. Sci. Res. Inst. Publ., 10, Springer, New York, 1988. 

\bibitem[McM91]{McM91}
C.~McMullen.
\newblock Rational maps and Teichm{\"u}ller space: analogies and open problems.
\newblock \url{https://people.math.harvard.edu/~ctm/papers/home/text/papers/probs-91/probs-91.pdf}, 1991.

\bibitem[McM94]{McM94}
C. McMullen.
\newblock {\em Complex Dynamics and Renormalization (AM-135)}.
\newblock Princeton University Press, 1994.

\bibitem[McM97]{McM97}
C. McMullen. 
\newblock Complex earthquakes and Teichm{\"u}ller theory, 
\newblock {\em J. Am. Math.
Soc.} 11 no. 2, p. 283--320, 1998.

\bibitem[Mil12]{Mil12}
J.~Milnor.
\newblock Hyperbolic components. With an appendix by A. Poirier. 
\newblock In {\em Conformal dynamics and hyperbolic geometry}, Contemp. Math., 573, 183--232, American Mathematical Society, Providence, RI, 2012.

\bibitem[Min99]{Min99}
Y.~Minsky.
\newblock The classification of punctured-torus groups.
\newblock {\em Ann. of Math. (2)}, 149:559--626, 1999.

\bibitem[Min10]{Min10}
Y.~N. Minsky.
\newblock {The classification of Kleinian surface groups I: Models and bounds}.
\newblock {\em Ann. of Math. (2)}, 171:1--107, 2010.

\bibitem[Mj14]{mahan-split}
M.~Mj.
\newblock {Cannon-Thurston maps for surface groups}.
\newblock {\em Ann. of Math. (2)}, 179:1--80, 2014.

\bibitem[Mj17]{mahan-kl}
M.~Mj.
\newblock {Cannon-Thurston Maps for Kleinian Groups}.
\newblock {\em Forum Math. Pi 5, e1, 49 pp.}, 2017.

\bibitem[MM23]{MM1}
M.~Mj and S.~Mukherjee.
\newblock Combining rational maps and {K}leinian groups via orbit equivalence.
\newblock \emph{Proc. Lond. Math. Soc (3)}, 126:1740--1809, 2023.

\bibitem[MM25]{MM2}
M.~Mj and S.~Mukherjee.
\newblock Matings, holomorphic correspondences, and a Bers slice. 
\newblock \emph{J. {\'E}c. polytech. Math.}, 12:1445--1502, 2025.

\bibitem[Mum65]{Mum65}
D.~Mumford.
\newblock {\em Geometric invariant theory}.
\newblock Ergebnisse der Mathematik und ihrer Grenzgebiete, (N.F.), Band 34
Springer-Verlag, Berlin-New York, 1965. vi+145 pp.

\bibitem[Ota01]{otal-book}
J.~P. Otal.
\newblock {The hyperbolization theorem for fibered 3-manifolds}.
\newblock {\em SMF/AMS Texts and Monographs, 7. American Mathematical Society, Providence,
  RI; Soci{\'e}t{\'e} Math{\'e}matique de France, Paris. xiv+126 pp.}, 2001.

\bibitem[Poi82]{Poi82}
H.~Poincar{\'e}.
\newblock Th{\'e}orie des groupes {F}uchsiens. 
\newblock {\em Acta Math.}, 1:1--76, 1882. 

\bibitem[Poi83]{Poi83}
H.~Poincar{\'e}.
\newblock M{\'e}moire sur les groupes {K}lein{\'e}ens.
\newblock {\em Acta Math.}, 3:49--92, 1883.

\bibitem[Pom75]{Pom75}
C.~Pommerenke.
\newblock \emph{Univalent functions}, with a chapter on quadratic differentials by Gerd Jensen.
\newblock \emph{Studia Math./Math. Lehrb{\"u}cher}, Band XXV,
Vandenhoeck \& Ruprecht, G{\"o}ttingen, 376 pp., 1975.


\bibitem[SW79]{scott-wall}
P.~Scott and T.~Wall. 
\newblock Topological methods in group theory. 
\newblock In {\em Homological group theory (Proc. Sympos., Durham, 1977)}, pp. 137--203, London Math. Soc. Lecture Note Ser., 36, Cambridge Univ. Press, Cambridge-New York, 1979.

\bibitem[Sul81]{Sul81}
D.~Sullivan.
\newblock On the ergodic theory at infinity of an arbitrary discrete group of hyperbolic motions.
\newblock In {\em Riemann surfaces and related topics: Proceedings of the 1978 Stony Brook Conference (State Univ. New York, Stony Brook, N.Y., 1978)}, pp. 465–496, {\em Ann. of Math. Stud.}, No. 97, Princeton University Press, Princeton, NJ, 1981.

\bibitem[Sul85a]{Sul85a}
D.~Sullivan.
\newblock Quasiconformal homeomorphisms and dynamics {I}: Solution of the
{F}atou-{J}ulia problem on wandering domains.
\newblock {\em Ann. of Math. (2)}, 122:401--418, 1985.

\bibitem[Sul85b]{Sul85b}
D.~Sullivan.
\newblock Quasiconformal homeomorphisms and dynamics II: Structural stability implies hyperbolicity for Kleinian groups.
\newblock {\em Acta Math.}, 155:243--260, 1985.

\bibitem[Thu86]{thurston-hypstr2}
W.~P. Thurston.
\newblock Hyperbolic {S}tructures on 3-{M}anifolds,{II}: {S}urface {G}roups and
  3-{M}anifolds which {F}iber over the {C}ircle.
\newblock {\em preprint, arXiv:math.GT/9801045}, 1986.

\bibitem[Thu22]{thurstonnotes}
W.~P. Thurston.
\newblock  The geometry and topology of three-manifolds, Vol. IV. 
\newblock Edited and with a preface by Steven P. Kerckhoff and a chapter by J. W. Milnor. American Mathematical Society, Providence, RI, 2022.

\end{thebibliography}
\end{document}